\documentclass[11pt,a4paper,oneside]{article}

\usepackage{amssymb}
\usepackage{amsmath,amsfonts,amsthm,amstext}
\usepackage{inputenc}
\usepackage{csquotes}
\usepackage{makeidx}
\usepackage{fancybox}
\usepackage{niceframe}
\usepackage{enumerate}
\usepackage{array}
\usepackage{subfigure}
\usepackage{newlfont}
\usepackage{verbatim}
\usepackage[dvips]{psfrag}
\usepackage[usenames,dvipsnames,svgnames,table]{xcolor}
\usepackage{float}
\usepackage[footnotesize]{caption}
\usepackage{indentfirst}
\usepackage[normalem]{ulem}
\usepackage{pst-text}
\usepackage{pgf}
\usepackage{wrapfig}
\usepackage{lmodern}
\usepackage[top=2.5cm,bottom=2cm,right=2cm,left=2.5cm]{geometry}
\usepackage{aliascnt}
\usepackage[colorlinks=false,linkcolor=purple]{hyperref}
\usepackage{graphicx,color}
\usepackage{epstopdf}
\usepackage{epsfig}
\usepackage{overpic}

\usepackage[backend=bibtex,style=alphabetic]{biblatex}
\allowdisplaybreaks

\input xy
\xyoption{all}

\graphicspath{{./FIG/}}

\newcommand{\s}{\Sigma}

\newcommand{\U}{\mathcal{U}}
\newcommand{\R}{\mathbb{R}}

\newcommand{\sgn}[1]{\mbox{sgn}\left( #1 \right)}
\newcommand{\0}{\mathbf{0}}
\newcommand{\e}{\varepsilon}
\newcommand{\tr}{\mathrm{tr}}

\newcommand{\G}{\mu}
\newcommand{\A}{\mathcal{A}}
\newcommand{\m}[1]{ {\scriptscriptstyle #1}}
\newcommand{\bu}{\bar{u}}
\newcommand{\bv}{\bar{v}}

\newaliascnt{defi}{definition}
\newtheorem{defi}[defi]{Definition}
\aliascntresetthe{defi}

\newaliascnt{theo}{definition}
\newtheorem{theo}[theo]{Theorem}
\aliascntresetthe{theo}

\newaliascnt{notation}{definition}
\newtheorem{notation}[notation]{Notation}
\aliascntresetthe{notation}

\newaliascnt{corol}{definition}
\newtheorem{corol}[corol]{Corolary}
\aliascntresetthe{corol}

\newaliascnt{lemma}{definition}
\newtheorem{lemma}[lemma]{Lemma}
\aliascntresetthe{lemma}

\newaliascnt{prop}{definition}
\newtheorem{prop}[prop]{Proposition}
\aliascntresetthe{prop}

\newaliascnt{rem}{definition}
\newtheorem{rem}[rem]{Remark}
\aliascntresetthe{rem}

\theoremstyle{definition}

\newtheorem{exmp}{Example}[section]

\numberwithin{equation}{section}

 \let\oldtheequation\theequation
    \makeatletter
    \def\tagform@#1{\maketag@@@{\ignorespaces#1\unskip\@@italiccorr}}
    \renewcommand{\theequation}{(\oldtheequation)}
    \makeatother
 
\makeindex   

\author{ Carles Bonet-Reves\footnote{\footnotesize BGSMATH, Dpt. Matem\`{a}tiques, Universitat Polit\`{e}cnica de Catalunya, Diagonal 647, 08028 Barcelona,Spain.} \,, 
Juliana Larrosa\footnote{ \footnotesize Departamento de Matem\'{a}tica, Universidade Federal de Santa Maria, Avenida Roraima 1000, 97195-000  Santa Maria, RS, Brasil.} , 
Tere M-Seara$^*$}

\title{Regularization around a generic codimension one fold-fold singularity}

\date{}

\begin{document}
	
	\maketitle
	
%%%%%%
% abstract
%%%%%%	
\begin{abstract}
	This paper is devoted to study the generic fold-fold singularity of
	Filippov systems on the plane,
	its unfoldings and its Sotomayor-Teixeira regularization.
	We work with general Filippov systems and provide  the bifurcation
	diagrams of the fold-fold singularity and their unfoldings, proving
	that, under some generic conditions, is a codimension one embedded
	submanifold of  the set of all Filippov systems.
	The regularization of this singularity is studied and its bifurcation
	diagram is shown. In the visible-invisible case, the use of geometric
	singular perturbation theory has been useful to give the complete diagram
	of the unfolding, specially the appearance and disappearance  of periodic
	orbits that are not present in the Filippov vector field. In the case
	of a linear regularization, we prove that the regularized system is
	equivalent to a general slow-fast system studied by Krupa and Szmolyan
	\cite{KrupaS01}.
\end{abstract}

\begin{center}
\end{center}

%%%%%%
% intro
%%%%%%

\section{Introduction}

In this paper, derived from the thesis \cite{larrosa}, we study the generic fold-fold singularity of Filippov systems on the plane, 
its unfoldings and its regularization, more concretely,  its Sotomayor-Teixeira regularization \cite{SotoTei}. 

The first part of the paper is devoted to study the fold-fold singularity. 
This singularity has been studied in \cite{Kuznetsov} and \cite{mst} by con\-si\-de\-ring some simple normal forms for the Filippov vector fields and 
their unfoldings, and also in the original book of Filippov \cite{Filippov}. 
A systematic study of the set of structurally stable Filippov vector fields was done in \cite{mst} but, besides the previously mentioned works, which study normal forms,
there does not exist a rigorous approach to the codimension one singularities. 
Our goal, realized in \autoref{thm:generic}, 
is to work with general Filippov systems and provide  the bifurcation diagrams of these singularities and their unfoldings, 
proving that the set of the fold-fold singularities, under some generic conditions, is a codimension one embedded submanifold of $\Xi_0$, the set 
of structurally stable Filippov systems.

The second part of the paper is dedicated to study  the regularization of the unfoldings of the fold-fold singularity and is a 
natural continuation of the paper \cite{tere}, where Filippov vector fields near a fold-regular point were considered.
It is known \cite{TeixeiraBuzziSilva,TeixeiraS12} that, under general conditions, in the so-called  sliding and escaping regions,
the regularized system has
a normally hyperbolic invariant manifold, attracting near the sliding region and repelling near the escaping one.
Furthermore, the flow of the regularized vector field reduced to  this invariant manifold
approaches to the Filippov flow. 
In \cite{tere}, these results were extended to visible tangency points, using asymptotic methods
following \cite{MischenkoR80}. 
The work \cite{KH15} extended these results  to $\mathbb{R}^3$ by use of blow-up methods.

The results in this work are mostly given in \cite{larrosa},  therefore the cumbersome computations are referred to it.

During the period of time of writing this paper, the work \cite{KristiansenH15}, where the authors study  this problem, came out.
In \cite{KristiansenH15} the authors perform some changes of variables to simplify the system and then 
study the normal form obtained using  blow-up methods and analyzing it in different charts (variables).
Their analytic approach completely characterizes the existence and the attracting/repelling character of the equilibrium points showing that in some
relevant cases, there is a curve in the parameter plane where the equilibrium of the system has a Hopf bifurcation
They also show that the  (sub/super critical) character of the Hopf bifurcation depends on the considered Filippov vector fields but also on the regularization
function. In fact, in formula (7.15) of that paper, the authors give an explicit formula for the Lyapunov coefficient at the Hopf bifurcation for the 
normal form system.
They also study the appearance and character of the family of periodic orbits at the Hopf bifurcation and their evolution. In the invisible-invisible case, they 
succeed in describing  the family as a smooth family of locally unique periodic orbits, that, in some cases, can undergo a saddle-node bifurcation.
In the visible-invisible case, they prove the existence of a curve in the parameter plane where a Maximal Canard occurs.
Moreover, they prove the existence of a family of locally unique ``big'' periodic orbits for parameters (exponentially) close to the Canard curve. 
The authors conjecture that the ``small'' curves near the Hopf bifurcation and the ``big'' curves near the Canard curve 
belong to the same  smooth family of locally unique periodic orbits.

The approach in our paper is mainly topological providing some new and slightly different
results which complement the ones obtained in  \cite{KristiansenH15}; 
one major goal is to give results directly checkable in a given system, for this reason we work in the original variables of the system, 
and we present its possible phase portraits. 
We use topological methods to  get the generic conditions which determine the phase portrait in terms of some intrinsic and explicit 
quantities that can be computed directly from the original system.
For this reason, although  \cite{KristiansenH15} already computed the values of the Hopf and Canard curves, we can not rely in their formulas (7.14) and (6.22)
because they are only valid for systems in normal form and we have done these computations for general vector fields in Propositions  
\ref{prop:canard}, \ref{prop:toptype}.

%The use of topological methods, in contrast with analytical ones, provides slightly different results which complement the ones obtained in \cite{KristiansenH15}.
We now present these different results and the new ones presented in this paper.

In the visible-invisible case we prove the existence of a  periodic orbit for any value of the parameters   
between the Canard and the Hopf curves in theorems \ref{thm:bdVIsuper} and \ref{thm:bdVIsub} whose stability depends on the relative position of these curves. 
We stress that this result is purely topological and follows the same kind of argument used in
the invisible-invisible case to prove the existence of a stable periodic orbit for any value of the parameters after the Hopf bifurcation  curve 
in Proposition \ref{prop:phiaeZ}.
Furthermore, in proposition \ref{pocanard} we give precise information about the region of existence of the ``big'' periodic orbit which appears close to the Canard curve, 
establishing that it exists before  the Canard curve when it is unstable and after the Canard curve when it is stable using again topological reasonings. 
Moreover, the stability of this ``big'' periodic orbit is studied and we show that, 
analogously to what happens at the Hopf bifurcation, it depends on the considered Filippov vector field but also on the regularization
function as formula  \eqref{Rsign} proves.

This topological approach does not answer the conjecture of \cite{KristiansenH15} but it gives a precise information about  
the domain of existence of the periodic  orbits in the visible-invisible case and their possible saddle-node bifurcations.

Moreover, in the visible-invisible case when the transition function is linear, 
we present some new results about the position of the curve in the parameter plane where the maximal Canard exists. 
We also provide the complete bifurcation diagram of the regularized system. 

In \cite{KristiansenH15}  a  Melnikov-based argument introduced in \cite{KrupaS01}  was  used  to continue the small  periodic orbits arising at the Hopf bifurcation. 
In section \ref{sec:melnikov} we develop the Melnikov method to compute periodic orbits and in Proposition \ref{prop:propiedadesM}  we study the properties of 
the Melnikov function  and show how this function can be used, as an alternative to the Lyapunov coefficient, to detect the subcritical/supercritical
character at the Hopf bifurcation in a given system. We also give conditions on  this function that guarantee global unicity of the periodic orbits both 
in the visible-visible  or the visible-invisible cases (see Proposition \ref{prop:phiaeZ} and Theorem \ref{thm:bdIIsuper}), 
and we show that it can be used  to compute the saddle-node bifurcations in concrete examples (see Proposition \ref{prop:melnikov}). 

Now we explain the contain of the paper.
We consider  a Filippov vector field $Z=(X,Y)$ having a fold-fold point, that we assume being at the origin $(x,y)=\0$,  
we take $Z^\alpha$ its unfolding, were $\alpha$ is the unfolding parameter,
and its regularization $Z^\alpha_\e$ (see \autoref{streg}),  where $\e$ is the regularization parameter.
Our goal is to see if the dynamics of $Z^\alpha_\e$ is equivalent, from a topological point of view, to the one of $Z^\alpha$.
The results are different depending on the fold-fold type, which can be visible-visible, invisible-invisible or visible-invisible.

As can be expected, the behavior of the regularized system 
$Z^\alpha_\e$ is similar to the one of $Z^\alpha$ if we fix $\alpha \ne 0$ and consider $\e$ small enough; if $Z^\alpha$ has a sliding zone in its switching surface 
and the sliding vector field has a pseudo equilibrium $Q(\alpha)$, then the regularized vector field $Z^\alpha_\e$ has an equilibrium $P(\alpha, \e)$ of the same type. 
Both conditions depend on the original vector field $Z=(X,Y)$ satisfying $X^1\cdot Y^1(\0)<0$. Analogously, when $Z^\alpha$ has a crossing periodic orbit, 
the regularized vector field $Z^\alpha_\e$ has a periodic orbit of the same type.

In the visible-visible case, both the unfolding $Z^\alpha$ and its regularization $Z^\alpha_\e$ have the same topological type if $\alpha$ and $\e$ are small enough:
the critical point $P(\alpha, \e)$ is a saddle point for $Z^\alpha_\e$ ($Q(\alpha)$ is a pseudo-saddle for $Z^\alpha$) and there is 
no other interesting dynamics near it.

The invisible-invisible case is more involved. 
In this case, the fold-fold  is  the so-called pseudo-focus case in the language of Filippov systems \cite{Kuznetsov}, 
and its  attracting or repelling character  can be checked studying the return map around it (c.f  \cite{MR637373}).
First, we see  that the character of the critical point $P(0, \e)$ of the regularization $Z_\e$ 
is independent of the character of the fold-fold point: $P(0, \e)$  can be a (repelling or attracting) focus or a center. 
One understands better the dynamics when  one considers the regularization of the unfoldings $Z^\alpha_\e$. 
It is known that $Z^\alpha$ has a pseudo-node $Q(\alpha)$. 
We see that $P(\alpha, \e)$ is a node  with the same character as $Q(\alpha)$ for fixed $\alpha\ne 0$ and $\e$ small enough. 
We also find a curve $\mathcal{D}$ in the parameter plane of the form $\e= C \alpha ^2+\mathcal{O}( \alpha ^3)$ where the critical point $P(\alpha, \e)$ 
becomes a focus and another curve $\mathcal{H}$ of the form $\alpha = \delta _\mathcal{H}\e+ \mathcal{O}(\e^2)$ 
where there is a Hopf bifurcation which creates
a periodic orbit $\Delta^{\alpha, *} _{\e}$ ($*=s,u$ since the orbit can be stable or unstable depending of the character of the Hopf bifurcation). 
On the other hand, it is well known (\cite{Kuznetsov}) that $Z^\alpha$ has a periodic crossing cycle $\Gamma^\alpha$ for $\alpha$ at one side of $0$. 
We can prove that for $\alpha$ and $\e$ at one side of the Hopf curve, $Z^\alpha_\e$ has a periodic orbit $\Gamma^{\alpha,*}_\e$ whose character is the 
opposite to the one of the critical point. 
Moreover, for fixed $\alpha$ and $\e$ small enough a periodic orbit $\Gamma^{\alpha,*}_\e= \Gamma ^\alpha + \mathcal{O}(\e)$ 
exists for the regularization  $Z^\alpha_\e$.
One would expect that the periodic orbit created at the Hopf bifurcation of the regularization $Z^\alpha_\e$ increases in size until it becomes 
$\Gamma^{\alpha,*}_\e$, but this is not always the case. 
Depending on the attracting/repelling character of the fold-fold given by the return map and sub/supercritical character  of the Hopf
bifurcation,  both periodic orbits can appear in a saddle-node bifurcation of periodic orbits and only the ``big one'' $\Gamma^{\alpha,*}_\e$ 
persists and becomes the cycle 
$\Gamma^\alpha$. 

In short: the periodic orbit arising from the non-smooth crossing cycle can either ``die'' at the Hopf bifurcation or coexist with the periodic orbit 
born at the Hopf bifurcation, and both die in a saddle node bifurcation of periodic orbits. 
It is important to stress that, as has been already observed in \cite{KristiansenH15}, the character of the Hopf bifurcation depends on the transition function. 

The dynamics is richer in the regularization of  the visible-invisible fold. 
The first observation is that the unfoldings of this fold are of different topological behavior depending on an intrinsic quantity of the original Filippov vector
field $Z$. 
In one case, the pseudo equilibrium $Q(\alpha)$ of the unfolding $Z^\alpha$ is a saddle point and its dynamics and the one of its  regularization $Z^\alpha_\e$ are
topologically equivalent. 
This case is similar to the visible-visible fold.

The other case is the one which presents the more interesting  dynamics. 
The pseudo-equilibrium $Q(\alpha)$ of the unfolding $Z^\alpha$ is a node and  the  behavior is similar to  the invisible-invisible case.
The critical point $P(\alpha, \e)$ is also  a node with the same character as $Q(\alpha)$ 
for fixed $\alpha \ne 0$ and $\e$ small enough. It becomes a focus when the parameters cross the parabola  $\mathcal{D}$
and suffers a Hopf bifurcation at the curve $\mathcal{H}$.

The most difficult question is to determine what happens with the periodic  orbit $\Delta^{\alpha,*} _{\e}$ that appears at the Hopf bifurcation because 
there are no periodic orbits in the unfolding $Z^\alpha$.
In order to understand this phenomenon we have to investigate the slow-fast nature of the regularized vector field $Z^\alpha_\e$ when written in scaled variables.
Using the  methods of singular perturbation theory, we have proved that this slow-fast system has a stable Fenichel manifold  
and an unstable one that coincide along a maximal Canard if  the parameters are in a curve $\mathcal{C}$ of the form 
$\alpha = \delta_\mathcal{C}\e + \mathcal{O}(\e ^{3/2})$. 
The existence of this maximal Canard creates a big periodic orbit $\Delta^{\alpha,\mathcal{C}} _{\e}$ 
(the so-called Canard explosion phenomenon, see \cite{KristiansenH15}).
Then, depending o the character of both periodic orbits, the interaction  between this ``big one'' $\Delta^{\alpha,\mathcal{C}} _{\e}$ 
with the ``small one'' $\Delta^{\alpha,*}_\e$,  emerging at the Hopf bifurcation,  creates a richer dynamics that makes the orbits 
disappear when they meet at a different saddle-node bifurcations. 
Our analysis shows that, analogously to the Hopf bifurcation, the attracting or repelling character of the periodic orbit arising at the Canard also depends
on the transition function.

The paper is organized as follows.
In \autoref{sec:revisited}, we recall the basic concepts of Filippov vector fields and the intrinsic
quantities which characterize the different types of fold-fold.  
The main result of this section is \autoref{thm:generic} where we prove that
the fold-fold singularity satisfying some generic conditions is a codimension one singularity. 
The proof of the theorem also gives the dynamics of its versal unfoldings
that will be needed in the following sections where we consider the Sotomayor-Teixeira regularization.

Section \ref{sec:regularization} considers the regularization $Z^\alpha_\e$ of an unfolding of the fold-fold singularity and the slow-fast system 
\ref{vsystem} associated to it.
The first part of this section is devoted to studying the critical points of $Z^\alpha_\e$and the  second,  to studying the critical manifolds of the slow-fast system.

Section \ref{regularizationunf} gives the dynamics of the regularized vector field. 
The section is separated in three cases, one for each type of fold.
The visible-visible case is the  simplest and is studied in \autoref{sec:VVunfreg}.

The invisible-invisible case is studied in \autoref{sec:IIunfreg}. 
The main results in this section are \autoref{prop:phiaeZ}, which prove the existence 
of the periodic orbit at one side of the Hopf bifurcation, independently of the nature of this bifurcation, and also guarantees that this orbit is near the crossing
cycle of the non-smooth system when $\alpha$ is fixed and $\e$ is small enough. 
Theorems \ref{thm:bdIIsuper} and \ref{thm:bdIIsub}  provide a complete description of the evolution of the dynamics when the parameters $(\alpha,\e)$
move around the origin. 
In particular, we observe that the regularized system may have saddle-node bifurcations of periodic orbits which do not exist in the unfolding $Z^\alpha$.
Following the ideas of \cite{KrupaS01} (see also \cite{KristiansenH15}), 
we also present in \autoref{prop:melnikov} some results about the use of a suitable Melnikov function  to give the local uniqueness of the periodic orbits 
and to compute the value of the parameters where the saddle node bifurcation takes place, if it exists. 
We conclude this section showing some examples that illustrate the behavior described in these results.

The visible-invisible case is studied in \autoref{sec:VIunfreg} and presents two different behaviors.
In \autoref{ssec:VIunfregG} we study the case that the critical point is a saddle, which is similar to the visible-visible case. 
In \autoref{ssec:VIunfregL} we analyze the case where the critical point is first a node, then  becomes a focus and finally undergoes a Hopf bifurcation.
In \autoref{prop:canard} we prove the existence of a maximal Canard in a curve of the parameter plane $(\alpha,\e)$.
Theorems \ref{thm:bdVIsuper} and \ref{thm:bdVIsub} provide the phase portrait of the system, including the behavior of the periodic orbits, depending
on the position of the Canard and the Hopf curves, as well as on the nature of the Hopf bifurcation. 
In Theorem \ref{prop:linearcanard} we see that system \ref{vsystem} can be transformed into the general slow-fast system studied in \cite{KrupaS01} 
by changes of variables if the transition function $\varphi$ is linear.
This completely determines the position of the Canard curve depending on  the sign of the Lyapunov coefficient at the Hopf bifurcation. 
Therefore, in the linear case, only the results of Theorem \ref{thm:bdVIsuper} are possible. 
These theorems are complemented with \autoref{pocanard}, where we give a formula for the coefficient which determines the stable/unstable character of the 
periodic orbit near the Canard and we show how the periodic orbit disappears in a so-called Canard explosion or 
in a saddle node bifurcation depending on its stability.
We conclude this section with some examples which illustrate the most interesting behaviors described in the section.

The aim of section \ref{sec:melnikov} is to recover the periodic orbits of moderate size of the system when $\alpha = \delta \e$ 
using classical perturbation theory after some scaling of the variable $x=\sqrt{\e} u$.
This provides the so-called Melnikov function $M(v;\delta)$, whose simple zeros give locally unique periodic orbits of the system.
Even if the existence of periodic orbits in the invisible-invisible case and in the visible-invisible case of focus type are obtained 
without using this function, it is useful to derive their uniqueness and to give a computable method to obtain the value of the parameters where 
the saddle-node bifurcations occur. 
For this reason, we think is worthwhile to dedicate a short section to this function, its properties and recover the results about periodic orbits in theorems
\ref{thm:bdIIsuper}, \ref{thm:bdIIsub}, \ref{thm:bdVIsuper} and \ref{thm:bdVIsub}.

Finally, we postpone to the Appendix the more technical proofs  of \autoref{prop:canard} in \autoref{proofcanard}, the proof of \autoref{prop:linearcanard} in 
\ref{sec:linearcanard}, and the proof of \autoref{prop:phiaeZ} in \ref{ssec:propphiaZproof}.
	
%%%%%%
% Section 1
%%%%%%

\section{Generic behavior of a Filippov system around a fold-fold singularity} \label{sec:revisited}

Let $\mathcal{Z}=\mathcal{Z}^r, \, r \geq 1$ be the set of all planar Filippov systems  defined in a bounded neighborhood $\U \subset \R^2$ of the origin, 
that is 
\begin{equation} \label{Filippov}
Z(x,y)= \begin{cases}
X(x,y), & f(x,y)>0  \\
Y(x,y), & f(x,y)<0
\end{cases},
\end{equation} 
where $X=(X^1,X^2),Y=(Y^1,Y^2), f \in \mathfrak{X}^r(\U), \,  r \geq 1$. 
As we want to study local singularities we assume $f(x,y)=y$ and that the dynamics on the discontinuity curve 
$\s=\U \cap f^{-1}(0)$ is given by the Filippov convention. 
We consider $\mathcal{Z} = \mathfrak{X}^r \times \mathfrak{X}^r$ with the product $\mathcal{C}^r$ topology.

Recall that, by the Filippov convention, as can be seen in \cite{Filippov}, the discontinuity curve is decomposed as the closure of the following regions:
\begin{eqnarray*}
	\s^c &=& \{ (x,0) \in \s: \, X^2 \cdot Y^2 (x,0)  >0 \}, \\
	\s^s &=& \{ (x,0) \in \s:\,  X^2(x,0)<0 \; \textrm{ and } \; Y^2 (x,0)  >0 \}, \\
	\s^e &=& \{ (x,0) \in \s:\,  X^2(x,0)>0 \; \textrm{ and } \; Y^2 (x,0)  <0 \}.
\end{eqnarray*} 
The flow through a point $p$ in the crossing region is the  concatenation of the flow of $X$ and $Y$ through $p$ in a consistent way. 
Over the regions $\s^{s,e}$, using $x$ as a variable in $\s$, the flow is given by the  \textit{sliding vector field}, denoted by $Z^s$ and given by 
\begin{equation} \label{slidingdef}
Z^s(x) = \frac{Y^2 \cdot X^1-X^2 \cdot Y^1}{Y^2-X^2}(x,0) = \frac{\det{Z}}{Y^2 - X^2}(x,0).
\end{equation}
where \begin{equation} \label{det} \det{Z(p)}=(X^1 \cdot Y^2 - X^2\cdot Y^1)(p), \, p \in \R^2 \end{equation}  

\begin{defi}
	The point $p =(x_p,0) \in \s^{s,e}$ is a pseudo-equilibrium of $Z$ if $Z^s(x_p)=0$ and it is a hyperbolic pseudo-equilibrium of $Z^s$, if $(Z^s)'(x_p) \neq 0$. 
	Moreover, 
	\begin{itemize}
		\item $p$ is a pseudo-node if $(Z^s)'(x_p)<0$ and $p \in \s^s$ or $(Z^s)'(x_p)>0$ and $p \in \s^e$;
		\item $p$ is a pseudo-saddle if $(Z^s)'(x_p)<0$ and $p \in \s^e$ or $(Z^s)'(x_p)>0$ and $p \in \s^s$.
	\end{itemize}
\end{defi}

It follows from \ref{slidingdef} that $(x_p,0) \in \s^{e,s}$ is a pseudo-equilibrium if, and only if, $\det{Z(x_p,0)}=0$. 
Moreover, the stability of a pseudo-equilibrium $(x_p,0) \in \s^{e,s}$ is determined by 
\begin{equation} \label{pstability}
(Z^s)'(x_p)= \frac{(\det{Z})_x}{Y^2-X^2}(x_p,0).
\end{equation}

When $p \in \s^{c,s,e}$, the vector fields $X$ and $Y$ are transverse to $\s$ at the point $p$, otherwise we have a tangency or fold point. 
In this paper we are going to deal with fold points.

%\cred{
\begin{notation}
	During this paper, given a function $h \in \mathfrak{X}^r(\U)$, we will denote its partial derivatives by $h_x= \frac{\partial h}{\partial x}$,
	$h_y= \frac{\partial h}{\partial y}$, $h_{xx}= \frac{\partial ^2 h}{\partial x^2}$, etc.
\end{notation}
%}

\begin{defi} \label{visibility}
	$p \in \s$ is a fold point of $X$ if $Xf(p)=X^2(p)=0$ and $X(Xf)(p) = X^2_x(p) \cdot X^1 (p) \neq 0$.  
	The fold is visible if $X(Xf)(p)>0$ and it is invisible if $X(Xf)(p)<0$. 
	Analogously, a fold point $p \in \s$ of $Y$ satisfies $Yf(p)=Y^2(p)=0$, and it is visible if $Y(Yf)(p)<0$ and invisible if $Y(Yf)(p)>0$. 
\end{defi}

Our purpose is to study vector fields $Z \in \mathcal{Z}$  having a fold-fold singularity, which we assume, without loss of generality, 
that is at the origin $\0=(0,0) \in \s$. 
That is, using that $f(x,y)=y$:
\begin{eqnarray} 
&&\begin{cases} Xf(\0)= X^2(\0) = 0 \\
X(Xf)(\0) = X^2_x(\0) \cdot X^1 (\0) \neq 0 
\end{cases} \label{eq:foldX} 
\\
&&\begin{cases} 
Yf(\0)=Y^2(\0)=0  \\
Y(Yf)(\0) =Y^2_x(\0) \cdot Y^1 (\0) \neq 0
\end{cases} \label{eq:foldY}
\end{eqnarray}

The fold-fold singularity has been studied in \cite{Kuznetsov} and \cite{mst} by con\-si\-de\-ring some normal forms for the Filippov vector fields and 
their unfoldings. 
In this section we present a detailed study of the bifurcation diagrams of these singularities, 
proving that the set of the fold-fold singularities, under some generic conditions, is a codimension one embedded submanifold of $\mathcal{Z}$.

Let $\Xi_0 \subset \mathcal{Z}$  the set of all locally $\s-$structurally stable Filippov systems defined on $\U$, that is, 
given  $Z \in \Xi_0 \subset \mathcal{Z}$ there exists 
a neighborhood $\U \subset \mathcal{Z}$ such that for all $\tilde{Z} \in \U$, $\tilde{Z} $ is topologically equivalent to $Z$, equivalently,   
there exists a homeomorphism $h$ which maps trajectories of 
$Z$ in trajectories of $\tilde{Z}$, preserving the regions of $\s$ and the sliding vector field (see \cite{mst}). 

\begin{defi}
	Consider $\mathcal{Z}_1 = \mathcal{Z} \setminus \Xi_0$. Let $Z, \tilde{Z} \in \mathcal{Z}_1$. We say that two unfoldings  $Z_\delta$ and $\tilde{Z}_{\tilde{\delta}}$, of $Z$ and $ \tilde{Z}$ respectively, are weak equivalent if there exists a homeomorphic 
	change of parameters $\tilde \delta= \mu(\delta)$, such that, for each $\delta$ the vector fields $Z_\delta$ and $\tilde{Z}_{\mu(\delta)}$ are locally $\s-$equivalent. 
	Moreover, given an unfolding $Z_\delta$ of $Z$ it is said to be a versal unfolding if every other unfolding $Z_{\tilde \delta}$ of 
	$Z$ is weak equivalent to $Z_\delta$.
\end{defi}

\begin{defi}\label{LambdaF}
	We define $\Lambda^F \subset \mathcal{Z}_1$ as the set of  Filippov systems which have a locally 
	$\s-$struc\-tu\-ral\-ly stable fold-fold.  
	More precisely, given $Z \in \Lambda^F$ there exists a neighborhood $\mathcal{V}_Z$ such that given $\tilde{Z} \in \mathcal{V}_Z \cap\mathcal{Z}_1$  then $Z$ 
	is locally $\s-$equivalent to $\tilde{Z}$ and their versal unfoldings are  weak equivalent.
\end{defi}

This section is devoted to prove the following theorem:

\begin{theo}
	\label{thm:generic}  
	Consider $\Lambda^F \subset \mathcal{Z}_1$ the set of all Filippov systems $Z$ which have a $\s-$structurally stable fold-fold singularity in the induced topology on 
	$ \mathcal{Z}_1$. 
	Then $Z \in \mathcal{Z}_1$ belongs to $\Lambda^F$ if and only if satisfies one of the following conditions: 
	\begin{enumerate}[(A)]
		\item 
		it is a visible-visible fold; \label{itm:A}
		\item 
		it is an invisible-invisible  fold which is a non degenerated fixed point for the ge\-ne\-ra\-li\-zed Poincar\'{e} return map. 
		See \ref{gencondii} for a precise definition; \label{itm:B} 
		\item  
		it is a visible-invisible fold and,  in the case where the sliding vector field  $Z^s(x)$ is defined, it must satisfy \label{itm:C} 
	\end{enumerate} 
	\begin{equation} 
	\label{bsliding} \gamma := Z^s(0) \neq 0.
	\end{equation}
	
	In addition, $\Lambda^F$ is a codimension one embedded submanifold of $\mathcal{Z}$.
\end{theo} 

\begin{rem}
	Theorem \ref{thm:generic} says that given a vector filed $Z$ satisfying the conditions of the Theorem if we consider unfoldings of the form:
	$$
	Z^\alpha = Z+\alpha \tilde Z+ \mathcal{O}(\alpha ^2)
	$$
	they all are equivalent if they are versal. 
	The condition for this unfoldings to be versal, roughly speaking, is that for $\alpha \ne 0$, $Z^\alpha$ has not a fold-fold singularity. 
	As we will see during the proof of this theorem this is equivalent to satisfy:
	\begin{equation}\label{versal}
	\frac{\tilde{Y}^{2}(\0)}{Y^2_x(\0)}-\frac{\tilde{X}^{2}(\0)}{X^2_x(\0)} \ne 0 .
	\end{equation}
\end{rem}

The rest of this section is devoted to prove this theorem.
As we are going to deal with local singularities, we will always work in a neighborhood of the origin without explicit mention. 

Next lemma, whose proof is straightforward, characterizes the cases where there is a region of sliding around the fold-fold point.

\begin{lemma} \label{lem:sdec} Suppose that the origin is a fold-fold point for $Z \in \mathcal{Z}$,  then:
	\begin{itemize}
		\item If the folds have the same visibility, $\s=\overline{\s^c}$ if $X^1 \cdot Y^1 (\0)<0$ and $\s=\overline{\s^e \cup \s^s}$ if $X^1 \cdot Y^1 (\0) >0$.
		\item If the folds have opposite visibility, $\s=\overline{\s^c}$ if $X^1 \cdot Y^1 (\0) >0$ and $\s=\overline{\s^e \cup \s^s}$ if $X^1 \cdot Y^1 (\0) <0$.
	\end{itemize} 
\end{lemma}

In the case that the sliding vector field \ref{slidingdef} is defined around the fold-fold $(0,0)$, by  \autoref{visibility} and \autoref{lem:sdec}, 
we have $(Y^2_x-X^2_x)(\0) \neq 0$. Therefore in this case, even if $Z^s$ is not defined at $x=0$, 
one can extend it by  the L'H\^{o}pital's rule: 
\begin{align} \label{eq:sliding}
	\gamma := Z^s(0) = \lim_{x \rightarrow 0} Z^s(x) = \frac{(\det{Z)_x}}{(Y^2_x - X^2_x)}(\0).  
\end{align} 
Thus the sliding vector field $Z^s$ is well defined at the origin.

\begin{lemma} \label{lem:neq0} 
	Suppose that the origin is a fold-fold point of $Z=(X,Y)$ then
	\begin{enumerate}[(a)]
		\item if both folds are visible, we have $(\det{Z)_x(\0)}>0$ when $X^1 \cdot Y^1 (\0)<0$ and $(\det{Z)_x(\0)}<0$ when $X^1 \cdot Y^1 (\0)>0$;
		\item if both folds are invisible, we have $(\det{Z)_x(\0)}<0$ when $X^1 \cdot Y^1 (\0)<0$ and $(\det{Z)_x(\0)} >0$ when $X^1 \cdot Y^1 (\0)>0;$
		\item if the folds have opposite visibility and $X^1 \cdot Y^1 (\0)<0$ and $Z^s$ satisfies hypothesis \ref{bsliding}, then $(\det{Z)_x(\0)} \neq 0.$ However, one can not decide, a priori, its sign. 
	\end{enumerate} 
\end{lemma}

\begin{proof} 
	It follows from \ref{eq:foldX}, \ref{eq:foldY} and the fact that $(\det{Z)_x (\0)} = (X^1 \cdot Y^2_x - Y^1 \cdot X_x^2)(\0).$
\end{proof}

\begin{corol} \label{corol:signZs} 
	If
	the sliding vector field $Z^s$ is defined around the fold-fold point, we have:
	\begin{itemize}
		\item 
		If the folds have the same visibility, then $\sgn{\gamma}=\sgn{X^1(\0)}$;
		\item 
		If the folds have opposite visibility, then $ \sgn{\gamma}=-\sgn{X^1 (\0) \cdot (\det{Z)_x}(\0)},$ 
	\end{itemize}
	where $\gamma$ is given in \ref{eq:sliding}.
\end{corol} 

\begin{corol} \label{corol:psequi} 
	Let $Z_0=(X_0,Y_0) \in \mathcal{Z}$ having a fold-fold at the origin satisfying the same hypotheses of \autoref{lem:neq0}. 
	Then there exist neighborhoods $Z_0 \in \U_0 \subset \mathcal{Z}$ and $\0 \in \mathcal{I}_0 \subset \s$ such that for each $Z \in \U_0$ 
	there exists a unique $P(Z) \in \mathcal{I}_0$ such that $\det{Z(P(Z),0)}=0$ and $\sgn{(\det{Z)_x(P(Z),0)}}=\sgn{(\det{Z_0})_x(\0)}.$  \end{corol} 
\begin{proof} 
	Let $Z_0 \in \mathcal{Z}$ satisfying the hypothesis of \autoref{lem:neq0}.  Let $\xi$ be the Frechet differentiable map 
	\begin{equation*} \begin{array}{cccc} \xi: & \mathcal{Z}  \times  \R & \rightarrow & \R \\
			& (Z,x) & \mapsto & \det{Z}(x,0)
		\end{array}
	\end{equation*}
	
	As $Z_0$ has a fold-fold at the origin, then $\xi(Z_0,0)=\det{Z_0}(\0)=0$ and by the \autoref{lem:neq0} we have $\xi_x(Z_0,0)=(\det{Z_0)_x(\0)} \neq 0$. 
	By the Implicit Function Theorem we obtain neighborhoods $Z_0 \in \U_0 \subset \mathcal{Z}$ and $0 \in \mathcal{I}_0 \subset \s$ and a Frechet 
	differentiable map $P:\U_0 \rightarrow \mathcal{I}_0$ satisfying $\xi(Z,x)=0$ if, and only if, $x=P(Z)$. 
	That is $\xi(Z,P(Z))=\det{Z}(P(Z),0)=0$ for all $Z \in \U_0$. 
	Moreover, we can assume without loss of generality that in this neighborhood we have $\sgn{(\det{Z)_x(P(Z),0)}}=\sgn{(\det{Z_0)_x(\0)}}.$
\end{proof}

In the case of the invisible fold-fold it is natural to consider the first return map (\cite{MR637373}). 
Next proposition, whose proof is straightforward and can be found in \cite{larrosa}, gives the main term of the Taylor expansion of the 
Poincar\'{e} map near a tangency point. 

\begin{prop}[Poincaré map for $X$ at a point $(x_0,y_0) \in \s_{y_0}$] \label{prop:frX} 
	Let $X$ be a smooth vector field having a fold point at 
	$p_0= (x_0,y_0) \in \s_{y_0}= \{ (x,y_0) : \, x \in \mathcal{I} \}$, where $\mathcal{I}=\mathcal{I}(y_0)$ is a neighborhood of $x_0$. 
	
	Then the Poincaré map $\phi_X^{p_0}$ is given by 
	\begin{equation}
	\begin{array}{llll}
	\phi_X^{p_0}: & \s_{y_0} & \rightarrow & \s_{y_0} \\
	&x & \mapsto & \phi_X^{p_0}(x) =2x_0 -x + \beta_X^{p_0} (x-x_0)^2 + \mathcal{O}(x-x_0)^3
	\end{array}
	\end{equation}
	where
	\begin{equation} \label{betap0}
	\beta_X^{p_0} = \frac{1}{3} \left[ -\frac{X^{2}_{xx}}{X^2_x} + 2\frac{X_x^1}{X^1} +2 \frac{X_y^2}{X^1} \right](p_0).
	\end{equation}
\end{prop}

Suppose that the vector field $Z$ has an invisible fold-fold point at $\0 \in \s$ with   
$\s=\overline{\s^c}$. Then,  it has sense to consider the first return map that, for 
convenience, we define on $\s^- = \{ (x,0), \ x \in \mathcal{I} : \, x<0 \}$  
\begin{equation} \label{gfirstreturn} 
\phi_Z: \s^- \rightarrow \s^-,
\end{equation}
by setting $p_0=\0$ and composing appropriately the Poincaré maps for $X$ and $Y$. Using Proposition \ref{prop:frX} we obtain
\begin{eqnarray}
\phi_Z(x) = \phi_Y \circ \phi_X(x)  &=& x + (\beta_Y - \beta_X) x^2 + \mathcal{O}(x^3), \mbox{ if } X_1(0)>0 \vspace{0.1cm}\label{eq:phi}\\
\phi_Z(x) = \phi_X \circ \phi_Y(x) &=& x + (\beta_X - \beta_Y) x^2 + \mathcal{O}(x^3), \mbox{ if } X_1(0)<0
\end{eqnarray}

The generic condition stated in  item (\ref{itm:B}) of  \autoref{thm:generic} for the invisible fold-fold is that 
\begin{equation} \label{gencondii} 
\G_Z = \beta_Y - \beta_X \neq 0.
\end{equation}

\begin{rem} \label{rem:pseudocycle} 
	If $X^1(\0)>0$, $\phi_Z$ is defined in $\s^-$ and is given by \eqref{eq:phi}. 
	The origin is an attractor fixed point for $\phi_Z$ if $\G_Z=\beta_{Y} - \beta_{X} >0$ 
	and it is a repellor in case  $\G_Z=\beta_{Y} - \beta_{X} <0$. Analogously for the case $X^1(\0)<0$.
	Nevertheless, it is important to stress that, as $\phi'_Z (0)=1$, the origin is never a hyperbolic fixed point of the first return map $\phi_Z$ 
	even in the generic case. 
	This will have consequences latter in \autoref{sec:regularization} when we study the regularization of the vector field $Z$.
\end{rem}

\begin{rem}
	An important detail that had not been observed in \cite{mst} and \cite{Kuznetsov} is that, even in the case $\s=\overline{\s^s \cup \s^e}$, 
	one needs to consider the first return map and impose the same generic condition \eqref{gencondii}. 
	Even though the first return map has no dynamical meaning in this case, the pseudo-cycles, which correspond to fixed points of $\phi_Z$, 
	must be preserved by $\s-$equivalences. 
	This map will be used in Section \ref{sec:IIunf} when we study the unfolding of an invisible fold-fold satisfying $X^1\cdot Y^1(\0)>0$ 
	and in this case we consider $\phi_Z = \phi_Y \circ \phi_X$ independently of the sign of $X^1(\0).$ 
\end{rem}

Now, we are able to state and prove  that conditions (\ref{itm:A}) to (C) in \autoref{thm:generic} which
characterize a codimension one embedded submanifold in $\mathcal{Z}$.

\begin{prop} \label{thm:manifold}
	The set 
	$\tilde \Lambda^F \subset \mathcal{Z}$ 
	of all Filippov systems which have a fold-fold at the origin satisfying the hypothesis 
	(\ref{itm:A}), (\ref{itm:B}) or (C) in \autoref{thm:generic} is an embedded co-dimension one subma\-ni\-fold of $\mathcal{Z}$. 
	That is, for each $Z^0 \in \tilde \Lambda^F$ there exist a map $\lambda:\mathcal{V}_0 \rightarrow \R$ where $\mathcal{V}_0  \subset \mathcal{Z}$ 
	is a neighborhood of $Z^0$ and $Z^0 \in \lambda^{-1}(0)=\mathcal{V}_0  \cap  \tilde \Lambda^F$ and $D \lambda_{Z^0} \neq 0.$
\end{prop}

\begin{proof} 
	Consider $Z^0 \in  \tilde \Lambda^F$. 
	Let $\U_0 \subset \mathcal{Z}$ be a neighborhood of $Z^0$ sufficiently small such that in this neighborhood the sign of 
	$X^1(x,0)$, $Y^1(x,0)$, $X^2_x(x,0)$ and $Y^2_x(x,0)$ is constant for $x \in \mathcal{I}_0\subset \R$. 
	Moreover, if $Z^0$ satisfies the hypothesis of \autoref{lem:neq0}, 
	suppose that $\sgn{(\det{Z)_x}(x,0)}=\sgn{(\det{Z^0)_x(\0)})}$ for all $Z \in \U_0$ and $x \in \mathcal{I}_0$.

	Consider the following Frechet differentiable map 
	\begin{equation*}
		\begin{array}{llll} 
			\xi:& \U_0 \times \R^2 & \rightarrow & \R^2 \\ & (Z,(p,q)) & \mapsto & (X^2(p,0),Y^2(q,0)) 
		\end{array}.
	\end{equation*} 
	
	Since $\0 \in \s$ is a fold-fold point $\xi(Z^0,\0)=\0$ and by \ref{eq:foldX} and \ref{eq:foldY}, 
	$$
	\det{D_{(p,q)}\xi(Z^0,\0)}= (X^0)^2_x \cdot (Y^0)^2_x (\0) \neq 0.
	$$
	Using the Implicit Function Theorem for $\xi$ there exist $\mathcal{V}_0^* \subset \U_0$ and a 
	%Frechet differentiable 
	map 
	\begin{equation} 
	\label{Tmap} T: Z=(X,Y) \in \mathcal{V}_0^* \subset \U_0 \mapsto (T_X,T_Y) \in \mathcal{I}_0 \times \mathcal{I}_0  \subset \R^2 
	\end{equation} 
	defined in a path connected open set such that $\xi(Z,(p,q))=(0,0)$ if, and only if, $(p,q)=(T_X,T_Y).$ 
	In other words, $\xi(Z,T(Z))=\0$ for every $Z \in \mathcal{V}_0^*$. 
	That is, $X$ and $Y$ have a fold point $(T_X,0)$ and $(T_Y,0)$ near the origin with the same visibility as the origin has for $X^0$ and $Y^0$.

	To show that $ \tilde \Lambda^F$ is a submanifold let consider the Frechet differentiable map 
	$$
	\begin{array}{llll} 
	\lambda_*:	& \mathcal{V}_0^* & \rightarrow 	& \R \\
	& Z 			& \mapsto 		& T_X - T_Y 
	\end{array}.
	$$
	It is clear that $Z \in \mathcal{V}_0^*$ has a fold-fold point near the origin if, and only if, $Z \in \lambda_*^{-1}(0)$. 
	Moreover, it is easy to see that when $Z^0$ satisfies items (A) or (C) then every 
	$Z \in \lambda_*^{-1}(0)$ also belongs to $ \tilde \Lambda^F$ and the fold-fold type is preserved. 
	In this case, set $\mathcal{V}_0=\mathcal{V}_0^*$. 
	
	When $Z^0$ satisfies (B) with $\G_{Z^0} \neq 0$ \ref{gencondii}, then there exists a neighborhood 
	$\tilde{\mathcal{V}}_0^1=\mathcal{V}_0^1 \cap \mathcal{Z}_1$, with $\mathcal{V}_0^1 \subset \mathcal{Z}$, 
	such that $\sgn{\G_Z}=\sgn{\G_{Z^0}}$, for all $Z \in \mathcal{V}_0^1$. 
	Therefore the fold-fold point has the same attractivity to $\phi_Z$ as the origin has to $\phi_{Z^0}$. 
	In this case set $\mathcal{V}_0=\mathcal{V}_0^* \cap \mathcal{V}_0^1$. 
	
	Consider then the map $\lambda = \lambda_* \big|_{\mathcal{V}_0}$. 
	It follows that $\lambda^{-1}(\0)=\mathcal{V}_0 \cap  \tilde \Lambda^F$. 
	To finish our proof, observe that 
	\begin{equation}\label{derivativelambda}
	D \lambda_{Z^0}(Z) = D(T_X - T_Y)_{Z^0}(Z) = \frac{X^2(\0)}{(X^0)^2_x(\0)} - \frac{Y^2(\0)}{(Y^0)^2_x(\0)},
	\end{equation}
	and therefore the resulting map is a non-vanishing linear map, what proves the desired result.
\end{proof}

The next step to finish the proof of Theorem \ref{thm:generic} is to prove that any versal unfoldings of $Z^0 \in \lambda^{-1}(0)$ are weak equivalent and consequently
$ \tilde \Lambda ^F=\Lambda ^F$. 
As we will see the behavior of the unfolding depends on the sign of $X^1 \cdot Y^1 (\0)$, but the study is completely analogous for $X^1(\0)$ positive or negative. 
Therefore, in what follows, we assume that $X^1(\0)>0$. 

Consider $\gamma: \alpha \in (-\alpha_0,\alpha_0) \mapsto Z^\alpha \in \mathcal{V}_0$ with $\alpha_0 \ll 1$, a versal unfolding of $Z^0$, 
where $\mathcal{V}_0$ is the neighborhood given in \autoref{thm:manifold}. 
Since $\gamma$ is transverse to $\lambda^{-1}(0)$ at $Z^0$, and the derivative of $\lambda$ is given in  \eqref{derivativelambda},
one can write $Z^\alpha = Z^0 + \tilde{Z} \alpha+\mathcal{O}(\alpha)^2$ with 
$$
\frac{\tilde{Y}^{2}(\0)}{(Y^0)^2_x(\0)}-\frac{\tilde{X}^{2}(\0)}{(X^0)^2_x(\0)} \ne 0 .
$$
This condition ensures that for $\alpha \ne 0$ small, the vector field $Z^\alpha$ has not a fold-fold.

Moreover, since $\lambda^{-1}(0) \subset \mathcal{V}_0$ is a codimension one embedded submanifold and $\mathcal{V}_0$ is path connected the set 
$\lambda^{-1}(0)$ splits  $\mathcal{V}_0$ in two connected components, namely, $\mathcal{V}_0^\pm = \lambda^{-1}(\R^\pm)$. 
Therefore in the sequel we suppose that $\gamma(-\alpha_0,0) \subset \mathcal{V}^-_0$ and $\gamma(0,\alpha_0) \subset \mathcal{V}^+_0$. 

By \autoref{thm:manifold}, applied to the particular curve $\gamma(\alpha)$, for each $\alpha \in (-\alpha_0,\alpha_0)$ 
there exist $T_X^{\alpha},T_Y^{\alpha} \in \s$ near the origin given by 
\begin{equation} \label{fXexpression}
T_X^{\alpha} = -\frac{\tilde{X}^2(\0)}{(X^0)^2_x(\0)} \alpha + \mathcal{O}(\alpha^2)  \ \text{and} \
T_Y^{\alpha} = -\frac{\tilde{Y}^2(\0)}{(Y^0)^2_x(\0)} \alpha + \mathcal{O}(\alpha^2).
\end{equation}

Then to assume $ T_X^{\alpha} - T_Y^{\alpha}>0$ for $\alpha>0$ is equivalent to 
\begin{equation}\label{foldsbe}
\frac{\tilde{Y}^{2}(\0)}{(Y^0)^2_x(\0)}-\frac{\tilde{X}^{2}(\0)}{(X^0)^2_x(\0)} >0 .
\end{equation}
Note that the points $(T_X^\alpha,0)$ and $(T_Y^\alpha,0)$ are the tangency points of the vector field $Z^\alpha$. 
Therefore, assumption \eqref{foldsbe} ensures that the tangency of the vector field $X^\alpha$ is on the left of 
the tangency of $Y^\alpha$ when $\alpha <0$ and otherwise when $\alpha >0$.

Once we  show that any unfolding has the same phase portrait, a systematic construction of the homeomorphism giving the topological  equivalences 
between them can be easily done using the ideas of \cite{mst}. 
Then any two unfoldings are weak equivalent. 
In conclusion, joining the result of \autoref{thm:manifold} and the following propositions \ref{prop:VVCunfolding}, \ref{prop:VVSunfolding}, \ref{prop:IICunfolding},
\ref{prop:IISunfolding} and \ref{prop:VISunfolding}, we prove \autoref{thm:generic}. In what follows, in order to avoid a huge amount of cases, we fix $X^1(\0)>0$.

%%%%%%%%%%%%%%%%%%%%%%%%%%%%%%%%%%%%%%%%%%%%%%%%%%%%%%%%%%%%%%%%%%%%%%%%%%%%%%%%%%%%%%%%%%%%%%%%%%%%%%%%%%%%%%% 
%%% VISIBLE FOLD-FOLD
%%%%%%%%%%%%%%%%%%%%%%%%%%%%%%%%%%%%%%%%%%%%%%%%%%%%%%%%%%%%%%%%%%%%%%%%%%%%%%%%%%%%%%%%%%%%%%%%%%%%%%%%%%%%%%%

\subsection{The versal unfolding of a visible fold-fold singularity} \label{sec:VVunf}

\begin{prop} \label{prop:VVCunfolding} 
	Let $Z \in \Lambda^F$ satisfying condition $(A)$ of \autoref{thm:generic} and $X^1 \cdot Y^1 (\0)>0$. 
	Let $\mathcal{V}_0$ be the neighborhood given by \autoref{thm:manifold}. 
	Then any smooth curve 
	$$
	\gamma: \alpha \in (-\alpha_0,\alpha_0) \mapsto Z^{\alpha} \in \mathcal{V}_0
	$$
	which is transverse to $\Lambda^F$ at $\gamma(0)=Z$ leads to the same topological behaviors in $\mathcal{V}^+_0$ and in $\mathcal{V}^-_0$. 
	Any  vector field $Z^{\alpha}$ has two visible fold points with a crossing region between them. 
	In the sliding and escaping regions there are no pseudo-equilibrium
	Therefore, there exists a weak equivalence between any two unfoldings of $Z$.
\end{prop}

\begin{proof} Since $X^1 \cdot Y^1(\0) >0$, $\s=\overline{\s^s \cup \s^e}$ for $\alpha=0$. 
	For $\alpha \neq 0$, we know the existence of the folds $T_X^{\alpha}$ and $T_Y^{\alpha}$ given in \ref{fXexpression}. 
	Observe that for $\alpha \ne 0$ we have $X^{\alpha 2}(x,0)<0$ if $x< T^{\alpha}_X$ and $Y^{\alpha 2}(x,0)<0$ if $x> T^{\alpha}_Y$. 
	Analogously, $X^{\alpha 2}(x,0)>0$ if $x> T^{\alpha}_X$ and $Y^{\alpha 2}(x,0)>0$ if $x< T^{\alpha}_Y$, see \autoref{VVSunf}. 
	Therefore, a crossing region appears between the folds and the discontinuity curve is decomposed as $\s = \s^s \cup \s^c \cup \s^e$, as follows: 
	\begin{equation} \label{vvdec1}
	\begin{aligned}
	\s^s &= \{ x \in \s: x < \min \{T_X^{\alpha},T_Y^{\alpha} \} \}, \\
	\s^c &= \{ x \in \s: x \in ( \min \{T_X^{\alpha},T_Y^{\alpha} \}, \max \{T_X^{\alpha},T_Y^{\alpha} \}) \}, \\
	\s^e &= \{ x \in \s : x > \max \{T_X^{\alpha} ,T_Y^{\alpha} \} \}.
	\end{aligned}
	\end{equation}
	
	Observe that
	\begin{equation}
	\det{Z^{\alpha}(T^{\alpha}_X,0)}=X^{\alpha 1} \cdot Y^{\alpha 2} (T^{\alpha}_X,0) \ \text{and} \
	\det{Z^{\alpha}(T^{\alpha}_Y,0)}=-X^{\alpha 2} \cdot Y^{\alpha 1} (T^{\alpha}_Y,0).
	\end{equation}
	Moreover, by definition  \ref{slidingdef} of the sliding vector field and using the fact of $T^\alpha_{X,Y}$ are fold points and 
	$\sgn{X^{\alpha1}(x,0)}=\sgn{Y^{\alpha1}(x,0)}>0$, we have 
	$$
	\sgn{(Z^{\alpha})^s(T^{\alpha}_X)}=\sgn{(Z^{\alpha})^s(T^{\alpha}_Y)}=\sgn{X^{\alpha 1}(\0)}>0.
	$$
	Then the sliding vector field of $Z^{\alpha}$ near $Z^0$ satisfies $\sgn{(Z^{\alpha})^s(x)}=\sgn{X^1(\0)}$.
	In particular there are no pseudo-equilibrium 
	
	By \ref{vvdec1}, for $\alpha<0$, the sliding vector field is defined for $x< T^{\alpha}_X$ and for $x>T^{\alpha}_Y$. 
	In addition, between the folds we have that $X^{\alpha2},Y^{\alpha2}>0$. 
	For $\alpha>0$, the sliding vector field is defined for $x< T^{\alpha}_Y$ and $x>T^{\alpha}_X$ and between the folds $X^{\alpha2},Y^{\alpha2}<0$.

	This proves that any unfolding of $Z \in \Lambda^F$ satisfying $(A)$ leads to vector fields with exactly the same behavior, that is, 
	the same $\s-$regions and singularities. 
	A sketch of a versal unfolding of the visible fold-fold satisfying $X^1\cdot Y^1(\0)>0$ can be seen in \autoref{VVSunf}.
\end{proof}

\begin{figure}[htb!]
	\centering
	\begin{tiny}
		\def\svgscale{0.3}
		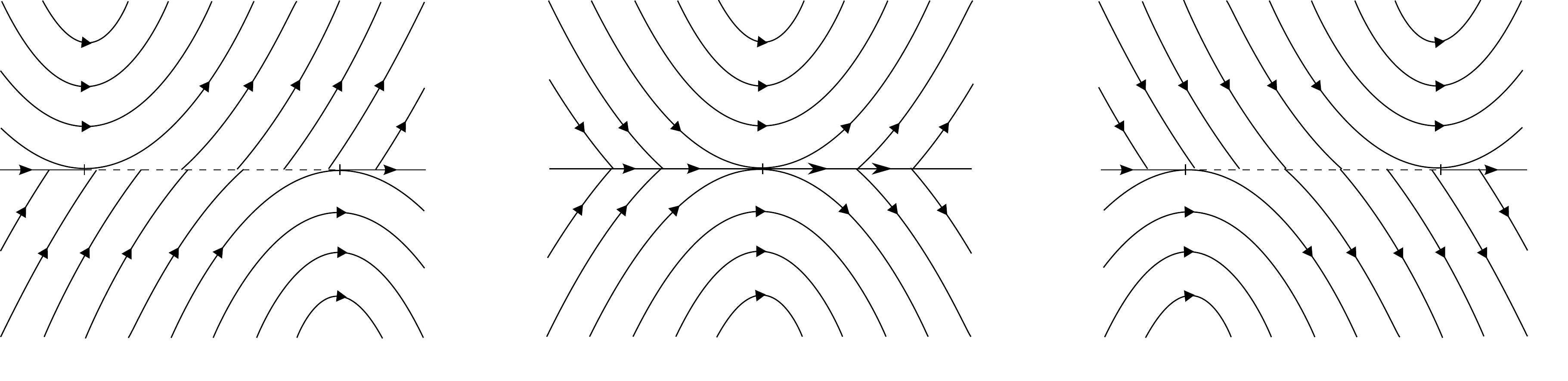
	\end{tiny}
	\caption{Versal unfolding for a visible fold-fold: $X^1 \cdot Y^1(\0)>0$}
	\label{VVSunf}
\end{figure}

\begin{prop} \label{prop:VVSunfolding} 
	Let $Z \in \Lambda^F$ satisfying condition $(A)$ of \autoref{thm:generic} and $X^1 \cdot Y^1 (\0)<0$. 
	Let $\mathcal{V}_0$ be the neighborhood given by \autoref{thm:manifold}. 
	Then any smooth curve 
	$$
	\gamma: \alpha \in (-\alpha_0,\alpha_0) \mapsto Z^{\alpha} \in \mathcal{V}_0
	$$
	which is transverse to $\Lambda^F$ at $\gamma(0)=Z$ leads to the same behaviors in $\mathcal{V}^+_0$ and in $\mathcal{V}^-_0$.  
	If $Z^\alpha \in \mathcal{V}_0^-$ ($Z^\alpha \in \mathcal{V}_0^+$), 
	it has two visible fold points with a escaping (sliding) region  between them, whose sliding vector field has a pseudo-saddle $Q(\alpha)=(x(\alpha),0)$. 
	Therefore, there exists a weak equivalence between any two unfoldings of $Z$.
\end{prop}

\begin{proof} 
	Since $X^1 \cdot Y^1(\0) <0$, $\s=\overline{\s^c}$ for $\alpha=0$. 
	For $\alpha \neq 0$, we know the existence of the folds $T_X^{\alpha}$ and $T_Y^{\alpha}$ given in  \ref{fXexpression}.
	%\ref{Tmap}. 
	Observe that $X^{\alpha 2}(x,0)<0$ if $x< T^{\alpha}_X$ and $Y^{\alpha 2}(x,0)>0$ if $x> T^{\alpha}_Y$. 
	Analogously, $X^{\alpha 2}(x,0)>0$ if $x> T^{\alpha}_X$ and $Y^{\alpha 2}(x,0)<0$ if $x< T^{\alpha}_Y$, see \autoref{VVCunf}. 
	Therefore, a piece of sliding ($\alpha >0$) or escaping  ($\alpha <0$) region appear between the folds.
	\begin{eqnarray} \label{vvdec2}
	\s^c &=& \{ x \in \s: x < \min \{T_X^{\alpha},T_Y^{\alpha} \} \} \cup \{ x \in \s: x > \max \{T_X^{\alpha},T_Y^{\alpha} \} \}  \\
	\s \setminus \s^c, &=& \begin{cases}
	\s^e =& \{ x \in \s: x \in  (T_X^{\alpha},T_Y^{\alpha}) \} \}, \, \alpha<0, \\
	\s^s =& \{ x \in \s: x \in  (T_Y^{\alpha},T_X^{\alpha}) \} \}, \, \alpha>0 .
	\end{cases}
	\end{eqnarray}

	Observe that $\det{Z^{\alpha} (T_X^{\alpha},0)} \cdot \det{Z^{\alpha}(T_Y^{\alpha},0)}<0$ for $\alpha \neq 0$. 
	Then there exists a point $x(\alpha) \in \s^{e,s}$ such $\det{Z^{\alpha}}(x(\alpha),0)=0$. 
	Moreover, by \autoref{lem:neq0} we know that $(\det{Z^0})_x(\0)>0$ and therefore by \autoref{corol:psequi}, 
	for $\alpha$ small enough $x(\alpha)$ is unique. Therefore $Q(\alpha)=(x(\alpha),0)$  is a pseudo-equilibrium of $(Z^{\alpha})^s$. 
	Moreover, by \ref{pstability} 
	\begin{eqnarray*}
		((Z^{\alpha})^s)'(x(\alpha)) 	&=& \frac{(\det{Z^{\alpha})_x}(Q(\alpha))}{(Y^{\alpha 2} -X^{\alpha 2})(Q(\alpha))}.
	\end{eqnarray*}
	As
	$(\det{Z^{\alpha})_x}(Q(\alpha))>0$ for $|\alpha| \ll 1$,  
	$((Z^{\alpha})^s)'(x(\alpha))$ is positive if $\alpha>0$ and it is negative if $\alpha<0$.  
	This fact implies that the point $Q(\alpha)$ is a pseudo-saddle of the sliding vector field. 
	
	This proves that any unfolding of $Z \in \Lambda^F$ satisfying $(A)$ with $X^1\cdot Y^1(\0)<0$ leads to vector fields with exactly the same topological invariants, 
	see \autoref{VVCunf}.
\end{proof}

\begin{figure}[!htb]
	\centering
	\begin{tiny}
		\def\svgscale{0.3}
		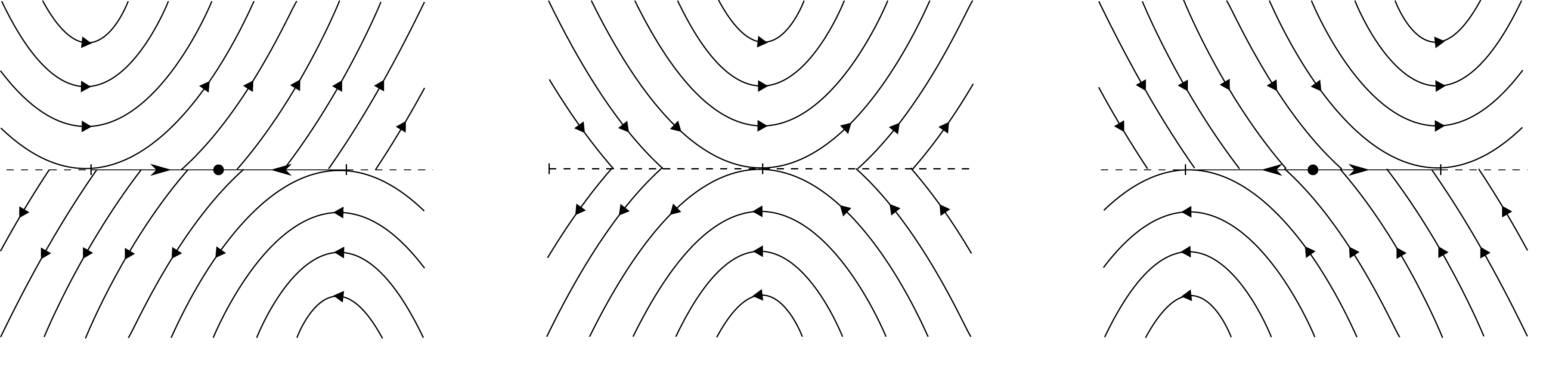
	\end{tiny}
	\caption{Versal unfolding for a visible fold-fold: $X^1 \cdot Y^1(\0)<0$}
	\label{VVCunf}
\end{figure}

%%%%%%%%%%%%%%%%%%%%%%%%%%%%%%%%%%%%%%%%%%%%%%%%%%%%%%%%%%%%%%%%%%%%%%%%%%%%%%%%%%%%%%%%%%%%%%%%%%%%%%%%%%%%%%% 
%%% INVISIBLE FOLD-FOLD
%%%%%%%%%%%%%%%%%%%%%%%%%%%%%%%%%%%%%%%%%%%%%%%%%%%%%%%%%%%%%%%%%%%%%%%%%%%%%%%%%%%%%%%%%%%%%%%%%%%%%%%%%%%%%%%
\subsection{The versal unfolding of a invisible fold-fold singularity} \label{sec:IIunf}

To study the unfoldings of a Filippov vector field $Z$ having an invisible fold-fold, 
we need to consider the generalized first return map \ref{gfirstreturn} around the fold-fold point. 
We will have four different types of bifurcations depending on the sign of $X^1 \cdot Y^1 (\0)$ and the attracting or repelling character of the return 
map. The case where $\s=\overline{\s^c}$ is the so called pseudo-Hopf bifurcation and it was studied in \cite{Kuznetsov} and \cite{mst} and it is a 
generic codimension one bifurcation if $\G_Z \neq 0$ (see \ref{gencondii}).

\begin{prop} \label{prop:IICunfolding} 
	Let $Z \in \Lambda^F$ satisfying condition $(B)$ of \autoref{thm:generic}, $X^1 \cdot Y^1 (\0)<0$ and $\G_Z \neq 0$ (see \ref{gencondii}). 
	Let $\mathcal{V}_0$ be the neighborhood given in \autoref{thm:manifold}. 
	Then any smooth curve 
	$$
	\gamma: \alpha \in (-\alpha_0,\alpha_0) \mapsto Z^{\alpha} \in \mathcal{V}_0
	$$
	which is transverse to $\Lambda^F$ at $\gamma(0)=Z$ leads to the same topological behavior in $\mathcal{V}^+_0$ and in $\mathcal{V}^-_0$.
	This behavior depends of the sign of $\G_Z$:
	\begin{enumerate}
		\item 
		If $\G_Z>0$:
		\begin{itemize}
			\item 
			Every  $Z \in \mathcal{V}^-_0$ has two invisible fold points and there  exists a region of sliding between them. 
			The sliding vector field has a stable pseudo-node  $Q(\alpha)= (x(\alpha),0)\in \Sigma^s$ which is  a global attractor.
			\item 
			Every  $Z \in \mathcal{V}^+_0$ has two invisible fold points and there  exists a region of escaping between them. 
			The sliding vector field has a unstable pseudo-node  $Q(\alpha)= (x(\alpha),0)\in \Sigma^e$ and there exists a crossing 
			stable periodic orbit $\Gamma^\alpha$ which is a global attractor.
		\end{itemize}
		\item 
		If $\G_Z<0$:
		\begin{itemize}
			\item 
			Every  $Z \in \mathcal{V}^-_0$ has two invisible fold points and there  exists a region of sliding between them. 
			The sliding vector field has a stable pseudo-node $Q(\alpha)= (x(\alpha),0)\in \Sigma^s$ and there exists a crossing 
			unstable periodic $\Gamma^\alpha$ orbit which is a global repellor.
			\item 
			Every  $Z \in \mathcal{V}^+_0$ has two invisible fold points and there  exists a region of escaping between them. 
			The sliding vector field has a unstable pseudo-node  $Q(\alpha)= (x(\alpha),0)\in \Sigma^e$ which is a global repellor.
		\end{itemize}
	\end{enumerate} 
	Therefore, there exists a weak equivalence between any two unfoldings of $Z$.
\end{prop}

\begin{proof}
	Since $X^1 \cdot Y^1(\0) <0$, $\s=\overline{\s^c}$ for $\alpha=0$. 
	For $\alpha \neq 0$, for all the  points $(x,0)$ between the folds $T_X^{\alpha}$ and $T^{\alpha}_Y$ the vector field $Z^{\alpha}$ satisfies 
	$X^{\alpha 2} \cdot Y^{\alpha 2}(x,0)<0$. 
	Therefore, a piece of sliding (for $\alpha <0$) or escaping  (for $\alpha >0$) region appears between the folds. 
	The discontinuity curve becomes $\s=\overline{\s^{e,s} \cup \s^c}$, where 
	\begin{eqnarray} \label{iidec}
	\s^c &=& \{ x \in \s: x < \min \{T_X^{\alpha},T_Y^{\alpha} \} \} \cup \{ x \in \s: x > \max \{T_X^{\alpha},T_Y^{\alpha} \} \}  \\
	\s \setminus \s^c, &=& \begin{cases}
	\s^s =& \{ x \in \s: x \in  (T_X^{\alpha},T_Y^{\alpha}) \} \}, \, \alpha<0, \\
	\s^e =& \{ x \in \s: x \in  (T_Y^{\alpha},T_X^{\alpha}) \} \}, \, \alpha>0 .
	\end{cases}
	\end{eqnarray}
	
	Since we have sliding motion defined on one side of the tangencies and crossing on the other, following \cite{mst}, the fold points are singular tangency points. 
	Using the same argument as in \autoref{prop:VVSunfolding}, there exists a unique $Q(\alpha) \in \s^{e,s}$ 
	such that $\det{Z^{\alpha}(Q(\alpha))}=0.$  
	By \autoref{lem:neq0}, we have $(\det{Z})_x(\0)<0$. Using the formulas for $(Z^s)'(x(\alpha))$  given in \ref{pstability} 
	the pseudo-equilibrium $Q(\alpha)$ is an stable pseudo-node when $\alpha<0$ and it is a unstable pseudo-node when $\alpha>0$.

	To give a complete description of the dynamics one needs to analyze the first return map around the fold-fold singularity: 
	\begin{equation*} \begin{array}{llll}
			\phi^{\alpha}: 	& \mathcal{D}^{\alpha}	& \rightarrow 	& \mathcal{I}^{\alpha} \\
			& x						& \mapsto		& \phi^{\alpha}(x) =\phi_Y^{\alpha} \circ \phi_{X}^{\alpha}(x)
		\end{array}.
	\end{equation*} 
	where 
	$$
	\begin{array}{rcl}
	\mathcal{D}^{\alpha}&=&\{ x \in \s: \, x < (\phi^{\alpha}_X)^{-1}(T^{\alpha}_Y) \}, \ 
	\mathcal{I}^{\alpha}= \{ x \in \s :\, x<T^{\alpha}_Y \}, \ \mbox{ if } \ \alpha<0 \ \text{and}\\ 
	\mathcal{D}^{\alpha}&=&\{ x \in \s: \, x < T^{\alpha}_X \}, \  
	\mathcal{I}^{\alpha}=\{ x \in \s: \, x < (\phi^{\alpha}_Y)(T^{\alpha}_X) \}, \ \mbox{ if }\ \alpha>0.
	\end{array}
	$$
	Using the expressions given for $\phi_X^\alpha$ and $\phi_Y^\alpha$ in \autoref{prop:frX} applied to $X^\alpha$ and $Y^\alpha$, the return map is given by 
	\begin{equation} \label{FRunfolding}
	\phi^{\alpha}(x)=2(T^{\alpha}_Y-T^{\alpha}_X) + x  
	-\beta_X (x-T_X^\alpha)^2+ 
	\beta_Y (2T_X^\alpha-T_Y^\alpha-x)^2 +O_3(x,T_X^\alpha,T_Y^\alpha)
	\end{equation}
	where $T_X^\alpha$ and $T_Y^\alpha$ are given  in \ref{fXexpression}. 
	\begin{figure}[h]
		\centering
		\begin{tiny}
			\def\svgscale{0.35}
			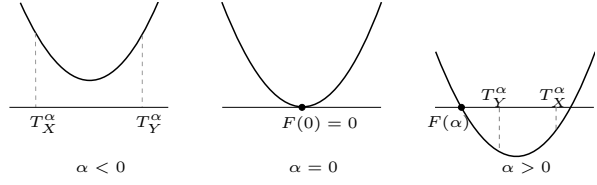
		\end{tiny}
		\caption{The graphic of $\Phi(\alpha,x)=\phi^{\alpha}(x) -x$ when $\G_Z>0$, for different values of the parameter $\alpha$.}
		\label{fig:IIreturnmap}
	\end{figure}
	
	To look for periodic orbits near the fold-fold point, we look for zeros of the
	the auxiliary map 
	\begin{equation} \label{Phiaux}
	\Phi: (\alpha,x) \in \mathcal{W} \subset (-\alpha_0,\alpha_0) \times \R \mapsto \Phi(\alpha,x)= \phi^{\alpha}(x) - x \in \R.
	\end{equation}
	
	The map $\Phi$ satisfies $\Phi(0,0)=0$, $\displaystyle{\frac{\partial}{\partial x} \Phi(0,0)=0}$ and  
	$\displaystyle{\frac{\partial^2}{\partial x^2} \Phi(0,0)=2\G_{Z} \neq 0}$. 
	Then, by the Implicit Function Theorem, for each $\alpha$ sufficiently small there exists a unique  $C(\alpha)$ near 
	$0 \in \s$ such that $\frac{\partial}{\partial x} \Phi(\alpha,C(\alpha))=0$. 
	Moreover, as $T_X^\alpha, T_Y^\alpha = O(\alpha)$ also $C(\alpha)= O(\alpha)$. 
	Thus the map $\Phi(\alpha, x)$ has a critical point at $C(\alpha)$ which is a maximum or minimum depending on the sign of $\G_Z$. 
	
	If $\G_Z>0$ (see \autoref{fig:IIreturnmap}) $C(\alpha)$ is a local minimum of $\Phi(\alpha,.)$. 
	If $\alpha<0$, then 
	$\Phi(\alpha,C(\alpha))=2(T^{\alpha}_Y-T^{\alpha}_X)+ O(\alpha ^2)>0$, being $C(\alpha)$ a minimum this means that $\Phi(\alpha,x)>0$, for $x \in \s$. 
	Therefore there are  no fixed points for $\phi^\alpha $ if $\alpha<0$. 
	If $\alpha>0$ we obtain $\Phi(\alpha,C(\alpha))<0$ and therefore $\Phi(\alpha,x)$ has two zeros. 
	Moreover, $\Phi(\alpha,T_Y ^\alpha)<0$, and we call $F(\alpha) \in \mathcal{D}^\alpha$ the zero of $\Phi$ satisfying $F(\alpha)< T^\alpha_Y$. 
	Therefore, the map $\phi^{\alpha}$ has a fixed point $F(\alpha)$ which corresponds to an attracting crossing cycle $\Gamma^\alpha$, 
	since 
	$$
	\frac{\partial}{\partial x} \phi^{\alpha}(F(\alpha))<1.
	$$
	
	Summarizing the case $\G_Z>0$: for $\alpha<0$ the vector field $Z^{\alpha}$ has an stable pseudo-node $Q(\alpha)$ and no crossing cycles exist. 
	When $\alpha>0$, the point $Q(\alpha)$ is a unstable pseudo-node and an attracting crossing cycle $\Gamma^\alpha$ through the point $(F(\alpha),0)$ appears. 
	Using that $T^\alpha_X$ and $T^\alpha_Y$ are $O(\alpha)$ one can compute 
	\begin{equation} \label{Falpha}
	F(\alpha) =\left(-\displaystyle{\sqrt{\frac{2(T^{\alpha}_X-T^{\alpha}_Y)}{\G_Z}} + \mathcal{O}(\alpha)},0\right).
	\end{equation}   
	
	When $\G_Z<0$ the point  $C(\alpha)$ is a local maximum of $\Phi(\alpha,x)$. 
	Thus an repelling crossing cycle exists for $\alpha<0$ and no crossing cycles appear for $\alpha>0$. 
	The nature of the pseudo-equilibrium $Q(\alpha)$ remains the same as in the case $\G_Z<0$, since its stability does not depend on $\G_Z$.
	
	Then any unfolding of $Z \in \Lambda^F$ satisfying $(B)$, $X^1 \cdot Y^1(\0)<0$ with $\G_Z \neq 0$ leads to vector fields with exactly the same topological invariants, 
	that is, the same $\s-$regions and singularities; see Figures \ref{fig:IIunfa} and \ref{fig:IIunfr}.
\end{proof}

\begin{figure}[h]
	\centering
	\begin{tiny}
		\def\svgscale{0.3}
		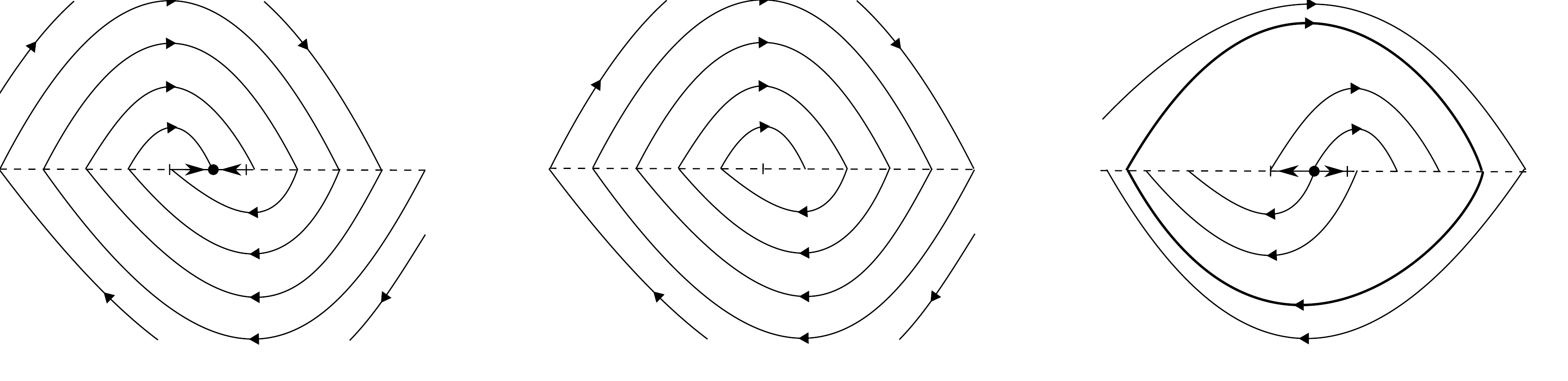
	\end{tiny}
	\caption{The unfolding of $Z \in \Lambda^F$ satisfying $(B)$, $X^1 \cdot Y^1(\0)<0$ and $\G_Z>0$.}
	\label{fig:IIunfa}
\end{figure}

\begin{figure}[h]
	\centering
	\begin{tiny}
		\def\svgscale{0.3}
		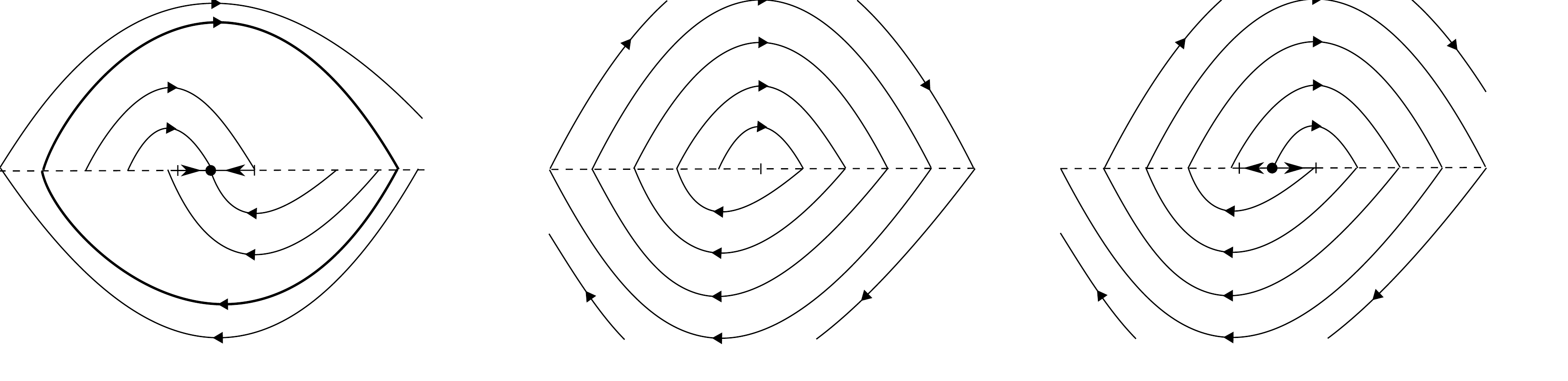
	\end{tiny}
	\caption{The unfolding of $Z \in \Lambda^F$ satisfying satisfying $(B)$, $X^1 \cdot Y^1(\0)<0$ and $\G_Z<0$.}
	\label{fig:IIunfr}
\end{figure}

\begin{prop} \label{prop:IISunfolding} 
	Let $Z \in \Lambda^F$ satisfying condition $(B)$ of \autoref{thm:generic}, $X^1 \cdot Y^1 (\0)>0$ and $\G_Z \neq 0$. 
	Let $\mathcal{V}_0$ be the neighborhood given in \autoref{thm:manifold}. 
	Then any smooth curve 
	$$
	\gamma: \alpha \in (-\alpha_0,\alpha_0) \mapsto Z^{\alpha} \in \mathcal{V}_0
	$$
	which is transverse to $\Lambda^F$ at $\gamma(0)=Z$ leads to the same behaviors in $\mathcal{V}^+_0$ and in $\mathcal{V}^-_0$. 
	For any $Z \in \mathcal{V}_0^\pm$ has two invisible folds with a crossing region between them. 
	In both cases, the sliding vector field has no pseudo-equilibrium Moreover, in the case $\G_Z>0$ then $Z \in \mathcal{V}^+_0$ has an pseudo-cycle and when $\G_Z<0$ then $Z \in \mathcal{V}^-_0$ has a pseudo-cycle. Therefore, there exists a weak equivalence between any two unfoldings of $Z$.
\end{prop}

\begin{proof} 
	Since $X^1 \cdot Y^1(\0)>0$, $\s=\overline{\s^s \cup \s^e}$ for $\alpha=0$. 
	For $\alpha \neq 0$, since for all points $(x,0)$ between the folds $T_X^{\alpha}$ and $T^{\alpha}_Y$ 
	the vector field $Z^{\alpha}$ satisfies $X^{\alpha 2} \cdot Y^{\alpha 2}(x,0)>0$. 
	Therefore, a crossing region appears between the folds. 
	The discontinuity curve becomes $\s=\overline{\s^{e} \cup \s^c \cup \s^s}$, where \begin{eqnarray} \label{iidec1}
	\s^e &=& \{ x \in \s: x < \min \{T_X^{\alpha},T_Y^{\alpha} \} \},  \\
	\s^c &=& \{ x \in \s: x \in  \left( \min \{T_X^{\alpha},T_Y^{\alpha} \},\max \{T_X^{\alpha},T_Y^{\alpha} \} \right) \}, \\
	\s^s &=& \{ x \in \s: x > \max \{T_X^{\alpha},T_Y^{\alpha} \} \}. 
	\end{eqnarray}
	Using the same argument as in \autoref{prop:VVCunfolding}, no pseudo-equilibrium appears for $\alpha \neq 0$. 
	As mentioned in \autoref{rem:pseudocycle}, one needs to consider the generalized first return map for this case. 
	
	For convenience, we set  
	$$
	\phi^{\alpha}_Z = \phi_Y^{\alpha} \circ \phi_X^{\alpha}=2(T^{\alpha}_Y-T^{\alpha}_X) + x 
	-\beta_X (x-T_X^\alpha)^2+ 
	\beta_Y (2T_X^\alpha-T_Y^\alpha-x)^2 +O_3(x,T_X^\alpha,T_Y^\alpha)
	$$ 
	By the same arguments of \autoref{prop:IICunfolding} a pseudo-cycle appears for $\alpha>0$ when $\G_Z>0$ and it appears for $\alpha<0$ when $\G_Z<0.$  
	
	Then any unfolding of $Z \in \Lambda^F$ satisfying $(B)$, $X^1 \cdot Y^1(\0)>0$ with $\G_Z \neq 0$ 
	leads to vector fields with exactly the same topological invariants, that is, the same $\s-$regions and singularities; 
	see Figures \ref{fig:IISunfa} and \ref{fig:IISunfr}.
\end{proof}

\begin{figure}[h]
	\centering
	\begin{tiny}
		\def\svgscale{0.3}
		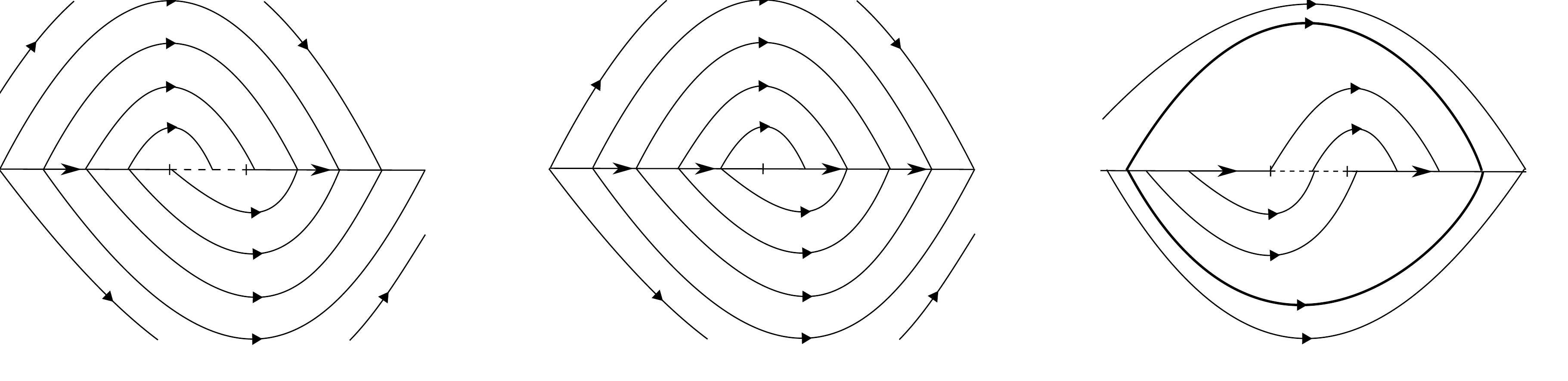
	\end{tiny}
	\caption{The unfolding of $Z \in \Lambda^F$ satisfying $(B)$, $X^1 \cdot Y^1(\0)>0$ and $\G_Z<0$.}
	\label{fig:IISunfa}
\end{figure}

\begin{figure}[h]
	\centering
	\begin{tiny}
		\def\svgscale{0.3}
		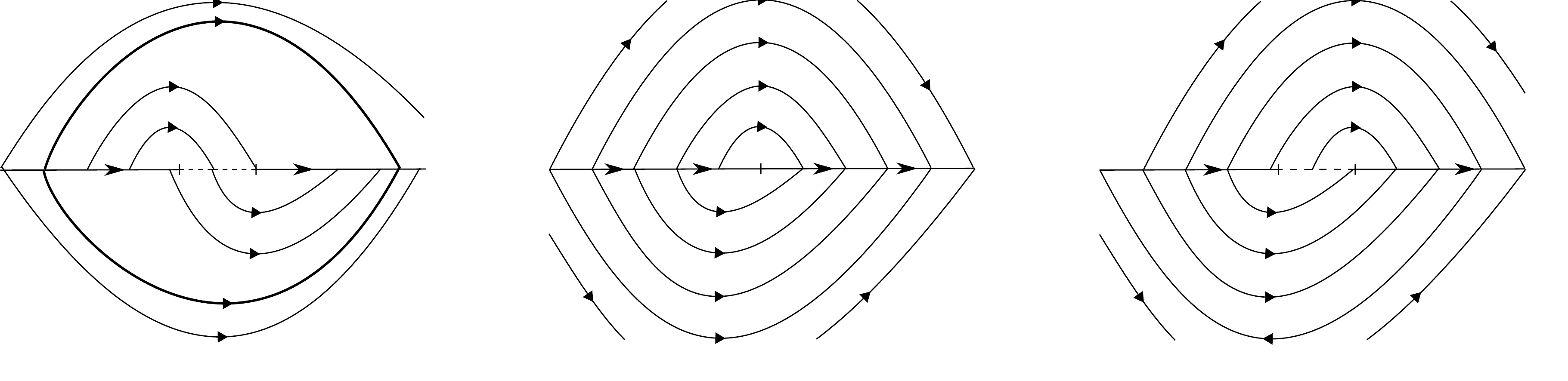
	\end{tiny}
	\caption{The unfolding of $Z \in \Lambda^F$ satisfying satisfying $(B)$, $X^1 \cdot Y^1(\0)>0$ and $\G_Z>0$.}
	\label{fig:IISunfr}
\end{figure}

%%%%%%%%%%%%%%%%%%%%%%%%%%%%%%%%%%%%%%%%%%%%%%%%%%%%%%%%%%%%%%%%%%%%%%%%%%%%%%%%%%%%%%%%%%%%%%%%%%%%%%%%%%%%%%% 
%%% VISIBLE-INVISIBLE FOLD
%%%%%%%%%%%%%%%%%%%%%%%%%%%%%%%%%%%%%%%%%%%%%%%%%%%%%%%%%%%%%%%%%%%%%%%%%%%%%%%%%%%%%%%%%%%%%%%%%%%%%%%%%%%%%%%
\subsection{The versal unfolding of a visible-invisible fold-fold singularity} \label{sec:VIunf}

This section is devoted to the study of the unfoldings of a Filippov vector field having a visible-invisible fold point. 
We have essentially three different bifurcations, two of them occur when the vector fields $X$ and $Y$ point at opposite directions at the fold-fold point and differ in the sign of $(\det{Z})_x(\0)$. 
The third occurs  when both vector fields point to the same direction. 

\begin{prop} \label{prop:VISunfolding} 
	Let $Z \in \Lambda^F$ satisfying condition $(C)$ of \autoref{thm:generic} and $X^1 \cdot Y^1 (\0)<0$. 
	Let $\mathcal{V}_0$ be the neighborhood given in \autoref{thm:manifold}. 
	Then any smooth curve 
	$$
	\gamma: \alpha \in (-\alpha_0,\alpha_0) \mapsto Z^{\alpha} \in \mathcal{V}_0
	$$
	which is transverse to $\Lambda^F$ at $\gamma(0)=Z$ leads to the same behaviors in $\mathcal{V}^+_0$ and in $\mathcal{V}^-_0$.

	This behavior depends of the sign of $(\det{Z})_x(\0)$:
	\begin{enumerate}
		\item 
		If $(\det{Z})_x(\0)>0$:
		\begin{itemize}
			\item 
			Every  $Z \in \mathcal{V}^-_0$ has a visible and an invisible fold points and there  exists a crossing region between them. 
			The sliding vector field has a pseudo-saddle $Q(\alpha)= (x(\alpha),0)\in \s^s$  situated on ``left'' of both folds.
			\item 
			Every  $Z \in \mathcal{V}^+_0$ has a visible and an invisible fold points and there  exists a crossing region between them. 
			The sliding vector field has a pseudo-saddle  $Q(\alpha)= (x(\alpha),0)\in \s^e$ situated on the ``right'' of both folds.
		\end{itemize}
		\item 
		If $(\det{Z})_x(\0)<0$:
		\begin{itemize}
			\item 
			Every  $Z \in \mathcal{V}^-_0$ has a visible and an invisible fold points and there  exists a crossing region between them. 
			The sliding vector field has a unstable pseudo-node $Q(\alpha)= (x(\alpha),0)\in \s^e$ situated on the ``right'' of both folds.
			\item 
			Every  $Z \in \mathcal{V}^+_0$ has a visible and an invisible fold points and there  exists a crossing region between them. 
			The sliding vector field has a stable pseudo-node $Q(\alpha)= (x(\alpha),0)\in \s^s$ situated on ``left'' of both folds.
		\end{itemize}
	\end{enumerate} 
	Therefore, there exists a weak equivalence between any two unfoldings of $Z$. 
\end{prop}

\begin{proof}
	Since $X^1 \cdot Y^1(\0) <0$, $\s=\overline{\s^s \cup \s^e}$ for $\alpha=0$. 
	For $\alpha \neq 0$, for all points $(x,0)$ between the folds $T_X^{\alpha}$ and $T^{\alpha}_Y$ the vector field $Z^{\alpha}$ 
	satisfies $X^{\alpha 2} \cdot Y^{\alpha 2}(x,0)>0$. 
	Therefore, a piece of crossing region appears between the folds for $\alpha \neq 0$. 
	The discontinuity curve becomes $\s=\overline{\s^s \cup \s^c \cup \s^e}$, where 
	\begin{equation} \label{videc}
	\begin{aligned}
	\s^s &= \{ x \in \s: x < \min \{T_X^{\alpha},T_Y^{\alpha} \} \}, \\
	\s^c &= \{ x \in \s: x \in ( \min \{T_X^{\alpha},T_Y^{\alpha} \}, \max \{T_X^{\alpha},T_Y^{\alpha} \}) \}, \\
	\s^e &= \{ x \in \s : x > \max \{T_X^{\alpha} ,T_Y^{\alpha} \} \}.
	\end{aligned}
	\end{equation}
	
	By condition $(C)$, $\gamma \neq 0$ (see \ref{eq:sliding}) implying $(\det{Z})_x(\0) \neq 0$, 
	then \autoref{corol:psequi} guarantees the existence of a unique point $Q(\alpha) \in \s$ such that $\det{Z^{\alpha}}(Q(\alpha))=0.$ 
	To check if $Q(\alpha)$ belongs to $\s^{s,e}$ and therefore it is a pseudo-equilibrium one must analyze separately the cases $(\det{Z})_x(\0)>0$ and $(\det{Z})_x(\0)<0$.
	
	\begin{figure}[h]
		\centering
		\begin{tiny}
			\def\svgscale{0.4}
			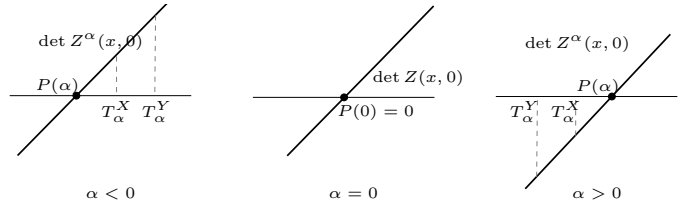
		\end{tiny}
		\caption{The curve $\det{Z^{\alpha}}(x,0)$ when $(\det{Z})_x(\0)$ for each $\alpha$.}
		\label{fig:VIdetposition}
	\end{figure}
	
	Suppose that $(\det{Z})_x(\0)>0$, so for $|\alpha| << 1$ we have $(\det{Z)_x^{\alpha}}(x,0)>0$ for $x \in \s$. 
	Therefore, the function $\det{Z^{\alpha}}: x \in \s \mapsto \det{Z^{\alpha}}(x,0) \in \R$ is increasing, see \autoref{fig:VIdetposition}. 
	Computing the values of $\det{Z^{\alpha}}(T^{\alpha}_X,0)$ and $\det{Z^{\alpha}}(T^{\alpha}_Y,0)$ we conclude that the pseudo-equilibrium 
	$Q(\alpha)$ belongs to $\s^e$ for $\alpha>0$ and it belongs to $\s^s$ when $\alpha<0$.  
	In addition, by \ref{pstability} we have that $(Z^s)'(x(\alpha))>0$ if $\alpha<0$ and $(Z^s)'(x(\alpha))<0$ if $\alpha>0$ and therefore   the point $Q(\alpha)$ is a pseudo-saddle. 
	
	When $(\det{Z)_x}(\0)<0$, the map $\det{Z(x,0)}$ is decreasing. Therefore, by the same argument the pseudo-equilibrium belongs to $\s^s$ if $\alpha>0$ and to $\s^e$ if $\alpha<0$. In this case, computing $(Z^s)'(x(\alpha))$ the point $Q(\alpha)$ is a pseudo-node.
	
	Then any unfolding of $Z \in \Lambda^F$ satisfying $(C)$ with $X^1 \cdot Y^1(\0)<0$ leads to vector fields with exactly the same topological invariants, 
	that is, the same $\s-$regions and singularities, depending on the sign of $(\det{Z})_x(\0)$, see Figures \ref{fig:VISunf+} and \ref{fig:VISunf-}.
	
\end{proof}

\begin{figure}[h]
	\centering
	\begin{tiny}
		\def\svgscale{0.3}
		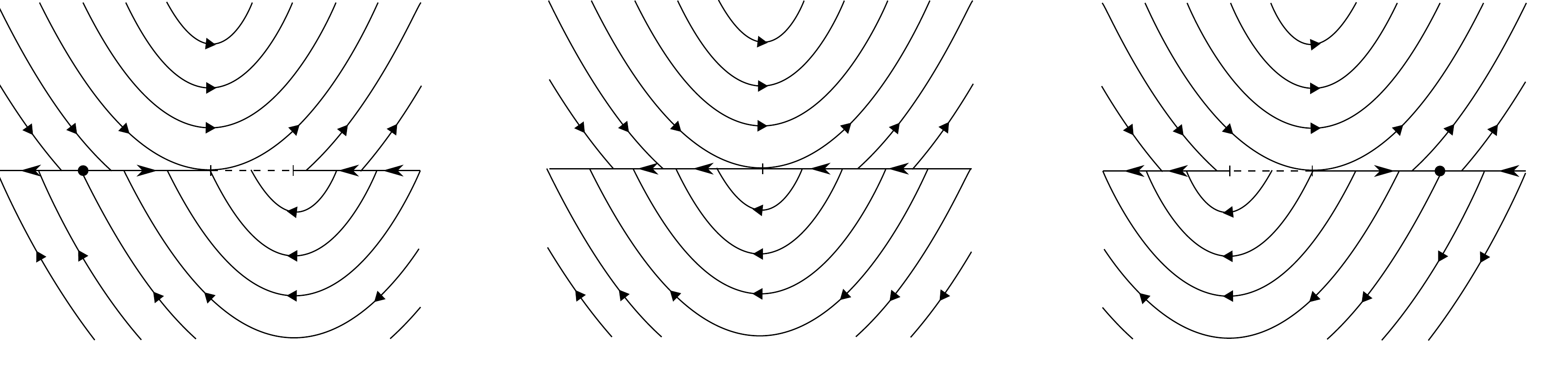
	\end{tiny}
	\caption{The unfolding of $Z \in \Lambda^F$ satisfying $(C)$ , $X^1 \cdot Y^1(\0)<0$ and $(\det{Z})_x(\0)>0$.}
	\label{fig:VISunf+}
\end{figure}

\begin{figure}[h]
	\centering
	\begin{tiny}
		\def\svgscale{0.3}
		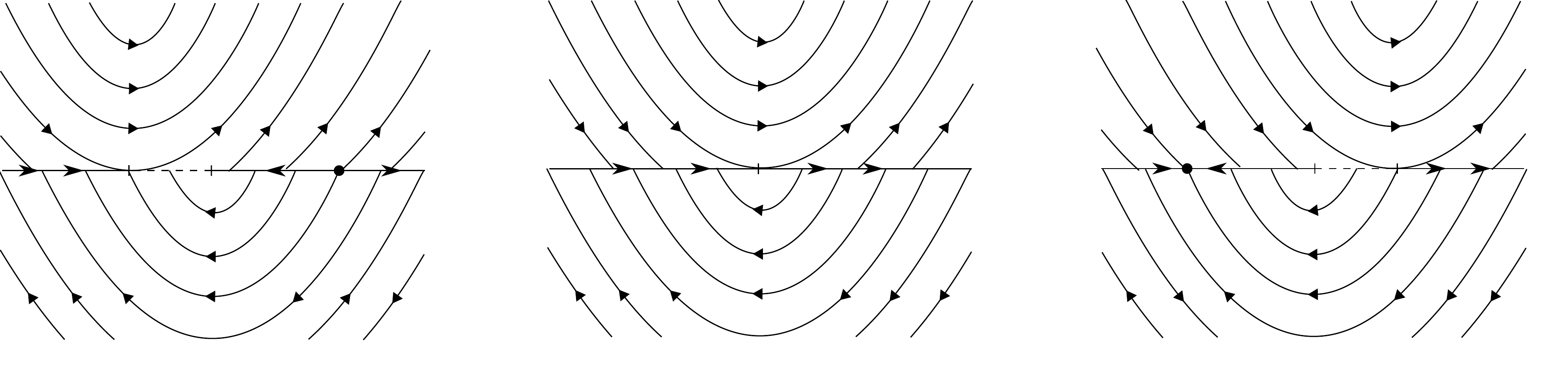
	\end{tiny}
	\caption{The unfolding of $Z \in \Lambda^F$ satisfying $(C)$, $X^1 \cdot Y^1(\0)<0$ and $(\det{Z})_x(\0)<0$.}
	\label{fig:VISunf-}
\end{figure}

\begin{prop} \label{prop:VICunfolding} 
	Let $Z \in \Lambda^F$ satisfying condition $(C)$ of \autoref{fig:IISunfa} and $X^1 \cdot Y^1 (\0)>0$. 
	Let $\mathcal{V}_0$ be the neighborhood given in \autoref{thm:manifold}. 
	Then any smooth curve 
	$$
	\gamma: \alpha \in (-\alpha_0,\alpha_0) \mapsto Z^{\alpha} \in \mathcal{V}_0
	$$
	which is transverse to $\Lambda^F$ at $\gamma(0)=Z$ leads to the same behaviors in $\mathcal{V}^+_0$ and in $\mathcal{V}^-_0$.

	Every  $Z \in \mathcal{V}^-_0$ has one visible and one invisible fold points and there  exists a region of escaping between them. 
	The sliding vector field has no pseudo equilibrium.
	
	Every  $Z \in \mathcal{V}^+_0$ has one visible and one invisible fold points and there  exists a region of sliding between them. 
	The sliding vector field has no pseudo equilibrium.
	
	Therefore, there exists a weak equivalence between any two unfoldings of $Z$.
\end{prop}

\begin{proof}
	Since $X^1 \cdot Y^1(\0)>0$, $\s=\overline{\s^c}$ for $\alpha=0$. 
	For $\alpha \neq 0$,  for all points $(x,0)$ between the folds $T_X^{\alpha}$ and $T^{\alpha}_Y$ the vector field $Z^{\alpha}$ satisfies 
	$X^{\alpha 2} \cdot Y^{\alpha 2}(x,0)<0$. 
	Therefore, a piece of sliding ($\alpha>0$) or escaping $(\alpha<0)$ region appears between the folds for $\alpha \neq 0$. 
	The discontinuity curve becomes $\s=\overline{\s^c \cup \s^{e,s}}$, where  \begin{eqnarray} \label{videc1}
	\s^c &=& \{ x \in \s: x < \min \{T_X^{\alpha},T_Y^{\alpha} \} \} \cup \{ x \in \s: x > \max \{T_X^{\alpha},T_Y^{\alpha} \} \}  \\
	\s \setminus \s^c &=& \begin{cases}
	\s^s =& \{ x \in \s: x \in  (T_Y^{\alpha},T_X^{\alpha}) \} \}, \, \alpha>0 \\
	\s^e =& \{ x \in \s: x \in  (T_X^{\alpha},T_Y^{\alpha}) \} \}, \, \alpha<0 
	\end{cases}
	\end{eqnarray}
	
	A simple computation shows that if $\alpha>0$ then $\det{Z^{\alpha}(x,0)}>0$ for all $x \in [T_Y^{\alpha},T_X^{\alpha}]$, 
	therefore $(Z^{\alpha})^s(x)>0$ in the sliding region. 
	Analogously, for $\alpha<0$ we have $\det{Z^{\alpha}(x,0)}<0$ for all $x \in [T_X^{\alpha},T_Y^{\alpha}]$, therefore $(Z^{\alpha})^s(x)>0$ in the escaping region.
	
	\begin{figure}[h]
		\centering
		\begin{tiny}
			\def\svgscale{0.3}
			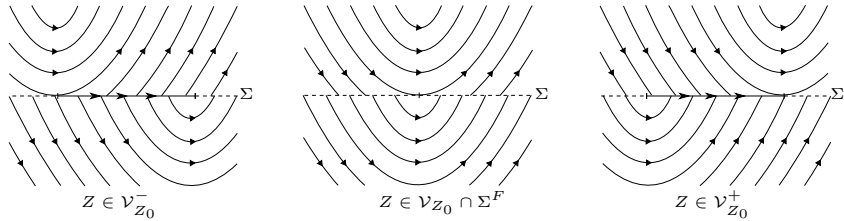
		\end{tiny}
		\caption{The unfolding of $Z \in \Lambda^F$ satisfying $(C)$ and $X^1 \cdot Y^1(\0)>0$.}
		\label{fig:VICunf}
	\end{figure}

	Then any unfolding of $Z \in \Lambda^F$ satisfying $(C)$ and $X^1 \cdot Y^1 (\0)>0$ leads to vector fields with exactly the behavior. 
	Therefore, there exist a weak equivalence between the unfoldings of $Z_0.$
\end{proof}

Joining the results obtained in \autoref{thm:manifold} and in Propositions \ref{prop:VVCunfolding} to \ref{prop:VICunfolding} we prove 
\autoref{thm:generic} stated at the beginning of this section. 
	
%%%%%%
% Section 3
%%%%%%
	
	\section{The regularization near some generic codimension one fold-fold singularity} \label{sec:regularization}
	
	In this section we study the regularization of the versal unfolding $Z^{\alpha}$ of $Z \in \Lambda^F$ 
	studied in  \autoref{sec:revisited}. 
	%\newline
	We will work with  the Sotomayor-Teixeira regularization (see \cite{SotoTei}) 
	which is the vector field $Z^\alpha_{\e}$ given by 
	\begin{eqnarray}\label{streg}
	Z^\alpha_{\e}(x,y) &=& \frac{1}{2} \left[ (X^\alpha+Y^\alpha)(x,y)+ \varphi \left( \frac{y}{\e} \right) (X^\alpha-Y^\alpha)(x,y) \right] \label{eq:regularized}
	\end{eqnarray} 
	\noindent where $\varphi$ is any sufficiently smooth transition function satisfying 
	\begin{equation} \label{transfunc}
	\varphi(v)= \begin{cases}
	1, & v \geq 1 \\
	-1, & v \leq -1
	\end{cases} \ 
	\mathrm{ and } \ \varphi'(v) > 0 \ \mathrm{for } \ v \in (-1,1).
	\end{equation}
	Observe that, rescaling time $t \rightarrow 2 t$, the vector field  \ref{eq:regularized} gives rise to the differential equations 
	\begin{equation} \label{system}
	\begin{cases}
	\dot{x}= G^1(x,y;\alpha,\e) \\
	\dot{y}= G^2(x,y;\alpha,\e)
	\end{cases} , 
	\end{equation} where 
	\begin{equation} \label{Gdef}
	G^i(x,y;\alpha,\e)= (X^{\alpha i}+Y^{\alpha i})(x,y)+ \varphi \left( \frac{y}{\e} \right) (X^{\alpha i}-Y^{\alpha i})(x,y), \, i=1,2.
	\end{equation}
	Performing the change $y= \e \cdot v$ in $\ref{system}$ we obtain the so called {\it slow system} $\bar Z^\alpha_\e$:
	\begin{equation} \label{vsystem}
	\begin{cases} 
	\dot{x} =& F^1(x,v;\alpha,\e)  \\
	\e \dot{v} =& F^2(x,v;\alpha,\e)
	\end{cases}
	\end{equation} 
	\noindent where 
	\begin{equation} \label{Fdef}
	F^i(x,v;\alpha,\e)= (X^{\alpha i}+Y^{\alpha i})(x,\e v)+ \varphi \left( v \right) (X^{\alpha i}-Y^{\alpha i})(x,\e v), \, i=1,2.
	\end{equation}
	After the change of time $\tau=\frac{t}{\e}$, system \ref{vsystem} becomes the so called {\it fast system}, 
	which is a  smooth vector field $\tilde{Z}^\alpha_\e$ depending on two parameters $(\alpha, \e)$: 
	\begin{equation} \label{fsystem}
	\displaystyle{
		\begin{cases} 
		x' =& \e F^1(x,v;\alpha,\e)  \\
		v' =& F^2(x,v;\alpha,\e)
		\end{cases} }
	\end{equation}
	\begin{rem}
		Even if systems \ref{vsystem} and \ref{fsystem} are formally slow-fast systems, when $|v|\ge 1$ these systems are the original smooth vector fields $X$ and $Y$ 
		written in variables $(x,v)=(x,\frac{y}{\e})$. 
		In particular, the existence of ``big periodic orbits'' in section \ref{sec:VIunfreg} will be a consequence of the slow-fast nature of these systems for 
		$|v|\le 1$ combined with the behavior of the original systems $X$ and $Y$ for $v\ge 1$ and $v\le -1$ respectively.
	\end{rem}

	\begin{rem}
		Since one can write 
		\begin{equation}\label{notaciounfolding}
		Z^\alpha = Z + \tilde Z \alpha + \mathcal{O}(\alpha^2), \quad Z=(X,Y), \quad \tilde Z=(\tilde X, \tilde Y),
		\end{equation}
		the regularized system can be written as 
		$Z^\alpha_\e = Z_\e + \tilde Z_\e \alpha + \mathcal{O}(\alpha^2)$, where $Z_\e$ and $\tilde Z_\e$ are the $\varphi-$regularization of $Z$ and $\tilde Z$, respectively. 
	\end{rem}
	
	\subsection{Critical points of the regularized system \texorpdfstring{$Z^\alpha_\e$}{Zae}} \label{critpoint}
	
	To understand the dynamics of $Z_\e ^\alpha$ we begin by studying its equilibrium points.
	
	\begin{lemma} \label{eqlemma} 
		There exist $\alpha_0,\e_0>0$ such that for every $-\alpha_0 < \alpha < \alpha_0$ and $0< \e < \e_0$ one has:
		\begin{enumerate}[(a)]
			\item 
			If $X^1 \cdot Y^1 (\0) >0$,  $Z^\alpha_\e$ has no critical points;
			\item 
			If $X^1 \cdot Y^1 (\0) <0$, $Z^\alpha_{\e}$ has a unique critical point:
			\begin{equation} \label{criticalpoint}
			P(\alpha,\e)=(x(\alpha,\e),\e v(\alpha,\e)) = Q(\alpha) + \mathcal{O}(\e) = (\bar x,0) \alpha + (x^*, v^*) \e + \mathcal{O}_2(\alpha,\e),
			\end{equation}
			where $Q(\alpha)$ is the pseudo-equilibrium of $Z^\alpha$, and $v^*$, $x^*$ and $\bar x$ are given in \ref{eq:criticalpoint}, \ref{eq:xline} and \ref{eq:xbar}.
		\end{enumerate}
	\end{lemma}
	
	\begin{proof}
		Using the change $y = \e v$, we look for zeros of the map 
		\begin{equation}
		F(x,v;\alpha,\e) =(F^1(x,v;\alpha,\e),F^2(x,v;\alpha,\e)).
		\end{equation} 
		
		At first we consider $F^1(x,v;\alpha,0)=0$, which is solvable if and only if for each $x$ there exists $v(x) \in (-1,1)$ satisfying  
		\begin{equation} \label{criticalpoint1}
		\varphi(v(x)) = - \frac{X^{\alpha 1}+Y^{\alpha 1}}{X^{\alpha 1}-Y^{\alpha 1}}(x,0)
		\end{equation}  
		
		If $X^1 \cdot Y^1 (\0) >0$ then \autoref{criticalpoint1} has no solution for $\alpha$ small enough, since the absolute value of the right-hand side of equation 
		\ref{eq:criticalpoint} is greater than one. 
		Therefore, by continuity the vector field $Z^\alpha_{\e}$ has no critical points near the origin, for $\e>0$ sufficiently small.
		
		If $X^1 \cdot Y^1 (\0) <0$ for $\alpha$ small enough the absolute value of the right-hand side of \autoref{criticalpoint1} is smaller than one, 
		then for each $x$ it admits a solution $v(x) \in (-1,1)$. Moreover, 
		$$
		F^2(x,v(x);\alpha,0) = -\frac{2 \det{Z^\alpha}}{Y^{\alpha 1} - X^{\alpha 1}}(x,0).
		$$
		
		By \autoref{corol:psequi},  applied to the vector field $Z^\alpha$, for small $\alpha$ there exists a unique solution $x(\alpha)=P(Z^\alpha)$ 
		near the origin such that $\det{Z^{\alpha}}({x(\alpha),0)})=0$ and 
		$\sgn{(\det{Z^\alpha})_x(x(\alpha),0)}=\sgn{(\det{Z)_x}(\0)}.$ 
		Therefore, setting $v(\alpha)=v(x(\alpha))$ we have $F(x(\alpha),v(\alpha);\alpha,0)=0$. Moreover, a straightforward computation shows that 
		$$
		\det{D_{(x,v)} F(x(\alpha),v(\alpha);\alpha,0)}= - 2 \varphi'(v(\alpha)) (\det{Z^\alpha)_x(x(\alpha),0)} \neq 0
		$$
		
		By applying the Implicit Function Theorem we obtain that $Z^\alpha_\e$ has a critical point 
		$P(\alpha,\e) = (x(\alpha),0) + \mathcal{O}(\e)$ for $\e>0$ sufficiently small. 
		Moreover, by the Chain's Rule we get the expressions for $v^*$, $x^*$ and $\bar x:$ using the notation given in \eqref{notaciounfolding}
		\begin{align}
		v^* &=\varphi^{-1} \left( - \frac{X^1+Y^1}{X^1-Y^1}(\0) \right) \label{eq:criticalpoint} \\
		x^* &= -\frac{(\det{Z})_y}{(\det{Z})_x} (\0) v^* \label{eq:xline} \\
		\bar x &= \frac{Y^1 \tilde X^2 - X^1 \tilde Y^2}{(\det{Z})_x}(\0) \label{eq:xbar}.
		\end{align}
		
		Observe that for $\alpha \neq 0$, the point $Q(\alpha)=(x(\alpha),0)$ is the pseudo-equilibrium of the sliding vector field which appears 
		in the unfolding $Z^\alpha$. That is, the critical point $P(\alpha,\e)$ that arises after the regularization derives from the pseudo-equilibrium of 
		$Z^{\alpha}$. Moreover $x(\alpha)=\overline{x} \alpha + \mathcal{O}(\alpha^2)$.
		
	\end{proof}
	\begin{rem}
		Observe that when $\alpha$ tends to zero, even though the pseudo-equilibrium $Q(\alpha)$ for $Z^s_\alpha$ disappears and becomes the 
		fold-fold point of $Z$, 
		the critical point $P(\alpha,\e)$ of $Z^\alpha_\e$ tends to $P(0,\e)$ which is the critical point of $Z_\e$.
	\end{rem}
	
	To obtain the topological character of the critical point $P(\alpha,\e)$ we need information about the determinant and the trace of 
	$DZ_{\e}^{\alpha}(P(\alpha,\e))$.
	
	\begin{prop} \label{corol:regnumbers}  
		Consider $Z \in \Lambda^F$ and $X^1\cdot Y^1 (\0)<0$. 
		Then at the critical point $P(\alpha,\e)$ we have
		\begin{enumerate}[(a)]
			\item \label{itm:unfdet1} 
			$\displaystyle{\det{DZ^{\alpha}_{\e}}(P(\alpha,\e))
				= -\frac{1}{\e} \left( 2 \varphi'(v(\alpha)) (\det{Z^\alpha)_x(x(\alpha),0)} + \mathcal{O} \left(\e \right) \right);}$ 
			\item \label{itm:unftr1} 
			$\displaystyle{\tr DZ^{\alpha}_{\e}(P(\alpha,\e))
				=\frac{1}{\e} \left( \varphi'(v(\alpha))(X^{\alpha 2}-Y^{\alpha 2})(x(\alpha),0) + \mathcal{O}(\e) \right),}$ for $\alpha \neq 0.$ 
			
			\noindent {\it Moreover}
			
			\item \label{itm:unfdet}   
			$\displaystyle{\det{DZ^{\alpha}_{\e}}(P(\alpha,\e)) 
				= -\frac{1}{\e} \left( 2 \varphi'(v^*) (\det{Z)_x(\0)} + \mathcal{O} \left( \alpha,\e \right) \right);}$ 
			\item \label{itm:unftr} 
			$\displaystyle{\tr {DZ^{\alpha}_{\e} (P(\alpha,\e))} 
				=\frac{1}{\e} \left( N(Z,\tilde{Z}) \alpha + M(Z) \e  + \mathcal{O}_2\left( \alpha,\e \right) \right)},$ 
		\end{enumerate} 
		
		\noindent {\it where $M(Z)$ and $N(Z,\tilde{Z})$ are constants given by} \begin{eqnarray*}
			M(Z) &=&  \left[ (X^1_x + Y^1_x) + \varphi(v^*)(X^1_x - Y^1_x) + (X^2_y + Y^2_y) + \varphi(v^*)(X^2_y - Y^2_y) \right. \\
			&+& \left. \varphi'(v^*) ((X^2_x-Y^2_x) x^* + (X^2_y-Y^2_y) v^*)  \right] (\0) \\
			N(Z,\tilde{Z})  &=& \frac{1}{(\det{Z})_x(\0)} \varphi'(v^*) (X^1 - Y^1) (\0) (Y^2_x \tilde{X}^2 - X^2_x \tilde{Y}^2)(\0)
		\end{eqnarray*}
		\noindent with $v^*, x^*$ 
		given in \ref{eq:criticalpoint} and  \ref{eq:xline} respectively.
	\end{prop}
	
	\begin{proof}
		The proof can be seen in \cite{larrosa}.
	\end{proof}
	
	\begin{rem}
		Since $\varphi'(v^*), (X^1-Y^1)(\0) \neq 0$ and the transversality condition \ref{foldsbe} guarantees that 
		$(Y^2_x \tilde{X}^2 - X^2_x \tilde{Y}^2)(\0) \neq 0$, the constant $N(Z,\tilde Z) \neq 0$ for any versal unfolding of $Z \in \Lambda^F$.
	\end{rem}

	\begin{prop} \label{prop:toptype}
		Let $Z \in \Lambda^F$ satisfying $X^1 \cdot Y^1 (\0)<0$ and
		$P(\alpha,\e)$ be the critical point of $Z^{\alpha}_\e$ given in \autoref{eqlemma}. 
		It follows that for $\alpha$ and $\e>0$ small enough:
		
		\begin{enumerate}[(a)]
			\item 
			If $(\det{Z})_x(\0)>0$ the critical point $P(\alpha,\e)$ is a saddle; \label{itm:g0}
			\item 
			If $(\det{Z})_x(\0)<0$ the critical point $P(\alpha,\e)$ is a node or a focus. 
			Moreover,: \label{itm:l0}
			\begin{enumerate}[({b}1)]
				\item  
				There exist a curve  $\mathcal{D}$ in the parameter plane $(\alpha,\e)$, given by 
				\begin{equation} \label{dcurve}
				\mathcal{D}= \left\{ (\alpha,\e) : \, \e = -\frac{N(Z,\tilde{Z})^2}{8\varphi'(v^*)(\det{Z)_x(\0)})} \alpha^2 + \mathcal{O}(\alpha^3) \right\},
				\end{equation}
				such that:
				\begin{enumerate}[(i)]
					\item 
					For $(\alpha,\e)$ bellow the curve $\mathcal{D}$ 
					the critical point $P(\alpha,\e)$ is a node;
					\item 
					For $(\alpha,\e)$ on the curve $\mathcal{D}$ 
					the critical point $P(\alpha,\e)$ is a degenerate node;
					\item 
					For $(\alpha,\e)$ above the curve $\mathcal{D}$ 
					the critical point $P(\alpha,\e)$ is a focus;
				\end{enumerate}
				\item 
				There exists a curve $\mathcal{H}$ in the parameter plane $(\alpha,\e)$, given by 
				\begin{equation} \label{hcurve}
				\mathcal{H} =\left\{ (\alpha,\e): \, \alpha =  \delta_\mathcal{H} \ \e + \mathcal{O}(\e^2) \right\}, \quad \delta_\mathcal{H}=-\frac{M(Z)}{N(Z,\tilde{Z})}.
				\end{equation}
				such that 
				the critical point $P(\alpha,\e)$ 
				undergoes a Hopf Bifurcation.
			\end{enumerate} 
		\end{enumerate}
	\end{prop}
	\begin{proof}
		The proof can be seen in \cite{larrosa}.
	\end{proof}
	
	\begin{figure}[H]
		\centering
		\begin{scriptsize}
			\def\svgscale{0.55}
			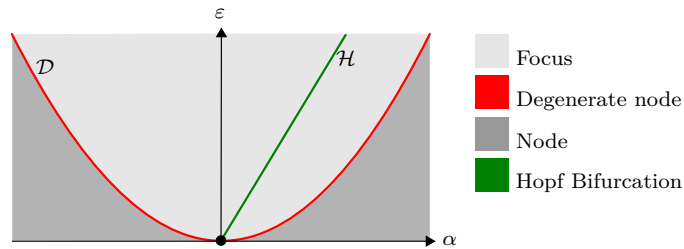
		\end{scriptsize}
		\caption{Topological type of the critical point $P(\alpha,\e)$ depending on each region of the $(\alpha,\e)-$parameter semi-plane.}
		\label{fig:toptype}
	\end{figure}
	
	\begin{corol} \label{corol:toptype}
		Let $Z \in \Lambda^F$ satisfying $X^1\cdot Y^1(\0)<0$. One has that 
		\begin{enumerate}[(a)]
			\item 
			if both folds are visible the critical point $P(\alpha,\e)$ is a saddle for $|\alpha|$ and $\e>0$ small enough;
			\item 
			if both folds are invisible, the topological type of the critical point $P(\alpha,\e)$ changes as described in item $(b)$ of \autoref{prop:toptype}. 
			Moreover, the critical point $P(\alpha,\e)$ is stable for $(\alpha,\e)$ on the left of the curve $\mathcal{H}$ and it is unstable for 
			$(\alpha,\e)$ on the right of the curve $\mathcal{H}$; 
			\item 
			if both folds have opposite visibility, then
			\begin{enumerate}[(c1)]
				\item 
				if $(\det{Z})_x(\0)>0$, $P(\alpha,\e)$ is a saddle for every $|\alpha|$ and $\e>0$ small enough;
				\item 
				if $(\det{Z})_x(\0)<0$, the topological type of the critical point $P(\alpha,\e)$ changes as described in item $(b)$ of \autoref{prop:toptype}. 
				Moreover, the critical point $P(\alpha,\e)$ is unstable for $(\alpha,\e)$ on the left of the curve $\mathcal{H}$ and it is stable for $(\alpha,\e)$ 
				on the right of the curve $\mathcal{H}$; 
			\end{enumerate}
		\end{enumerate}
	\end{corol}
	\begin{proof}
		The proof is a consequence of  \autoref{lem:neq0} and Propositions \ref{corol:regnumbers} and \ref{prop:toptype}.
	\end{proof}
	
	\begin{rem}
	Over the curve $\mathcal{H}$ the character of the Hopf bifurcation is determined by the first Lyapunov coefficient $\ell_1$, see \cite{holmes}. 
		If $\ell_1>0 $ the bifurcation is subcritical and gives rise to the existence of an unstable periodic orbit. If $\ell_1<0$ a 
		stable periodic orbit arises at the Hopf bifurcation.
		However, the computation of $\ell_1$ leads to a cumbersome expression which does not add any relevant information. 
		The only important thing is that it sign depends on the vector field $Z$ but also of the regularization function $\varphi$ as 
		%we will see in examples \ref{ex:IIsuper} to \ref{ex:ch} and 
		was already observed in \cite{KristiansenH15}, 
		where a formula for this coefficient for some suitable normal forms of $Z^\alpha$ was given (see formula (7.15) of that paper).
		In  \autoref{sec:melnikov} we will relate this coefficient with the derivative of a suitable Melnikov function at the point $(0,v^*)$ 
		(see Proposition \ref{prop:propiedadesM}). 
	\end{rem}

	\begin{rem} 
		It is worth to mention that, when the non-smooth vector field $Z$ has an invisible-invisible fold, the stability of the critical point $P(0,\e)$ of $Z_\e$ 
		is not related with the stability of the fold-fold given by the first return map  \ref{gfirstreturn}. 
		Let us recall that the stability of the fold-fold point depends of $\G_Z$ given in \ref{gencondii} and the stability of the focus depends of the sign of 
		$\tr{DZ_{\e}(P(0, \e))}$ given in \ref{corol:regnumbers}. 
		Due to the cumbersome expression of these coefficients one could think that it is possible to relate the sign of both quantities but we will see 
		in \autoref{ex1} that the signs of these coefficients are totally independent.
		
		In fact, this is not surprising because the fold-fold point is a linear center for the return map $\phi_Z$ and its stability can be changed by a small perturbation. 
	\end{rem}
	
	\begin{exmp}[From an attractive invisible fold-fold in the non-smooth vector field to a focus and a ``linear'' center in its regularization] \label{ex1}
		Consider the one parameter family $Z_{\eta} = (X_{\eta}, Y)$ where  
		\begin{equation} \label{Ex1unfolding}
		\begin{cases}
		X_{\eta}(x,y) = (-1 + \eta x, x) \\
		Y (x,y) = (1, 2x+x^2)
		\end{cases}
		\end{equation}
		Observe that $Z_{\eta}$ is not a versal unfolding of $Z_0$, since we have an invisible fold-fold at the origin for all values of $\eta$. 
		By Proposition \ref{prop:frX}, the return map  associated to this family is given by 
		\begin{equation} \label{ex1returnmap}
		\phi_{\eta}(x) = x + \frac{1}{3} (1-2\eta) x^2 -\frac{5}{9} x^3 + \mathcal{O}(x^4).
		\end{equation}
		For every $\eta<\frac{1}{2}$, $\G_Z =\frac{1}{3}(1-2\eta)>0$ and therefore the origin is an stable fixed point for the 
		return map $\phi_{\eta}$ in \ref{gfirstreturn}.
		
		Using a smooth transition function $\varphi$  as in \autoref{transfunc}, the regularized system reads 
		$$
		Z_{\e}^{\eta}(x,y)= \begin{cases}
		\dot{x}=  \left[ \eta x + \varphi(\frac{y}{\e})(-2 +\eta x) \right], \\
		\dot{y}= \left[3x +x^2 + \varphi(\frac{y}{\e})(-x-x^2) \right].
		\end{cases}. 
		$$
		The critical point for this system is  point $P(\e)=(0,\e v*)$ where $\varphi(v^*)=0$. 
		At $P(\e)$ we have 
		$$
		DZ_{\e}^{\alpha}(P(\e)) =
		\left(
		\begin{array}{cc}
		\eta & -2\frac{1}{\e }\varphi'(v^*) \\
		3 & 0 \\
		\end{array}
		\right) 
		\Rightarrow  \,
		\begin{cases} 
		\det DZ_{\e}^{\eta}(P(\e)) = \displaystyle{\frac{6}{\e}\varphi'\left( v^* \right)} > 0 \\
		\tr DZ_{\e}^{\eta}(P(\e)) = \displaystyle{ \eta }
		\end{cases}.
		$$
		It directly follows that the origin is an stable focus for $\eta<0$, a linear center for $\eta= 0$ and an unstable focus for $\eta >0$.
		
	\end{exmp}
	
	Observe that 
	$Z_{\eta}$ suffers a codimension two bifurcation when $\eta=\frac{1}{2}$, without moving the tangencies apart. 
	By varying $\eta$ around $\eta=\frac{1}{2}$, the fold-fold changes its stability but $P(\e)$ remains an unstable focus.
	
	\subsection{Critical manifolds of the regularized system \texorpdfstring{$Z^{\alpha}_\e$}{Zae}} \label{subsec:fenichelunf}
	
	In this section we will study the critical manifolds the slow-fast systems \ref{vsystem} and \ref{fsystem}.
	Setting $\e=0$ in \ref{vsystem}, the critical manifold $\Lambda^{\alpha}_0$
	is given by
	$$
	\Lambda^{\alpha}_0 = \left\{ (x,v): \, F^2(x,v;\alpha,0)=(X^{\alpha 2}+ Y^{\alpha 2})(x,0) + \varphi(v)(X^{\alpha 2}- Y^{\alpha 2})(x,0) =0 \right\}. 
	$$
	If $(x,0) \in \s^c$, $F^2(x,v;\alpha,0) \neq 0$ and therefore the critical manifold is not  defined.
	If $(x,0) \in \s^{e,s}$ the equality 
	\begin{equation} \label{phialpha}
	\varphi(v)=\frac{(X^\m{\alpha 2} + Y^\m{\alpha 2})(x,0)}{(Y^\m{\alpha 2} - X^\m{\alpha 2})(x,0)},
	\end{equation}
	is well defined and it is solvable. Therefore the critical manifold is given by 
	\begin{equation} \label{unfcrit}
	\Lambda_0^{\alpha} = \left\{ (x,v) : v = m_0^\alpha(x), \ x \in \s^e \cup \s^s \right\}. 
	\end{equation}
	where \begin{equation}\label{eq:m0}
	m_0^\alpha(x)= \varphi^{-1} \left( \frac{X^\m{\alpha 2} + Y^\m{\alpha 2}}{Y^\m{\alpha 2} - X^\m{\alpha 2}}(x,0) \right).
	\end{equation}
	Observe that, for $\alpha \neq 0$, we have:
	%\[
	$m_0^\alpha(T_X^\alpha)=1$, 
	%\quad 
	$m_0^\alpha(T_Y^\alpha)=-1$,  
	%\]
	where $(T_X^\alpha,0)$ and $(T_Y^\alpha,0)$ are the fold points of the vector fields $X^\alpha$ and $Y^\alpha$ given in \ref{fXexpression}.
	Moreover, 
	\begin{equation}\label{vderivative}
	\frac{d}{dx} m_0^\alpha(x) = \frac{d}{dx} \left( \varphi^{-1} \left(\frac{X^\m{\alpha 2} + Y^\m{\alpha 2}}{Y^\m{\alpha 2} - X^\m{\alpha 2}}(x,0) \right) \right)  
	= K^\alpha(x,0) \cdot \left( (X^\m{\alpha 2})_x Y^\m{\alpha 2} - (Y^\m{\alpha 2})_x X^\m{\alpha 2}  \right)(x,0) 
	\end{equation}
	where $K^\alpha(x,0)=\displaystyle{\frac{1}{\varphi'(m_0^\alpha(x)) ((X^\alpha-Y^\alpha)(x,0))^2}}>0$ 
	for $\alpha \neq 0$. 
	As $\varphi'(\pm 1)=0$,  when $x$ tends to the tangency points we have  
	$\displaystyle{\lim \frac{d}{dx} m_0^\alpha(x) = \pm \infty},$ 
	therefore $\Lambda^{\alpha}_0$ reaches  $v= \pm 1$ at these points vertically. 
	
	The stability  of the critical manifold $\Lambda^\alpha_0$ for system \ref{fsystem} is given by 
	\begin{equation} \label{Fstability}
	\frac{\partial}{\partial v} F^{2}(x,v;\alpha ,\e) \bigg|_{\e=0} = \varphi'(v)(X^{\alpha 2}-Y^{\alpha 2}) (x,0),
	\end{equation}
	thus the critical manifold is repelling if it is defined over  a escaping region and attracting if it is defined over a 
	sliding region, see \cite{SotoTei}. 
	
	As it was seen in \cite{SotoTei}, the dynamics induced over the critical manifold $\Lambda^\alpha_0$ is equivalent to the dynamics of the sliding vector field 
	$(Z^\alpha)^s$,  defined in $\s^e \cup \s^s$, 
	since by a simple computation and \ref{slidingdef} we obtain 
	\begin{equation} \label{induceddyn}
	\dot{x} = F^1(x,v;\alpha,0) \big|_{\Lambda^{\alpha}_0}  = 2 \left( Z^\m{\alpha} \right)^\m{s}(x).
	\end{equation}
	Therefore, if the sliding vector field $(Z^\alpha)^s$ has a pseudo-equilibrium $Q(\alpha)=(x(\alpha),0)$, 
	the induced dynamics in $\Lambda _0 ^\alpha$  has a critical point at $(x(\alpha),m_0^\alpha(x(\alpha))$ 
	which has the same stability as the pseudo-equilibrium $Q(\alpha)$.
	
	For $\alpha=0$, since the origin is a fold-fold point, one can write $F^2(x,v;0,0) = x \cdot \A(x,v)$ where
	\begin{equation} \label{eq:R1}
	\A(x,v) = (1+\varphi(v)) X^2_x(\0) +(1-\varphi(v)) Y^2_x(\0) + \mathcal{O}(x)\ .
	\end{equation} 
	Therefore, in this case, the critical manifold $\Lambda^0_0$ decomposes as $\Lambda^0_0 = C_0 \cup \Lambda_0$, where 
	$$
	\ C_0 = \{ (x,v) : \ x=0 \} \ \text{and} \ \Lambda_0 = \{ (x,v) : \ \A(x,v)=0 \}.
	$$
	Moreover, at $C_0$, \eqref{Fstability} is identically zero,
	therefore $C_0$ is not a hyperbolic critical manifold of system \ref{fsystem}. 
	We will see in sections \ref{sssec:SVcritical} and \ref{ssec:OVcritical} that the critical manifold $\Lambda_0$ can be empty depending on the folds visibility. 
	
	During the rest of this paper we will restrict ourselves to the study of the regularization of 
	$Z^\alpha$ in the case that $(X^1\cdot Y^1) (\0)<0$.
	The dynamics of the other case is straightforward and can  be found in \cite{KristiansenH15}.
	
	\subsubsection{Folds with the same visibility} \label{sssec:SVcritical}
	
	When  the folds have the same visibility, for $\alpha = 0$ \autoref{visibility} implies that $X^2_x \cdot Y^2_x (\0)>0$ and 
	hence $\A(x,v) \neq 0$ for $(x,v)$ in a neighborhood of the origin, therefore $\Lambda_0=\{ (x,v): \, \A(x,v)=0 \} = \varnothing$.
	The critical manifold is $\Lambda^0_0 = C_0$ and it is not hyperbolic, see \autoref{fig:SVcriticalVV}. 
	
	We saw in Propositions \ref{prop:VVSunfolding} and \ref{prop:IICunfolding} that, for $\alpha \neq 0$, an sliding 
	or escaping  region appears between the two fold points. 
	Therefore the critical manifold $\Lambda_0^{\alpha}$, given in \ref{unfcrit}, is a smooth curve connecting the points $(T_X^{\alpha},1)$ and $(T_Y^{\alpha},-1)$. 
	In addition, using \autoref{visibility} and \ref{vderivative}, we obtain that $\Lambda_0^{\alpha}$ is an increasing curve if $\alpha>0$ and decreasing if $\alpha<0$. 
	Adding the results about the sliding and escaping regions of sections \ref{sec:VVunf} and \ref{sec:IIunf} we obtain:
	\begin{itemize}
		\item 
		In the visible-visible case, see \autoref{fig:SVcriticalVV}
		\begin{itemize}
			\item 
			If $\alpha <0$, $\Lambda _0 ^\alpha=\Lambda _0 ^{\alpha,u}$ is a decreasing curve connecting the points $(T_X^{\alpha},1)$ and $(T_Y^{\alpha},-1)$ and is repelling. 
			The point $(x(\alpha), m_0^\alpha(x(\alpha))\in \Lambda _0 ^{\alpha,u}$ is stable for the induced dynamics.
			\item 
			If $\alpha >0$, $\Lambda _0 ^\alpha=\Lambda _0 ^{\alpha,s}$ is an increasing curve connecting the points $(T_Y^{\alpha},-1)$ and $(T_X^{\alpha},1)$ and  is attracting.
			The point $(x(\alpha), m_0^\alpha(x(\alpha))\in \Lambda _0 ^{\alpha,s}$ is unstable for the induced dynamics. 
		\end{itemize}
		\item 
		In the invisible-invisible case, see \autoref{fig:SVcriticalII}
		\begin{itemize}
			\item 
			If $\alpha <0$, $\Lambda _0 ^\alpha=\Lambda _0 ^{\alpha,s}$ is a decreasing curve connecting the points $(T_X^{\alpha},1)$ and $(T_Y^{\alpha},-1)$ and is attracting.
			The point $(x(\alpha),m_0^\alpha(x(\alpha))\in \Lambda _0 ^{\alpha,s}$ is stable for the induced dynamics.
			\item 
			If $\alpha >0$, $\Lambda _0 ^\alpha=\Lambda _0 ^{\alpha,u}$ is an increasing curve connecting the points $(T_Y^{\alpha},-1)$ and $(T_X^{\alpha},1)$ and  is repelling.
			The point $(x(\alpha), m_0^\alpha(x(\alpha))\in \Lambda _0 ^{\alpha,u}$ is unstable for the induced dynamics.
		\end{itemize}
	\end{itemize}
	
	\begin{figure}[htb]
		\centering
		\begin{tiny}
			\subfigure[\label{fig:SVcriticalVV} The visible-visible fold]{ \def\svgscale{0.5} 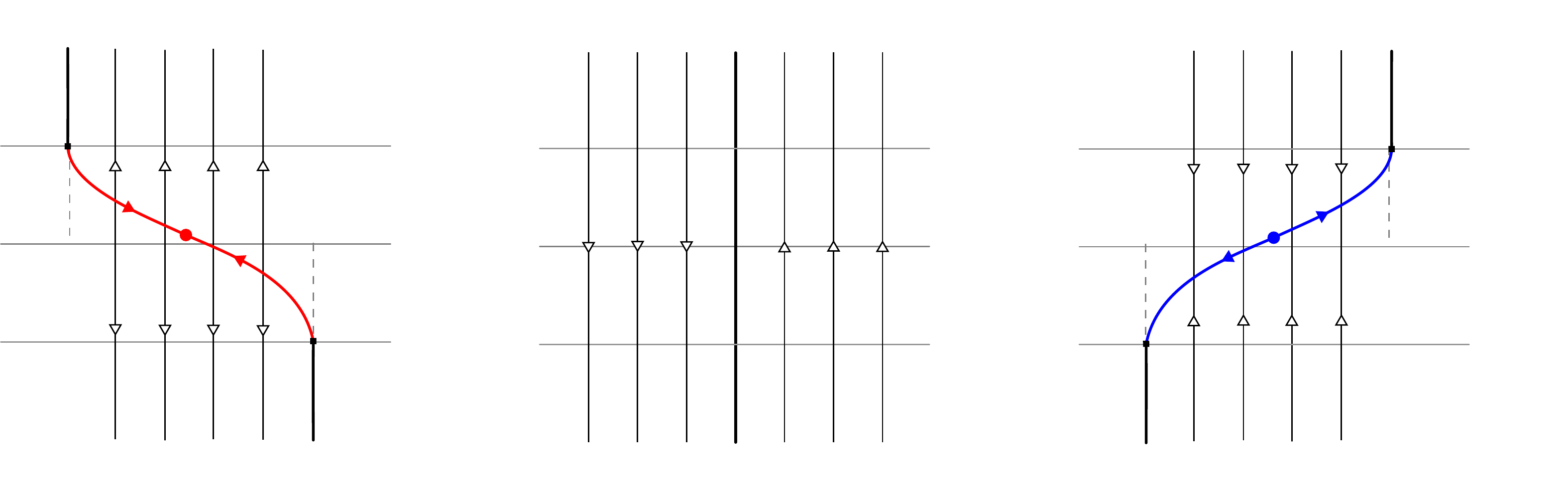}
			\subfigure[\label{fig:SVcriticalII} The invisible-invisible fold]{\def\svgscale{0.5} 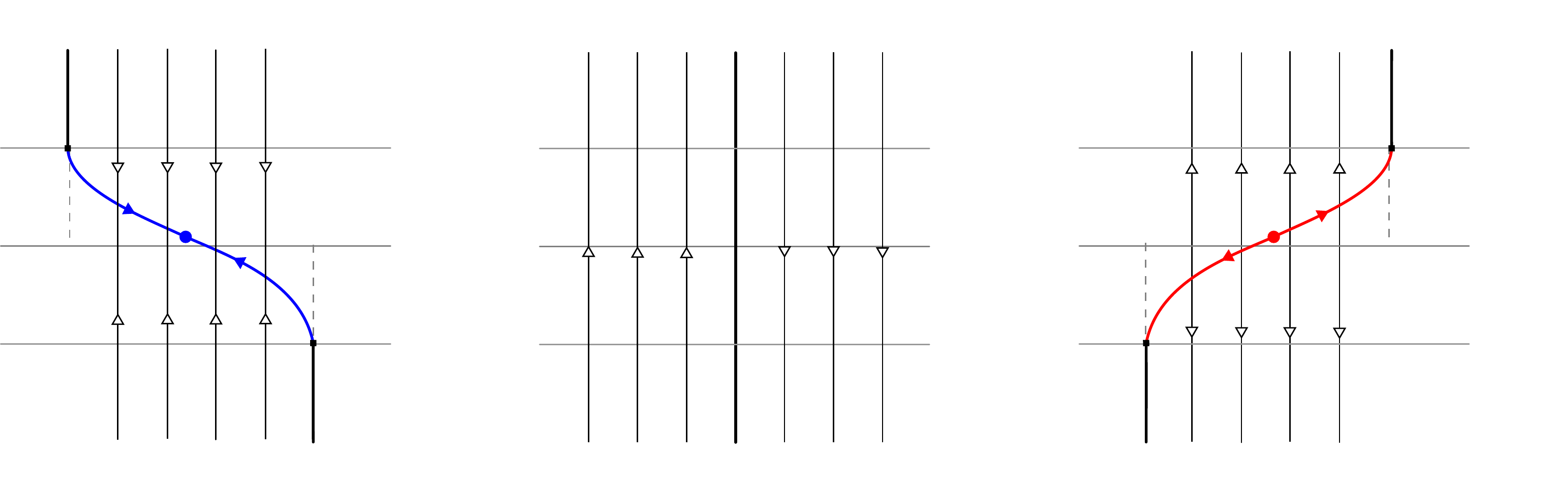}
		\end{tiny}
		\caption{The critical manifold when the folds have the same visibility for different values of $\alpha$.}
		\label{fig:SVcritical}
	\end{figure}
	
	\subsubsection{Folds with opposite visibility} \label{ssec:OVcritical}
	
	When  the folds have opposite visibility, for $\alpha=0$ there exists a curve $v=m_0(x)$ defined in a neighborhood of $x=0$ such that $\A(x,m_0(x))=0$, where
	\begin{equation} \label{IVslowmanifold}
	\begin{array}{rcl}
	m_0(x)&=& \bar v + \mathcal{O}(x)\\
	\bar v &=& \varphi^{-1} \left( \frac{(X^2_x + Y^2_x)}{(Y^2_x - X^2_x)}(\0) \right)
	\end{array}
	\end{equation}
	which is transverse to the line $x=0$ at the point  $(0,\bar v)$. 
	
	Using \ref{Fstability}, for $\alpha = 0$, we have two hyperbolic critical manifolds 
	$$
	\Lambda_0^s = \{ (x,v) : \ v = m_0(x), \ x<0 \}, \quad \Lambda_0^u = \{ (x,v): v = m_0(x), \  x>0 \}
	$$ 
	which are attracting and repelling, respectively. 
	Therefore, $\Lambda_0 = \Lambda_0^s \cup \Lambda_0^u,$  see Figures \ref{fig:OVcritical+} and \ref{fig:OVcritical-}, for $\alpha=0$.

	For $\alpha \neq 0$, we have seen in  \autoref{prop:VISunfolding} that   a crossing region appears between 
	the fold points $(T^\m{\alpha}_\m{X},0)$ and $(T^\m{\alpha}_\m{Y},0)$. 
	Therefore there exist two critical manifolds: $\Lambda^{\alpha,s}_0$, which is defined for $x<\min\{T^\m{\alpha}_\m{X},T^\m{\alpha}_\m{Y} \}$ and is attracting, and 
	$\Lambda^{\alpha,u}_0$ which is defined for $x>\max\{T^\m{\alpha}_\m{X},T^\m{\alpha}_\m{Y} \}$ and is repelling. 
	
	Adding the results about the sliding and escaping regions of section \ref{sec:VIunf} we obtain:
	\begin{itemize}
		\item 
		$(\det{Z)_x(\0)}>0$, see \autoref{fig:OVcritical+} 
		\begin{itemize}
			\item 
			If $\alpha <0$, $\Lambda _0 ^{\alpha,s}$ is a increasing and attracting curve containing the point $(T_X^{\alpha},1)$. 
			$\Lambda _0 ^{\alpha,u}$ is a increasing and repelling curve containing the point $(T_Y^{\alpha},-1)$. 
			The point $(x(\alpha), m_0^\alpha(x(\alpha)) \linebreak \in \Lambda _0 ^{\alpha,s}$ is unstable for the induced dynamics.
			\item 
			If $\alpha >0$, $\Lambda _0 ^{\alpha,s}$ is a decreasing and attracting curve containing the point $(T_Y^{\alpha},-1)$. 
			$\Lambda _0 ^{\alpha,u}$ is a decreasing and repelling curve containing the point $(T_X^{\alpha},1)$. 
			The point $(x(\alpha), m_0^\alpha(x(\alpha)) \linebreak \in \Lambda _0 ^{\alpha,u}$ is stable for the induced dynamics.
		\end{itemize}
		\item 
		$(\det{Z)_x(\0)}<0$,  see \autoref{fig:OVcritical-}
		\begin{itemize}
			\item 
			If $\alpha <0$, $\Lambda _0 ^{\alpha,s}$ is a increasing and attracting curve containing the point $(T_X^{\alpha},1)$. 
			$\Lambda _0 ^{\alpha,u}$ is a increasing and repelling curve containing the point $(T_Y^{\alpha},-1)$. 
			The point $(x(\alpha), m_0^\alpha(x(\alpha)) \linebreak \in \Lambda _0 ^{\alpha,u}$ is unstable for the induced dynamics.
			\item 
			If $\alpha >0$, $\Lambda _0 ^{\alpha,s}$ is a decreasing and attracting curve containing the point $(T_X^{\alpha},1)$.
			$\Lambda _0 ^{\alpha,u}$ is a decreasing and repelling curve containing the point $(T_Y^{\alpha},-1)$. 
			The point $(x(\alpha), m_0^\alpha(x(\alpha)) \linebreak \in \Lambda _0 ^{\alpha,s}$ is stable for the induced dynamics.
		\end{itemize}
	\end{itemize}
	
	\begin{figure}[htb]
		\centering
		\begin{tiny}
			\subfigure[\label{fig:OVcritical+}The visible-invisible fold satisfying $(\det{Z)_x(\0)>0}$]{ \def\svgscale{0.5} 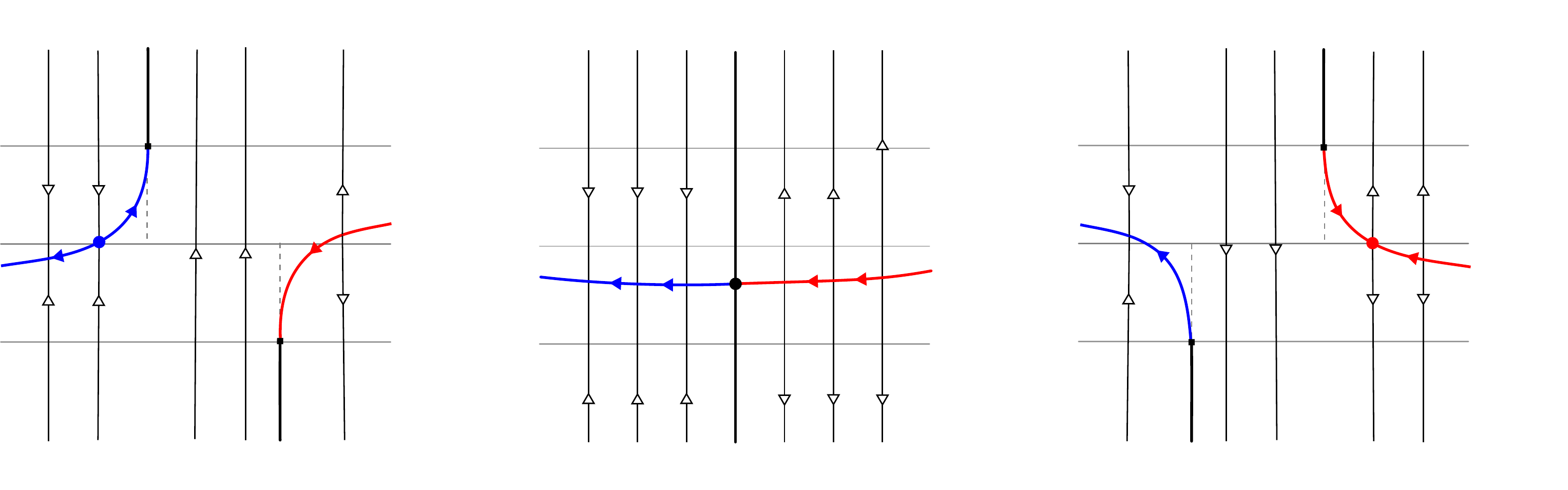}
			\subfigure[\label{fig:OVcritical-}The visible-invisible fold satisfying $(\det{Z)_x(\0)<0}$]{\def\svgscale{0.5} 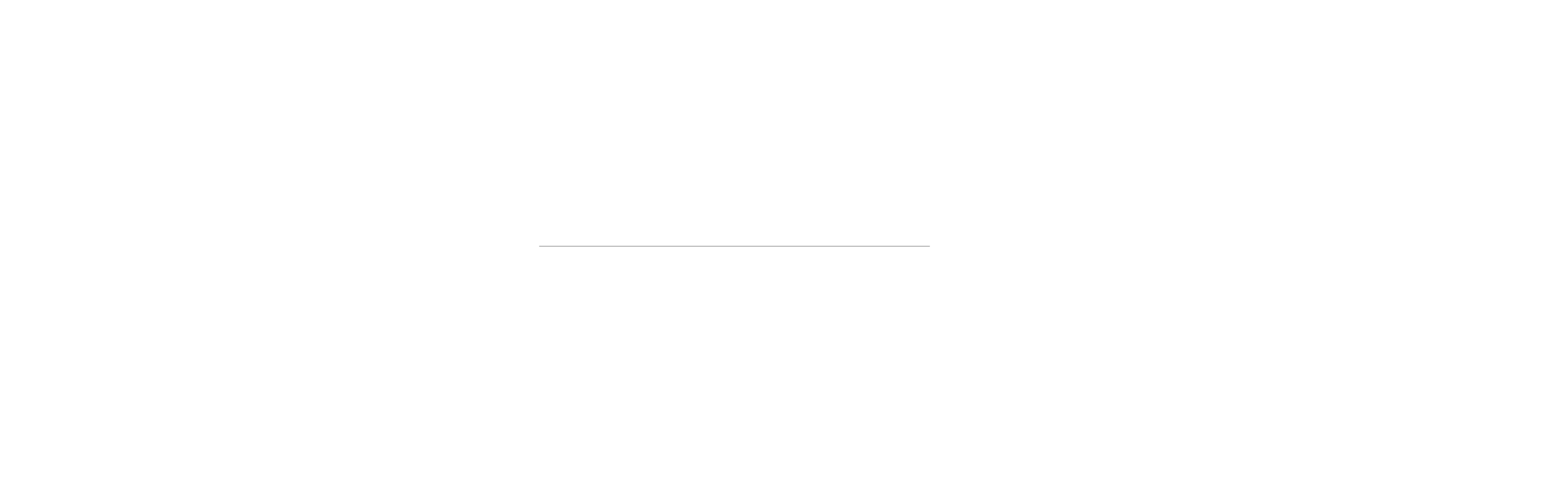}
		\end{tiny}
		\caption{The critical manifold when the folds have the opposite visibility for different values of $\alpha$.}
		\label{fig:OVcritical}
	\end{figure}

	\begin{rem} 
		When $\alpha$ tends to zero the tangency points $T_{X}^\alpha$ and $T_{Y}^\alpha$ meet in the fold-fold singularity. 
		Therefore, when the folds have the same visibility, the critical manifold $\Lambda^{\alpha}_0$ tends to the vertical line $C_0=\{ x=0 \}$ which is not hyperbolic. 
		When the folds have opposite visibility, the two critical manifolds $\Lambda^{\alpha,s}_0$ and  $\Lambda^{\alpha,u}_0$ join in the degenerated hyperbola 
		$C_0 \cup \Lambda_0^s \cup \Lambda^u_0 $ (see \autoref{fig:OVcritical+}, \autoref{fig:OVcritical-}).
	\end{rem}
	
	\begin{rem} 
		In all cases, the dynamics over the critical manifold for $\alpha \neq 0$ is equivalent to the dynamics of the sliding vector field 
		$(Z^\alpha)^s$ studied in \autoref{sec:revisited}. 
		Therefore, for each fixed $\alpha \neq 0$ and $\e>0$ sufficiently small, the dynamics of the regularized vector field is faithful 
		to the dynamics of the unfoldings of $Z$ studied in that section.
	\end{rem}
	\begin{rem}   \label{rem:vsbvposition} 
		The points $v^*$, given in \ref{eq:criticalpoint},  and $\bar v$ given in \ref{IVslowmanifold} satisfy the following relation
		$$ 
		\varphi(v^*) - \varphi(\bv)  = C (\det{Z})_x(\0), 
		$$
		where $C= \displaystyle{\frac{2}{(X^1-Y^1)(\0) (X^2_x-Y^2_x)(\0)}}>0$.  
		
		Since $\varphi$ is an increasing map, we have that $-1 < \bar v < v^* <1 $, if $(\det{Z})_x(\0)>0$ and $-1 < v^* < \bar v <1 $, if $(\det{Z})_x(\0)<0$.
		The relative positions of the points $(0,v^ *)$ and $(0,\bar v)$ will be important to describe the global dynamics of the regularized vector field in \autoref{sec:VIunfreg}.
	\end{rem}

%%%%%%
% Section 4
%%%%%%
	
	\section{The bifurcation diagram of the regularized vector field \texorpdfstring{$Z^\alpha_\e$}{Zae}}\label{regularizationunf}
	
	The aim of this section is to understand the relation between the bifurcation diagram of the versal unfolding $Z^\alpha$ of
	$Z \in \Lambda^F$ and its regularization $Z^\alpha_\e$. 
	
	As we will see in section \ref{sec:VVunfreg} the dynamics of the regularized vector field $Z^\alpha_\e$ is very similar to the dynamics of
	the unfolding $Z^\alpha$ in the case of the
	visible-visible fold. When we study the invisible-invisible fold  in section  \ref{sec:IIunfreg} and the visible-invisible one in section
	\ref{sec:VIunfreg} we will see that the regularization may create new periodic orbits and bifurcations which were not present in the unfolding $Z^\alpha$.
	
	\subsection{Visible-visible case} \label{sec:VVunfreg}
	
	When both folds are visible, by \autoref{corol:toptype} the critical point $P(\alpha,\e)$ is a saddle which is $\e-$close to the pseudo-equilibrium $Q(\alpha)$.
	Using the results in \autoref{sssec:SVcritical} and applying the Fenichel Theorem, for each fixed $\alpha \neq 0$ and any compact set between the fold points
	$(T^\alpha_X,1)$ and $(T^\alpha_Y,-1)$, for $0<\e<\e_0(\alpha)$, there exists a normally hyperbolic invariant manifold
	$\Lambda^\alpha_\e$ which is $\e-$close to $\Lambda^\alpha_0$ (see \autoref{fig:SVcriticalVV}).
	Moreover, for $\alpha<0$, $\Lambda^\alpha_\e=\Lambda^{\alpha,u}_\e$ is repelling and is the stable manifold of the saddle point $P(\alpha,\e)$ and for
	$\alpha>0$, $\Lambda^\alpha_\e=\Lambda^{\alpha,s}_\e$ is attracting and is its the unstable manifold.
	
	A simple computation shows that for each fixed $\alpha \neq 0$ and for $\e > 0$  the vector field $X^{\alpha}(x,\e v)$ has a unique visible fold point at
	$v=1$ at $T^{\alpha, \e}_{X}=T^\alpha_{X} + \mathcal{O}(\e)$ (see \eqref{fXexpression}).
	Analogously, the vector field $Y^\alpha(x,\e v)$ has a unique visible fold at $v=-1$ at the point $T^{\alpha, \e}_{Y}=T^\alpha_{Y} + \mathcal{O}(\e)$.
	Moreover
	\begin{align}
	T^{\alpha, \e}_{X} =& - \left(\frac{\tilde{X}^2}{X^2_x}(\0) \right) \alpha - \left( \frac{X^2_y}{X^2_x}(\0) \right) \e
	+ \mathcal{O}_2(\alpha,\e), \label{epsXtangency} \\
	T^{\alpha, \e}_{Y} =& - \left(\frac{\tilde{Y}^2}{Y^2_x}(\0) \right) \alpha - \left( \frac{Y^2_y}{Y^2_x}(\0) \right) \e
	+ \mathcal{O}_2(\alpha,\e). \label{epsYtangency}
	\end{align}
	
	Observe that for $x < T^{\alpha, \e}_X$ the vector $X^\alpha(x,1)$ points inward to the regularization zone and points outwards to the regularization zone for
	$x>T^{\alpha, \e}_X$.
	Analogously, for $x < T^{\alpha, \e}_Y$ the vector $Y^\alpha(x,-1)$ points outwards to the regularization zone for $x<T^{\alpha, \e}_Y$
	and inwards to the regularization zone for $x>T^{\alpha,\e}_Y$.
	
	\begin{figure}[!htb]
		\centering
		\begin{scriptsize}
			\def\svgwidth{0.9\textwidth}
			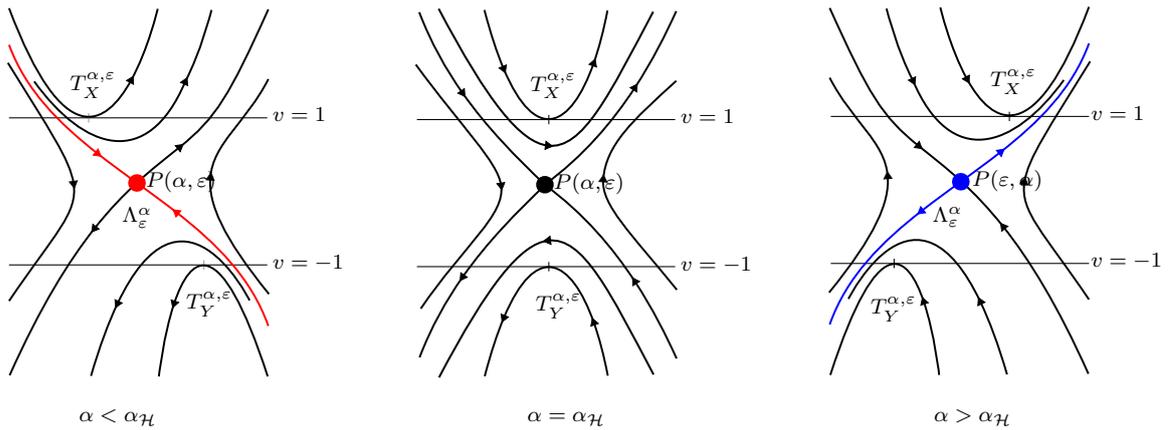
		\end{scriptsize}
		\caption{The phase portrait of $Z^{\alpha}_{\e}$.}
		\label{fig:VVUR2a}
	\end{figure}
	
	The above information and the fact that the dynamics over the Fenichel Ma\-ni\-fold $\Lambda^\alpha_\e$ is equivalent to the one over
	the critical manifold $\Lambda^\alpha_0$, gives:
	\begin{itemize}
		\item
		for $\alpha<0$ and $\e$ small enough, the Fenichel manifold $\Lambda_{\e}^{\alpha}$, which is the stable manifold of the saddle point $P(\alpha,\e)$,
		intersects the section $\{ v=1 \}$ on the left of the tangency point $T^{\alpha, \e}_{X}$ and it intersects the section $\{ v=-1 \}$
		on the right of the tangency point $T^{\alpha, \e}_{Y}$,
		\item
		for $\alpha>0$ and $\e$ small enough, the Fenichel manifold $\Lambda_{\e}^{\alpha}$, which is the unstable manifold of the saddle point $P(\alpha,\e)$,
		intersects the section $\{ v=1 \}$ on the right of the tangency point $T^{\alpha, \e}_{X}$ and it intersects the section $\{ v=-1 \}$
		on the left of the tangency point $T^{\alpha, \e}_{Y}$.
	\end{itemize}
	Observe that, for $\alpha =0$ and $\e$ small enough, even if one can not apply Fenichel theorem, we know that $P(0,\e)$ is a saddle with stable and unstable manifolds.
	By the exposed above,  the phase portrait of $Z^{\alpha}_{\e}$ must look as in \autoref{fig:VVUR2a}.
	
	\begin{rem} Over the curve $\mathcal{H}$, given in \ref{hcurve} the matrix $DZ^\alpha_\e(P(\alpha,\e))$ has two real eigenvalues with same absolute value.
		Therefore, the critical point $P(\alpha,\e)$ is a neutral saddle and the qualitative behavior of the system reminds the behavior of the Filippov system
		$Z^\alpha$ for $\alpha=0$. Then in some sense, the dynamics of $Z$ is ``continued'' over the curve $\mathcal{H}$.
	\end{rem}
	
	\subsection{The invisible-invisible case} \label{sec:IIunfreg}
	
	When both folds are invisible, for each $\alpha \neq 0$ and $\e>0$ small enough, by \autoref{corol:toptype} and in agreement with \cite{SotoTei}, 
	the  point $P(\alpha,\e)$ is a node
	with the same character that the pseudo-node $Q(\alpha)$ of $Z^\alpha$.
	
	Using the results about the critical manifold given in \autoref{sssec:SVcritical}, we can apply the Fenichel Theorem in any compact set between the points
	$(T^\alpha_X,1)$ and $(T^\alpha_Y,-1)$, obtaining  that for $0<\e<\e_0(\alpha)$, there exists a normally hyperbolic invariant manifold
	$\Lambda^\alpha_\e$ which is $\e-$close to $\Lambda^\alpha_0$ (see \autoref{fig:SVcriticalII}).
	Moreover, for $\alpha<0$, $\Lambda^\alpha_\e=\Lambda^{\alpha,s}_\e$ is attracting and for $\alpha>0$, $\Lambda^\alpha_\e=\Lambda^{\alpha,u}_\e$ is repelling.
	In both cases, these manifolds contain the node $P(\alpha,\e)$ and they are its weak manifold.
	
	We now consider the tangency points $T^{\alpha, \e}_{X,Y}$ given in \ref{epsXtangency} and \ref{epsYtangency}, see \autoref{fig:IIUR2}.
	For $x < T^{\alpha, \e}_X$ the vector $X^\alpha(x,1)$ points outward to the regula\-ri\-za\-tion zone and points inwards to the regularization zone for
	$x>T^{\alpha, \e}_X$.
	Analogously, for $x < T^{\alpha, \e}_Y$ the vector $Y^\alpha(x,-1)$ points inwards to the regularization zone and outwards to the
	regularization zone for $x>T^{\alpha,\e}_Y$.
	
	The above information and the fact that the dynamics over the Fenichel Ma\-ni\-fold $\Lambda^\alpha_\e$ is the same of the critical manifold
	$\Lambda^\alpha_0$  gives:
	\begin{itemize}
		\item
		for $\alpha<0$ and $\e$ small enough, the Fenichel manifold $\Lambda_{\e}^{\alpha,s}$, which is the weak manifold of the stable  node $P(\alpha,\e)$,
		intersects the section $\{ v=1 \}$ on the right of the tangency point $T^{\alpha, \e}_{X}$ and it intersects the section $\{ v=-1 \}$
		on the left of the tangency point $T^{\alpha, \e}_{Y}$,
		\item
		for $\alpha>0$ and $\e$ small enough, the Fenichel manifold $\Lambda_{\e}^{\alpha,u}$, which is the weak manifold of the  unstable node $P(\alpha,\e)$,
		intersects the section $\{ v=1 \}$ on the left of the tangency point $T^{\alpha, \e}_{X}$ and it intersects the section $\{ v=-1 \}$
		on the right of the tangency point $T^{\alpha, \e}_{Y}$.
	\end{itemize}
	The phase portrait of the vector field $Z^\alpha_\e$ for $(\alpha,\e)$ below the parabola $\mathcal{D}$ is given in \autoref{fig:IIUR2}.

	\begin{figure}[!htb]
		\centering
		\begin{scriptsize}
			\def\svgscale{0.4}
			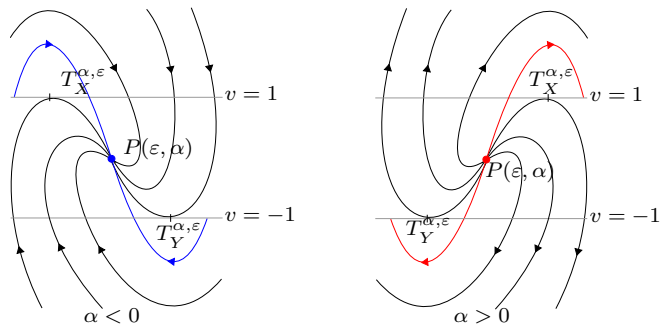
		\end{scriptsize}
		\caption{Phase portrait of $Z^{\alpha}_\e$ in the invisible-invisible case for  $(\alpha,\e)$ below the curve $\mathcal{D}$ for $\alpha \neq 0$.}
		\label{fig:IIUR2}
	\end{figure}

	When $(\alpha,\e)$ are above the parabola $\mathcal{D}$ the point  $P(\alpha,\e)$ is a focus.
	Since it undergoes a Hopf bifurcation, a periodic orbit arises at one side of the curve $\mathcal{H}$.
	The nature of the Hopf bifurcation depends on the first Lyapunov coefficient $\ell_1(\e,\alpha(\e))$ (see \cite{holmes}, p. 152).

	Now we are going to investigate the persistence of the crossing cycle $\Gamma^\alpha$ which appears in the unfolding $Z^\alpha$ when $\alpha \cdot \G_Z >0$ as seen in
	\autoref{prop:IICunfolding}.
	We will also study the relation between the limit cycle which raises from the Hopf bifurcation and the periodic orbit $\Gamma^{\alpha, s}_{\e}$ which 
	raises from
	the crossing cycle $\Gamma^\alpha$.
	
	\begin{prop}  \label{prop:phiaeZ}
		Consider $Z \in \Lambda^F$ having an invisible fold-fold  satisfying $X^1\cdot Y^1(\0)<0$. Consider the coefficient $\G_Z$ in \eqref{gencondii}.
		If $\G_Z>0$, one has that:
		\begin{enumerate}[(a)]
			\item
			For $\alpha, \e>0$ sufficiently small and such that $(\alpha,\e)$ are on the right of the curve $\mathcal{H}$ (see \ref{hcurve}), there exists a stable
			periodic orbit $\Gamma^{\alpha,s}_{\e}$ of the vector field $Z^{\alpha}_{\e}$;
		\end{enumerate}
		Moreover,
		\begin{enumerate}[(b1)]
			\item
			If $\alpha >0$, there exists $\e_{0}(\alpha)$ such that for $0<\e<\e(\alpha)$, $\Gamma^{\alpha, s}_{\e}$ is the unique stable hyperbolic periodic orbit of
			$Z^\alpha_\e$.
			Furthermore, $\Gamma^{\alpha, s}_{\e} = \Gamma^\alpha + \mathcal{O}(\e)$, where $\Gamma^\alpha$ is given in \autoref{prop:IICunfolding}.
			\item
			If $\alpha<0$, the system $Z^\alpha_\e$ has no periodic orbits for $0<\e<\e(\alpha)$ .
			
		\end{enumerate}
		\begin{enumerate}[(c1)]
			\item 
			There exists a Melnikov  function $M(v,\delta)$  given in \eqref{eq:melnikov}, whose properties are given in \autoref{prop:propiedadesM}, such that 
			for  $\alpha = \delta \e + O(\e^2)$, such that  $\delta>\delta_{\mathcal{H}}$   and $\e$ small enough,
			calling
			$(v_\e^s,0) = \Gamma^{\alpha, s}_{\e}\cap \{x=0\}= (v^s,0)+ O(\e)$, the value $v^s$
			satisfies:
			$$
			M(v^s,\delta)=0, \ \frac{\partial M}{\partial v} (v^s,\delta) \le 0,
			$$
			and $\Gamma^{\alpha, s}_{\e}$ is locally unique if $\displaystyle{\frac{\partial M}{\partial v} (v^s,\delta)<0}$. 
			\item
			Moreover, if $M(v,\delta)$ is strictly concave ($\frac{\partial ^2M}{\partial v^2} (v,\delta) < 0$) then the periodic orbit 
			$\Gamma^{\alpha,s}_\e$ is unique and disappears at $\alpha_\mathcal{H}$.
		\end{enumerate}
		If $\G_Z<0$ one has an analogous results changing signs of the parameters.
	\end{prop}
	
	\begin{proof}
		The proof of this proposition can be found in \autoref{ssec:propphiaZproof}.
	\end{proof}

	Next we analyze the relation between the stable periodic orbit $\Gamma^{\alpha,s}_{\e}$ and a periodic orbit which arises from the
	Hopf bifurcation.
	
	The following theorems give us the bifurcation diagram of $Z^\alpha_\e$ in each case, depending on the signs of
	$\G_Z$ and the first Lyapunov coefficient $\ell_1(\alpha(\e),\e)$.

	\begin{theo}[Invisible fold-fold: \texorpdfstring{$\G_Z>0$}{GZ>0} and \texorpdfstring{$\ell_1(\alpha(\e),\e)<0$}{l<0}] \label{thm:bdIIsuper}
		Let $Z \in \Lambda^F$ having an invisible fold-fold satisfying $X^1\cdot Y^1(\0)<0$ and  $\G_Z>0$.
		Suppose that the first Lyapunov coefficient $\ell_1(\alpha(\e),\e)$ at the Hopf bifurcation of $Z^\alpha_\e$ is negative.
		Let $\alpha_\mathcal{D}^\pm(\e_0)$ and $\alpha_\mathcal{H}(\e_0)$ be the intersections between the line $\e=\e_0$ and the curves
		$\mathcal{D}^\pm$ (the negative and positive parts of $\mathcal{D}$ given in \ref{dcurve}) and $\mathcal{H}$ (\ref{hcurve}), respectively.
		One has that
		\begin{itemize}
			\item
			For $\alpha<\alpha_\mathcal{D}^-(\e_0)$ sufficiently small the critical point $P(\alpha,\e)$ is an stable node;
			\item
			For $\alpha_\mathcal{D}^-(\e_0)<\alpha<\alpha_\mathcal{H}(\e_0)$ the critical point $P(\alpha,\e)$ is an stable focus;
			\item
			When $\alpha=\alpha_\mathcal{H}(\e_0)$ a supercritical Hopf bifurcation takes place;
			\item
			For $\alpha$ values such that $\alpha_\mathcal{H}(\e_0)<\alpha<\alpha_\mathcal{D}^+(\e_0)$,
			the critical point $P(\alpha,\e)$ is an unstable focus and there exist a stable limit cycle $\Gamma^{\alpha, s}_{\e}$.
			\item
			When  $\alpha>\alpha_\mathcal{D}^+(\e_0)$ the critical point $P(\alpha,\e)$ is a unstable node and the limit cycle $\Gamma^{\alpha, s}_{\e}$
			persists for $\alpha>\alpha_\mathcal{D}^+(\e_0)$.
			The cycle $\Gamma^{\alpha, s}_{\e}$ tends to $\Gamma^\alpha$ when $\e$ goes to zero.
			\item
			If the Melnikov function $M(v,\delta)$ is strictly concave for $\delta$ close enough to $\delta_\mathcal{H}$ the stable periodic 
			$\Gamma^{\alpha, s}_{\e}$ disappears at the Hopf bifurcation.
			Therefore, the periodic orbit which rises from the Hopf bifurcation and the one which is given by \autoref{prop:phiaeZ} are the same.\end{itemize}
	\end{theo}
	
	\begin{proof}
		The character of the critical point $P(\alpha,\e)$ is given by \autoref{corol:toptype}.
		The existence of the periodic orbit $\Gamma^{\alpha, s}_{\e}$ is given by \autoref{prop:phiaeZ}.
		Moreover, when the first Lyapunov coefficient $\ell_1(\e,\alpha(\e))$ is negative, a supercritical Hopf bifurcation occurs,
		creating a stable periodic orbit $\tilde \Gamma^{\alpha,s}_{\e}$ near the critical point $P(\alpha,\e)$ for $(\alpha,\e)$ to the right of the curve $\mathcal{H}$.
		As both periodic orbits are given, in first order, by the zeros of $M(v,\delta)$, when this function is strictly concave both periodic orbits have to coincide.
	\end{proof}
	
	\begin{figure}[!htb]
		\centering
		\begin{tiny}
			\def\svgwidth{0.9\textwidth}
			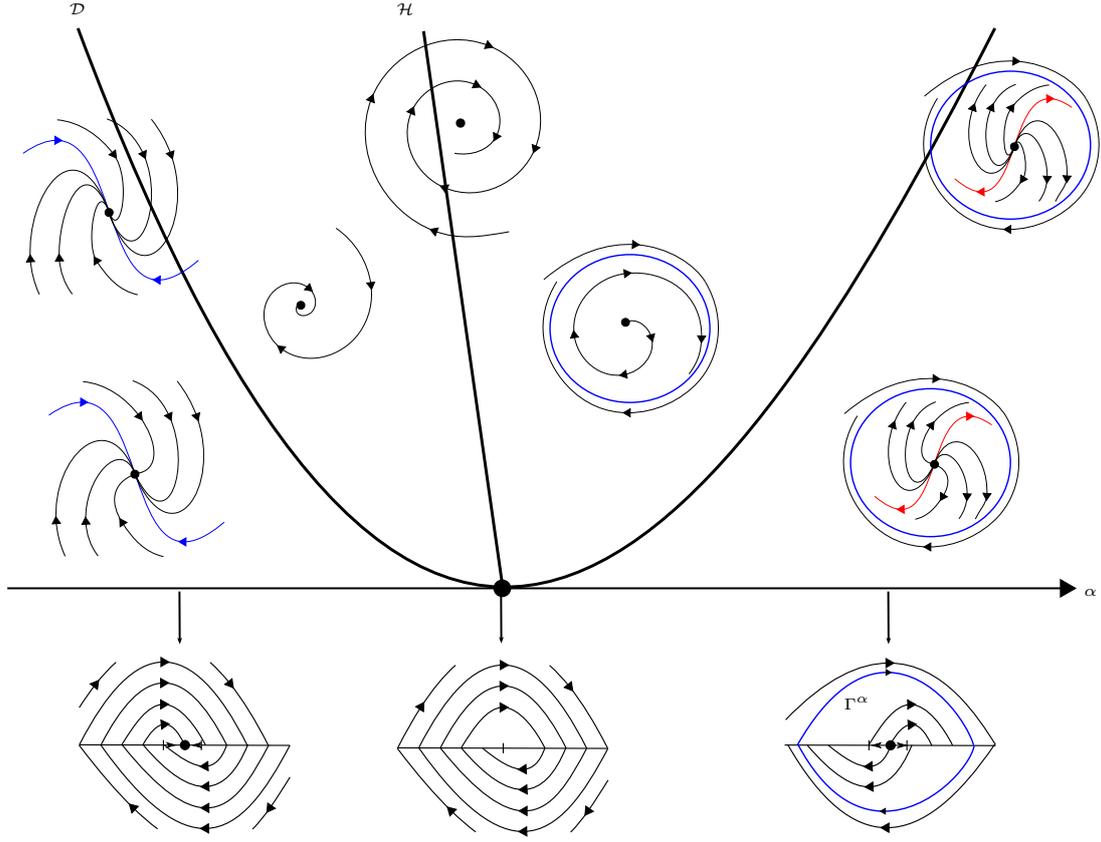
		\end{tiny}
		\caption{Bifurcation diagram of $Z^\alpha_\e$ when $\ell_1(\e,\alpha(\e))<0$ and $Z$ has an invisible fold-fold  with $\G_Z>0$. A stable limit cycle exists for 
			$\alpha>\alpha_\mathcal{H}$.}
		\label{fig:IIUR4}
	\end{figure}
	
	The  cycle $\Gamma^{\alpha, s}_{\e}$, given by \autoref{thm:bdIIsuper}, tends to the non smooth crossing cycle
	$\Gamma^\alpha$ when $\e$ tends to zero.
	Summarizing, the small orbit arising from the Hopf bifurcation and the regularized periodic orbit coming from the pseudo-Hopf
	bifurcation of the non smooth system are, generically, continuation one of the other.
	
	The bifurcation diagram of $Z^{\alpha}_\e$ in the two parameter space is sketched in \autoref{fig:IIUR4}.

	\begin{theo} [$\G_Z>0$ and $\ell_1(\alpha(\e),\e)>0$] \label{thm:bdIIsub}
		Assume the same hypothesis of \autoref{thm:bdIIsuper} but $\ell_1(\alpha(\e),\e)$ positive. Then one has that:
		\begin{itemize}
			\item
			For $\alpha<\alpha_\mathcal{D}^-(\e_0)$ sufficiently small the critical point $P(\alpha,\e)$ is a stable node;
			\item
			For $\alpha_\mathcal{D}^-(\e_0)< \alpha<\alpha_\mathcal{H}(\e_0)$ the critical point $P(\alpha,\e)$ is a stable focus;
			\item
			There exists a curve $\mathcal{S}$ in the parameter plane such that for $\alpha<\alpha_\mathcal{S}(\e_0)$ the generalized Poincar\'{e} return map $\phi^{\alpha}_{\e}$
			has no fixed points.
			\item
			For $\alpha_\mathcal{S}(\e_0)<\alpha<\alpha_\mathcal{H}(\e_0)$, the critical point $P(\alpha,\e)$, which is a stable focus,
			coexists with a pair of periodic orbits  $\Gamma^{\alpha, s}_{\e}$ and $\Delta^{\alpha,u}_{\e}$, which are stable and unstable, respectively.
			For $\alpha=\alpha_\mathcal{H}(\e_0)$ a subcritical Hopf bifurcation takes place.
			This implies the disappearance of the unstable periodic orbit $\Delta^{\alpha,u}_{\e}$ for $\alpha>\alpha_\mathcal{H}$;
			\item
			For $\alpha>\alpha_\mathcal{H}(\e_0)$, the stable limit cycle $\Gamma^{\alpha, s}_{\e}$ persists.
			Moreover, the critical point $P(\alpha,\e)$ is a unstable focus $\alpha_\mathcal{H}(\e_0)<\alpha<\alpha_\mathcal{D}^+(\e0)$
			and becomes a unstable node when $\alpha>\alpha_\mathcal{D}^+(\e_0)$. 
			The cycle $\Gamma^{\alpha, s}_{\e}$ tends to $\Gamma^\alpha$ when $\e$ goes to zero.
		\end{itemize}
	\end{theo}
	\begin{proof}
		The character of the critical point $P(\alpha,\e)$ is given by \autoref{corol:toptype}.
		The existence of the stable periodic orbit $\Gamma^{\alpha, s}_{\e}$ for $\alpha>\alpha_\mathcal{H}(\e_0)$,  is given by \autoref{prop:phiaeZ}.
		When the parameter values reach the curve $\mathcal{H}$, a subcritical Hopf bifurcation occurs and an unstable periodic orbit
		$\Delta^{\alpha,u}_{\e}$ appears for $\alpha<\alpha_\mathcal{H}(\e_0)$.
		Observe that, for these parameter values, because of the attracting character of the generalized first return map $\phi^{\alpha}_\e$ far from the origin,
		and the presence of the unstable periodic orbit $\Delta^{\alpha,u}_{\e}$, the Poincar\'{e}-Bendixson theorem guarantees
		%, generically,
		the persistence of the periodic orbit $\Gamma^{\alpha, s}_{\e}$ for $\alpha<\alpha_\mathcal{H}(\e_0)$ small enough.
		On the other hand, if we fix $\alpha<0$, \autoref{prop:phiaeZ} says that system $Z^\alpha_\e$ has no periodic orbit for $\e<\e(\alpha)$.
		Therefore, must exist a curve $\mathcal{S}$ in the parameter space where the ``total'' first return map $\phi^{\alpha}_\e$ has a 
		bifurcation.
		Therefore there exists a value $\alpha_\mathcal{S}(\e_0)$, where these two periodic orbits collide for $\alpha=\alpha_\mathcal{S}(\e_0)$
		and then disappear for $\alpha< \alpha_\mathcal{S}(\e_0)$.
	\end{proof}
	
	The next proposition, whose proof is given in \autoref{sec:melnikov}, 
	provides quantitative information about the periodic orbits given in theorem \ref{thm:bdIIsub} in terms of the Melnikov function  $M(v,\delta)$,.

	\begin{prop}[The Saddle-node bifurcation] \label{prop:melnikov}
		Consider the  Melnikov  function $M(v,\delta)$  given in \eqref{eq:melnikov}. Then, under the hypotheses of \autoref{thm:bdIIsub}:
		\begin{itemize}
			\item
			For  $\alpha = \delta \e + O(\e^2)$, such that  $\delta<\delta_{\mathcal{H}}$ sufficiently close, and  $\e$ small enough, calling
			$(v_\e^u,0) = \Delta^{\alpha, u}_{\e}\cap \{x=0\}= (v^u,0)+ O(\e)$, the value $v^u$
			satisfies:
			$$
			M(v^u,\delta)=0, \ \frac{\partial M}{\partial v} (v^u,\delta) >0,
			$$
			\item
			If moreover we assume that $\frac{\partial^3 M}{\partial v^3} (v,\delta) \ne 0$ for $\delta$ near $\delta_\mathcal{H}$, 
			for $ \delta_\mathcal{S}<\delta<\delta_\mathcal{H}$, the function $M$ has  two zeros corresponding to the periodic orbits
			$\Delta^{\alpha,u}_{\e}$ and $\Gamma^{\alpha, s}_{\e}$ given in \autoref{thm:bdIIsub}.
			When $\alpha=\alpha_\mathcal{S}(\e_0)$ a Saddle-Node bifurcation  takes place.
			\item
			The curve $\mathcal{S}$ can be found as  $\alpha=\alpha_\mathcal{S}= \delta_\mathcal{S}\e + O(\e^2)$, where $\delta_\mathcal{S}$ is the solution of the
			(linear in $\delta$) equations:
			$$
			M(v_\mathcal{S},\delta_\mathcal{S})=0, \ \frac{\partial M}{\partial v} (v_\mathcal{S},\delta_\mathcal{S})=0
			$$
		\end{itemize}
	\end{prop}
	
	\begin{figure}[!htb]
		\centering
		\begin{tiny}
			\def\svgwidth{0.9\textwidth}
			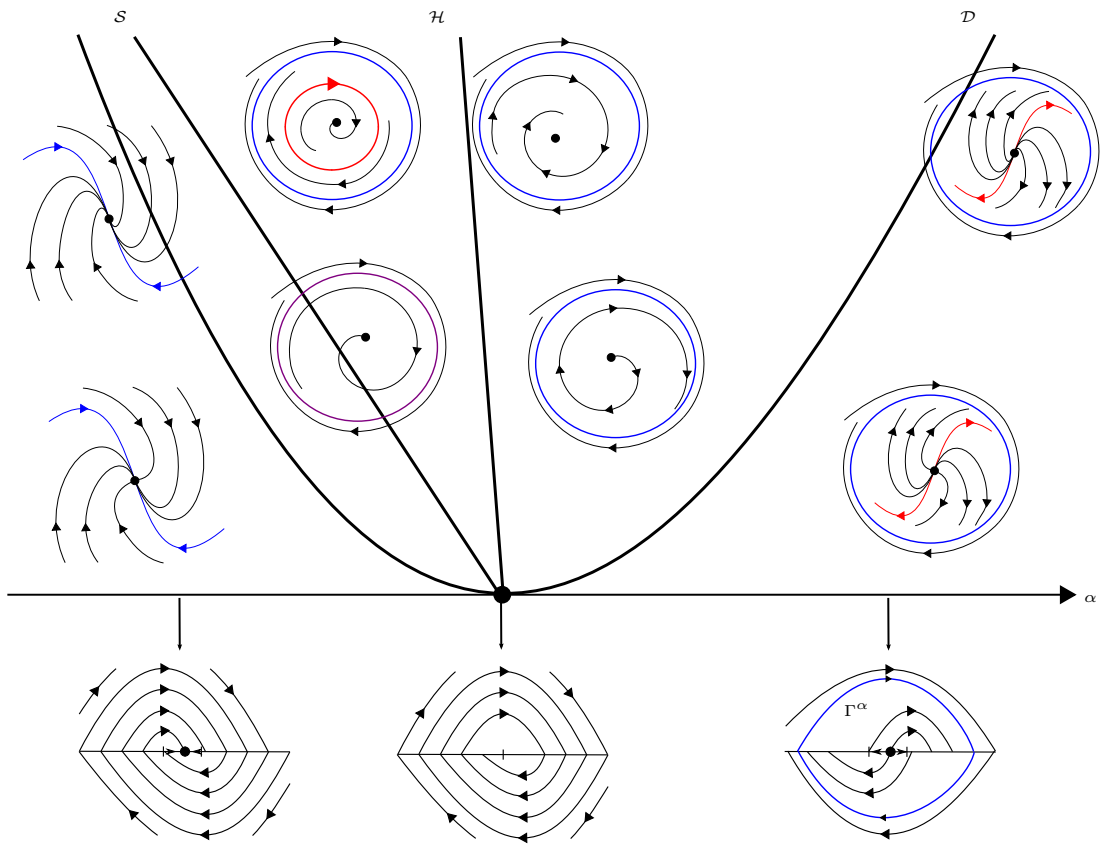
		\end{tiny}
		\caption{Bifurcation diagram of $Z^\alpha_\e$ when $\ell_1(\e,\alpha(\e))>0$ and $Z$ has an invisible-invisible fold  with $\G_Z>0$.
			Two periodic orbits arise at a saddle-node bifurcation, the unstable one disappears at the Hopf bifurcation.}
		\label{fig:IIUR6}
	\end{figure}

	In the next two examples, which satisfy \autoref{versal} and therefore are versal unfoldings of the fold-fold singularity,
	we illustrate the behaviors stated in Theorems \ref{thm:bdIIsuper} and \ref{thm:bdIIsub}.
	We also show that both behaviors can be achieved by the same piecewise vector field considering different transition maps $\varphi$.
	For this purpose, let  $Z^\alpha=(X,Y)$ be the piecewise vector field given by
	\begin{equation} \label{ex:IIns}
	Z^{\alpha}(x,y) = \begin{cases}
	X^\alpha(x,y) = \left(1-2 x, -x + \alpha \right) \\
	Y(x,y)= \left(-1,-x \right)
	\end{cases}
	\end{equation}
	which has an invisible-invisible fold at the origin for $\alpha=0$.
	The coefficient $\G_Z=\frac{4}{3}$, therefore, the origin is an stable fixed point for the Poincar\'{e} map $\phi_Z$.
	
	The $\varphi-$regularization $Z^\alpha_\e(x,y)$ in coordinates $y=\e v$, is given by
	\begin{equation} \label{ex:IIreg}
	\begin{cases}
	\dot{x}= -2x + 2\varphi(v)(1-x) \\
	\e \dot{v} = -2x+ \alpha(1+ \varphi(v))
	\end{cases} \ |v| \leq 1.
	\end{equation}
	
	\begin{exmp}[Supercritical Hopf bifurcation for the invisible-invisible fold] \label{ex:IIsuper} Consider  $\varphi$ as in \ref{transfunc} with
		\begin{equation} \label{phi3}
		\varphi(v)=- \frac{1}{2}v^3+\frac{3}{2}v, \, \textrm{for} \, v \in (-1,1).
		\end{equation}
		The critical point is
		$
		P(\alpha,\e)=\left( \frac{1}{2}\alpha,0 \right) + \mathcal{O}_2(\alpha,\e)
		$, and the curves $\mathcal{D}$ and $\mathcal{H}$ are given by:
		\begin{eqnarray*}
			\mathcal{D}&=&\left\{ (\alpha,\e) : \e = \frac{3}{32}\alpha^2 + \mathcal{O}(\alpha^3) \right\}  \\
			\mathcal{H}&=&\left\{ (\alpha,\e) : \alpha = \frac{4}{3}\e + \mathcal{O}(\e^2) \right\}
		\end{eqnarray*}
		The first Lyapunov coefficient is 
		$
		\ell_1(\alpha(\e),\e)=\frac{1}{\sqrt{\e}} \left( -\frac{1}{3 \sqrt{6}} + \mathcal{O}(\e) \right)
		$, therefore  the Hopf bifurcation is supercritical.
		
		\begin{figure}[htb!]
			\centering
			\subfigure[\label{fig:IISuper-a-03}]{\includegraphics[scale=0.55]{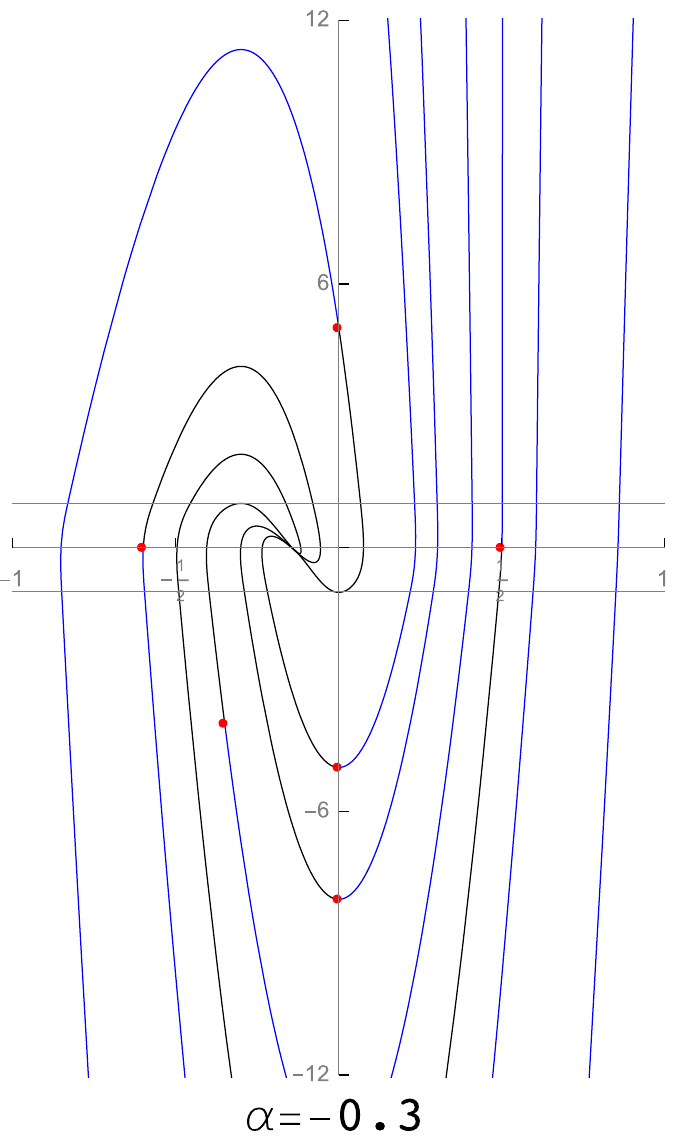}}  \hfill
			\subfigure[\label{fig:IISuper-a-01}]{\includegraphics[scale=0.55]{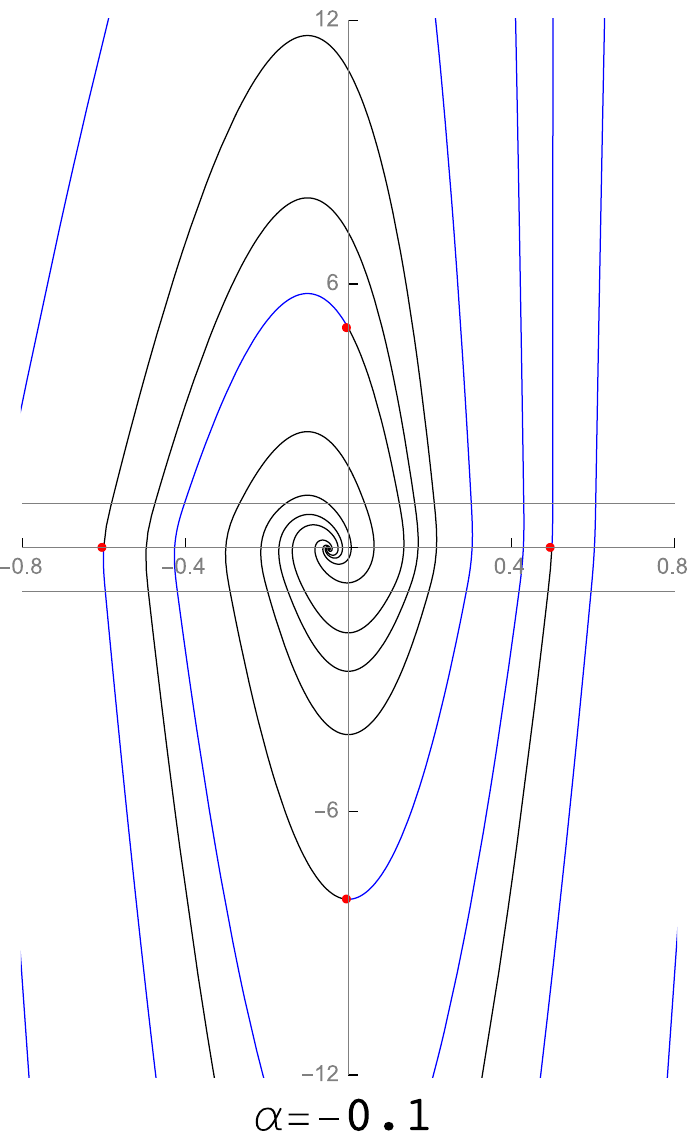}}  \hfill
			\subfigure[\label{fig:IISuper-a0006}]{\includegraphics[scale=0.55]{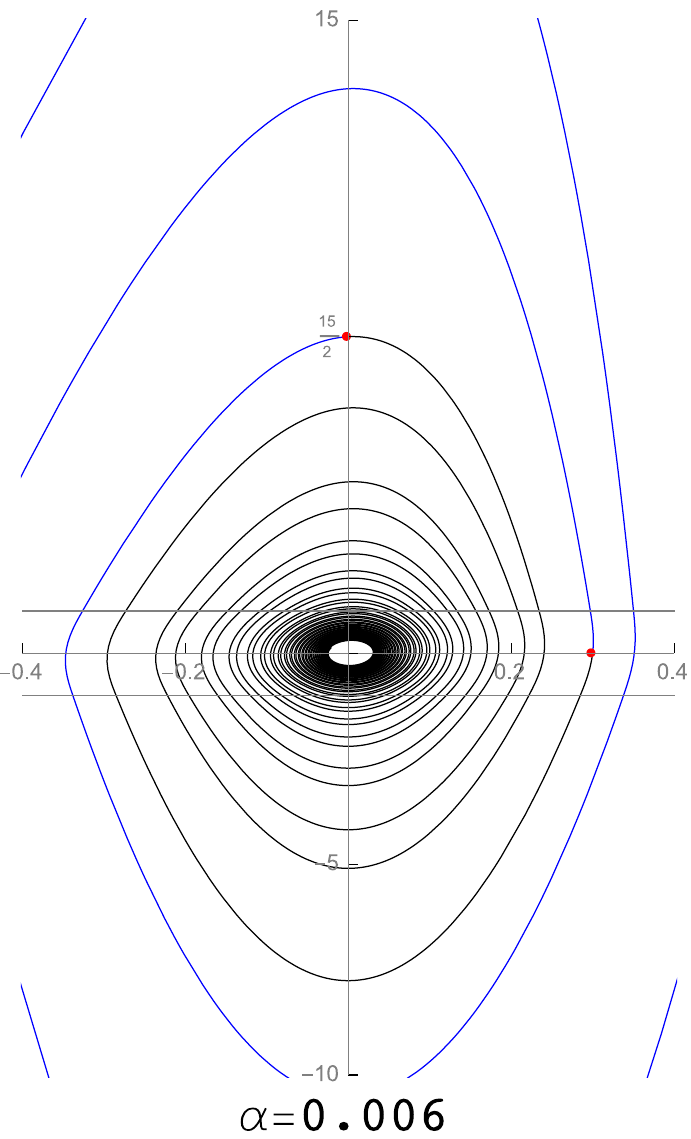}} \hfill \\
			\subfigure[\label{fig:IISuper-a001}]{\includegraphics[scale=0.55]{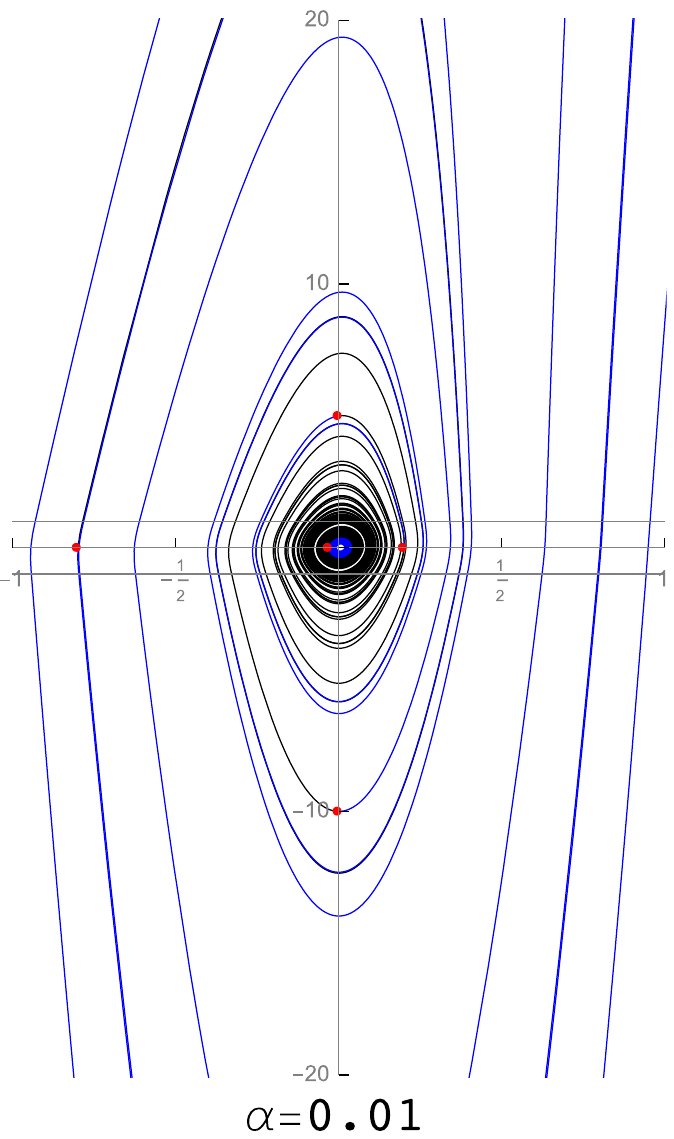}} \hfill
			\subfigure[\label{fig:IISuper-a01}]{\includegraphics[scale=0.55]{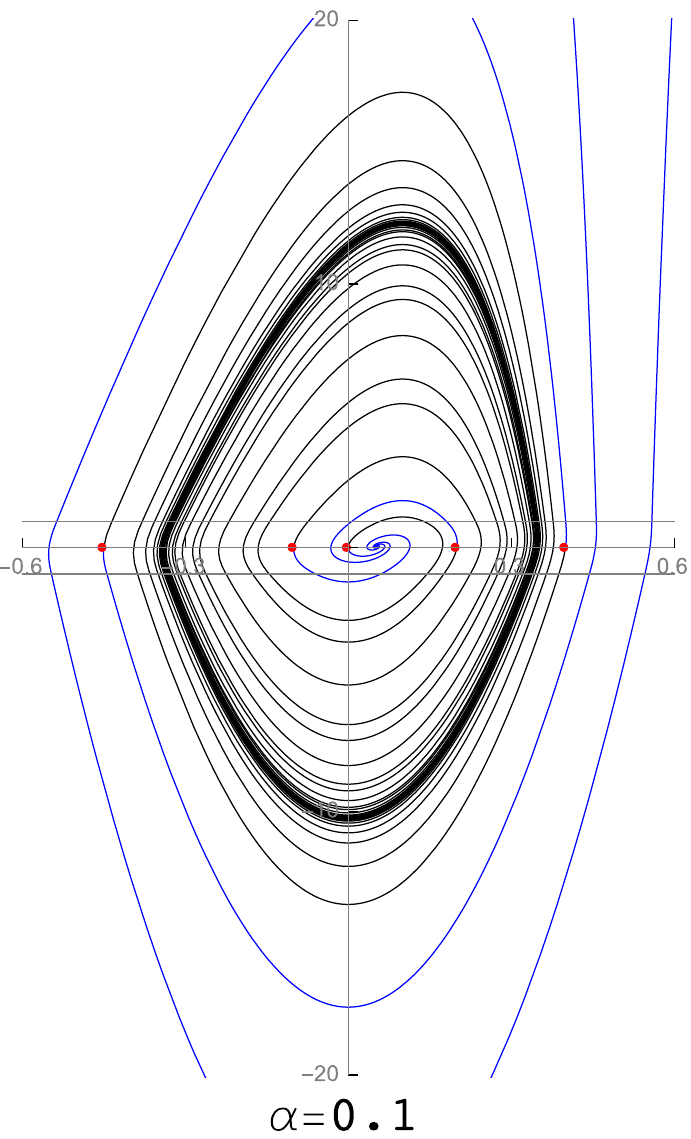}} \hfill
			\subfigure[\label{fig:IISuper-a03}]{\includegraphics[scale=0.55]{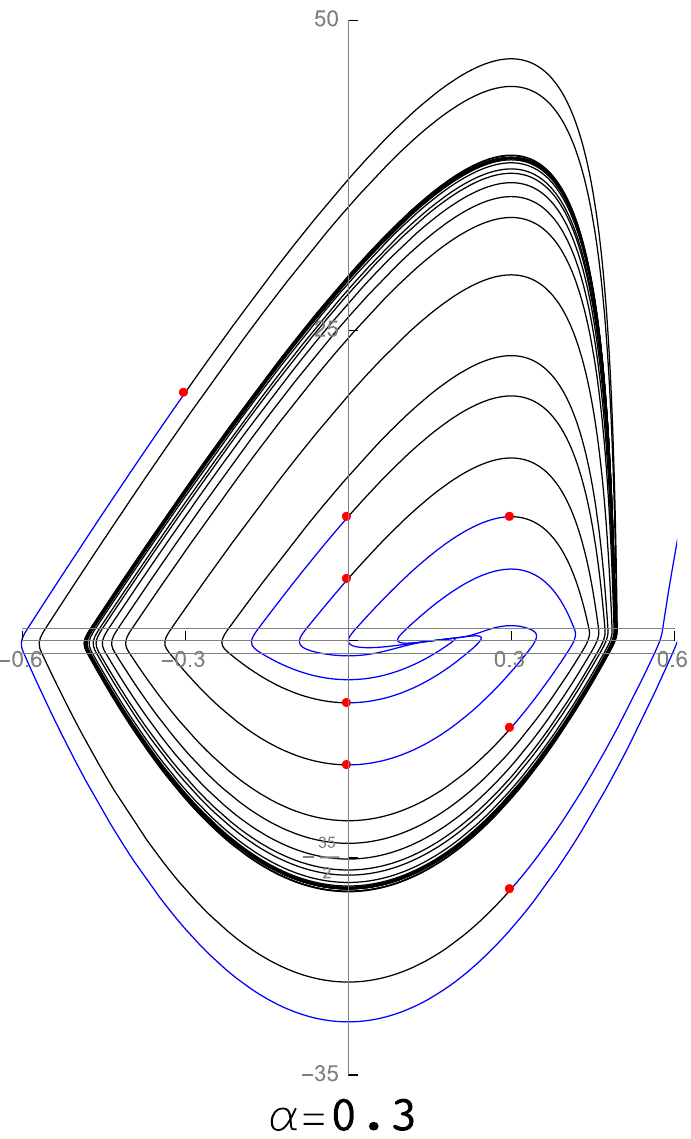}} \\
			\begin{scriptsize}
				\textcolor{red}{$\blacksquare$}  Initial condition \quad  \textcolor{blue}{$\blacksquare$} Negative time \quad  \textcolor{black}{$\blacksquare$} Positive Time
			\end{scriptsize}
			\caption{\autoref{ex:IIsuper} - The evolution of the dynamics while increasing the value of $\alpha$ for  $\e=0.006$.}
			\label{fig:IISuper}
		\end{figure}

		We fix $\e=0.006$ and vary the parameter $\alpha$.
		The intersection of the line $\e=0.006$ with the parabola $\mathcal{D}$   occurs at the points
		$ \alpha^\pm_{\mathcal{D}} \approx \pm 0.25$ and the intersection with $\mathcal{H}$  at $\alpha_{\mathcal{H}}\approx 0.008.$

		In \autoref{fig:IISuper} we can see the changes on the phase portrait of $Z^\alpha_\e$ when we vary the parameter $\alpha.$
		In \autoref{fig:IISuper-a-01}, $\alpha>\alpha^-_\mathcal{D}$: the node becomes a stable focus.
		In \autoref{fig:IISuper-a0006}, $\alpha_\mathcal{D}^-<\alpha < \alpha_\mathcal{H}$: the stable focus begins to loose strength.
		In \autoref{fig:IISuper-a001}, $\alpha>\alpha_{\mathcal{H}}$: the critical point is an unstable focus and a small stable limit cycle $\Gamma^{\alpha, s}_{\e}$
		inside the regularization zone appears.
		In \autoref{fig:IISuper-a01} the stable limit cycle $\Gamma^{\alpha, s}_{\e}$ is no more located inside the regularization zone.
		In \autoref{fig:IISuper-a03}, $\alpha>\alpha^+_{\mathcal{D}}$: the critical point becomes an unstable node and the limit cycle still $\Gamma^{\alpha, s}_{\e}$
		persists outside the regularization zone.
		In \autoref{fig:IISupermelnikov} we show  the behavior of the Melnikov function $M(v;\delta)$, which is strictly concave, 
		for different values of $\delta$ and which has a zero
		for $\delta >\delta _\mathcal{H}$ corresponding to $\Gamma ^{\alpha,s}_\e$.
	\end{exmp}

	\begin{exmp}[Subcritical Hopf bifurcation for the invisible-invisible fold]  \label{ex:IIsub} 
		Consider the transition map
		\begin{equation} \label{phi5}
		\varphi(v)=-v^5+\frac{3}{2}v^3+\frac{v}{2} \, \textrm{for} \, v \in (-1,1).
		\end{equation}
		
		The critical is point $P(\alpha,\e)=\left(\frac{1}{2}\alpha, 0 \right) + \mathcal{O}_2(\alpha,\e)$, and the bifurcation curves are given by
		\begin{eqnarray}
		\mathcal{D} &=& \left\{ (\alpha,\e) : \, \e = \frac{1}{32} \alpha^2 + \mathcal{O}(\alpha^3)   \right\} \\
		\mathcal{H} &=& \left\{ (\alpha,\e) : \, \alpha = 4 \e + \mathcal{O}(\e^2) \right\}
		\end{eqnarray}
		
		\begin{figure}[htb!]
			\centering
			\subfigure[\label{fig:ExIISuba0}]{\includegraphics[scale=0.6]{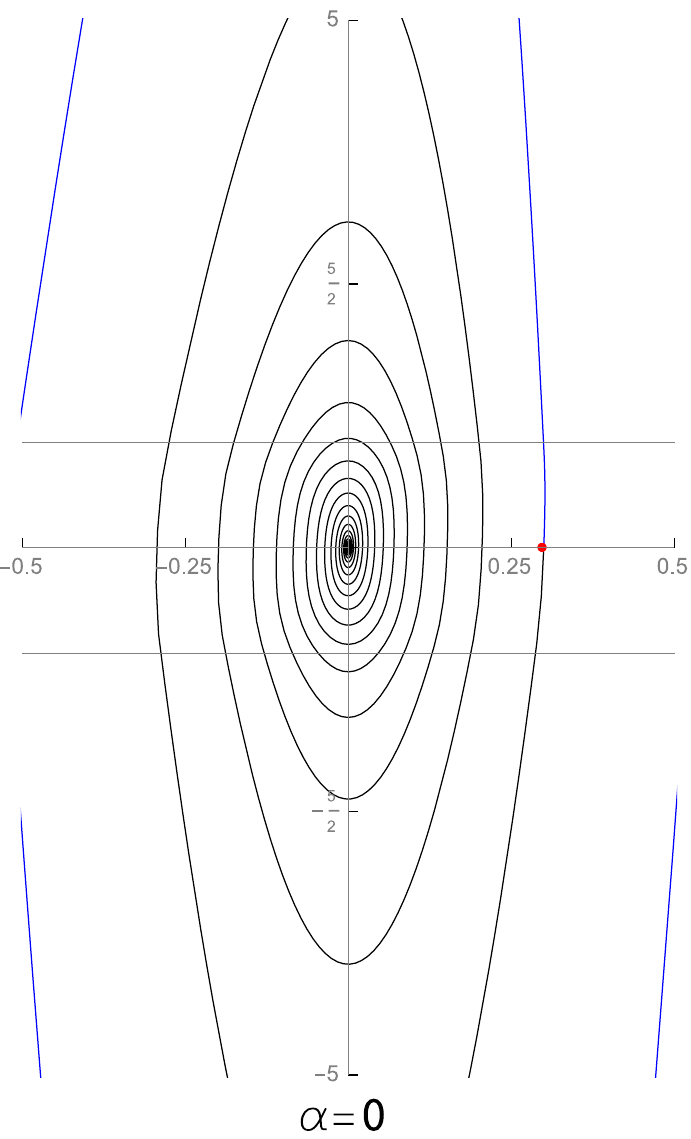}} \hspace{1cm}
			\subfigure[\label{fig:ExIISuba0012}]{\includegraphics[scale=0.6]{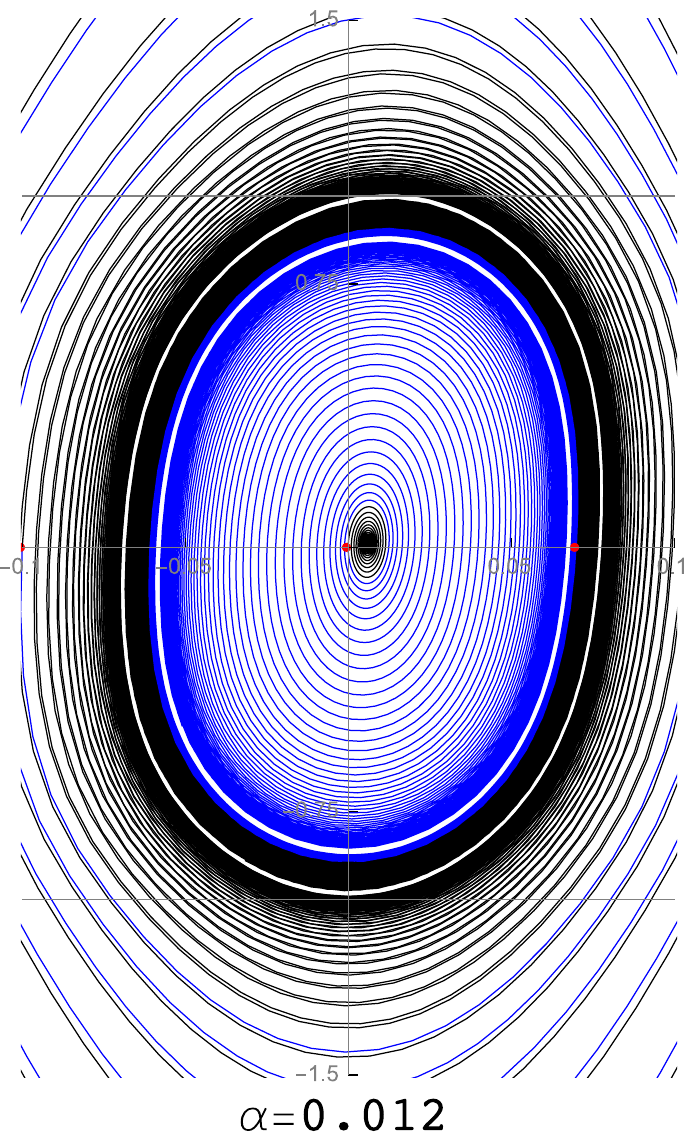}} \hspace{1cm}
			\subfigure[\label{fig:ExIISuba0014}]{\includegraphics[scale=0.6]{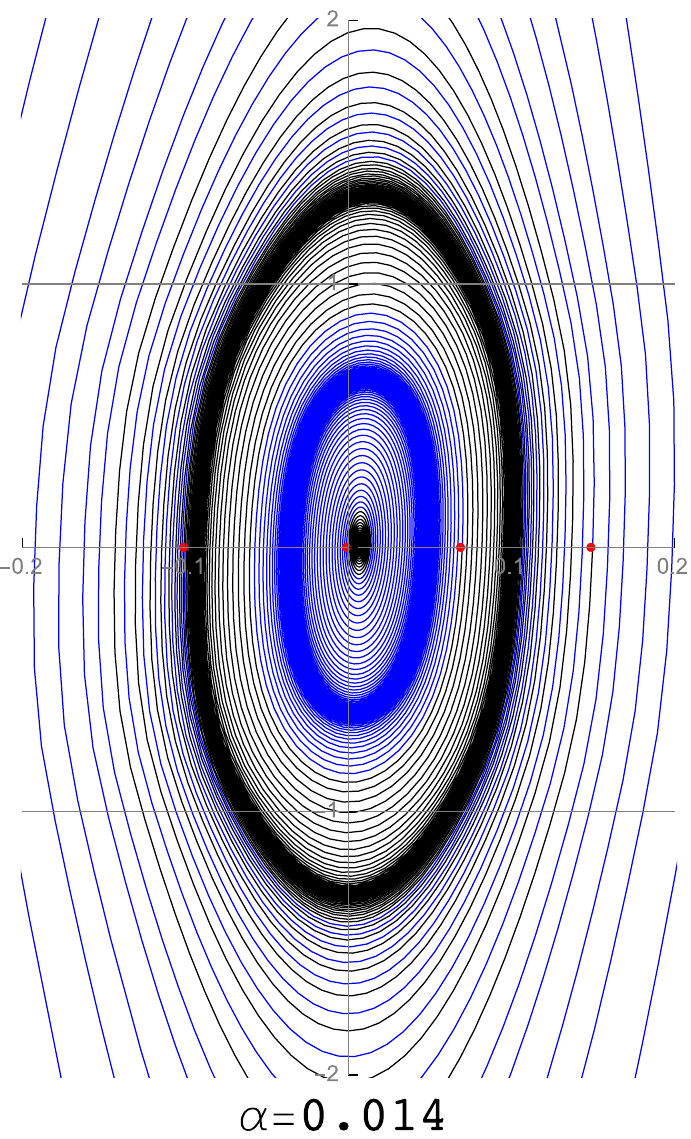}} \hspace{1cm}
			\subfigure[\label{fig:ExIISuba0015}]{\includegraphics[scale=0.6]{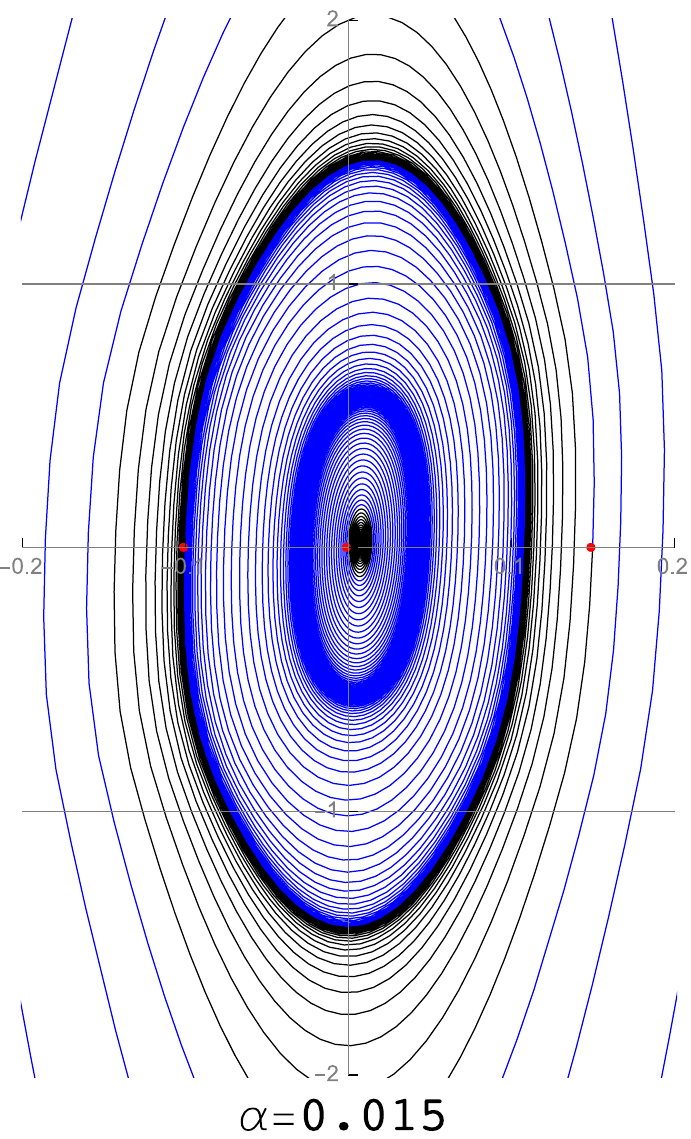}} \hspace{1cm}
			\subfigure[\label{fig:ExIISuba0021}]{\includegraphics[scale=0.6]{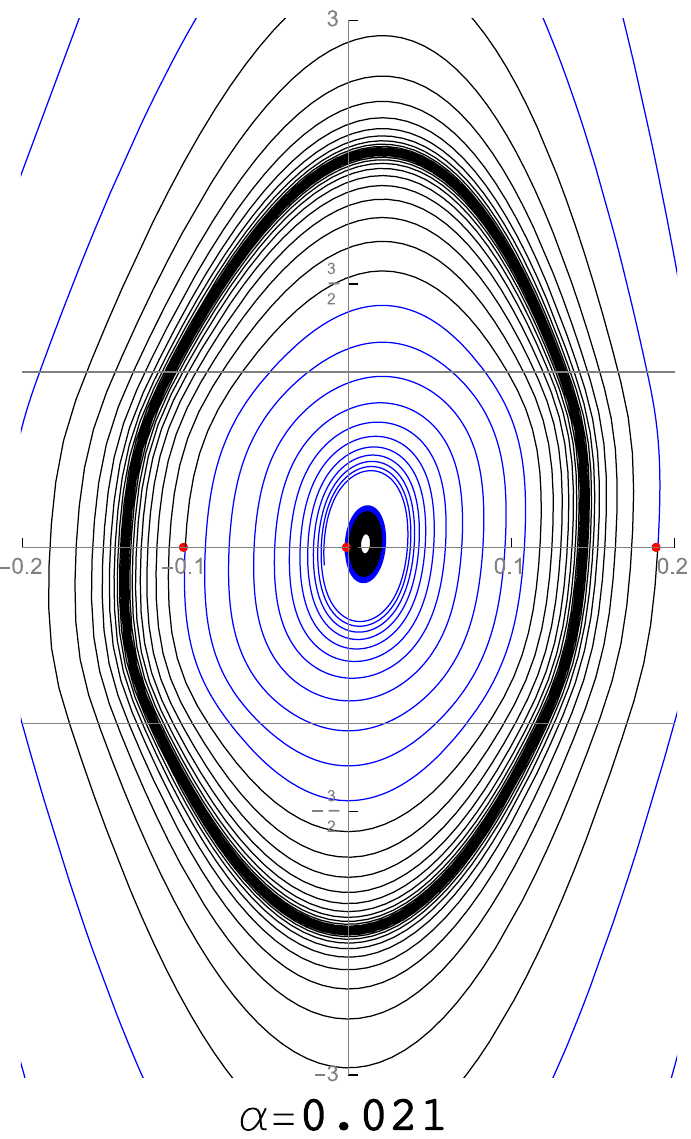}} \hspace{1cm}
			\subfigure[\label{fig:ExIISuba01}]{\includegraphics[scale=0.6]{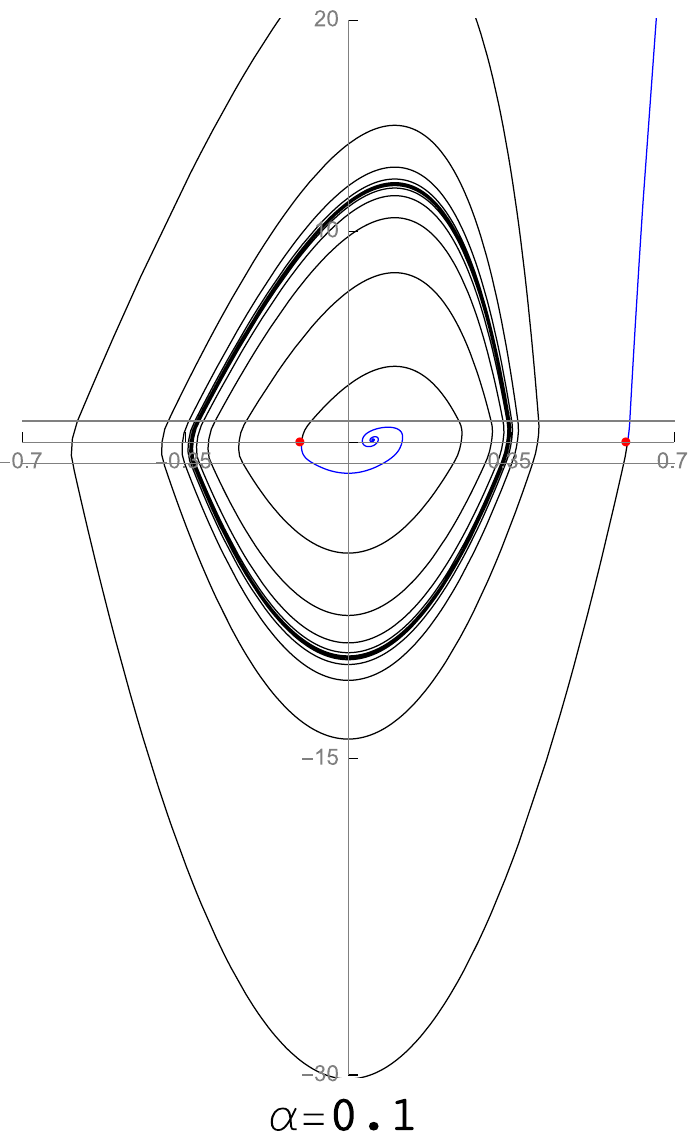}} \hspace{1cm}
			\begin{scriptsize}
				\textcolor{red}{$\blacksquare$}  Initial condition \quad  \textcolor{blue}{$\blacksquare$} Negative time \quad  \textcolor{black}{$\blacksquare$} Positive Time
			\end{scriptsize}
			\caption{ \autoref{ex:IIsub}: The evolution of the dynamics while increasing the value of $\alpha$ for $\e=0.006.$}
			\label{fig:ExIISub}
		\end{figure}
		
		The Lyapunov coefficient is $\ell_1(\e,\alpha(\e))=\frac{1}{\sqrt{\e}} \left( \frac{9}{\sqrt{2}}+ \mathcal{O(\e)} \right)
		$, therefore the Hopf bifurcation is subcritical.
		
		Fix $\e=0.006$. In \autoref{fig:ExIISub}, we present the simulations for different values of $\alpha$.
		In \autoref{fig:ExIISuba0}, $\alpha=0$: the stable focus $P(\alpha,\e)$ is a global attractor.
		In Figure \ref{fig:ExIISuba0012}, $\alpha=0.012$: we detect the presence of two periodic orbits, an smaller one $\Delta^{\alpha,u}_\e$
		which is unstable and a bigger one  $\Gamma^{\alpha, s}_{\e}$ which is stable.
		This means that the saddle-node value $\alpha_\mathcal{S}$ belongs to the interval $I_\mathcal{S}=(0.011,0.012)$.	
		In Figures \ref{fig:ExIISuba0012}, \ref{fig:ExIISuba0014} and \ref{fig:ExIISuba0015},
		we can see that the amplitude of the stable periodic orbit
		$\Gamma^{\alpha, s}_{\e}$ increases whereas the  unstable one $\Delta^{\alpha,u}_\e$ decreases while we increase the value of $\alpha.$
		In \autoref{fig:ExIISuba0021} and \autoref{fig:ExIISuba01}, for $\alpha$ greater than $\alpha_\mathcal{H}$, the critical point
		$P(\alpha,\e)$ is an unstable focus and only the stable periodic orbit
		$\Gamma^{\alpha, s}_{\e}$ persists. 
		We also show, in \autoref{fig:IISubmelnikov} the behavior of the Melnikov function $M(v;\delta)$ for different values of $\delta$.
		In yellow we show $M(v,\delta_\mathcal{S})$ which has a zero and is also a maximum. 
		For $\delta_\mathcal{S}<\delta <\delta _\mathcal{H}$ the function has two zeros.
		At $\delta = \delta_\mathcal{H}$, in black,  the zero with positive slope  disappears and only the big one with negative slope persists for 
		$\delta>\delta_\mathcal{H}$,
		corresponding to $\Gamma^{\alpha, s}_\e$.
	\end{exmp}
	
	\begin{figure}[htb!]
		\centering
		\subfigure[\label{fig:IISupermelnikov}]{\includegraphics[width=0.421\textwidth]{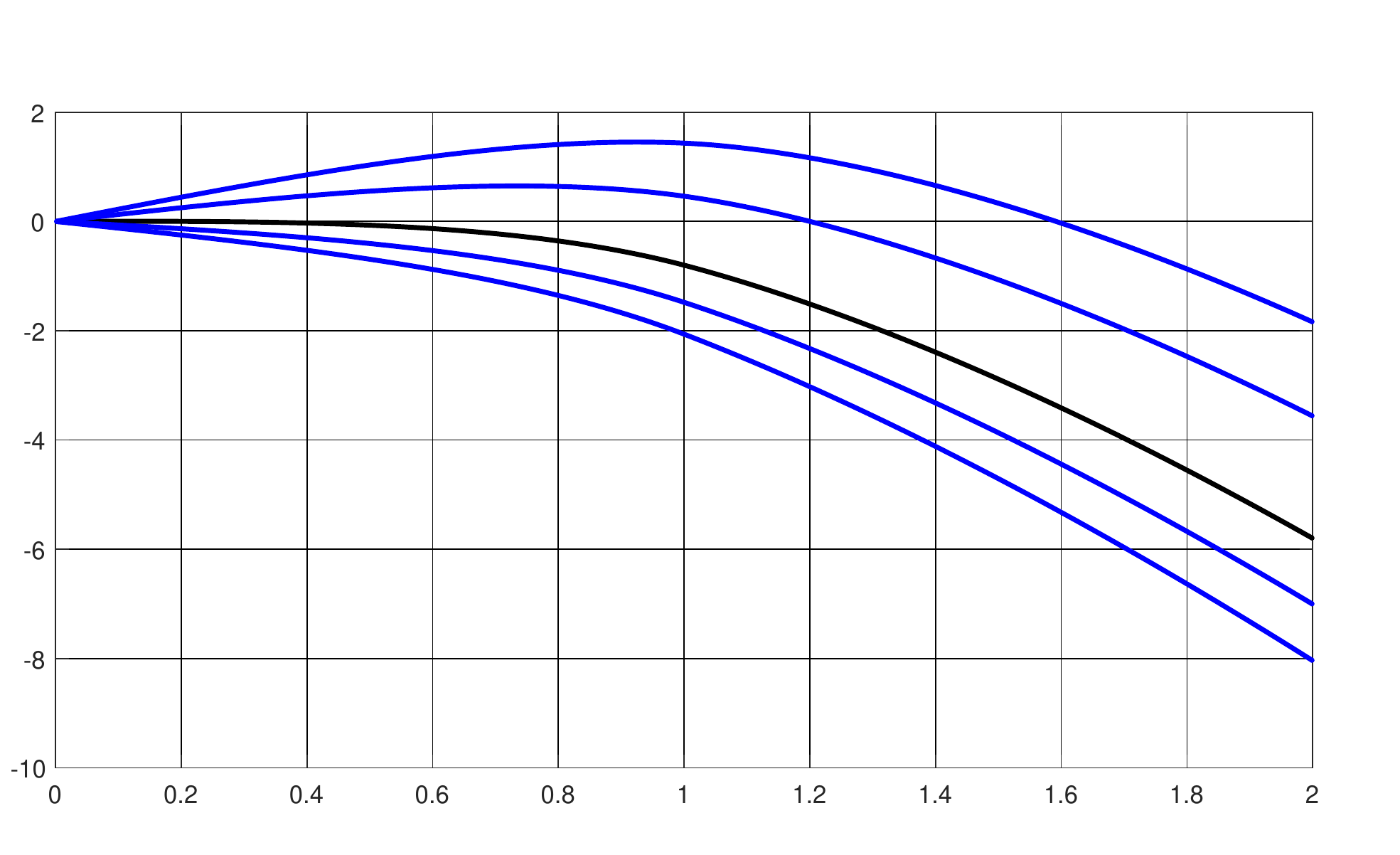}} \hspace{0.3cm}
		\subfigure[\label{fig:IISubmelnikov}]{\includegraphics[width=0.45\textwidth]{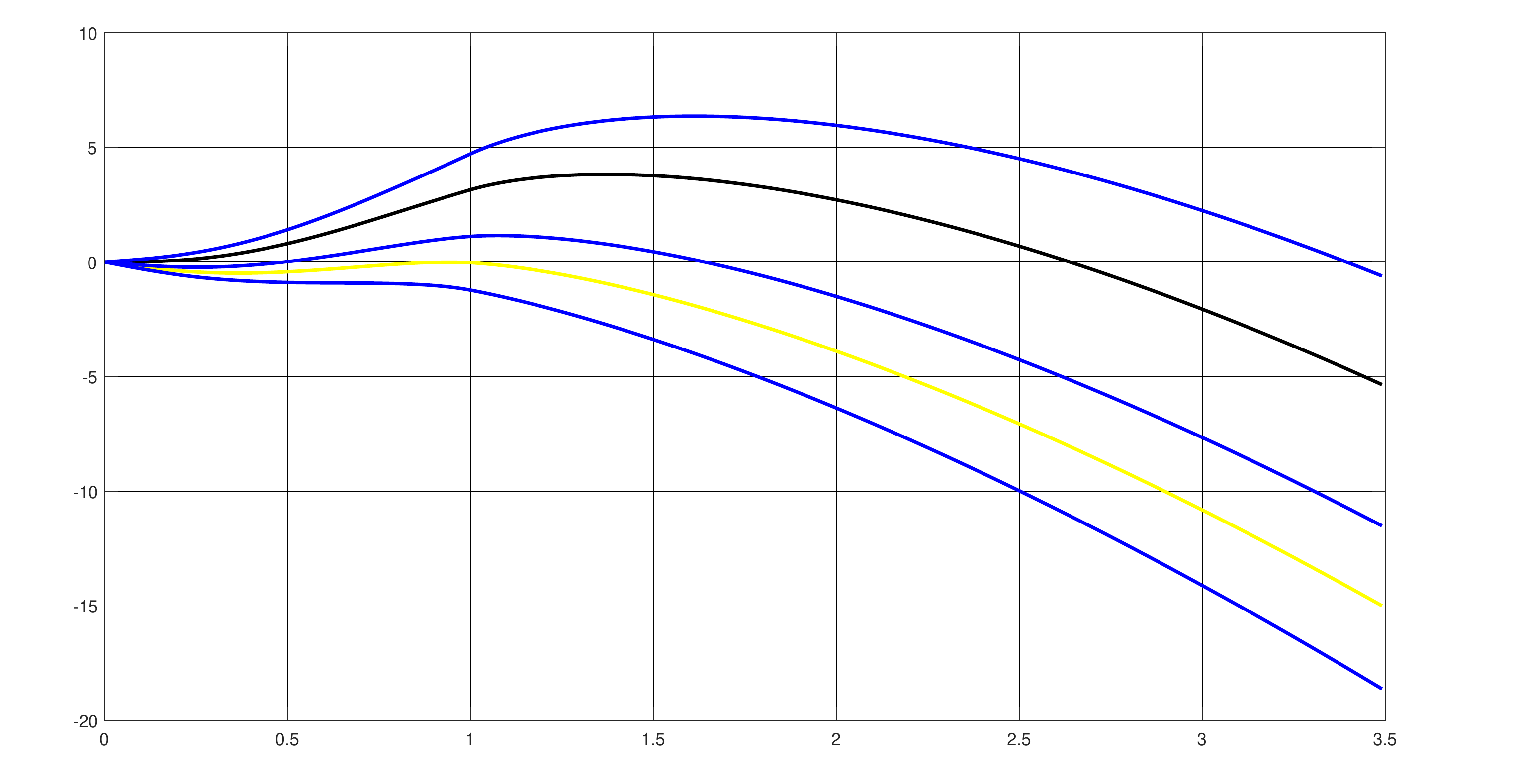}}
		\caption{(a)  The Melnikov function for example \ref{ex:IIsuper}:   $\delta = \delta _\mathcal{H}$ in black is supercritical,
			$M(v,\delta)$ has a zero with negative slope for $\delta <\delta_\mathcal{H}$. (b) The Melnikov function for example \ref{ex:IIsub}:   $\delta = \delta _\mathcal{H}$ in black is subcritical,
			$\delta = \delta _{S}$ in yellow. For $\delta _{S}<\delta <\delta _\mathcal{H}$ $M$ has two zeros, only one survives for  $\delta _\mathcal{H}<\delta$.}
	\end{figure}
	
	\subsection{The visible-invisible case} \label{sec:VIunfreg}
	
	In this section we study the case where the folds have opposite visibility,  the vector field $Z$ satisfies $X^1 \cdot Y^1(\0)<0$  
	and  $(\det{Z})_x(\0) \neq 0,$  (see \ref{eq:sliding}).
	The dynamics of the system $Z^\alpha_\e$ is more involved when $(\det{Z})_x(\0) < 0$, and it is studied in next subsection.
	
	\subsubsection{The focus case: \texorpdfstring{$(\det{Z)}_x(\0)<0$}{det<0}} \label{ssec:VIunfregL}
	
	When $(\det{Z)}_x(\0)<0$, for each $\alpha \neq 0$ and $\e$ small enough, by \autoref{corol:toptype}  the critical point $P(\alpha,\e)$ is a node with the same
	character that the pseudo-node $Q(\alpha)$ of $Z^\alpha$.
	This behavior persists for all values of $(\alpha,\e)$ below the parabola $\mathcal{D}$ given in \ref{dcurve}.

	Using the results about the critical manifolds given in \autoref{ssec:OVcritical} and applying Fenichel theory,
	for each fixed $\alpha \neq 0$ and any compact set of the critical manifolds $\Lambda^{\alpha,s}_0$ and
	$\Lambda^{\alpha,u}_0$ excluding the tangency points $(T^\alpha_X,1)$ and $(T^\alpha_Y,-1)$, for $0<\e<\e_0(\alpha)$,
	there exist two normally hyperbolic invariant manifold $\Lambda^{\alpha,s}_\e$ and $\Lambda^{\alpha,u}_\e$ which are $\e-$close to
	$\Lambda^{\alpha,s}_0$ and $\Lambda^{\alpha,u}_0$, respectively (see \autoref{fig:OVcritical-}).

	Moreover, in this case, the Fenichel manifold $\Lambda^{\alpha,u}_\e$ is the weak manifold of the unstable node $P(\alpha,\e)$ for
	$\alpha<0$ and  $\Lambda^{\alpha,s}_\e$ is the weak manifold of the stable node $P(\alpha,\e)$ for $\alpha>0$.
	
	As we have seen in \autoref{sec:VVunfreg}, the vector field $X^{\alpha}(x,\e v)$ has a unique visible fold point at
	$(T^{\alpha, \e}_{X},1)$
	and $Y^{\alpha}(x,\e v)$ has a unique invisible fold point at
	$(T^{\alpha, \e}_{Y},-1)$
	(see \autoref{epsXtangency} and \autoref{epsYtangency}).
	
	Observe that for $x < T^{\alpha, \e}_X$ the vector $X^\alpha(x,1)$ points inward to the regularization zone and points outwards to the regularization zone for
	$x>T^{\alpha, \e}_X$.
	Analogously, for $x < T^{\alpha, \e}_Y$ the vector $Y^\alpha(x,-1)$ points inwards to the regularization zone for $x<T^{\alpha, \e}_Y$
	and outwards to the regularization zone for $x>T^{\alpha,\e}_Y$.
	
	The above information and the fact that the dynamics over the Fenichel Ma\-ni\-folds $\Lambda^{\alpha, s/u}_\e$ is equivalent to the one  over  the critical manifolds
	$\Lambda^{\alpha, s/u}_0$, it follows that (see \autoref{fig:VILUR2}):
	\begin{itemize}
		\item
		for $\alpha<0$, the stable Fenichel manifold $\Lambda_{\e}^{\alpha,s}$ intersects the section $\{ v=1 \}$ on the right of the tangency point $T^{\alpha, \e}_{X}$.
		The unstable Fenichel manifold $\Lambda_{\e}^{\alpha,u}$, which is the weak manifold of the unstable  node $P(\alpha,\e)$,
		can intersect or not the section $v=-1$.
		If this intersection occurs it is located to the right of the tangency point $T^{\alpha, \e}_{Y}$,
		\item
		for $\alpha>0$, the unstable Fenichel manifold $\Lambda_{\e}^{\alpha,u}$ intersects the section $\{ v= 1 \}$ on the left of the tangency point $T^{\alpha, \e}_{X}$.
		The stable Fenichel manifold $\Lambda_{\e}^{\alpha,s}$, which is the weak manifold of the  stable node, can intersect or not the section $v=-1$.
		If this intersection occurs it is located to the left of the tangency point $T^{\alpha, \e}_{Y}$.
	\end{itemize}
	
	When $(\alpha,\e)$ is above the parabola $\mathcal{D}$ the critical point $P(\alpha,\e)$ becomes a focus, which is unstable for $\alpha<\alpha_\mathcal{H}(\e)$,
	undergoes  a Hopf bifurcation for $\alpha=\alpha_\mathcal{H}(\e)$ and is stable for $\alpha>\alpha_\mathcal{H}(\e)$.
	The main point here is
	that, since there are no periodic orbits in the bifurcation diagram of $Z^\alpha$, $\alpha \ne 0$, it must exist a curve in the parameter space such that, on
	this curve, the limit cycle which raises from the Hopf bifurcation disappears. It is at this point that the slow-fast nature of system \eqref{vsystem} plays a role,
	because the evolution of the periodic orbit will be influenced by the evolution of the Fenichel manifolds when the parameters vary.
	
	When $\alpha = \delta \varepsilon$, the critical manifolds of $Z^\alpha_\e$  are the same as the ones for $\alpha =0$.
	Then, as we saw in \autoref{ssec:OVcritical} there are two critical manifolds
	$\Lambda_0^{s,u}$ given by:
	$$
	\Lambda_0^s = \{ (x,v), \ v=m_0(x), x<0\}, \
	\Lambda_0^u = \{ (x,v), \ v=m_0(x), x>0\},
	$$
	which are normally hyperbolic (attracting and repelling respectively) and we restrict them to $|x|>\kappa$ for a small but fixed $\kappa>0$.
	Applying Fenichel theorem, for any compact subset $\mathcal{K}$ of the critical manifolds, and $\e$ small enough, we know the existence of two normally
	hyperbolic invariant manifolds
	$$
	\Lambda_\varepsilon^{s} = \{ (x,v)\in \mathcal{K}, \ v=m^s (x;\varepsilon), x<-\kappa\}, \
	\Lambda _\varepsilon^{u} = \{ (x,v)\in \mathcal{K}, \ v=m^u (x;\varepsilon), x>\kappa\},
	$$
	with $m^{s,u}(x;\e)=m_0(x)+ \mathcal{O}(\e)$.
	
	In \autoref{prop:canard} we prove the existence of a maximal Canard by 
	looking for $\alpha = \delta_\mathcal{C} \e + \mathcal{O}(\e^\frac{3}{2})$ such that the stable and unstable Fenichel manifolds
	can be extended up to $x=0$ and coincide.
	
	Moreover, we will see in \autoref{prop:linearcanard} that when $\varphi$ is linear,
	the regularized system $Z^\alpha_\e$ in the regularization zone $|v|\le 1$
	can be transformed, after a change of variables, into a normal form
	studied by Krupa-Szmolyan in \cite{KrupaS01}.
	This will completely determine the position of the curves where the Hopf bifurcation and the maximal Canard occur.
	Therefore, in this case, we provide a complete description of the bifurcation diagram of the regularized system.
	Later, in \autoref{ex:ch}, we see that this result is not true when $\varphi$ is not a linear map.
	Therefore, for non linear regularization the bifurcation diagram depends strongly of the transition map $\varphi$.
	
	\begin{figure}[!htb]
		\centering
		\begin{tiny}
			\def\svgscale{0.4}
			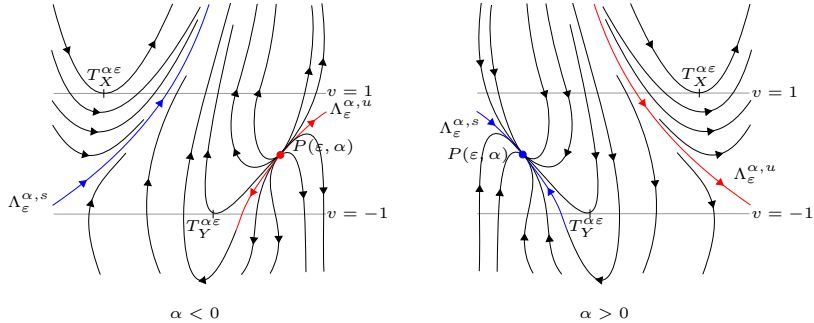
		\end{tiny}
		\caption{The phase portrait of $Z^{\alpha}_\e$ for $(\alpha,\e)$ below the curve $\mathcal{D}$.}
		\label{fig:VILUR2}
	\end{figure}

	\begin{prop} \label{prop:canard}
		Let $Z \in \Lambda^F$ having a visible-invisible fold, satisfying  $X^1\cdot Y^1(\0)<0$  and  $(\det{Z})_x(\0) < 0$.
		Then, for $\alpha = \delta \e$ the stable and unstable Fenichel manifolds $\Lambda^{\alpha,s}_\e$ and $\Lambda^{\alpha,u}_\e$ of the
		regularized system $Z^\alpha_\e$ can be extended up to $x=0$.
		Moreover:
		$$
		\Lambda^{\alpha,s,u}_\e\cap \{x=0\}= (0, \bar v+ \mathcal{O}(\e^{\frac{1}{2}})),
		$$
		where $\bar v=m_0(0)$ is given in \autoref{IVslowmanifold},
		and  the system has a maximal Canard for
		\begin{equation} \label{ccurve1}
		\mathcal{C} = \left\{ (\alpha,\e) : \, \alpha=\alpha_\mathcal{C}(\e) =  \delta_\mathcal{C} \e + \mathcal{O}\left( \e^{\frac{3}{2}} \right) \right\},
		\end{equation}
		and $\e$ small enough, where $\displaystyle{ \delta_C= - \frac{M_0 M_3 + M_1 M_4}{M_2 M_4}}$ and the constants $M_i$ are given by
		\begin{align} \label{mis1}
		\begin{aligned}
		M_0 &=\left( X^1 + Y^1 + \varphi(\bv) (X^1-Y^1) \right) (\0), \\
		M_1 &= \bv \left( X^2_{y} + Y^2_{y} + \varphi(\bv) \left( X^2_{y}-Y^2_{y} \right) \right) (\0), \\
		M_2 &= \left( \tilde X^2 + \tilde Y^2  + \varphi(\bar v) (\tilde X^2 - \tilde Y^2)  \right)(\0), \\
		M_3 &=\frac{1}{2} \left( X^2_{xx} + Y^2_{xx} + \varphi(\bv) \left( X^2_{xx}-Y^2_{xx} \right) \right)(\0), \\
		M_4	&= \varphi'\left(\bv\right) \left( X^2_{x} - Y^2_{x} \right)(\0).
		\end{aligned}
		\end{align}
		Moreover, for $\alpha= \alpha_\mathcal{C}(\e)$, one has that $
		\Lambda^{\alpha,s}_\e\cap \{x=0\}=\Lambda^{\alpha,u}_\e\cap \{x=0\}= (0, \bar v+ \mathcal{O}(\e)),
		$
	\end{prop}
	\begin{proof}
		The proof of this proposition is done using asymptotic methods and it is deferred to subsection \ref{proofcanard} in the appendix.
	\end{proof}

	Next, we will see how  the maximal Canard obtained in \autoref{prop:canard}  plays an important role in the behavior of the periodic orbits of the system.
	
	During this section, we fix $\mathcal{C} < \mathcal{H}$, that is, the curve  $\mathcal{C}$ where the
	Canard trajectory takes place is located to the left of the curve  $\mathcal{H}$ where the Hopf bifurcation happens.
	The other case can be done analogously.
	In Theorems \ref{thm:bdVIsuper} and  \ref{thm:bdVIsub}, completed by \autoref{pocanard} we present the bifurcation diagram of the regularized system
	$Z^\alpha_\e$ depending on the sign of the first Lyapunov coefficient $\ell_1(\alpha(\e),\e)$ over the curve $\mathcal{H}$ and the sign of the
	``way-in,way-out'' function \eqref{eq:R}.
	
	\begin{theo} \label{thm:bdVIsuper}
		Let $Z \in \Lambda^F$ having an visible-invisible fold-fold singularity at the origin sa\-tis\-fying  $X^1\cdot Y^1(\0)<0$ and $(\det{Z})_x(\0)<0$.
		Suppose that on the curve $\mathcal{H}$ (given in \ref{hcurve}), the first Lyapunov coefficient $\ell_1(\alpha(\e),\e)<0$ and for each $\e$
		sufficiently small $\alpha_\mathcal{C}(\e)<\alpha_\mathcal{H}(\e)$.
		
		Let $\alpha_\mathcal{D}^\pm(\e_0)$, $\alpha_\mathcal{H}(\e_0)$ and $\alpha_\mathcal{C}(\e_0)$ be the intersection of the curves
		$\mathcal{D}^\pm$(\ref{dcurve}), $\mathcal{H}$(\ref{hcurve}) and $\mathcal{C}$(\ref{ccurve}) with the line $\e=\e_0$, respectively.
		We have the following:
		\begin{itemize}
			\item
			For $\alpha<\alpha_\mathcal{D}^-(\e_0)$ the critical point $P(\alpha,\e)$ is a unstable node;
			\item
			For $\alpha_\mathcal{D}^-(\e_0)<\alpha<\alpha_\mathcal{C}(\e_0)$ the critical point $P(\alpha,\e)$ is a unstable focus;
			\item
			For $\alpha=\alpha_\mathcal{C}(\e_0)$ the stable and unstable Fenichel manifolds of system
			$Z^\alpha_\e$ coincide along  a maximal Canard and there exists an stable periodic orbit $\Delta^{\alpha,s}_{\e}$
			for $\alpha>\alpha_\mathcal{C}(\e_0)$.
			\item
			The periodic orbit $\Delta^{\alpha,s}_{\e}$ persists for $\alpha_\mathcal{C}<\alpha<\alpha_\mathcal{H}(\e_0)$.
			\item
			For $\alpha=\alpha_\mathcal{H}(\e_0)$ a supercritical Hopf bifurcation takes place.
			The critical point $P(\alpha,\e)$ becomes an stable focus.
			\item
			Moreover, the critical point $P(\alpha,\e)$ is a stable focus for $\alpha_\mathcal{H}(\e_0)<\alpha<\alpha_\mathcal{D}^+(\e_0)$ and a stable node
			$\alpha>\alpha_\mathcal{D}^+(\e_0)$.
			\item  
			If the Melnikov function $M(v,\delta)$ is strictly concave for $\delta$ near $\delta_\mathcal{H}$, 
			then the periodic orbit $\Delta^{\alpha,s}_\e$ is unique and disappears at $\alpha=\alpha_\mathcal{H}(\e_0).$
		\end{itemize}
	\end{theo}
	
	\begin{proof}
		The character of the critical point is given by \autoref{corol:toptype}.
		The dynamics for $(\alpha,\e)$ below the parabola $\mathcal{D}$ has been discussed in the beginning of this section.
		We are going focus on the dynamics above the parabola $\mathcal{D}$.
		
		When $\alpha_\mathcal{D}(\e_0)^-<\alpha<\alpha_\mathcal{H}(\e_0)$ the critical point $P(\alpha,\e)$ is a unstable focus,
		see \autoref{fig:VILURch1}.

		In the case $\alpha_\mathcal{C}(\e_0) < \alpha_\mathcal{H}(\e_0)$ for each fixed $\e$ small enough, the maximal Canard occurs before the  Hopf bifurcation.
		
		For $\alpha<0$ and $\e>0$ sufficiently small, the stable Fenichel manifold $\Lambda^{\alpha,s}_\e$ becomes unbounded for positive  time.
		The same occurs for $\alpha>0$ for the unstable Fenichel manifold $\Lambda_\e^{\alpha,u}$ in negative time, see \autoref{fig:VILUR2}.
		
		When $\alpha = \delta \e$, the critical manifold $\Lambda_0^\alpha$ associated to the vector field
		$Z^\alpha_\e$ is equal to the critical manifold
		$\Lambda_0=\Lambda_0^s \cup \Lambda_0^u$ associated to the vector field $Z_\e^\alpha$, $\alpha=0$.
		therefore, when $\alpha \rightarrow \alpha_\mathcal{C}(\e_0)^\pm$, both the stable and unstable the Fenichel manifolds become ``flattened'' until they
		coincide at $\alpha=\alpha_\mathcal{C}(\e_0)$, see \autoref{fig:VISuper2}.
		However, by continuity, for $\alpha< \alpha_\mathcal{C}(\e_0)$
		the stable Fenichel manifold is above the unstable one and the opposite occurs for $\alpha>\alpha_\mathcal{C}(\e_0)$.
		
		\begin{figure}
			\centering
			\begin{tiny}
				\def\svgwidth{0.8\textwidth}
				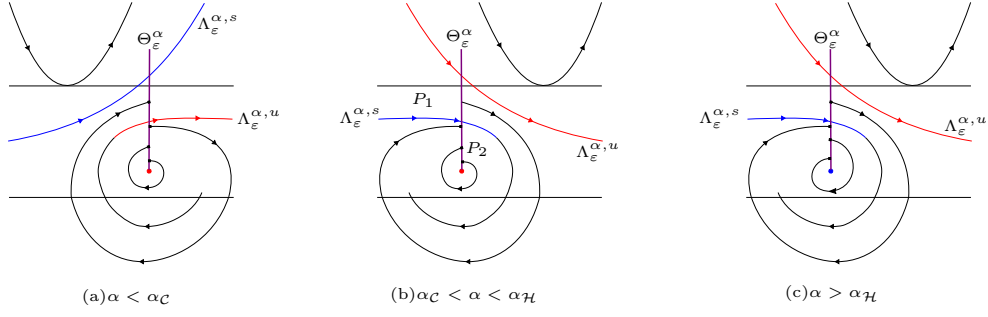
			\end{tiny}
			\caption{The first return map associated to the regularized system $Z^\alpha_\e$ when the folds have opposite visibility.}
			\label{fig:VILURhcop}
		\end{figure}
		
		At the point $\alpha=\alpha_{\mathcal{C}}(\e_0)$, the manifolds $\Lambda^{\alpha,u}_\e$ and $\Lambda^{\alpha,s}_\e$ coincide as
		in \autoref{fig:VISuper2} and the unstable focus $P(\alpha,\e)$ is below them
		(see \autoref{rem:vsbvposition}). 
		This change on the relative position of the critical manifolds $\Lambda^{\alpha,s}_\e$ and $\Lambda^{\alpha,u}_\e$, 
		gives raise to a stable periodic orbit which persists for  $\alpha_\mathcal{C}(\e_0)<\alpha<\alpha_\mathcal{H}(\e_0)$.
		
		\begin{figure}[!htb]
			\centering
			\begin{tiny}
				\subfigure[\label{fig:VISuper1}]{\def\svgscale{0.4} 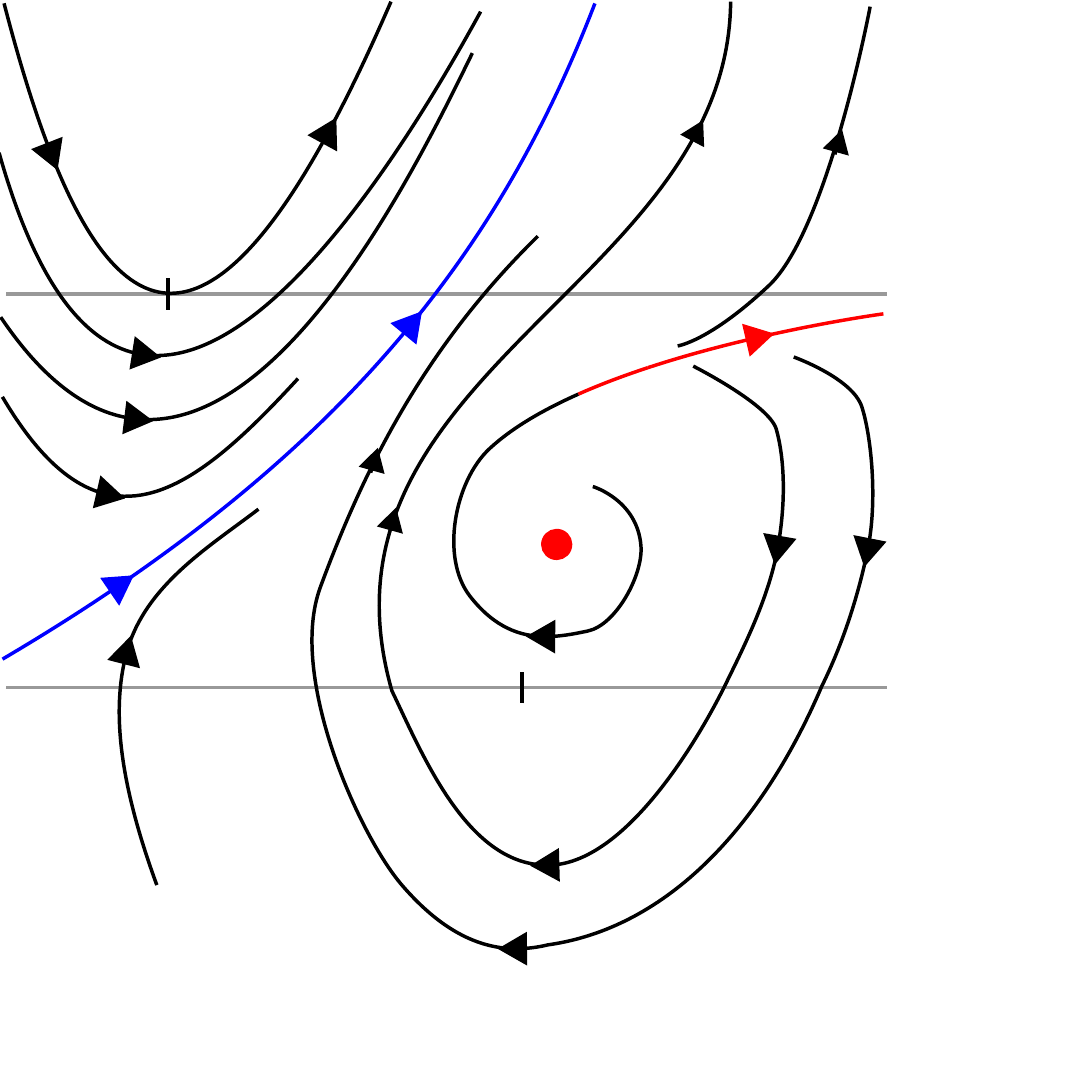}  \hspace{1cm}
				\subfigure[\label{fig:VISuper2}]{\def\svgscale{0.4} 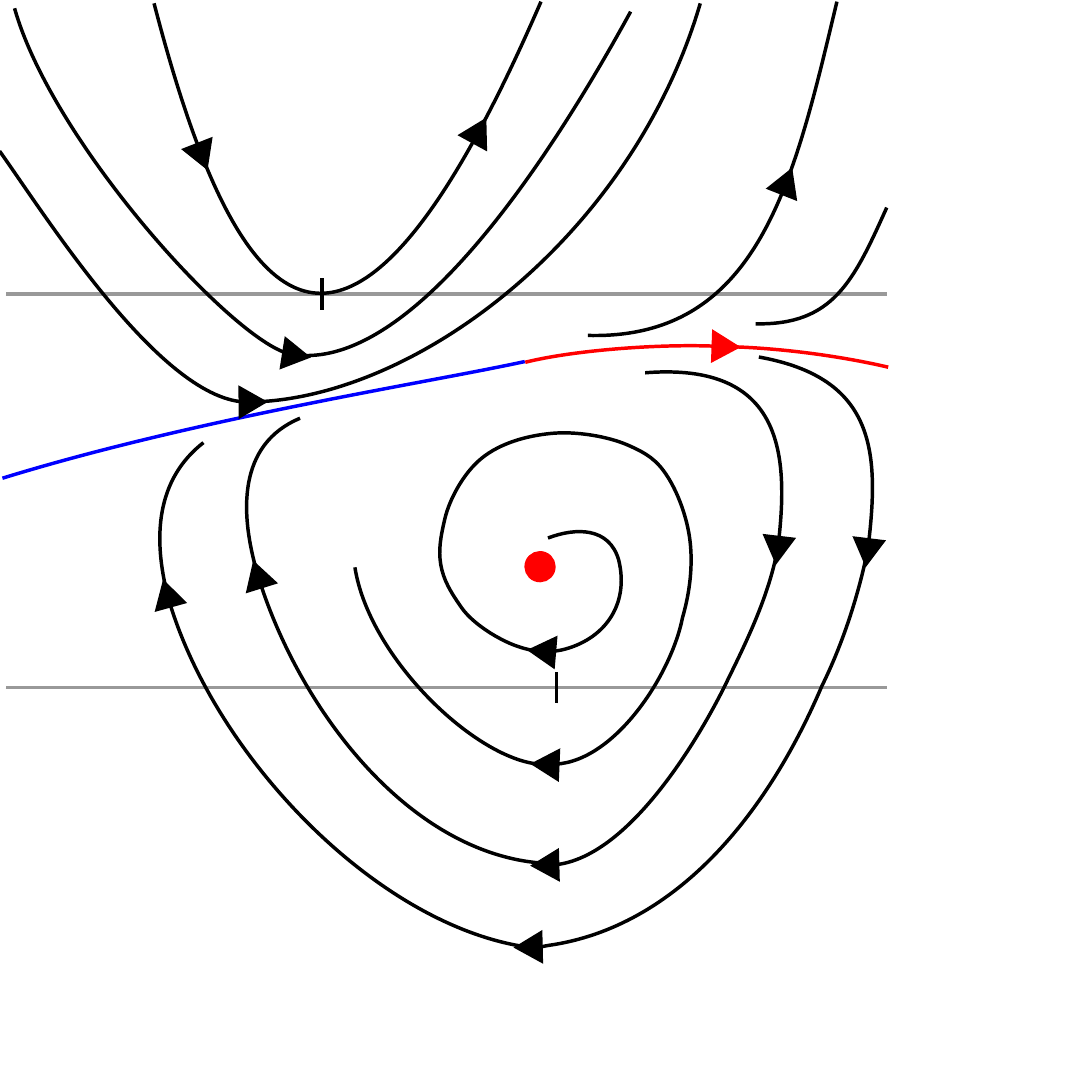}  \hspace{1cm} \\
				\subfigure[\label{fig:VISuper3}]{\def\svgscale{0.4} 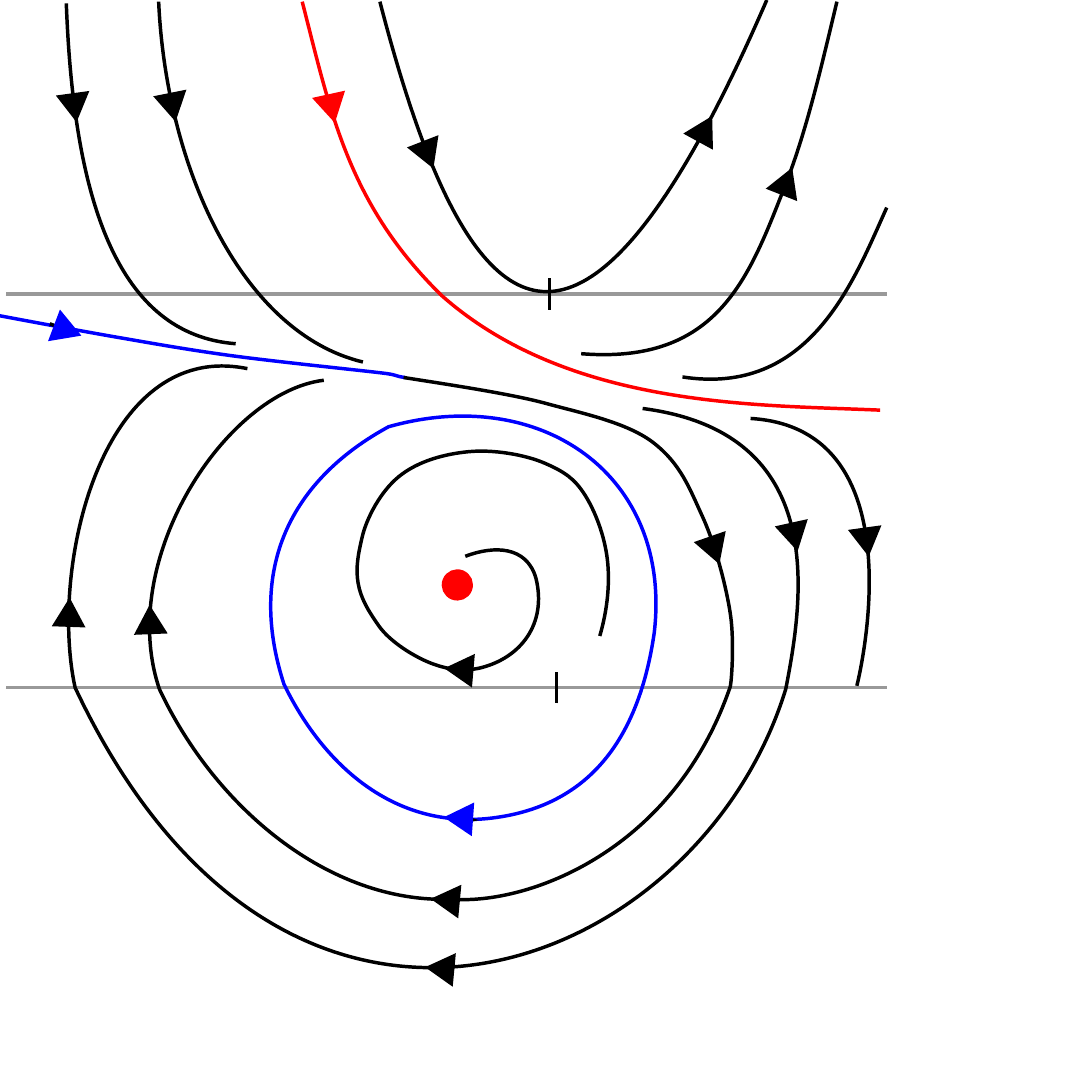}  \hspace{1cm}
				\subfigure[\label{fig:VISuper4}]{\def\svgscale{0.4} 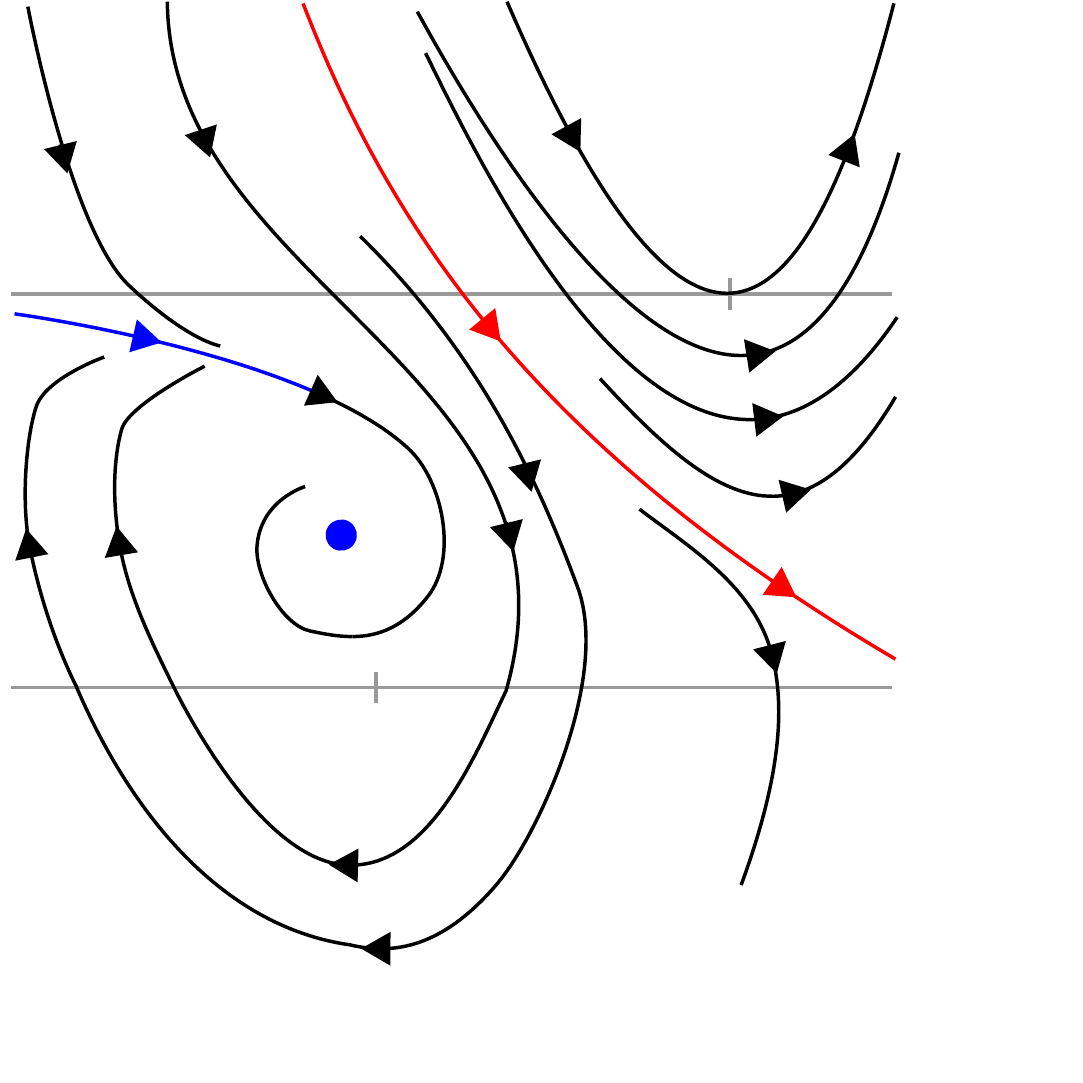}
			\end{tiny}
			\caption{The supercritical Hopf bifurcation on the regularization of a visible-invisible fold-fold satisfying $\alpha_{\mathcal{C}}<\alpha_{\mathcal{H}}$.}
			\label{fig:VISubHCdiagram}
		\end{figure}
		
		In fact, let $\Theta$ be a cross section transverse to the stable Fenichel manifold $\Lambda^{\alpha,s}_\e$ and containing
		the point $P(\alpha,\e)$, consider the return map $\pi :\Theta\to \Theta$.
		Take $P_1 \in \Theta$ which is above and sufficiently close to $\Lambda^{\alpha,s}_\e$, its trajectory  crosses the line $v=-1$ and then
		intersects the section $\Theta$ below $\Lambda^{\alpha,s}_\e$; therefore $\pi (P_1)-P_1<0$.
		Opposed to this behavior, since the point $P(\alpha,\e)$ is a unstable focus, 
		the trajectory of any initial condition $P_2 \in \Theta$ near the focus
		intersects, for negative time, $\Theta$ even closer to the focus point $P(\alpha,\e)$, see
		\autoref{fig:VILURhcop}; therefore $\pi (P_2)-P_2>0$.
		
		This way, the Bolzano-Weierstrass Theorem,  guarantees the existence of a fixed point of the return map $\pi$ which gives 
		an stable periodic orbit $\Delta^{\alpha,s}_{\e}$ for all $\alpha_\mathcal{C}(\e_0)<\alpha<\alpha_\mathcal{H}(\e_0)$.
		
		When $\alpha=\alpha_\mathcal{H}(\e_0)$, the critical point $P(\alpha,\e)$ undergoes to a supercritical Hopf bifurcation, therefore, 
		for $\alpha<\alpha_\mathcal{H}(\e_0)$ there exists a small stable periodic orbit $\bar{\Delta}^{\alpha,s}_\e$ near the critical point 
		$P(\alpha,\e)$. 
		For $\alpha>\alpha_\mathcal{H}(\e_0)$ the critical point changes its stability and becomes a stable focus, 
		therefore, there are no periodic orbits near the critical point.
		
		On the other hand, if the Melnikov function $M(v,\delta)$ is strictly concave for $\delta$ near $\delta_\mathcal{C}$ and 
		$\alpha=\delta \e$, system $Z^\alpha_\e$ has a unique periodic orbit for parameter values in this interval. 
		Therefore it follows that the two stable periodic orbits $\Delta^{\alpha,s}_\e$ and $\bar{\Delta}^{\alpha,s}_\e$ coincide. 
		This means that, in this case, the periodic orbit which raises from the canard becomes smaller until it disappear in the Hopf bifurcation. 
		Finally, for $\alpha>\alpha_\mathcal{H}(\e_0)$ the regularized system has no periodic orbits.
	\end{proof}

	\begin{theo} \label{thm:bdVIsub}  Consider the same hypothesis of \autoref{thm:bdVIsuper} but now $\ell_1(\alpha(\e),\e)>0$.
		
		Let $\alpha_\mathcal{D}^\pm(\e_0)$, $\alpha_\mathcal{H}(\e_0)$ and $\alpha_\mathcal{C}(\e_0)$ be the intersection of the line $\e=\e_0$ with the
		curves $\mathcal{D}^\pm$, $\mathcal{H}$ and $\mathcal{C}$, respectively. Generically, we have the following
		\begin{itemize}
			\item
			For $\alpha< \alpha_\mathcal{D}^-(\e_0)$ the critical point $P(\alpha,\e)$ is a unstable node;
			\item
			For $\alpha_\mathcal{D}(\e_0)^-<\alpha<\alpha_\mathcal{H}(\e_0)$ the critical point $P(\alpha,\e)$ is a unstable focus;
			\item
			When $\alpha=\alpha_\mathcal{C}(\e_0)$ the system $Z^\alpha_\e$ has a Canard trajectory;
			\item
			For $\alpha_\mathcal{C}(\e_0)<\alpha<\alpha_\mathcal{H}(\e_0)$ there exists a stable periodic orbit $\Delta^{\alpha,s}_\e$ and
			the critical point $P(\alpha,\e)$ is a unstable focus;
			\item
			When $\alpha=\alpha_\mathcal{H}(\e_0)$ a subcritical Hopf bifurcation takes place.
			The critical point $P(\alpha,\e)$ becomes a stable focus and a unstable periodic orbit $\Delta^{\alpha,u}_{\e}$ appears;
			\item
			The pair of periodic orbits coexist for $\alpha_\mathcal{H}(\e_0)<\alpha<\alpha_\mathcal{S}(\e_0)$;
			\item
			When $\alpha>\alpha_\mathcal{S}(\e_0)$ there are no periodic orbits.
			Moreover, the critical point $P(\alpha,\e)$ is an stable focus when $\alpha_\mathcal{H}(\e_0)<\alpha<\alpha_\mathcal{D}^+(\e_0)$
			and an stable node for $\alpha>\alpha_\mathcal{D}^+(\e_0)$.
		\end{itemize}
	\end{theo}
	
	\begin{proof}
		The character of the critical point  and the argument which gives
		the existence of the stable periodic orbit  $\Delta^{\alpha,s}_{\e}$ for $\alpha_\mathcal{C}<\alpha<\alpha_\mathcal{H}$ 
		is already proved in  \autoref{thm:bdVIsuper}, see \autoref{fig:VILURch3}.
		
		At $\alpha > \alpha_\mathcal{H}(\e_0)$, the point $P(\alpha,\e)$ becomes an stable focus.
		Since the Hopf bifurcation is subcritical ($\ell_1(\alpha(\e),\e)>0$),
		a unstable periodic orbit $\Delta^{\alpha,u}_{\e}$ appears for $\alpha> \alpha_\mathcal{H}(\e_0)$, as shown in \autoref{fig:VILURch4}.
		Observe that the stable periodic orbit $\Delta^{\alpha,s}_{\e}$
		persistence is guaranteed by the first return map and therefore, both periodic orbits coexist.
		
		Since for fixed $\alpha>0$ and $\e>0$ small enough, $Z^\alpha_\e$ has no periodic orbits, for each $\e$ it must exist a value of
		$\alpha_\mathcal{S}(\e_0)$ such that the two orbits collapse and then disappear, as illustrated in Figures
		\ref{fig:VILURch5} and \ref{fig:VILURch6}, respectively. 
		One expects that, in the simplest case, for $\alpha=\alpha_\mathcal{S}(\e_0)$ the periodic orbits collide in a saddle-node bifurcation.
	\end{proof}
	
	\begin{figure}[!htb]
		\centering
		\begin{tiny}
			\subfigure[\label{fig:VILURch1}]{\def\svgscale{0.4} \input{./FIG/VILURch1.pdf_tex}}  \hspace{1cm}
			\subfigure[\label{fig:VILURch2}]{\def\svgscale{0.4} \input{./FIG/VILURch2.pdf_tex}}  \hspace{1cm}
			\subfigure[\label{fig:VILURch3}]{\def\svgscale{0.4} \input{./FIG/VILURch3.pdf_tex}}  \hspace{1cm}
			\subfigure[\label{fig:VILURch4}]{\def\svgscale{0.4} 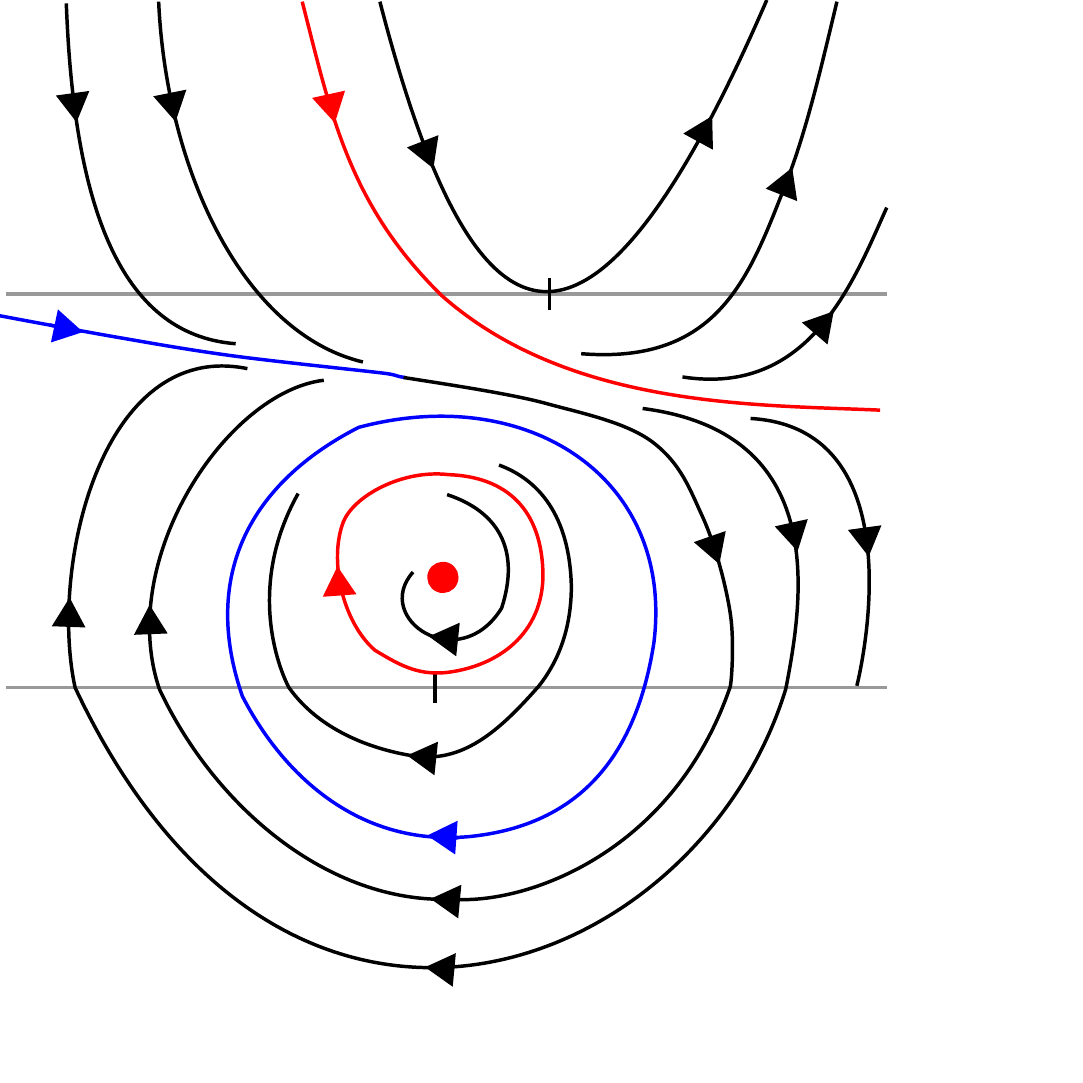}  \hspace{1cm}
			\subfigure[\label{fig:VILURch5}]{\def\svgscale{0.4} 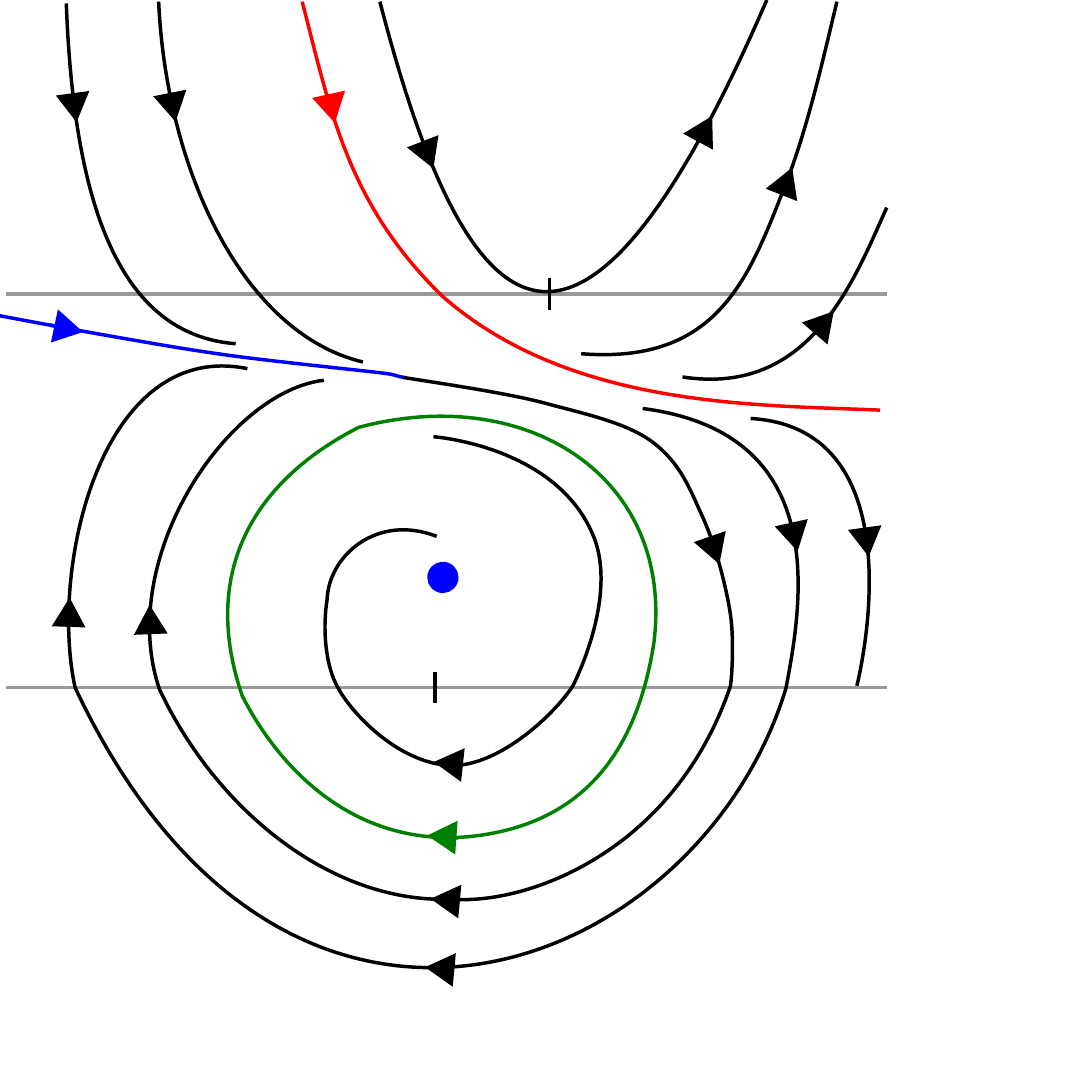}  \hspace{1cm}
			\subfigure[\label{fig:VILURch6}]{\def\svgscale{0.4} 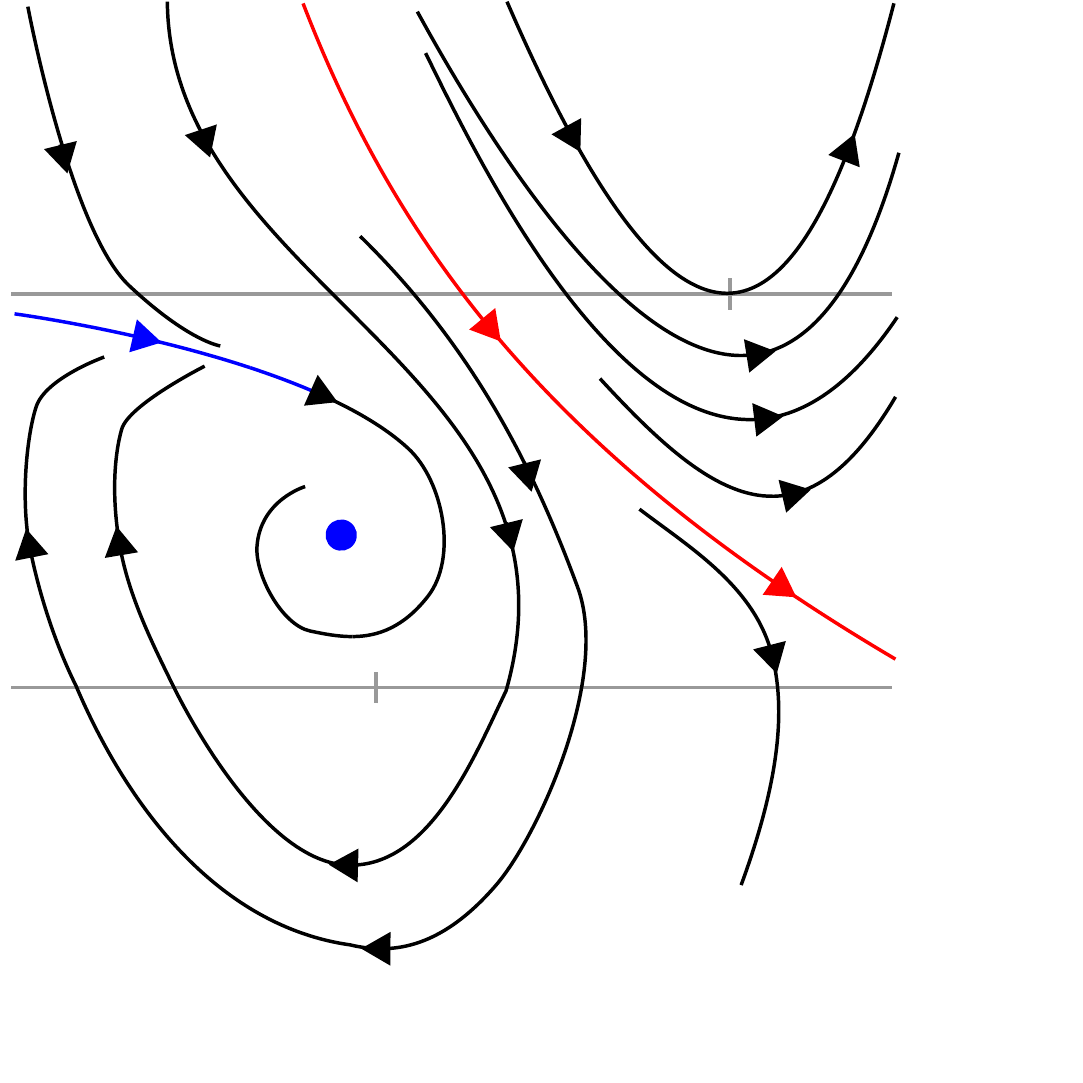}
		\end{tiny}
		\caption{The bifurcation diagram for each fixed $\e_0>0$ when $P(\alpha,\e)$ suffers a subcritical Hopf bifurcation and
			$\alpha_\mathcal{H}(\e)>\alpha_\mathcal{C}(\e)$.}
		\label{fig:VISubCHdiagram}
	\end{figure}
	
	\begin{rem}
		In this case one can obtain a proposition similar to \autoref{prop:melnikov} which gives the condition for the periodic orbits to collide in a saddle-node bifurcation 
		and, when  it exists, the value of the saddle node bifurcation parameters using the Melnikov function $M(v,\delta)$ given in \eqref{eq:melnikov}.
	\end{rem}
	
	Theorems  \ref{thm:bdVIsuper} and \ref{thm:bdVIsub} show that the regularization of the unfolding of a visible-invisible fold with $(\det{Z})_x(\0)<0$
	and $\alpha_\mathcal{C}<\alpha_\mathcal{H}$ behaves very closely to the classical slow-fast systems studied by Krupa and Szmolyan in \cite{krupa}.
	Next theorem shows that, when the transition function is linear, both systems are $\mathcal{C}^r-$conjugated inside the regularization zone.
	
	\begin{theo}[Linear regularization] \label{prop:linearcanard}
		Let $Z$ having a visible-invisible fold-fold point at the origin satisfying $X^1\cdot Y^1(\0)<0$ and $(\det{Z})_x(\0)<0$.
		Consider $Z^\alpha$  an unfolding of $Z$ and $Z^\alpha_\e$ its regularization with a linear transition function: $\varphi(v)=v$ for $v \in (-1,1)$.
		In addition, suppose that $\alpha = \delta \e$.
		Then there exists a change of variables which transforms system \ref{vsystem} in the region $|v|\le 1$ into the general slow-fast system \ref{eq:KSgeneral}.
		As a consequence  the maximal Canard occurs for $\alpha_\mathcal{C} = \delta_\mathcal{C} \e +\mathcal{O}(\e^{3/2})$
		\begin{equation}
		\delta_\mathcal{C} = \delta_\mathcal{H} + \bar{A} ,
		\end{equation}
		with $$
		\bar{A}=-\frac{1}{N(Z,\tilde{Z})} \left(-\frac{A_{10} A_3}{A_6^2} + \frac{A_9}{A_6} - A_1 \right).
		$$
		which has the same sign as the Lyapunov coefficient $\ell_1(\alpha(\e),\e)$ at the point $P(\alpha,\e).$
	\end{theo}
	\begin{proof}
		The proof of this proposition  is deferred to subsection \ref{sec:linearcanard} in the appendix.
	\end{proof}

	Consequently, when $\varphi$ is linear, if $\delta _\mathcal{C}<\delta_\mathcal{H}$, we have that $\ell_1 <0$ and therefore
	the hypothesis of \autoref{thm:bdVIsub} can not be fulfilled. Consequently
	the dynamics of the regularized system $Z^\alpha_\e$ in this case is always given by \autoref{thm:bdVIsuper}.
	One could think that it would  be possible to state an analogous result for the regularized system $Z^\alpha_\e$ with a nonlinear transition map.
	However, in  \autoref{ex:ch}, using the cubic transition map \ref{phi3}, we show that in these cases the dynamics is not always equivalent to the Krupa-Szmolyan system.
	
	\subsubsection{The function \texorpdfstring{$R$}{R}: disappearance of the periodic orbit after the occurrence of the maximal Canard}
	
	In this section we  give a description of how the stable periodic orbit $\Delta^{\alpha,s}_{\e}$ obtained in Theorems \ref{thm:bdVIsub} and \ref{thm:bdVIsuper},
	and which exists for $ \alpha >\alpha _{\mathcal{C}}$, disappears after the maximal Canard occurs for $ \alpha <\alpha _{\mathcal{C}}$.
	We present here an argument different from one used in these theorems, and which is independent of the character of  the critical point $P(\alpha, \e)$,
	which shows that, exponentially close (respect to $\varepsilon$) to the parameter value $\alpha = \alpha _{\mathcal{C}}(\e)$, a
	``big'' (of order $\mathcal{O}(\frac{1}{\e})$ in the $(x,v)$ plane) periodic orbit $\Delta^{\alpha,\mathcal{C}}_{\e}$ exists.
	The different mechanisms that make the stable orbit $\Delta^{\alpha,s}_{\e}$ ``disappear'' depend on the (attracting or repelling) character of
	$\Delta^{\alpha,\mathcal{C}}_{\e}$.
	
	The existence of the big periodic orbit $\Delta^{\alpha,\mathcal{C}}_{\e}$ is also shown in \cite{KristiansenH15} using the ideas of \cite{KrupaS01},
	but we present it here because our argument is different.
	As usually happens in singular perturbed problems (see \cite{eckahaus,KrupaS01}),
	periodic orbits of size $ \mathcal{O}(1)$ in the variable $y$ ($\mathcal{O}(\frac{1}{\e})$ in the variable $v$) exist when
	the parameter $\alpha$ is close to  the  value $\alpha_\mathcal{C}= \delta_{\mathcal{C}}\e + \mathcal{O}(\e^{3/2})$ where the maximal Canard exists, in fact, exponentially close.
	
	The reasoning which gives the existence of these  ``big'' periodic orbits will be made in the original variables $(x,y)$  of the problem and is the following.
	Take $y^*<0$, small but independent of $\e$.
	Consider $\alpha = \delta \e$,  the section $\Theta =\{ (0,y), y \le y^*\}$ and  $\Theta_\e=\{ (0,y), \ y \le \e \}$ and consider the  maps
	$
	\pi^s, \ \pi ^u( \cdot;\delta) : \Theta \to \Theta_\e
	$, which follow the flow in positive or negative times until it meets  the section $\Theta_\e$.
	These maps are well defined if the Fenichel manifolds are close enough, therefore for $\alpha=\delta \e$ close enough to $\alpha _{\mathcal{C}}$.
	
	Fix $ y\in \Theta$.
	Depending of the position of the Fenichel manifolds, which depends on the sign of $\delta-\delta_{\mathcal{C}}$,
	the sign of $f(\delta)= \pi^s(y;\delta)-\pi^u(y;\delta)$ changes (see  \autoref{fig:R1}). By Bolzano theorem
	it must exist a value of $\delta = \delta (y;\e)$ such that $f(\delta)=\pi^s(y;\delta)-\pi^u(y;\delta)=0$,
	and therefore a periodic orbit $\Delta^{\alpha,\mathcal{C}}_{\e}$ of system \ref{vsystem} passing through $(0,\frac{y}{\e})$ exists for this value of $\delta$.
	Moreover, calling $\Lambda^{\alpha,s,u}_\e\cap \{x=0\}= v^{s,u}$, for $\delta = \delta (y;\e)$:
	$$
	0=f(\delta)= \pi^s(y;\delta)-\pi^u(y;\delta)= \pi^s(y;\delta)-y^s(\delta)+ y^s(\delta)-y^u(\delta)+ y^u(\delta)-\pi^u(y;\delta)
	$$
	where $y^{s,u}= v^{s,u}\e$.
	Due to the exponential attraction of the Fenichel manifolds,  we know that $\pi^{s/u}(y;\delta)-y^{s/u}(\delta)=\e \mathcal{O}(e^{-\frac{c(y)}{\e}})$.
	Using that
	$y^s(\delta)-y^u(\delta)= \e ^{3/2}C(\delta -\delta _{\mathcal{C}})+\mathcal{O}(\e^{2})$, (see \autoref{Distcanard})
	we obtain that, for $\e$ small enough:
	$$
	\delta (y;\e)=\delta _{\mathcal{C}}+  \mathcal{O}\left(e^{-\frac{c(y)}{\e}}\right), \quad  c(y)>0.
	$$
	
	\begin{figure}[!htb]
		\centering
		\begin{tiny}
			\def\svgscale{0.5}
			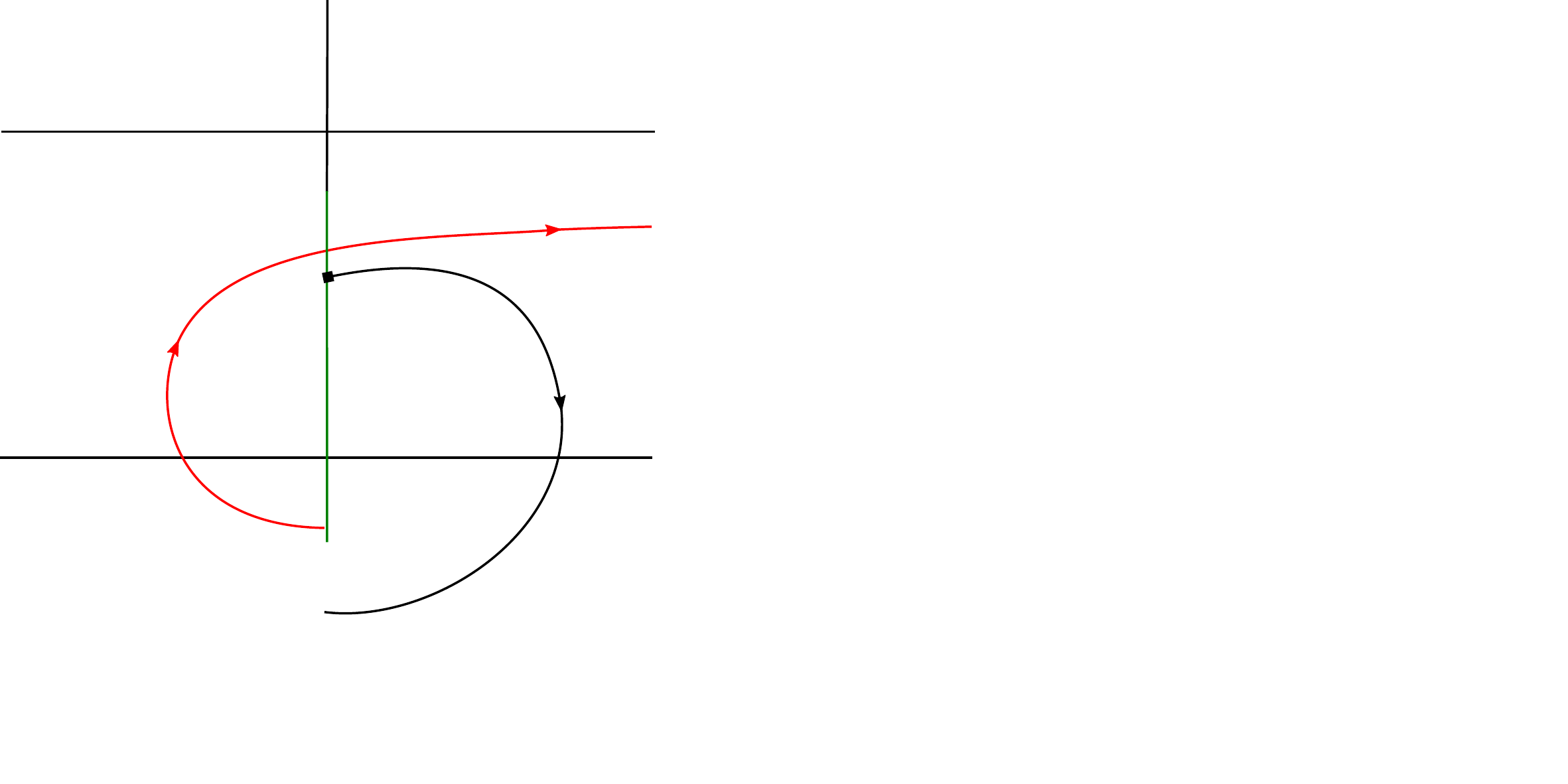
		\end{tiny}
		\caption{The maps $\pi{u,s}(y,\delta)$ for $\alpha=\delta \e$ close to the value $\alpha=\alpha_\mathcal{C}$}
		\label{fig:R1}
	\end{figure}
	
	The orbit $\Delta^{\alpha,\mathcal{C}}_{\e}$ which arises from the point $(0,v)=(0,\frac{y}{\e})$ is given, in first order,
	by the singular orbit having four pieces: $\Delta= \Delta_1\cup \Delta _2\cup \Delta_3\cup \Delta_4$.
	\begin{itemize}
		\item
		$\Delta_1$ is the orbit through $(0,y)$ of the vector field $Y$ written in $(x,v)=(x,\frac{y}{\e})$:
		$$
		\Delta_1 =\{ \varphi_Y (t; 0,y),\ -t_1\le t\le t_2\}, \mbox{where} \ \varphi_Y (-t_1; 0,y)=(x,0),  \ \varphi_Y (t_2; 0,y)=(\bar x,0)
		$$
		where $\varphi_Y$ denotes the flow of the vector field $Y$.
		\item
		$\Delta_2$ is the vertical orbit through $(\bar x,-1)$:
		$$
		\Delta_2 =\{ (\bar x, v(t))\}, \mbox{where} \ \e v'(t) = F^2(\bar x, v(t);0,0), \ t_2 \le t \le t_2+ \e \bar t:=t_3
		$$
		\item
		$\Delta_3$ is the piece of maximal Canard between the points $(x, m_0(x))$ and $(\bar x, m_0(\bar x))$ (see \autoref{IVslowmanifold}):
		$$
		\Delta_3 =\{ (x(t), m_0(x(t))\}, \mbox{where} \ x'(t) = F^1(x(t), m_0(x(t));0,0), \ t_3 \le t \le  t_4\ .
		$$
		\item
		$\Delta_4$ is the vertical orbit through $(x,\bar v)$:
		$$
		\Delta_4 =\{ (x, v(t))\}, \mbox{where} \ \e v'(t) = F^2( x, v(t);0,0), \ t_4 \le t \le t_4+ \e \tilde t:=t_5
		$$
	\end{itemize}
	The stability of the periodic orbit $\Delta^{\alpha,\mathcal{C}}_{\e}$ arising from $\Delta$ is given by:
	$$
	\int_{0}^{T} \textrm{Div} F(\Delta^{\alpha,\mathcal{C}}_{\e}(t),\delta \e,\e)dt
	$$
	where $T=T(\e)$ is its period.
	Using the form of the equations and the fact that $\varphi (v)=-1$ for $v\le -1$, this integral  is given in first order respect to $\e$
	by the so called ``way-in, way-out'' function:
	$$
	\frac{1}{\e} R= \frac{1}{\e}\int _{t_3}^{t_4} \varphi '(m_0(x(t)) (X^2-Y^2) (x(t),0) dt.
	$$
	which correspond to the integration along the ``Canard piece'' $\Delta_3$.
	It is important to mention that, even if our system is not slow fast for $v\le -1$ (it is simply the vector field $Y$
	written in variables $(x, v=\frac{y}{\e})$), and the time spent in $\Delta _1$ is of the same order that the one in $\Delta_3$, the fact that $\varphi$ is constant
	for $v \le -1$ makes the classical ``way-in, way-out'' function be  also the dominant term of this integral.
	
	Changing variables in the integral $s=x(t)$, and using that $ds= F^1(x(t), m_0(x(t));0,0)dt$ and that
	$\bar x= \phi_Y(x)$, where $\phi_Y$ is the Poincar\'{e} map associated to the vector field $Y$ near its invisible fold given in \autoref{prop:frX},
	one obtains a suitable form for this function, parameterized by the coordinate $x$:
	\begin{equation}\label{eq:R}
	R(x)=-\int _{\phi_Y(x)}^{x} \varphi '(m_0(s)) \frac{(X^2-Y^2)^2}{2(\det{Z)}}(s,0)ds .
	\end{equation}
	Note that in \cite{KrupaS01} the authors  parametrize this function by the ``hight'' of the periodic orbit that in our case corresponds to $y= \mathcal{O}(x^2)$.
	
	For $0<x<1$, using that $\phi_Y(x)=-x +\beta_Y x^2+O(x^3)$, as given in \autoref{prop:frX},  the Taylor expansion of $R(x)$ is:
	$$
	R(x)= -\left(\varphi'(\bar v) G'(0)\beta_Y+\frac{2}{3}\varphi''(\bar v)m_0'(0) G'(0)+\frac{\varphi'(\bar v)}{3}G''(0)\right) x^3 + O(x^4).
	$$
	where
	$
	\displaystyle{G(x)=\frac{(X^2-Y^2)^2}{2(\det{Z)}}(x,0)}
	$.
	Using the expression for $m_0(x)$ given in \ref{IVslowmanifold} and $G(x)$ one obtains:
	$$
	R(x) = A x^3 +O(x^4)
	$$
	with
	$$
	A=- \frac{G'(0)}{3}
	\left[
	\varphi'(\bar v)
	\left\{
	3\beta_Y +2 \left( 
	\frac{X^2_{xx}-Y^2_{xx}}{X^2_{x}-Y^2_{x}}-\frac{ (\det {Z)_{xx}} }{ 2(\det{Z)_x} } \right)\right\}
	+\frac{2\varphi''(\bar v)}{\varphi'(\bar v)}\frac{X^2_{xx}Y^2_{x}-Y^2_{xx}X^2_{x}}{(Y^2_{x}-X^2_{x})^2}\right](\0)
	$$
	and $\displaystyle{G'(0)= \frac{(X_x^2-Y_x^2)^2}{2(\det{Z)_x}}(\0)<0}$.
	
	Then, the periodic orbit $\Delta^{\alpha,\mathcal{C}}_{\e}$ of size $O(1)$ originated at the so called ``Canard explosion'' is stable if
	\begin{equation} \label{Rsign}
	B=\left[
	\varphi'(\bar v)
	\left\{
	3 \beta_Y +2 \left( 
	\frac{X^2_{xx}-Y^2_{xx}}{X^2_{x}-Y^2_{x}}-\frac{ (\det {Z)_{xx}} }{ 2(\det{Z)_x} } \right) \right\}
	+\frac{2\varphi''(\bar v)}{\varphi'(\bar v)}\frac{X^2_{xx}Y^2_{x}-Y^2_{xx}X^2_{x}}{(Y^2_{x}-X^2_{x})^2}\right](\0)<0
	\end{equation}
	and unstable otherwise.
	Moreover, the orbit $\Delta^{\alpha,\mathcal{C}}_{\e}$, for $\alpha = \delta (x)\e$  stays stable (unstable) while the function
	$R(x)$ stays negative  (positive).
	
	Next proposition tells us in what region of the parameter plane the periodic orbit $\Delta^{\alpha,\mathcal{C}}_{\e}$, which raises from the maximal Canard, 
	appears.
	Moreover, it also explains the relation between the periodic orbit $\Delta^{\alpha,\mathcal{C}}_{\e}$ and the stable  periodic orbit 
	$\Delta^{\alpha,s}_{\e}$ given in
	Theorems \ref{thm:bdVIsuper} and \ref{thm:bdVIsub}.
	Recall that the stable  periodic orbit $\Delta^{\alpha,s}_{\e}$ obtained in these theorems exists for $(\alpha,\e)$ between the curves
	$\mathcal{C}$ and $\mathcal{H}$, independently of the sign of $B$.
	
	\begin{prop} \label{pocanard}
		Let $B$ be the quantity given in \ref{Rsign}, then:
		\begin{itemize}
			\item
			If $B<0$ the periodic orbit $\Delta^{\alpha,\mathcal{C}}_{\e}$ is stable and appears for $(\alpha,\e)$ located on the right of the curve
			$\mathcal{C}$.
			In addition, for $(\alpha,\e)$ on the left of the curve $\mathcal{C}$, for any compact set $\mathcal{K}$ containing the critical point
			$P(\alpha,\e)$, for $\e$ small enough,
			given an initial condition inside $\mathcal{K}$ then its trajectory tends to $P(\alpha,\e)$ backward in time and  leaves  $\mathcal{K}$
			forward in time.
			\item
			If $B>0$ the  periodic orbit $\Delta^{\alpha,\mathcal{C}}_{\e}$ is unstable and appears for $(\alpha,\e)$ located on the left of the curve
			$\mathcal{C}$.
			Moreover, for each $\e>0$ small enough, there exists a value $\alpha_{\tilde{\mathcal{S}}}$ satisfying
			$\alpha_\mathcal{D}^-(\e)<\alpha_{\tilde{\mathcal{S}}}(\e)<\alpha_\mathcal{C}(\e)$ such that
			\begin{itemize}
				\item
				For $\alpha_{\tilde{\mathcal{S}}}(\e) <\alpha<\alpha_\mathcal{C}(\e)$ the unstable periodic orbit 
				$\Delta^{\alpha,\mathcal{C}}_{\e}$
				coexists with the smaller stable periodic orbit $\Delta^{\alpha,s}_{\e}$ given in  
				Theorems \ref{thm:bdVIsuper} and \ref{thm:bdVIsub}.
				\item
				For $\alpha<\alpha_{\tilde{\mathcal{S}}}(\e)$ there are no periodic orbits.
			\end{itemize}
		\end{itemize}
	\end{prop}
	\begin{proof}
		At first, we prove the region of the parameter plane where the periodic orbit appears.
		Suppose $B<0$ and that $\Delta^{\alpha,\mathcal{C}}_{\e}$ exists for $(\alpha,\e)$ on the left of $\mathcal{C}$.
		Therefore, the Poincar\'{e} map defined in the cross section $\Theta$ (see \autoref{thm:bdVIsuper}) is repelling near the unstable
		manifold and attracting near the periodic orbit $\Delta^{\alpha,\mathcal{C}}_{\e}$.
		This reasoning guarantees the existence of a bigger unstable
		periodic $\tilde{\Delta}^{\alpha,u}_\e$ which is located below the unstable manifold and its interior contains $\Delta^{\alpha,\mathcal{C}}_{\e}$ 
		which is a contradiction with $B<0$. 
		The case $B>0$ can be proved analogously.
		The persistence of the periodic orbit $\Delta^{\alpha,s}_{\e}$ for $\alpha<\alpha_\mathcal{C}(\e)$
		is given by the return map 
		as we did 
		in \autoref{thm:bdVIsub} using the repelling character of $\Delta^{\alpha,\mathcal{C}}_{\e}$, for an illustration see \autoref{fig:pocanard}. 
	\end{proof}
	
	\begin{figure}[!htb]
		\centering
		\begin{tiny}
			\def\svgwidth{0.7\textwidth}
			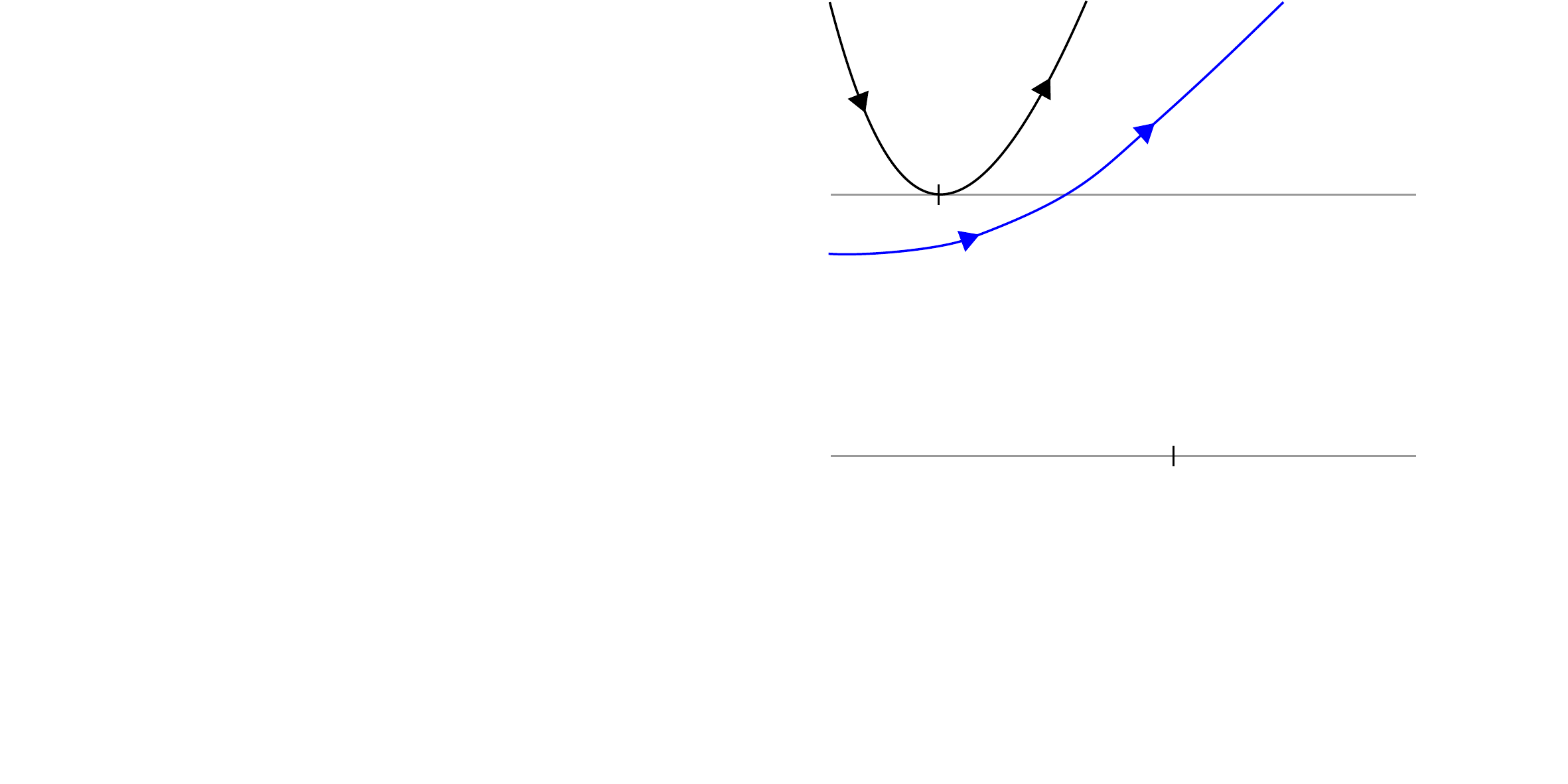
		\end{tiny}
		\caption{The periodic orbits of $Z^\alpha_\e$ for $\alpha<\alpha_\mathcal{C}$ when $B>0$.
			A saddle-node occurs at $\alpha = \alpha _{\mathcal{\tilde S}}.$}
		\label{fig:pocanard}
	\end{figure}
	
	\begin{rem}	When $B<0$, using the information given in \autoref{thm:bdVIsuper} we have two stable periodic orbits 
		$\Delta^{\alpha,C}_\e$ and $\Delta^{\alpha,s}_\e$ for $\alpha$ near the value $\alpha_\mathcal{C}$. 
		Therefore, in this situation, the simplest case is when these two periodic orbits coincide. 
		Moreover, when $B>0$, since we have two periodic orbits $\Delta^{\alpha,C}_\e$ and ${\Delta}^{\alpha,s}_\e$ for 
		$\alpha_{\tilde{\mathcal{S}}}<\alpha <\alpha_\mathcal{C}$ and no periodic for $\alpha<\alpha_{\tilde{\mathcal{S}}}$, 
		the simplest case, is when the two periodic orbits collide in a saddle-node bifurcation for $\alpha=\alpha_{\tilde{\mathcal{S}}}$ 
		and then disappear for $\alpha<\alpha_{\tilde{\mathcal{S}}}$. 
		
		The above discussion, joined to the results given in Theorems \ref{thm:bdVIsuper}, \ref{thm:bdVIsub} and  \autoref{pocanard},  suggest that
		we have four simplest possibilities for the bifurcation diagram of $Z^\alpha_\e$ having a visible-invisible fold with $(\det{Z})_x(\0)<0$.
		Moreover, if $R(x)<0$  $\forall x \ge 0$, then, for any $x>0$, one can find $\e$ small enough such that for
		$\delta=\delta (x;\e)$ the stable periodic orbit $\Delta^{\alpha,\mathcal{C}}_{\e}$ exists until $\alpha =\delta(x;\e)\e>\alpha_\mathcal{C}(\e)$.
		The orbit increases unboundedly as $\alpha$ approaches $\alpha_\mathcal{C}(\e)$ and ``disappears at infinity''.
		This is the so-called Canard explosion. Nevertheless, due to the lack of compactness, fixing a value of $\e$ small enough so that all the previous results about the Fenichel manifolds are valid,
		we can only find the periodic orbit for $|x|\le x^* (\e)$ and therefore for $\alpha$ close but until a certain distance of the $\alpha _\mathcal{C}(\e)$.
		Summarizing,  the bifurcation diagram of $Z^\alpha_\e$
		is then exactly as in Figures \ref{fig:VISubHCdiagram} or \ref{fig:VISubCHdiagram}, depending on $\ell_1(\alpha,\e)$ sign.
		
		However, if $B>0$, then we insert \autoref{fig:pocanard} between subfigures $(a)$ and $(b)$ of Figures \ref{fig:VISubHCdiagram}
		and \ref{fig:VISubCHdiagram}, depending on $\ell_1(\alpha,\e)$ sign.
		We summarize the four possible bifurcation diagram in \autoref{fig:POVI}.
	\end{rem}
	
	\begin{figure}[!htb]
		\centering
		\begin{tiny}
			\def\svgwidth{0.8\textwidth}
			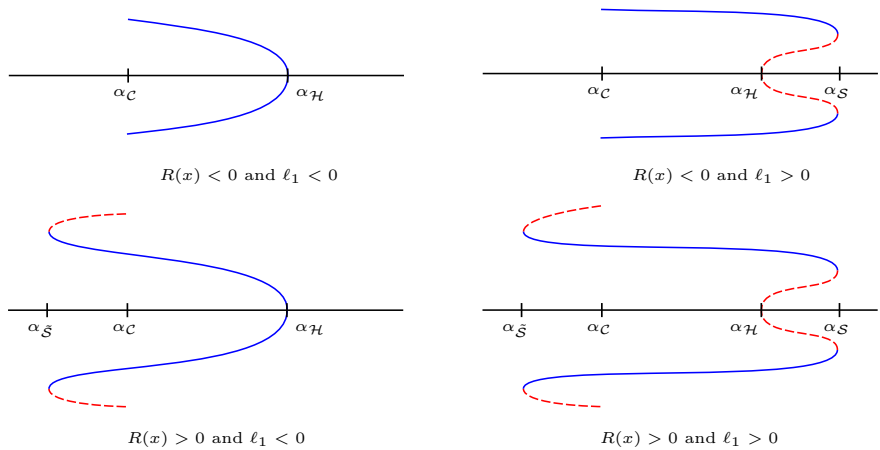
		\end{tiny}
		\caption{The four simplest- possible bifurcation diagrams for the periodic orbits of system $Z^\alpha_\e$ having a visible-invisible fold-fold and satisfying
			$(\det{Z})_x(\0)<0$, depending on the sign of $B$ and the first Lyapunov coefficient $\ell_1$.}
		\label{fig:POVI}
	\end{figure}
	
	In the sequel, we present examples which exhibit the behaviors stated in Theorems \ref{thm:bdVIsuper} and \ref{thm:bdVIsub} and in \autoref{pocanard}.
	First,  consider $Z^\alpha$ given by
	\begin{equation} \label{ex:VIns}
	Z_\alpha(x,y)=\begin{cases}
	X_\alpha(x,y) = (1+2x, x+ \frac{7}{2}y -\alpha) \\
	Y(x,y) = (-1,-3x)
	\end{cases}
	\end{equation}
	which has a visible-invisible fold at the origin for $\alpha=0$ and satisfies $(\det{Z})_x(\0)=-2$.
	
	The $\varphi$-regularization $Z^\alpha_\e(x,y)$ in  coordinates $(x,v)$ where $y=\e v$ has the form
	\begin{equation} \label{ex:VIreg}
	Z^\alpha_\e (x,\e v) = \begin{cases}
	\dot{x} = 2x + \varphi(v) \left( 2+ 2x  \right), \\
	\e \dot{v} =-2x-\alpha + \frac{7 \e v }{2}+ \varphi(v) \left( 4x - \alpha +\frac{7 \e v }{2}  \right),
	\end{cases}
	\end{equation}

	\begin{exmp}[Supercritical Hopf bifurcation for the visible-invisible fold]  \label{ex:VIsuper}
		Consider the transition map $\varphi(v)$
		$$
		\varphi(v) = v^5 +\frac{3}{2}v^3+\frac{1}{2}v, \, \textrm{for} \, v \in (-1,1).
		$$
		The  point is  $P(\alpha,\e)= \left(-\frac{1}{2}\alpha + \mathcal{O}_2(\alpha,\e), 0 + \mathcal{O}_2(\alpha,\e) \right)
		$ and the bifurcation curves are:
		%$\mathcal{D}$, and $\mathcal{H}$ are given by
		\begin{eqnarray*}
			\mathcal{D} &=& \left\{ (\alpha,\e): \e = \frac{9}{32} \alpha^2  +\mathcal{O}(\alpha^3) \right\} \\
			\mathcal{H} &=& \left\{ (\alpha,\e): \alpha = \frac{11}{3} \e  +\mathcal{O}(\e^2) \right\}.
		\end{eqnarray*}
		The first Lyapunov coefficient is $
		\ell_1(\alpha(\e),\e)=\frac{1}{\sqrt{\e}} (-23.15 + \mathcal{O}(\e)),
		$
		therefore  the  Hopf bifurcation is supercritical. 
		The Canard trajectory occurs over the curve
		$$
		\mathcal{C} = \left\{ (\alpha,\e): \alpha = 1.98 \  \e  +\mathcal{O}(\e^{3/2}) \right\},
		$$
		\noindent therefore, the stated in \autoref{thm:bdVIsuper} holds.
		
		We fix $\e=0.01$ and vary $\alpha$ sufficiently small.
		We obtain $\alpha^\pm_\mathcal{D}(\e_0) \approx \pm 0.18 $, $\alpha_\mathcal{H}(\e_0) \approx 0.03$ and $\alpha_\mathcal{C}(\e_0) \approx 0.019$
		obtained by the intersection between the line $\e=\e_0$ and curves $\mathcal{D}$, $\mathcal{H}$ and $\mathcal{C}$ respectively.
		
		\begin{figure}[!htb]
			\centering
			\begin{tiny}
				\subfigure[\label{fig:ExVIchSupera0015}]{\includegraphics[scale=0.6]{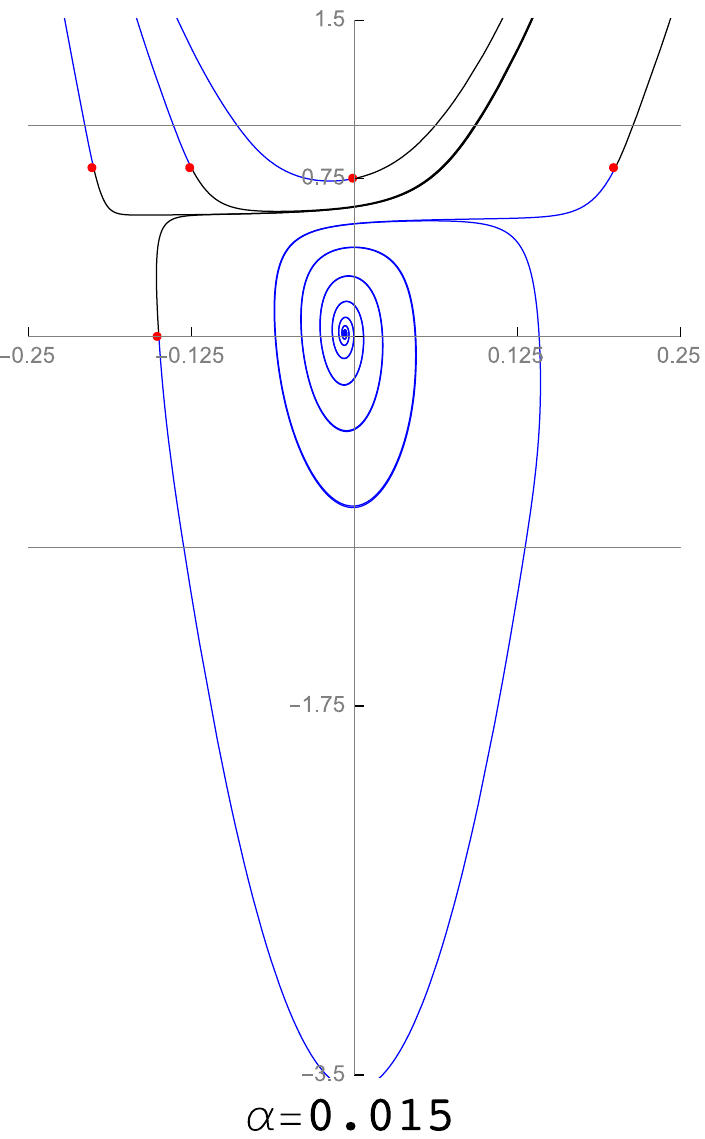}}  \hspace{1cm}
				\subfigure[\label{fig:ExVIchSupera0019}]{\includegraphics[scale=0.6]{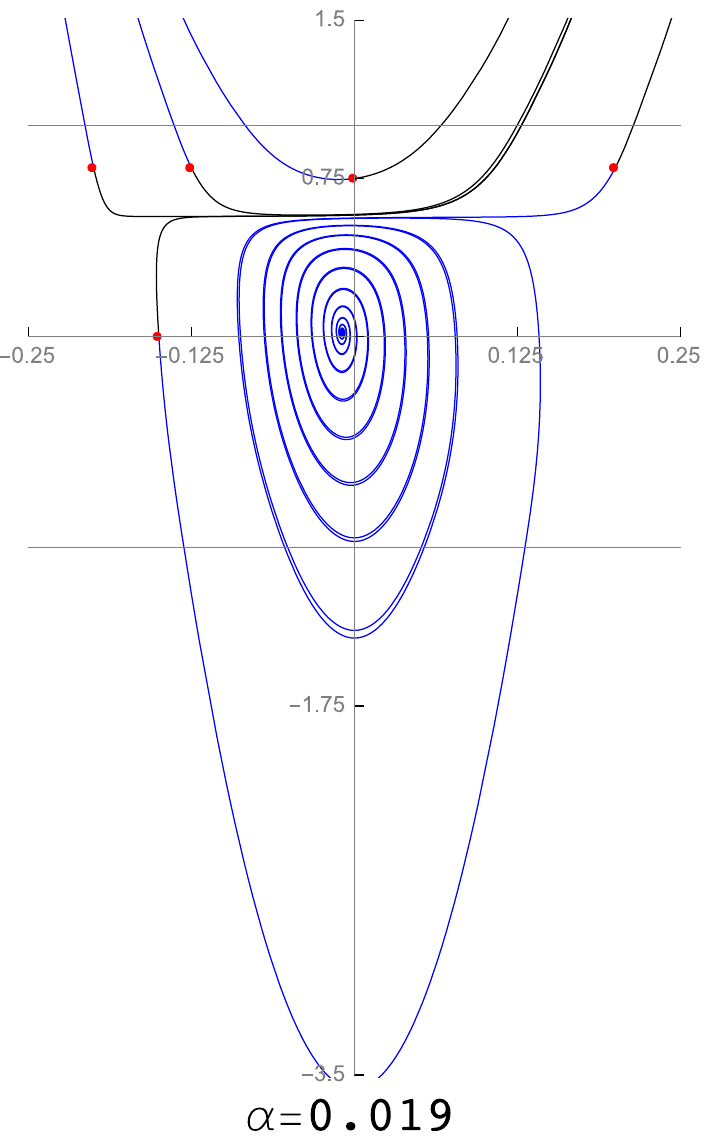}}  \hspace{1cm}
				\subfigure[\label{fig:ExVIchSupera002}]{\includegraphics[scale=0.6]{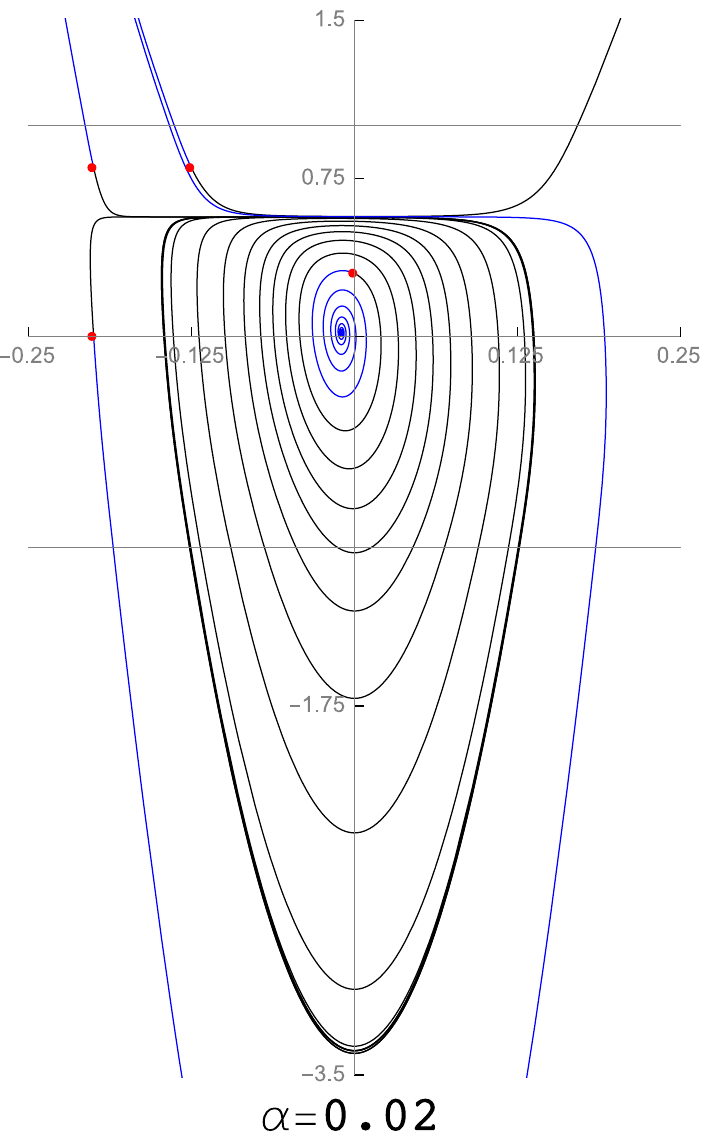}}  \\
				\subfigure[\label{fig:ExVIchSupera0025}]{\includegraphics[scale=0.6]{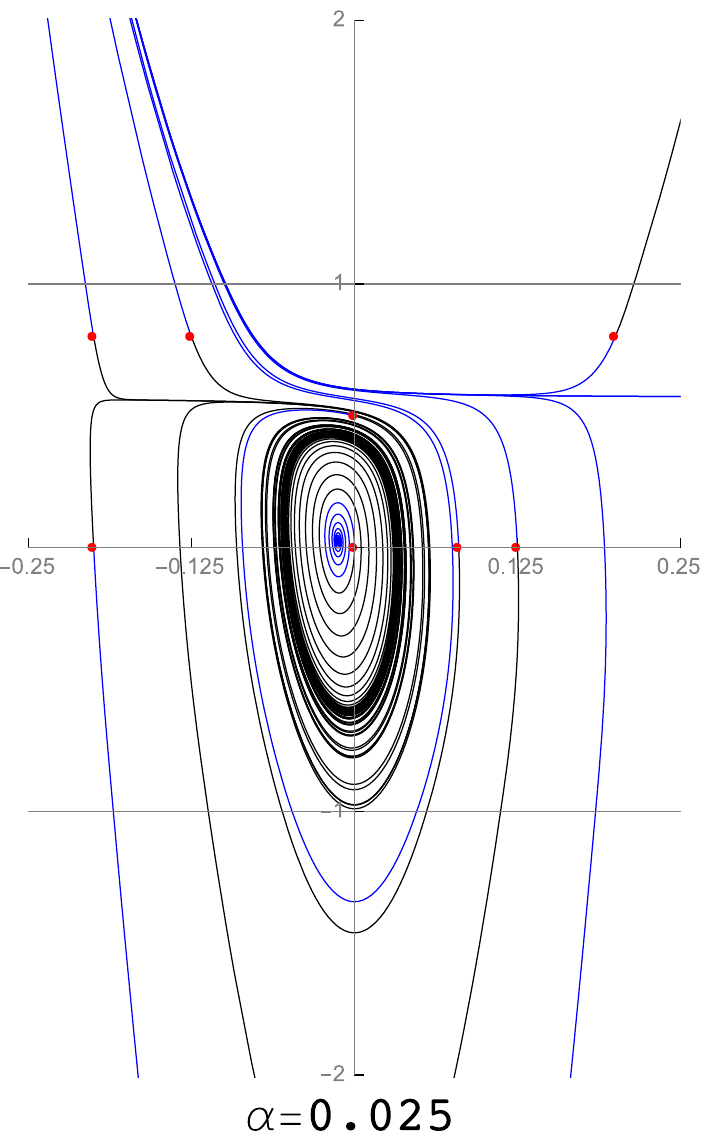}}  \hspace{1cm}
				\subfigure[\label{fig:ExVIchSupera003}]{\includegraphics[scale=0.6]{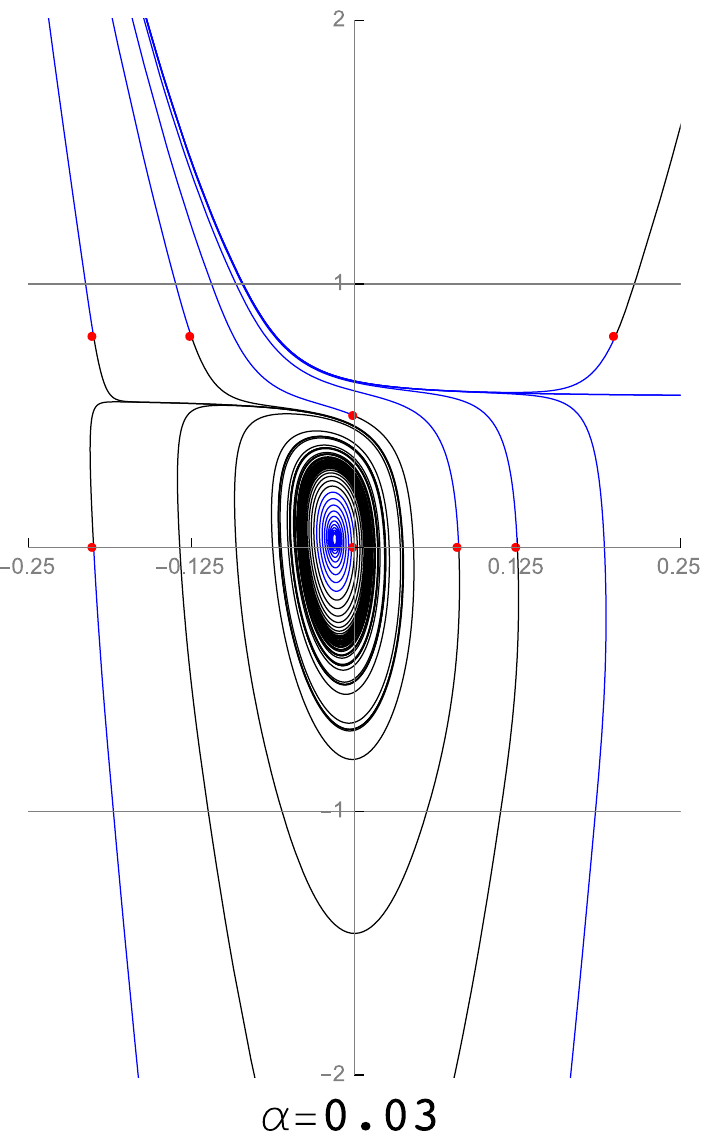}}  \hspace{1cm}
				\subfigure[\label{fig:ExVIchSupera005}]{\includegraphics[scale=0.6]{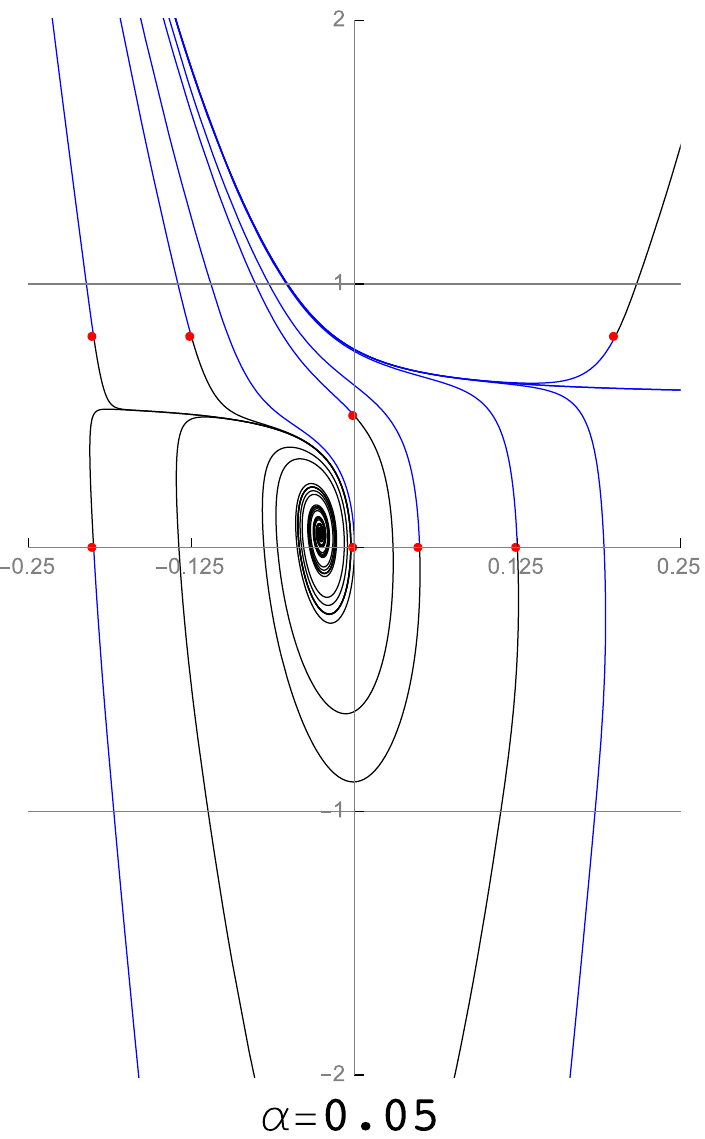}} 
			\end{tiny} \\
			\begin{scriptsize}
				\textcolor{red}{$\blacksquare$}  Initial condition \quad  \textcolor{blue}{$\blacksquare$} Negative time \quad  \textcolor{black}{$\blacksquare$} Positive Time
			\end{scriptsize}
			\caption{\autoref{ex:VIsuper}: Trajectories for the regularized vector field $Z^\alpha_\e$ for different values of $\alpha$ and $\e=0.01.$}
			\label{fig:ExVIsuper}
		\end{figure}

		In \autoref{fig:ExVIchSupera0015}, $\alpha=0.015$: the stable Fenichel manifold $\Lambda^{\alpha,s}_\e$ is above the unstable Fenichel manifold $\Lambda^{\alpha,u}_\e$ and the critical point $P(\alpha,\e)$ is an unstable focus.
		In \autoref{fig:ExVIchSupera0019}, $\alpha=0.019$: the Fenichel manifolds $\Lambda^{\alpha,u}_\e$ and $\Lambda^{\alpha,s}_\e$ are becoming closer.
		In \autoref{fig:ExVIchSupera002}, $\alpha=0.02$: the Canard has occurred and a big stable orbit $\Delta^{\alpha,\mathcal{C}}_{\e}=\Delta^{\alpha,s}_{\e}$ appears.
		In Figures \ref{fig:ExVIchSupera002} to \ref{fig:ExVIchSupera003}, one can see that the amplitude of the stable periodic orbit decreases while  $\alpha$ approaches the value $\alpha_\mathcal{H}.$
		
		In \autoref{fig:ExVIchSupera005}, $\alpha=0.05$: the subcritical Hopf bifurcation has occurred and the stable periodic orbit no longer exists. 
		The critical point $P(\alpha,\e)$ is an stable focus and there are no  periodic orbits. 
		In Figure \ref{fig:exemple43melnikov} we show the behavior of the Melnikov function for this example.
	\end{exmp}
	%\end{equation}

	\begin{exmp}[Subcritical Hopf bifurcation for the visible-invisible fold] \label{ex:ch}
		
		Consider the cubic transition map \eqref{phi3}. 
		The critical point is $	P(\alpha,\e)= \left(-\frac{1}{2}\alpha + \mathcal{O}_2(\alpha,\e), \mathcal{O}_2(\alpha,\e) \right)$, and the bifurcation curves $\mathcal{D}$, and $\mathcal{H}$ are given by:
		\begin{eqnarray*}
			\mathcal{D} &=& \left\{ (\alpha,\e): \e = 0.84 \ \alpha^2  +\mathcal{O}(\alpha^3) \right\} \\
			\mathcal{H} &=& \left\{ (\alpha,\e): \alpha =1.22 \ \e  +\mathcal{O}(\e^2) \right\}.
		\end{eqnarray*}
		The first Lyapunov coefficient is $
		\ell_1(\alpha(\e),\e)=\frac{1}{\sqrt{\e}} (0.57 + \mathcal{O}(\e)),
		$
		therefore the  Hopf bifurcation is subcritical.
		\begin{figure}[!htb]
			\centering
			\begin{tiny}
				\subfigure[\label{fig:ExVISchSuba1}]{\includegraphics[scale=0.6]{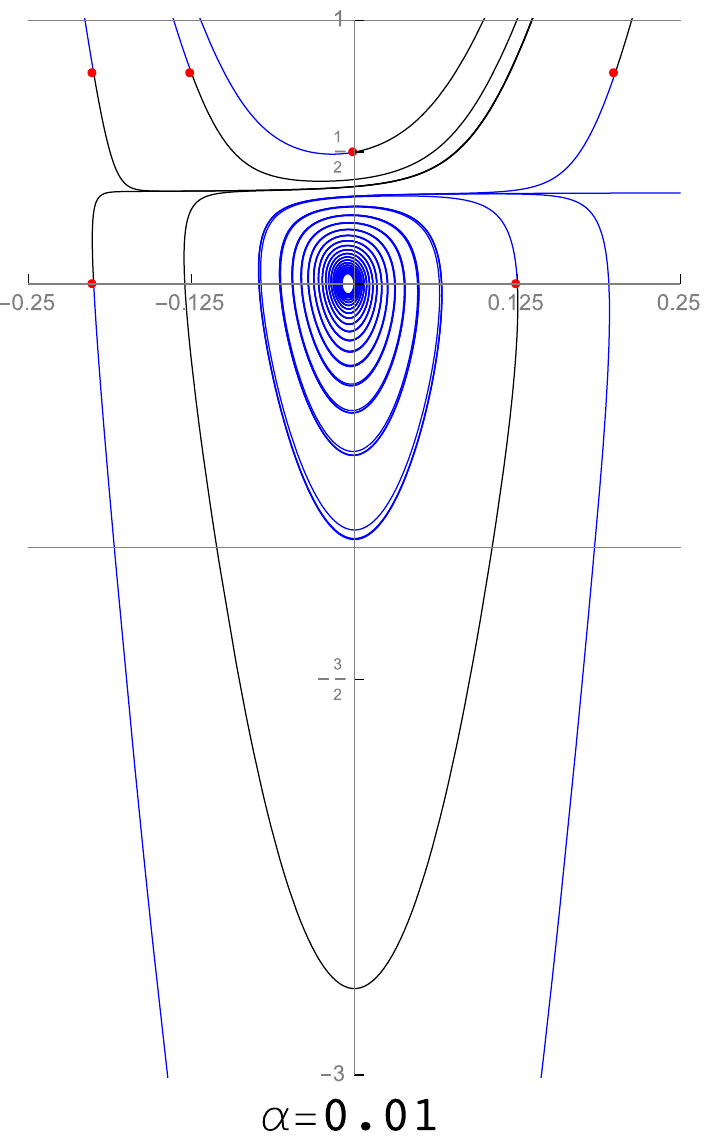}}  \hspace{1cm}
				\subfigure[\label{fig:ExVISchSuba1216}]{\includegraphics[scale=0.6]{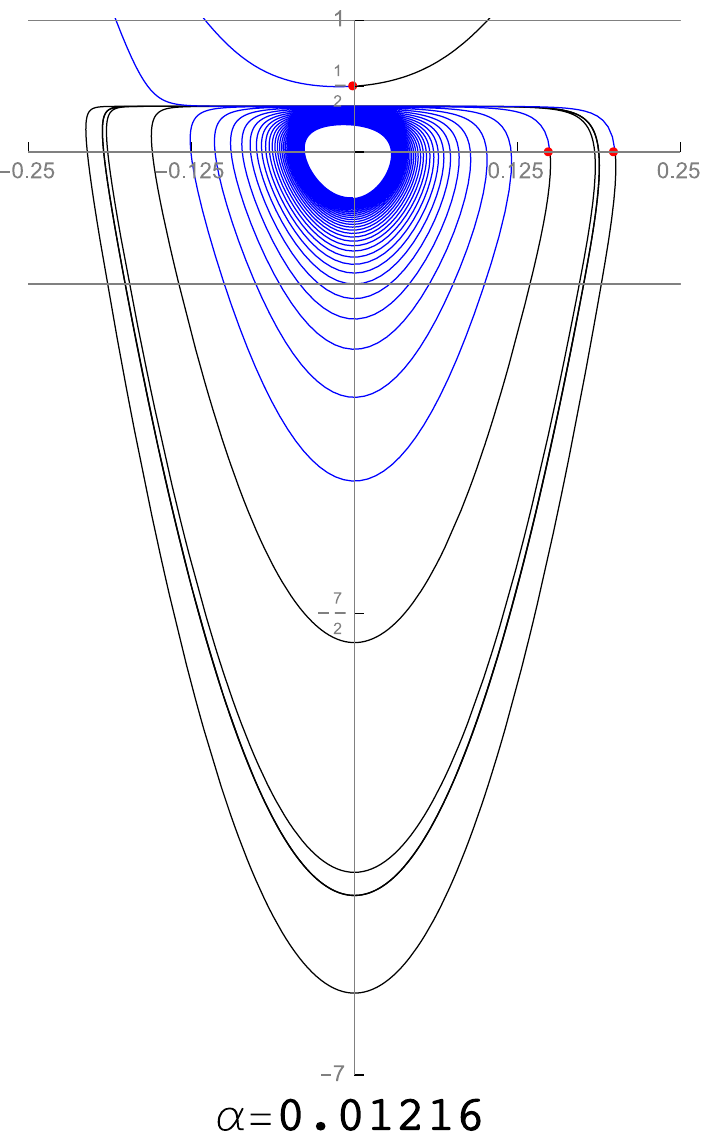}}  \hspace{1cm}
				\subfigure[\label{fig:ExVISchSuba123}]{\includegraphics[scale=0.6]{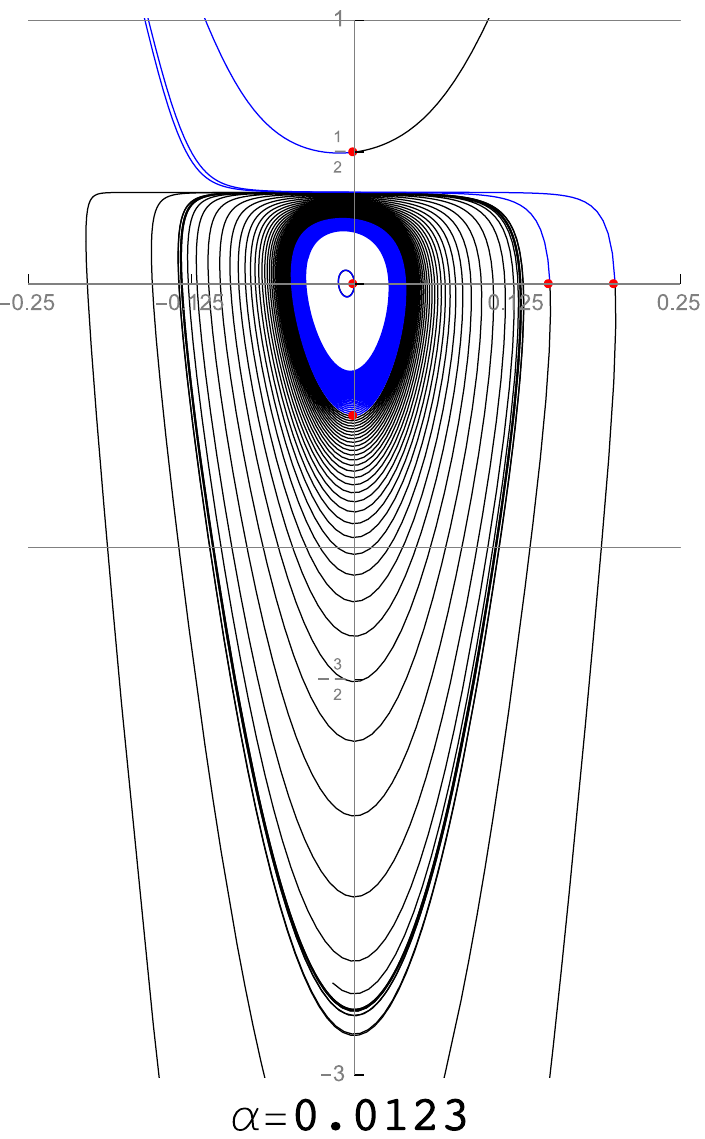}}  \hspace{1cm}
				\subfigure[\label{fig:ExVISchSuba125}]{\includegraphics[scale=0.6]{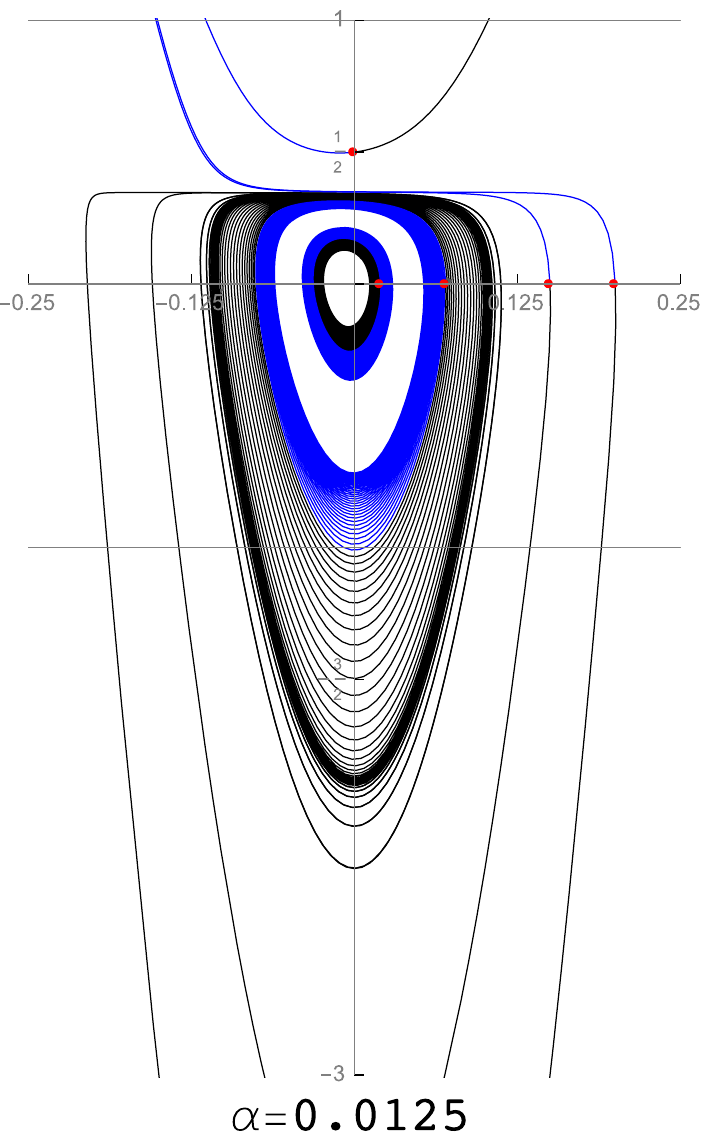}}  \hspace{1cm}
				\subfigure[\label{fig:ExVISchSuba126}]{\includegraphics[scale=0.6]{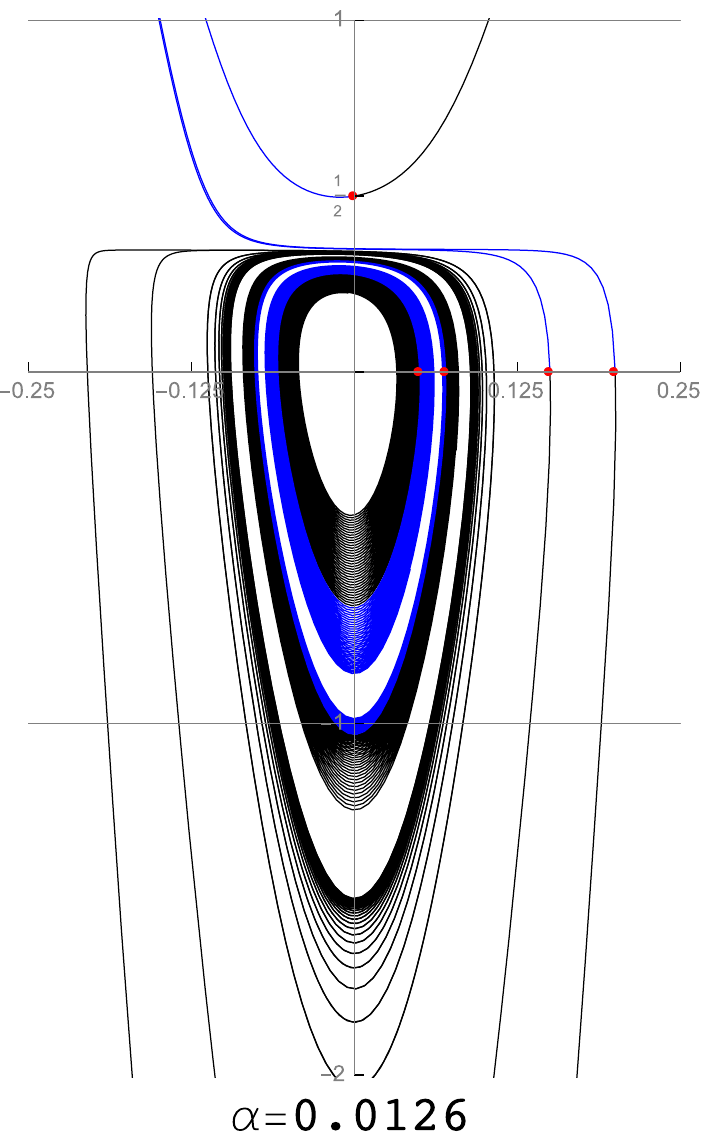}}  \hspace{1cm}
				\subfigure[\label{fig:ExVISchSuba127}]{\includegraphics[scale=0.6]{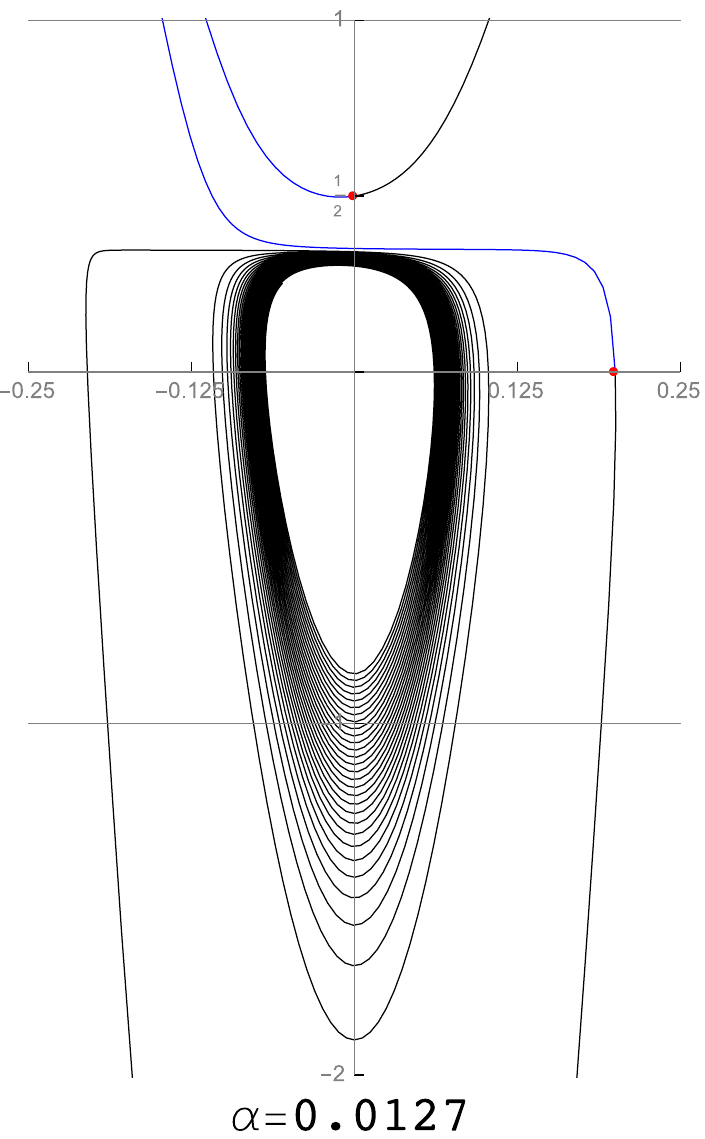}}
			\end{tiny} \\ 
			\begin{scriptsize}
				\textcolor{red}{$\blacksquare$}  Initial condition \quad  \textcolor{blue}{$\blacksquare$} Negative time \quad  \textcolor{black}{$\blacksquare$} Positive Time
			\end{scriptsize}
			\caption{Some trajectories of the regularized system $Z^\alpha_\e$ of \autoref{ex:ch} for different values of the parameter $\alpha$ and $\e=0.01$.}
			\label{fig:ExVISubch}
		\end{figure}
		The Canard trajectory occurs over the curve
		$$
		\mathcal{C} = \left\{ (\alpha,\e): \alpha = 1.21 \ \e  +\mathcal{O}(\e^{3/2}) \right\} 
		$$
		therefore, the stated in \autoref{thm:bdVIsub} holds.
		
		Fix $\e_0=0.01$ and consider values of $\alpha^\pm_\mathcal{D}(\e_0) \approx \pm 0.1 $,
		$\alpha_\mathcal{H}(\e_0) \approx 0.01222$ and $\alpha_\mathcal{C}(\e_0) \approx 0.01215$
		obtained by the intersection between the line $\e=\e_0$ and curves $\mathcal{D}$, $\mathcal{H}$ and $\mathcal{C}$, respectively.
		We are going to focus our attention to the region inside the parabola $\mathcal{D}$.
		
		In Figures \ref{fig:ExVISchSuba1} to \ref{fig:ExVISchSuba127} one can see the evolution of the dynamics of $Z^\alpha_\e$ while we vary the parameter $\alpha.$ In \autoref{fig:ExVISchSuba1}, $\alpha=0.01$,  $\alpha<\alpha_\mathcal{C}(\e_0)$:
		the stable Fenichel manifold is above the unstable one and the critical point is an unstable focus.
		In \autoref{fig:ExVISchSuba1216}, $\alpha=0.01216$: the Canard already happened.
		There exist a big stable periodic orbit $\Delta^{\alpha,s}_{\e}$ and an unstable focus.
		In \autoref{fig:ExVISchSuba123}, $\alpha=0.0123$: the subcritical Hopf bifurcation has occurred and the critical point $P(\alpha,\e)$ is a stable focus.
		An small unstable periodic $\Delta^{\alpha,u}_{\e}$ appears.
		In Figures \ref{fig:ExVISchSuba123} and \ref{fig:ExVISchSuba126}, the two periodic orbits $\Delta^{\alpha,u}_{\e}$, $\Delta^{\alpha,s}_{\e}$ 
		coexist until the parameter $\alpha$ reaches the value
		$\alpha_\mathcal{S}(\e_0)$.
		In \autoref{fig:ExVISchSuba127}, $\alpha=0.0127$: there are no periodic orbits and the stable focus $P(\alpha,\e)$ is global stable.
		This means that the value $\alpha_\mathcal{S}(\e_0)$ given by the intersection between the curve $\mathcal{S}$ and the line $\e=0.01$
		belongs to $I_\mathcal{S}=(0.01216,0.01217)$. Therefore, for $\alpha>\alpha_\mathcal{S}$ we have only an stable focus and no periodic orbits.
		In Figure \ref{fig:ex:chmelnikov} we show the behavior of the Melnikov function for this example.
		
	\end{exmp}

	\begin{figure}[htb!]
		\centering
		\subfigure[\label{fig:exemple43melnikov}]{\includegraphics[width=0.45\textwidth]{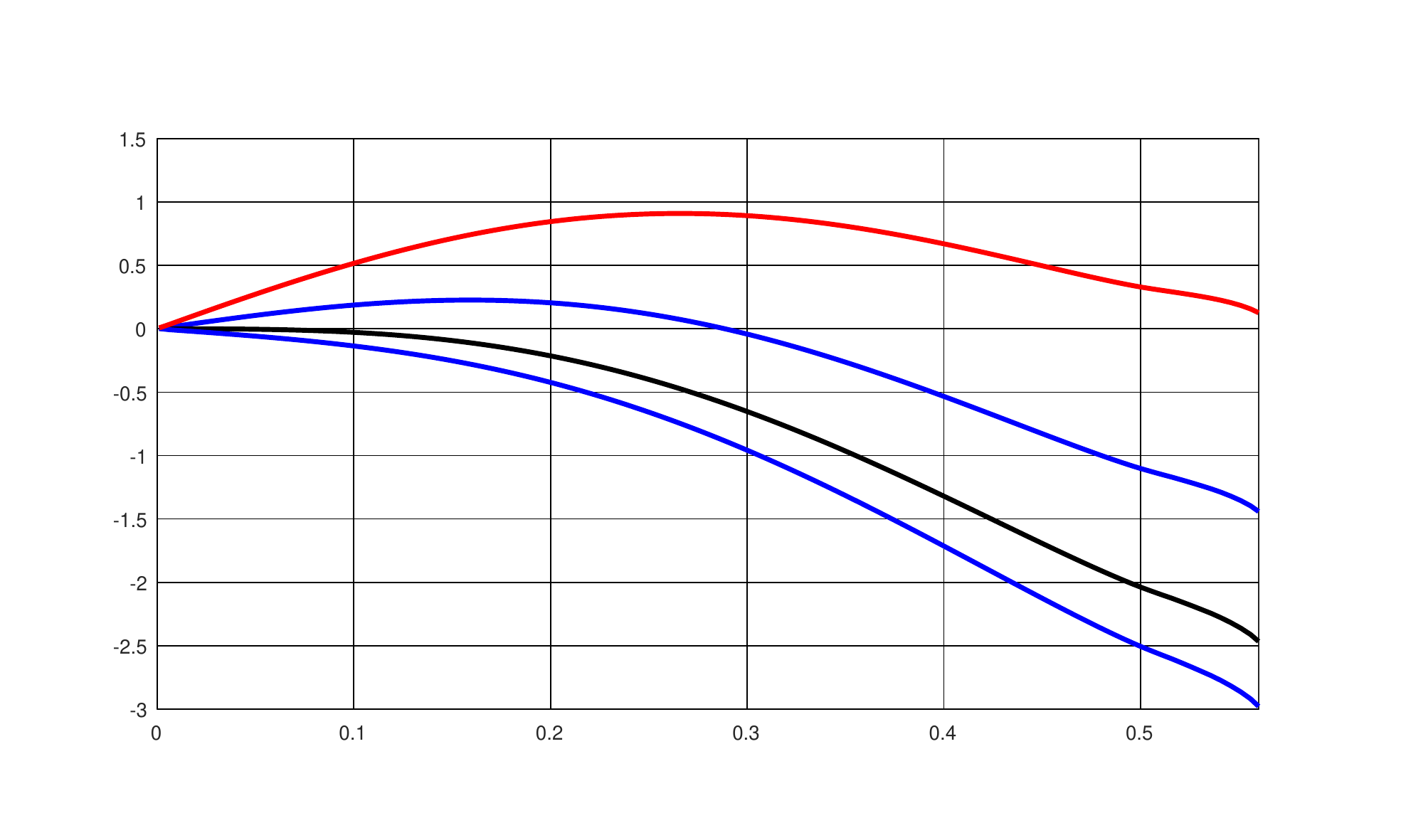}} \hspace{0.3cm}
		\subfigure[\label{fig:ex:chmelnikov}]{\includegraphics[width=0.455\textwidth]{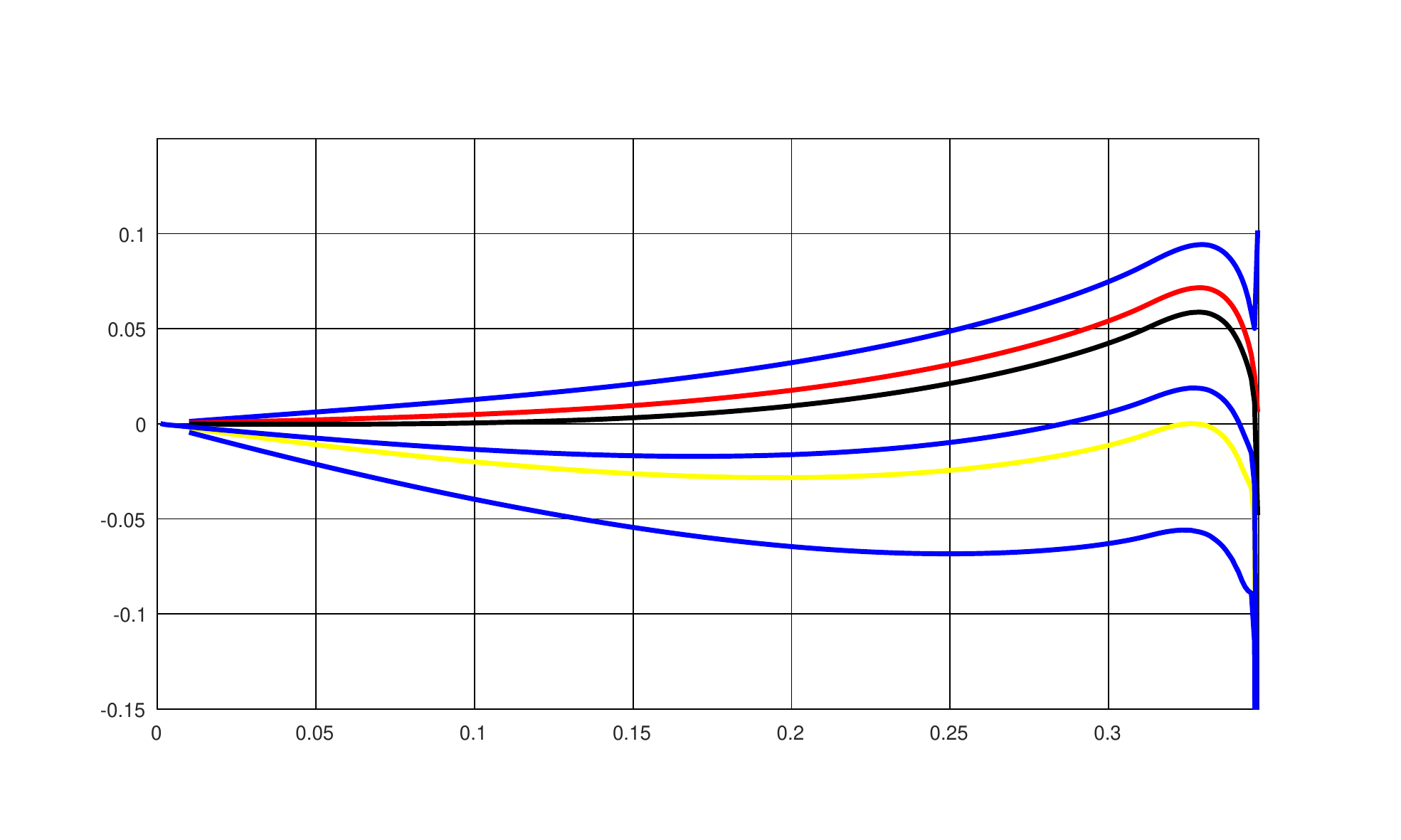}}
		\caption{(a) The Melnikov function for example \ref{ex:VIsuper}:   $\delta = \delta _\mathcal{H}$ in black is supercritical,  $\delta = \delta _\mathcal{C}$ in red. For
			$\delta _\mathcal{C}\le \delta \le  \delta _\mathcal{H}$, $M(v,\delta)$ has one zero with negative slope.
			(b) The Melnikov function for \autoref{ex:ch}:   $\delta = \delta _\mathcal{H}$ in black is subcritical, $\delta = \delta _\mathcal{C}$ in red,
			$\delta = \delta _\mathcal{S}$ in yellow. $M(v,\delta)$ has one zero with negative slope for
			$\delta _\mathcal{H}\le \delta = \delta _\mathcal{H}$, two zeros for $\delta _\mathcal{H}\le \delta \le  \delta _\mathcal{S}$ and no zeros for
			$\delta > \delta _\mathcal{S}$ and $\delta < \delta _\mathcal{C}$}
	\end{figure}

	\begin{rem}
		We want to emphasize that, due to theorem \ref{prop:linearcanard}, the hypothesis of \autoref{thm:bdVIsub} can not be fulfilled if
		we use a linear regularization function $\varphi$.
		Therefore, the importance of \autoref{ex:ch} is to show that the hypothesis of \autoref{thm:bdVIsub} are achievable for a
		$\mathcal{C}^2$ transition map.
	\end{rem}
	
	In what follows, we provide an example to demonstrate that the sign of the coefficient $B$ depends directly of the transition function $\varphi.$ For examples \ref{ex:Bnegativo} and \ref{ex:Bpositivo}, consider the vector field
	
	\begin{equation}
	\label{Bfield}
	Z_\alpha(x,y)=
	\begin{cases}
	X_\alpha= (1+0.2x,-\alpha +x \left(8 x^2+3 x+1\right)-4 y) \\
	Y(x,y)= (-1,-x \left(8 x^2+3 x+3\right))
	\end{cases}
	\end{equation}

	\begin{exmp}[A stable periodic orbit near the Canard: $B<0$] \label{ex:Bnegativo}
		\begin{figure}[!htb]
			\centering
			\begin{tiny}
				\subfigure[\label{fig:Bnegativo000140}]{\includegraphics[scale=0.6]{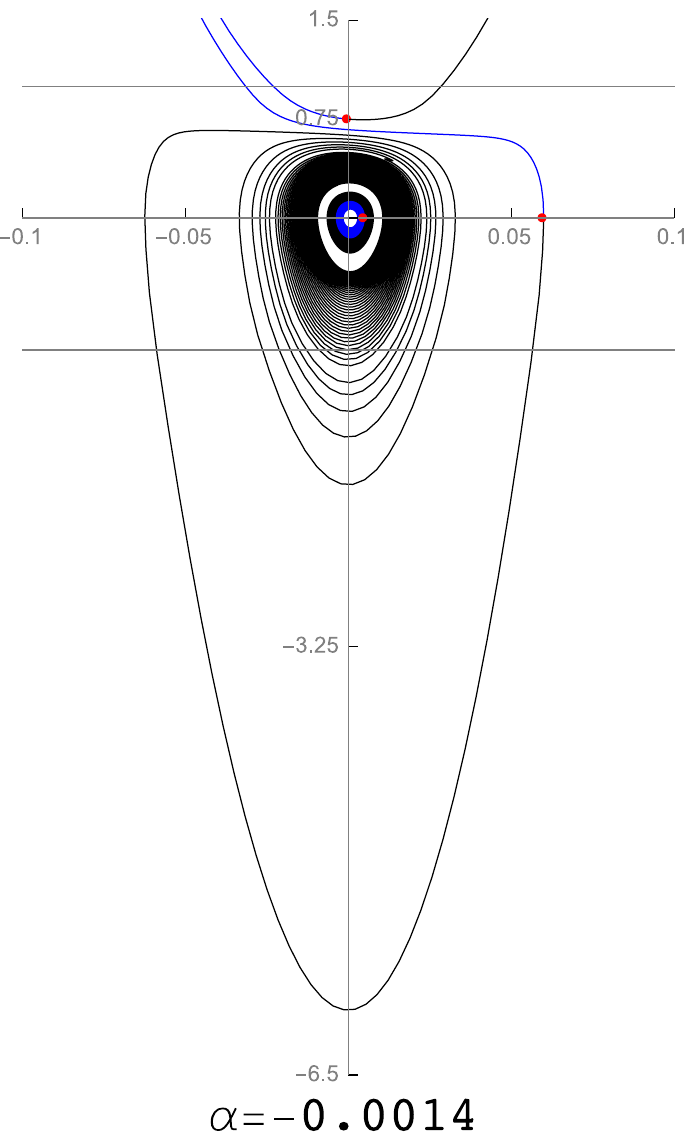}} \hspace{1cm}
				\subfigure[\label{fig:Bnegativo000216}]{\includegraphics[scale=0.6]{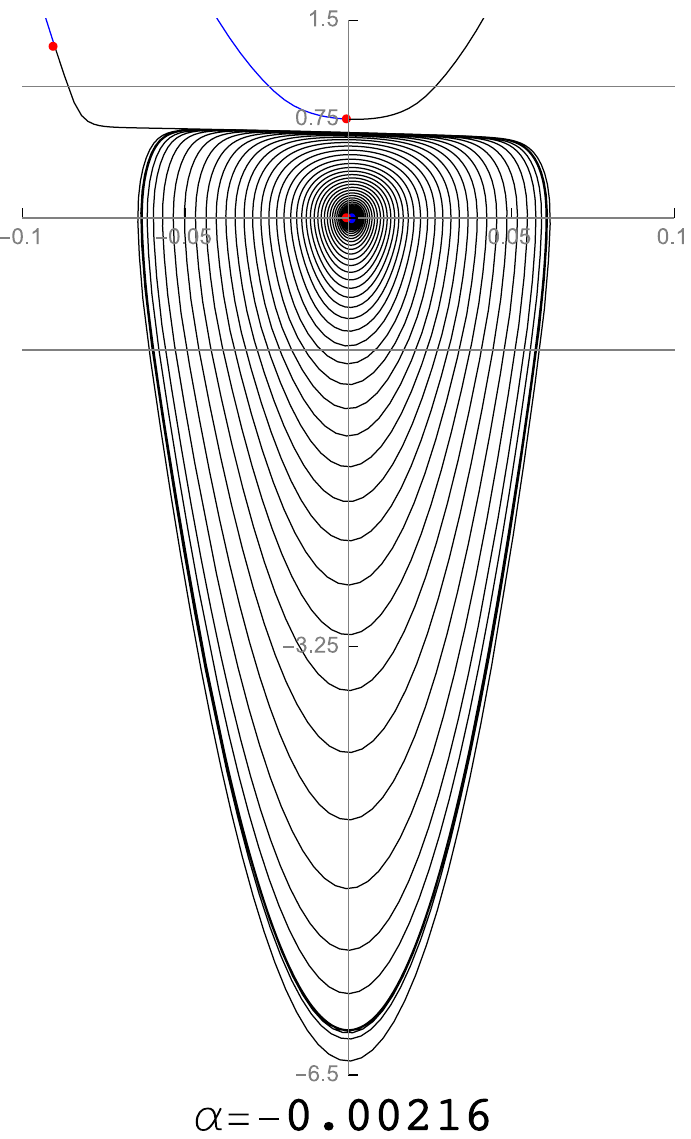}}  \hspace{1cm}
				\subfigure[\label{fig:Bnegativo000217}]{\includegraphics[scale=0.6]{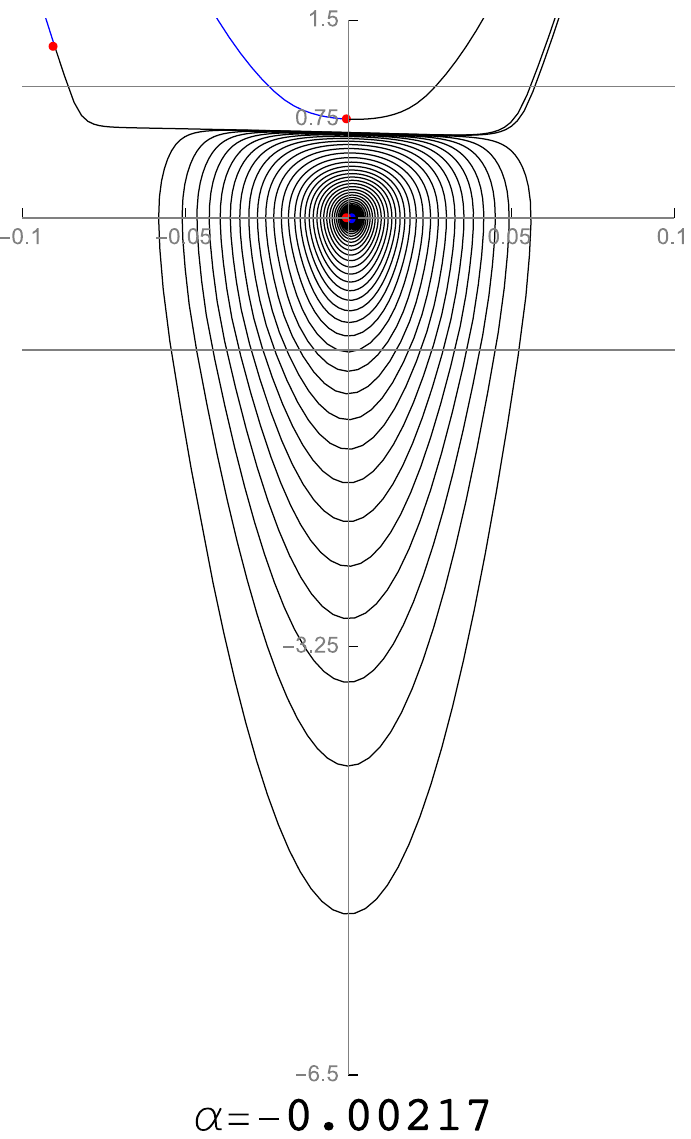}}  \hspace{1cm}
			\end{tiny} \\
			\begin{scriptsize}
				\textcolor{red}{$\blacksquare$}  Initial condition \quad  \textcolor{blue}{$\blacksquare$} Negative time \quad  \textcolor{black}{$\blacksquare$} Positive Time
			\end{scriptsize}
			\caption{Some trajectories of the regularized system $Z^\alpha_\e$ of \autoref{ex:Bnegativo} for different values of the parameter $\alpha$ and $\e=0.001$.}
			\label{fig:Bnegativo}
		\end{figure}
		
		Consider the piecewise vector field \ref{Bfield} and let $$	\varphi(v)= -\frac{5 v^7}{2}+\frac{9 v^5}{2}-2 v^3+v, \, v \in (-1,1).$$
		The critical point is $
		P(\alpha,\e)= \left(-\frac{1}{2}\alpha + \mathcal{O}_2(\alpha,\e), \mathcal{O}_2(\alpha,\e) \right)$	
		and the bifurcation curves $\mathcal{D}$, and $\mathcal{H}$ are given by:
		\begin{eqnarray*}
			\mathcal{D} &=& \left\{ (\alpha,\e): \e = 0.562 \ \alpha^2  +\mathcal{O}(\alpha^3) \right\} \\
			\mathcal{H} &=& \left\{ (\alpha,\e): \alpha =-1.26 \ \e  +\mathcal{O}(\e^2) \right\}.
		\end{eqnarray*}
		The first Lyapunov coefficient is
		$
		\ell_1(\alpha(\e),\e)=\frac{1}{\sqrt{\e}} (-6.3 + \mathcal{O}(\e)),
		$
		therefore  the  Hopf bifurcation is supercritical.
		
		The Canard trajectory occurs over the curve
		$$
		\mathcal{C} = \left\{ (\alpha,\e): \alpha = -2.167 \ \e  +\mathcal{O}(\e^{3/2}) \right\}.
		$$
		The coefficient $B=-2.17 <0$ (see \ref{Rsign}),  therefore, the stated in the first item of \autoref{pocanard} holds.

		Fix $\e_0=0.001$ and consider values of $\alpha^\pm_\mathcal{D}(\e_0) \approx \pm 0.042 $,
		$\alpha_\mathcal{H}(\e_0) \approx -0.0012$ and $\alpha_\mathcal{C}(\e_0) \approx -0.0021$
		obtained by the intersection between the line $\e=\e_0$ and curves $\mathcal{D}$, $\mathcal{H}$ and $\mathcal{C}$, respectively.
		
		We  focus our attention in the region 
		close to the canard curve $\mathcal{C}$, in order to show the behavior explained in the first item of \autoref{pocanard}.
		
		\begin{figure}[htb!]
			\centering
			\includegraphics[width=0.6\textwidth]{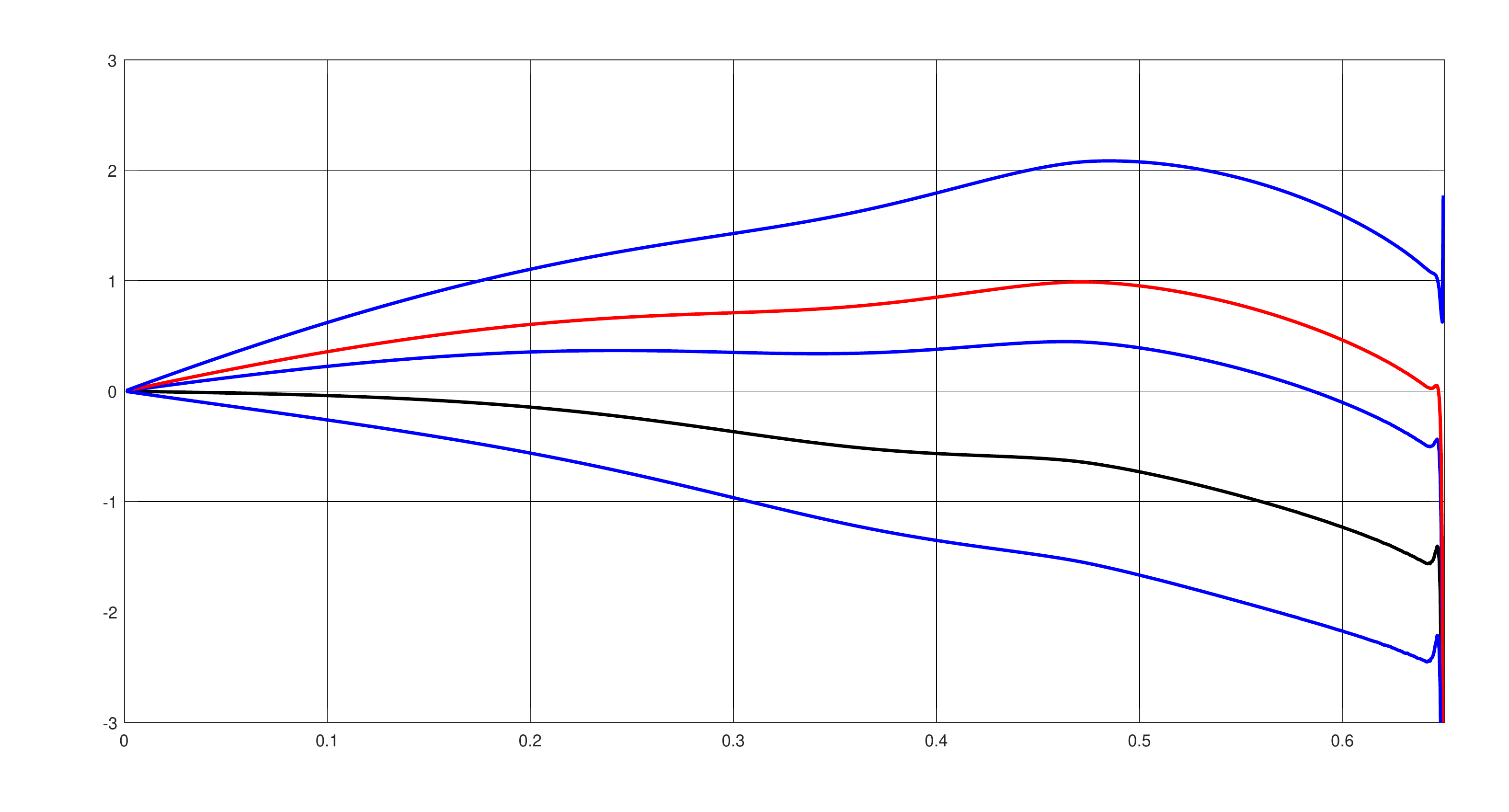}
			\caption{The Melnikov function for example \ref{ex:Bnegativo}:  $\delta = \delta _\mathcal{H}$ in black is supercritical, $\delta = \delta _\mathcal{C}$ in red. $M(v,\delta)$ has only one zero for  $\delta$ near $\delta_\mathcal{C}$ and it has negative slope.
			}
			\label{fig:ex:MelnikovBnegativo}
		\end{figure}
		
		In Figure \ref{fig:Bnegativo} we illustrate some trajectories of the vector field $Z^\alpha_\e$ for $\alpha=\delta \e$ with $\delta$ near the value $\delta_\mathcal{C}$. Observe that, for these values of $\alpha$ the critical point $P(\alpha,\e)$ is always an unstable focus.
		
		In \autoref{fig:Bnegativo000140}, $\alpha=-0.00140$, $\alpha_\mathcal{C} \ll \alpha \leq   \alpha_\mathcal{H}$: a stable periodic orbit exists and the stable Fenichel manifold is above the unstable one.
		In Figure \ref{fig:Bnegativo000216}, $\alpha=-0.00216$, $\alpha \leq \alpha_\mathcal{C}$, one can see a big stable periodic orbit, which corresponds to  the periodic orbit $\Delta^{\alpha,C}_\e$. 
		Finally, in \autoref{fig:Bnegativo000217}, $\alpha=-0.00217$, $\alpha> \alpha_\mathcal{C}$: the Canard already happened and the unstable Fenichel manifold is above the stable one. The periodic orbit $\Delta^{\alpha,C}_\e$ no longer exist. 
		In \autoref{fig:ex:MelnikovBnegativo} one can see the Melnikov function for different values of $\delta$ when $\alpha=\delta \e$.	
	\end{exmp}
	
	\begin{exmp}[A unstable periodic orbit near the Canard: $B>0$] \label{ex:Bpositivo}
		\begin{figure}[!htb]
			\centering
			\begin{tiny}
				\subfigure[\label{fig:Bpositivo0001000}]{\includegraphics[scale=0.6]{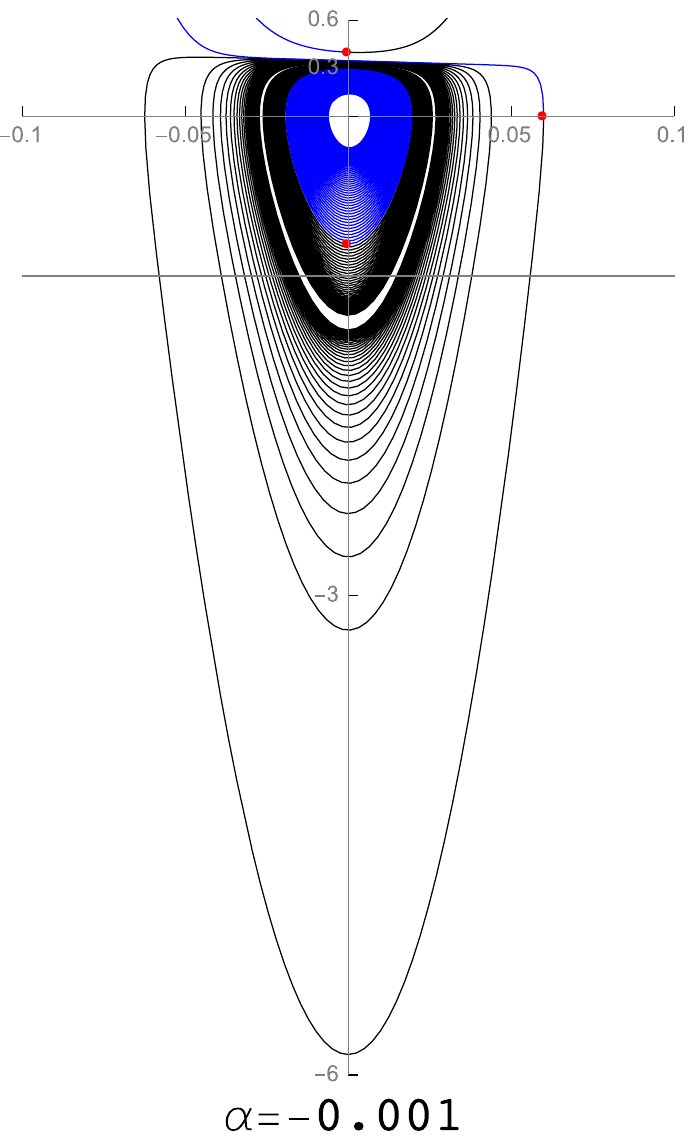}} \hspace{1cm}	
				\subfigure[\label{fig:Bpositivo000101014}]{\includegraphics[scale=0.6]{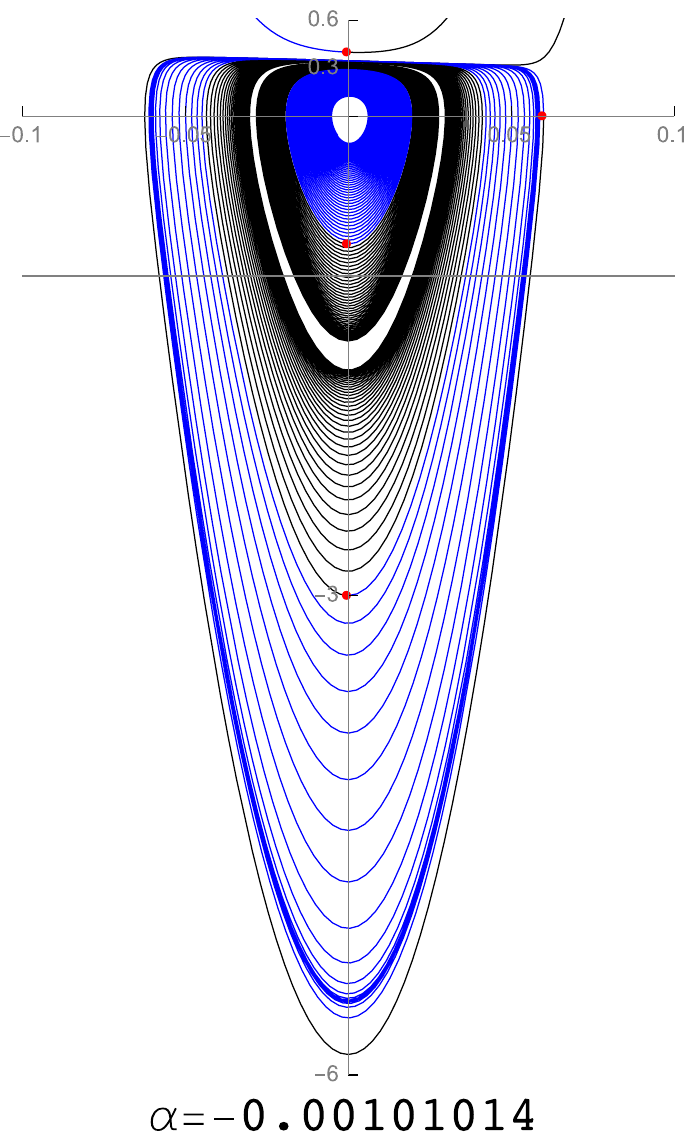}} \hspace{1cm} \linebreak
				\subfigure[\label{fig:Bpositivo000101800}]{\includegraphics[scale=0.6]{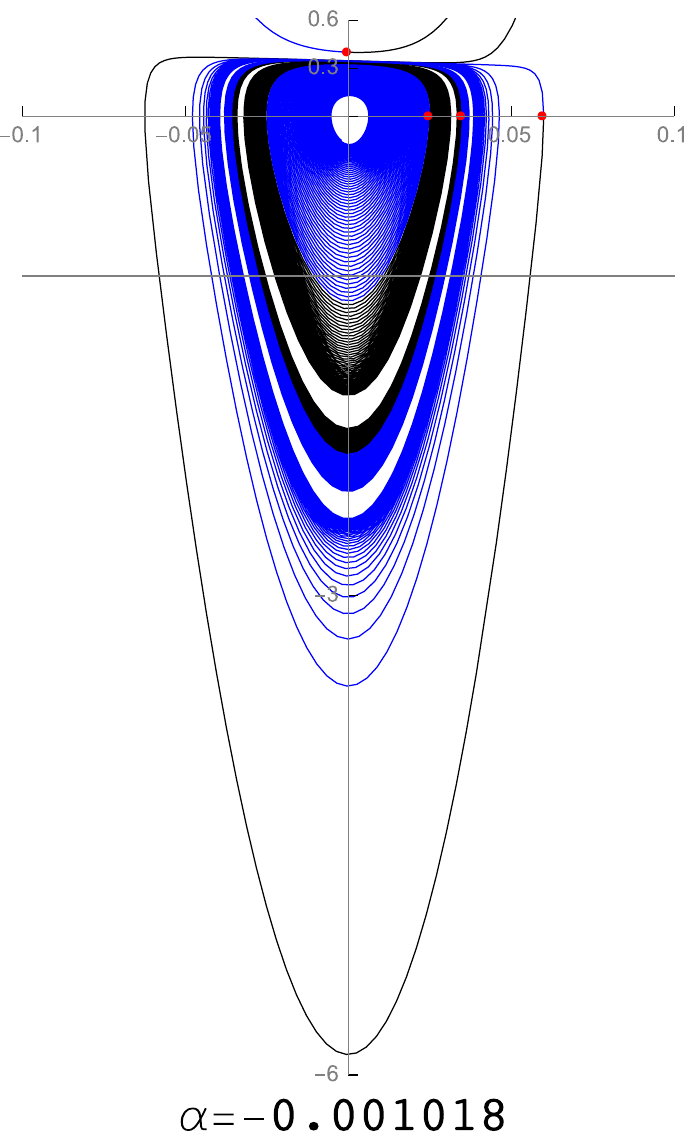}}  \hspace{1cm}
				\subfigure[\label{fig:Bpositivo000102000}]{\includegraphics[scale=0.6]{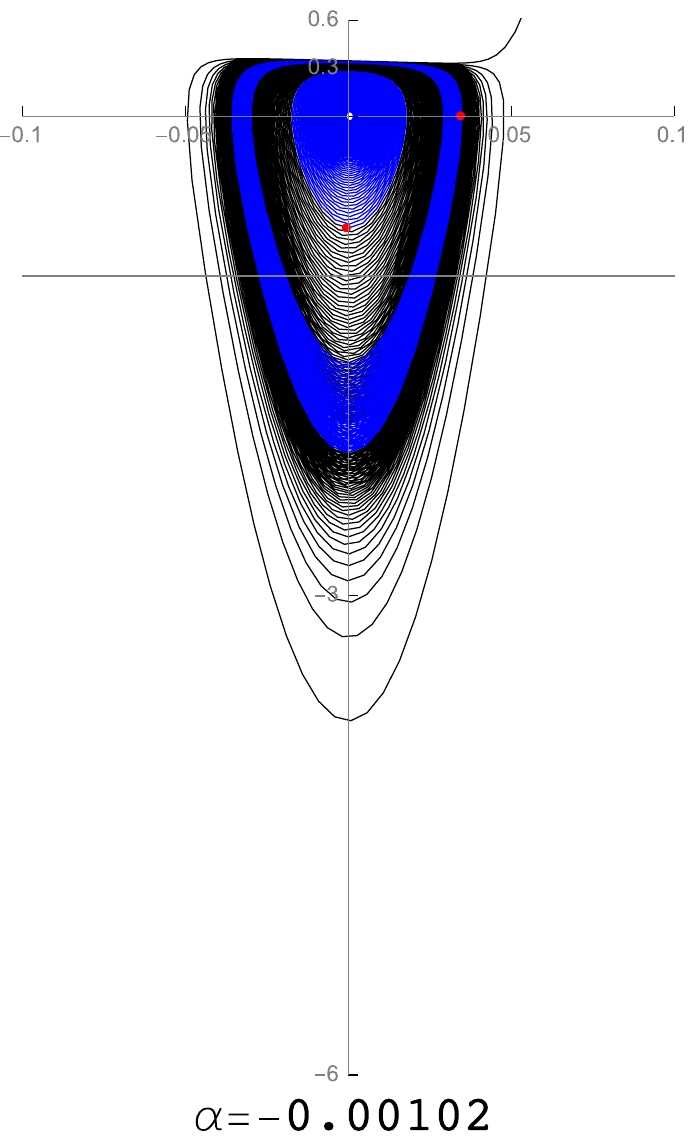}}  \hspace{1cm}
			\end{tiny} \\
			\begin{scriptsize}
				\textcolor{red}{$\blacksquare$}  Initial condition \quad  \textcolor{blue}{$\blacksquare$} Negative time \quad  \textcolor{black}{$\blacksquare$} Positive Time
			\end{scriptsize}
			\caption{Some trajectories of the regularized system $Z^\alpha_\e$ of \autoref{ex:Bpositivo} for different values of the parameter $\alpha$ and $\e=0.001$.}
			\label{fig:Bpositivo}
		\end{figure}
		
		Consider the piecewise vector field \ref{Bfield} and the cubic transition map \eqref{phi3}.
		
		The critical point is $
		P(\alpha,\e)= \left(-\frac{1}{2}\alpha + \mathcal{O}_2(\alpha,\e), \mathcal{O}_2(\alpha,\e) \right)$	
		and the bifurcation curves $\mathcal{D}$, and $\mathcal{H}$ are given by:
		\begin{eqnarray*}
			\mathcal{D} &=& \left\{ (\alpha,\e): \e =0.8437 \ \alpha^2  +\mathcal{O}(\alpha^3) \right\} \\
			\mathcal{H} &=& \left\{ (\alpha,\e): \alpha =-0.8444  \ \e  +\mathcal{O}(\e^2) \right\}.
		\end{eqnarray*}
		The first Lyapunov coefficient is
		$
		\ell_1(\alpha(\e),\e)=\frac{1}{\sqrt{\e}} (-1.09 + \mathcal{O}(\e)),
		$
		therefore  the  Hopf bifurcation is supercritical.
		
		The Canard trajectory occurs over the curve
		$$
		\mathcal{C} = \left\{ (\alpha,\e): \alpha = -1.01013 \ \e  +\mathcal{O}(\e^{3/2}) \right\}.
		$$
		The coefficient $B=5.66>0$ (see \ref{Rsign}),  therefore, the stated in the second item of \autoref{pocanard} holds.
		The system $Z^\alpha_\e$ has a saddle-node bifurcation of periodic orbits in some curve $\tilde{\mathcal{S}}$ located to the left of the canard curve 
		$\mathcal{C}.$
		
		Fix $\e_0=0.001$ and consider values of $\alpha^\pm_\mathcal{D}(\e_0) \approx \pm 0.1 $,
		$\alpha_\mathcal{H}(\e_0) \approx -0.0084$ and $\alpha_\mathcal{C}(\e_0) \approx -0.0101013$
		obtained by the intersection between the line $\e=\e_0$ and curves $\mathcal{D}$, $\mathcal{H}$ and $\mathcal{C}$, respectively.

		Once again, we  focus our attention in the region close to the canard curve $\mathcal{C}$, in order to show the behavior explained in \autoref{pocanard}.
		
		\begin{figure}[htb!]
			\centering
			\includegraphics[width=0.6\textwidth]{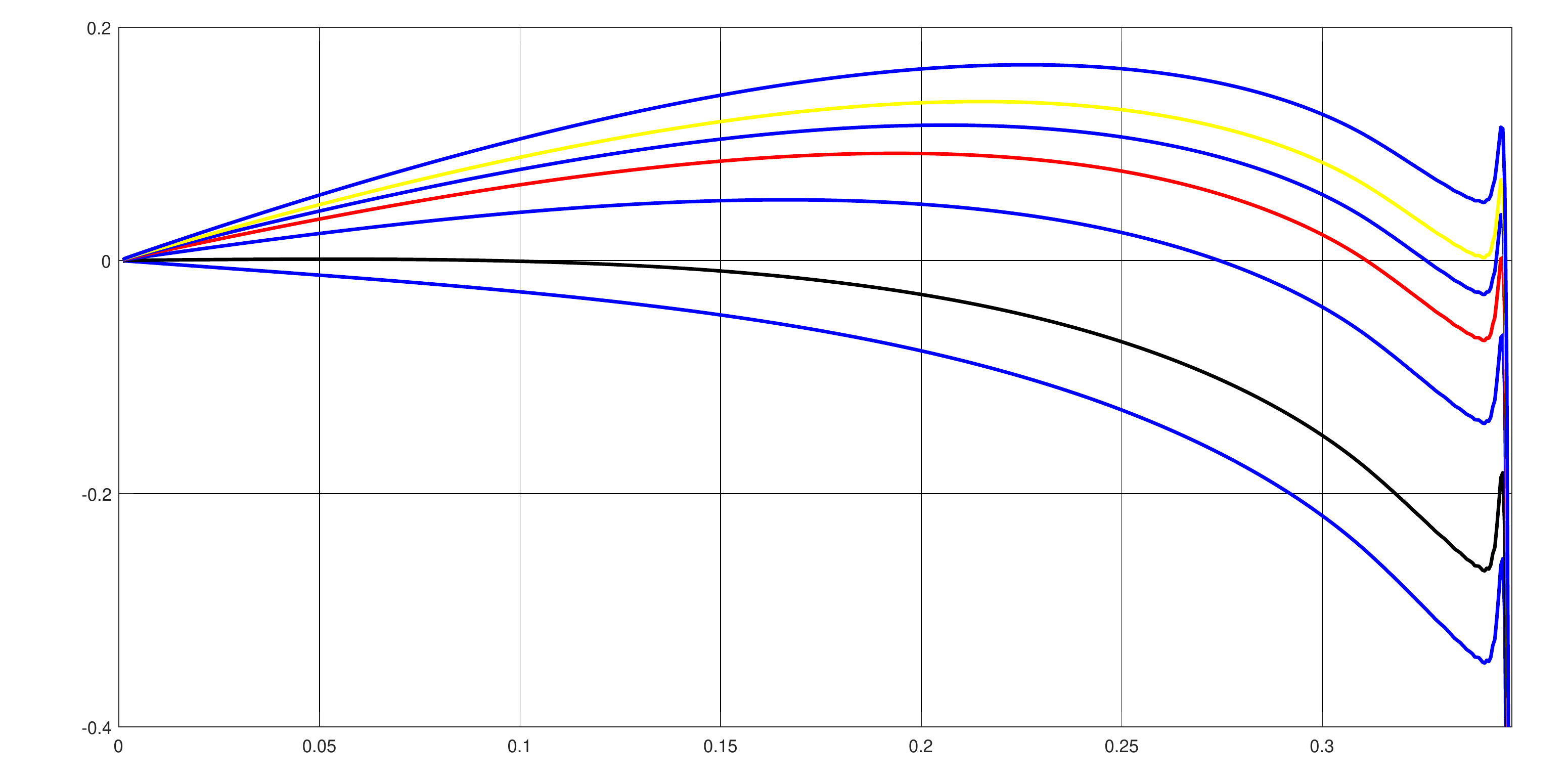}
			\caption{The Melnikov function for example \ref{ex:Bpositivo}:  $\delta = \delta _\mathcal{H}$ in black is subcritical, $\delta = \delta _\mathcal{C}$ in red,
				$\delta = \delta _{\tilde {\mathcal{ S}}}$ in yellow. $M(v,\delta)$ has no zeros for  $\delta < \delta _{\tilde {\mathcal{S}}}  $,
				two zeros for $\delta _{\tilde {\mathcal{ S}}}\le \delta \le  \delta _\mathcal{C}$, one zero with negative slope  for
				$\delta _\mathcal{C}<\delta < \delta _\mathcal{H}$ and no zeros for $\delta > \delta _\mathcal{H}$.
			}
			\label{fig:ex:MelnikovBpositivo}
		\end{figure}
		
		In Figures \ref{fig:Bpositivo0001000} to  \ref{fig:Bpositivo000102000} we illustrate some trajectories of the vector field $Z^\alpha_\e$ for $\alpha=\delta \e$ with $\delta$ near the value $\delta_\mathcal{C}$. Observe that, for these values of $\alpha$ the critical point $P(\alpha,\e)$ is always an unstable focus. We vary the parameter values from right to left on the bifurcation diagram, that is, from the Hopf bifurcation to the Canard direction.

		In \autoref{fig:Bpositivo0001000}, $\alpha=-0.001$, $\alpha_\mathcal{C} < \alpha < \alpha_{\mathcal{H}}$: a small stable periodic orbit exists and the unstable Fenichel manifold is bellow the stable one. 
		In \autoref{fig:Bpositivo000101014}, $\alpha=-0.00101014$, $\alpha_{\tilde{  \mathcal{S}}} < \alpha \lesssim \alpha_\mathcal{C}$: the stable Fenichel manifold is above the unstable one. For this value of $\alpha$, one can see two periodic orbits: a bigger one which is unstable and corresponds to the periodic orbit $\Delta^{\alpha,C}_\e$ and an smaller one, which is stable and correspond to the periodic orbit $\Delta^{\alpha,s}_\e$.
		
		Since for $\alpha=-0.001018$ we have two periodic orbits and for $\alpha=-0.0102$ we have no periodic orbits (see figures \ref{fig:Bpositivo000101800} and \autoref{fig:Bpositivo000102000})   it follows that the value of $\alpha_{\tilde{\mathcal{S}}} (\e_0)$ given by \autoref{pocanard} belongs to $I_{ \tilde{\mathcal{S}}}=(-0.0102,-0.01018)$. Moreover, in these pictures, we see that the size of the smaller stable periodic orbit increases while the size of the big unstable periodic orbit decreases.
		In \autoref{fig:ex:MelnikovBpositivo} one can see the Melnikov function for different values of $\delta$ when $\alpha=\delta \e$.
		
	\end{exmp}

	\subsubsection{The saddle case: \texorpdfstring{$(\det{Z)}_x(\0)>0$}{det>0}} \label{ssec:VIunfregG}
	
	When $(\det{Z)}_x(\0)>0$, for  any $\alpha $ and $\e>0$ sufficiently small, by \autoref{eqlemma} and \autoref{prop:toptype} the regularized vector field 
	$Z^{\alpha}_\e$ has a critical point $P(\alpha,\e)$ which is a saddle.
	
	Analogously to the visible-visible case, although the point $P(\alpha,\e)$ maintains the same character for all $\alpha,\e$,
	a ``bifurcation'' on the Fenichel manifolds when $\alpha$ crosses the $\alpha=0$ value occurs.
	
	Using the results about the critical manifolds  in \autoref{ssec:OVcritical}, for each fixed $\alpha \neq 0$ and any compact set of the critical manifolds $\Lambda^{\alpha,s}_0$ and
	$\Lambda^{\alpha,u}_0$ excluding the tangency points $(T^\alpha_X,1)$ and $(T^\alpha_Y,-1)$, for $0<\e<\e_0(\alpha)$,
	there exist two normally hyperbolic invariant manifold $\Lambda^{\alpha,s}_\e$ and $\Lambda^{\alpha,u}_\e$ which are $\e-$close to
	$\Lambda^{\alpha,s}_0$ and $\Lambda^{\alpha,u}_0$, respectively (see \autoref{fig:OVcritical+}).
	Moreover,
	\begin{itemize}
		\item for $\alpha<0$, $\Lambda^{\alpha,s}_\e$ is the unstable manifold of the saddle point $P(\alpha,\e),$
		\item for $\alpha>0$, $\Lambda^{\alpha,u}_\e$ is the stable manifold of the saddle point $P(\alpha,\e).$
	\end{itemize}
	
	Observe that for $x < T^{\alpha, \e}_X$ the vector $X^\alpha(x,1)$ points inward to the regularization zone and points outwards to the regularization
	zone for $x>T^{\alpha, \e}_X$. Analogously, for $x < T^{\alpha, \e}_Y$ the vector $Y^\alpha(x,-1)$ points inwards to the regularization zone for
	$x<T^{\alpha, \e}_Y$ and outwards to the regularization zone for $x>T^{\alpha,\e}_Y$.
	
	\begin{figure}[!htb]
		\centering
		\begin{tiny}
			\def\svgscale{0.4}
			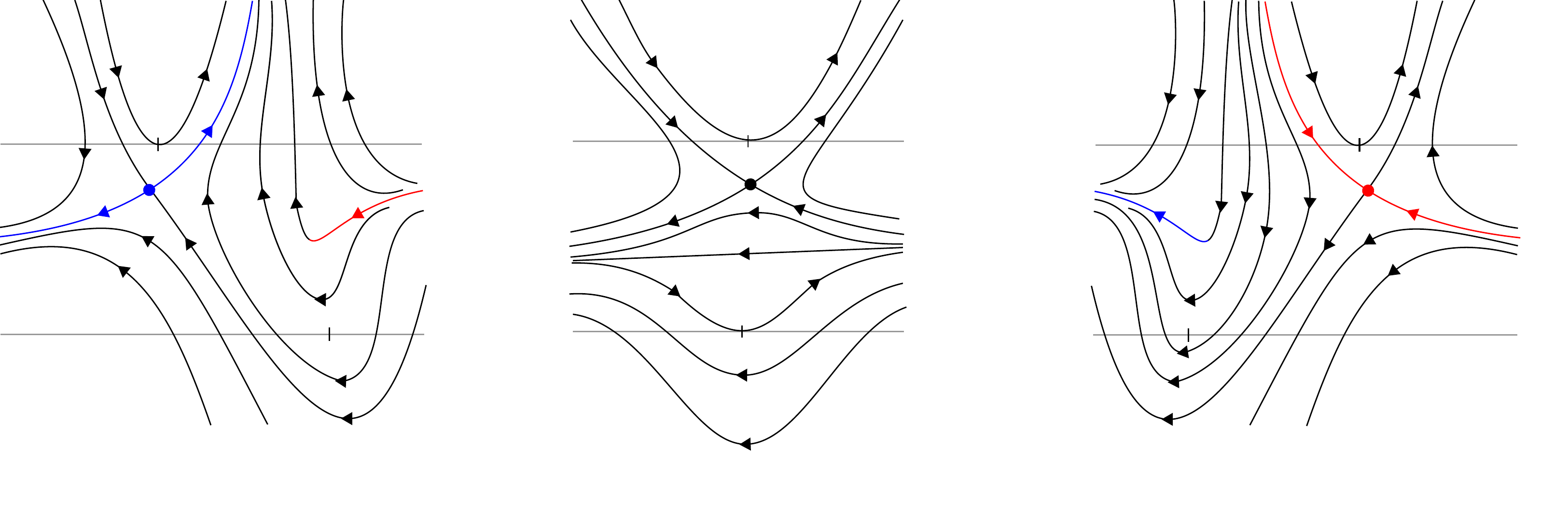
		\end{tiny}
		\caption{The phase portrait of $Z^{\alpha}_\e$ for $\e>0$ and different $\alpha$.}
		\label{fig:VIGUR2}
	\end{figure}
	
	The above information and the fact that the dynamics over the Fenichel Ma\-ni\-fold $\Lambda^\alpha_\e$ is the same of the critical manifold $\Lambda^\alpha_0$,
	it follows that
	\begin{itemize}
		\item for $\alpha<0$, the stable Fenichel manifold $\Lambda_{\e}^{\alpha,s}$ intersects the section $\{ v=1 \}$ on the right of the tangency point $T^{\alpha, \e}_{X}$. The unstable Fenichel manifold $\Lambda_{\e}^{\alpha,u}$ can intersect or not the section $v=-1$, if this intersection occurs it is located to the right of the tangency point $T^{\alpha, \e}_{Y}$,
		\item for $\alpha>0$, the unstable Fenichel manifold $\Lambda_{\e}^{\alpha,u}$ intersects the section $\{ v=1 \}$ on the left of the tangency point $T^{\alpha, \e}_{X}$. The stable Fenichel manifold $\Lambda_{\e}^{\alpha,s}$ can intersect or not the section $v=-1$, if this intersection occurs it is located to the left of the tangency point $T^{\alpha, \e}_{Y}$.
	\end{itemize}
	The phase portrait of $Z^{\alpha}_{\e}$ for $\alpha \neq 0$ is given in \autoref{fig:VIGUR2}.
	
	\begin{rem}
		Observe that even if the  Canard trajectory also exists when $(\det{Z})_x(\0)>0$ it does  not play a special role in this case.
		However, it is illustrated in \autoref{fig:VIGUR2}.
	\end{rem}
	
%%%%%%
% Section 5
%%%%%%

\section{Melnikov Method} \label{sec:melnikov}

This section will be devoted to study the existence and bifurcations of periodic orbits when $(\det{Z)}_x(\0)<0$ and therefore the fixed point 
is a focus, using the  Melnikov method. We will prove \autoref{prop:propiedadesM} and later, in \autoref{ssec:propphiaZproof} 
we will proof \autoref{prop:melnikov}, using the results in this section.
In fact, one could easily recover the results of  Theorems \ref{thm:bdIIsuper},  \ref{thm:bdIIsub},
for the invisible-invisible case, and give the proof of \autoref{prop:melnikov}, and also Theorems \ref{thm:bdVIsuper} and \ref{thm:bdVIsub} for the 
visible-invisible case using classical perturbation theory, also known as Melnikov theory. 
We will see that the Melnikov function contains all the information about the periodic orbits of the system while their size in the $x$-direction 
is ``small'' respect to the parameter $\e$.

The first observation is that   
examples \ref{ex:IIsuper} , \ref{ex:IIsub} ,  \ref{ex:VIsuper}, \ref{ex:ch}, \ref{ex:Bnegativo} and  \ref{ex:Bpositivo}  show that the periodic orbit which arises in 
the Hopf bifurcation and the possible saddle-node bifurcations of the system
occur in  a region of the phase space where $x= O(\sqrt{\varepsilon})$.

To analyze the system in this region  when  $\alpha =\delta \e $ we perform the change $u=\sqrt{\varepsilon}x$ in 
system \eqref{vsystem} and  a suitable scaling of time 
$t=\sqrt{\e}\tau$. 
Calling $\gamma= \sqrt{\e}$ the system reads:
\begin{equation}  \label{gammasystem}
\begin{aligned}
\frac{d u}{d\tau}&=  F^1(0,v;0,0) + \gamma F^1_x(0,v;0,0) u +O(\gamma ^2)\\
\frac{d v}{d\tau}&=  F^2_x(0,v;0,0) u + \gamma \left( \frac{1}{2} F^2_{xx}(0,v;0,0) u^2 + F^2_\e (0,v;0,0) + \delta \tilde{F}^2(0,v;0,0)\right)  +\mathcal{O}(\gamma ^2)
\end{aligned}
\end{equation}
Where, fixing $M>0$,  the $\mathcal{O}(\gamma ^2)$ terms are bounded by  $K \gamma ^2$ for $|u|\le M$, and $|\gamma v|\le M$, and $K=K(M)$.

The main observation is that for $\gamma=0$ the system is integrable and the function
$$
H(u,v)= \frac{u^2}{2}+ V(v), \quad V(v)= -\int_{v^*}^{v}\frac{F^1(0,r;0,0)}{F^2_x(0,r;0,0)} dr,
$$
where  $(0,v^*)$  is a  critical point and  $v^*$ is given in \ref{eq:criticalpoint},
is a first integral of the system.

As $(\det{Z})_x(\0)<0$, one has that 
$V(v^*)=0$, $V'(v^*) = \frac{F^1(0,v^*;0,0)}{F^2_x(0,v^*;0,0)} =0$ and 
$$
V''(v^*)=-\frac{\varphi'(v^*) ((X^1-Y^1)(\0))^2}{2(\det{Z})_x(\0)}>0.
$$

Therefore the point $(0,v^*)$ is a non-linear center  surrounded by periodic solutions.
We want to check which of these solutions persists when $\gamma \ne 0$ small enough, depending of the value of $\delta$. 
For Hamiltonian systems, one can apply the classical 
Melnikov method to show that, generically, fixing $\delta$, one periodic orbit can persist after the perturbation if some open condition is satisfied. 
For the system at hand, even if it is not Hamiltonian, it is autonomous and therefore one can change time to get a Hamiltonian 
system and then  apply the classical theory for the Hamiltonian case (see, for instance \cite{holmes, wiggins}, \cite{tes}).

Nevertheless, to make the paper more self contained, we will do this simple computation here without changes of time.
Consider the section
$$
\Theta= \{ (0,v), \  v\ge v^*\}
$$
and the Poincar\'{e} return map $\pi:\Theta\to \Theta$, given by
$
\pi(0,v_0)=(0,v_1)
$
where $v_1= v(T)$ and $T=T(v_0)$ is the time such that the orbit $(u(t),v(t))$ with initial condition $(u(0), v(0))=(0,v_0)$, satisfies $u(T)=0$ and $v(T)\ge v^*$. 
Our goal is to give an asymptotic formula for $v_1-v_0$.
The main observation is that we know a priori that $v_1-v_0= O(\gamma)$, because for $\gamma =0$ all the orbits are periodic.
Therefore, we observe that:
$$
H(0,v_1)-H(0,v_0)= V(v_1)-V(v_0) = V'(v_0) (v_1-v_0) + O(\gamma ^2)
$$
and consequently:
\begin{equation}\label{eq:distance}
v_1-v_0= \frac{H(0,v_1)-H(0,v_0)}{ V'(v_0)} + O(\gamma ^2).
\end{equation}

Using that $H(u,v)$ is a first integral of system \ref{gammasystem} for $\gamma=0$ and the Fundamental Theorem of Calculus applied to the function 
$f(t)=H(u(t), v(t))$, we have that
\begin{eqnarray*}
	H(0,v_1)-H(0,v_0)&=& 
	\int _{0}^{T}\frac{\partial H}{\partial u} \frac{d u}{dt}+\frac{\partial H}{\partial v} \frac{d v}{dt}(u(t), v(t)) dt\\
	&=&
	\gamma   \int _{0}^{T} \left( F^1_x(0,v(t);0,0) u(t)^2 + V'(v(t)) \left[ \frac{1}{2} F^2_{xx} (0,v(t);0,0) u(t)^2  \right. \right. \\
	&+&  
	\left. \left . F^2_\e(0,v(t);0,0)+ \delta \tilde{F}^2(0,v(t);0,0) \right] \right) dt +O(\gamma ^2).
\end{eqnarray*}

If we take the initial value $v_0 \in [v^*, V]$ for any fixed $V>v^*$ in invisible-invisible case or 
%$v^*<v<\bar v$ 
$V<\bar v$ in the visible-invisible case, 
one can ensure that, if $\gamma$ is small enough, the solution 
$u(t)=u_p(t)+\mathcal{O}(\gamma)$, $v(t)=v_p(t)+\mathcal{O}(\gamma)$, $T=T_0+\mathcal{O}(\gamma)$, where $(u_p(t),v_p(t))$ 
is the periodic solution of the unperturbed system with initial condition $(0,v_0)$ and $T_0=T_0(v_0)$ its period.
Then, using \eqref{eq:distance} we obtain:
$$
v_1-v_0= \gamma M(v_0, \delta)+ \mathcal{O}(\gamma ^2)
$$
where
\begin{eqnarray*}
	M(v_0,\delta) & =&\frac{1}{V'(v_0)} \int _{0}^{T_0} \left( F^1_x(0,v_p(t);0,0) u_p(t)^2 + V'(v_p(t)) \left[ \frac{1}{2} F^2_{xx} (0,v_p(t);0,0) u_{p}(t)^2  \right. \right. \\
	&+& \left.  \left. F^2_\e(0,v_p(t);0,0)+ \delta \tilde{F}^2(0,v_p(t);0,0) \right] \right) dt .
\end{eqnarray*}

The function $M(v_0, \delta)$ is  known as the Melnikov function. 
To simplify its expression let's $0\le \bar T_0\le T_0$ be the time where the orbit $(u_p(t), v_p(t))$ intersects $u=0$ 
for the first time at a point $(0,\bar v_0)$, with $\bar v_0 <v^*$. 
Splitting the integral from $0$ to $T_0$ into two integrals from $0$ to $\bar T_0$ and from $\bar T_0$ to $T_0$, 
and  changing  variables in both integrals
to $v=v_p(t)$, and using that $H(u_p(t), v_p(t))= H(0,v_0)=V(v_0)$, one obtains:
\begin{eqnarray*}
	M(v_0,\delta) & =& \frac{-2}{V'(v_0)} \int _{\bar v_0}^{v_0}  
	\left\{ 2F^1_x(0,v;0,0) (V(v_0)-V(v)) + V'(v) \left( F^2_{xx}(0.v;0,0) (V(v_0)-V(v)) \right. \right. \\
	&+& \left. \left. F^2_\e(0,v;0,0) + \delta \tilde{F}^2(0,v;0,0) \right) \right\}
	\frac{dv}{F^2_x(0,v;0,0) \sqrt{2(V(v_0)-V(v))}} 
\end{eqnarray*}
where $\bar v_0$ satisfies $V(v_0)= V(\bar v_0)$. Now, integrating by parts the second term in the integral:

\begin{eqnarray}
M(v_0,\delta) &=& 
\frac{-2}{V'(v_0)} \int_{\bar v_0}^{v_0} \left[ \frac{F^1_x(0,v;0,0)}{F^2_x(0,v;0,0)} +
\frac{\partial}{\partial v} \left\{\frac{1}{F^2_x(0,v;0,0)} \left( F^2_{xx}(0,v;0,0) (V(v_0)-V(v))   \right. \right. \right. \nonumber \\
&+& \left. \left. \left. F^2_\e(0,v;0,0)+   \delta \tilde F^2(0,v;0,0) \right)  \right\} \right] \sqrt{2(V(v_0)-V(v))}  dv \label{eq:melnikov}\\
&=& \frac{-2}{V'(v_0)} \int_{\bar v_0}^{v_0} f(v,v_0,\delta)  \sqrt{2(V(v_0)-V(v))} dv \nonumber
\end{eqnarray}
where:
\begin{eqnarray*} \label{fexpression}
	f(v,v_0,\delta)= &=& 
	\frac{F^1_x(0,v;0,0)}{F^2_x(0,v;0,0)} +
	\frac{\partial}{\partial v} \left\{\frac{1}{F^2_x(0,v;0,0)} \left( F^2_{xx}(0,v;0,0) (V(v_0)-V(v))   \right. \right.  \\
	&+& \left. \left.  F^2_\e(0,v;0,0)+   \delta \tilde F^2(0,v;0,0) \right)  \right\} .
\end{eqnarray*}

\begin{prop}\label{prop:propiedadesM}
	The Melnikov function $M(v,\delta)$ satisfies:
	\begin{itemize}
		\item 
		$ M(v^*,\delta)=0, \ \forall \delta$;
		\item
		$\frac{\partial M}{\partial v}(v^*, \delta_\mathcal{H})=0$;
		\item
		In the invisible-invisible case 
		$\frac{\partial^2}{\partial v \partial \delta} M(v^*,\delta_\mathcal{H})>0$;
		\item
		In the visible-invisible case 
		$\frac{\partial^2}{\partial v \partial \delta} M(v^*,\delta_\mathcal{H})<0$;
		\item
		If $\frac{\partial^2}{\partial v ^2} M(v^*,\delta_\mathcal{H})>0$ the Hopf  bifurcation is subcritical (the Lyapunov coefficient $\ell_1 >0$);
		\item
		If $\frac{\partial^2}{\partial v ^2} M(v^*,\delta_\mathcal{H})<0$ the Hopf bifurcation is supercritical (the Lyapunov coefficient $\ell_1 <0$).
	\end{itemize}
\end{prop}

\begin{proof}
	It is clear that the integral vanishes at $v^*$ for any $\delta$ because in this case $\bar v_0=v_0=v^*$, 
	but as $V'(v^*)=0$ we need more information about the integral at $v^*$.
	For this reason we split the integral between $\bar v_0$ and $v^*$ and between $v^*$ and $v_0$. In the first integral we perform the change
	$v=v ^*+ t(\bar v_0-v^*)$ and in the second one $v=v ^*+ t( v_0-v^*)$ obtaining:
	
	\begin{eqnarray*}
		&&\int _{\bar v_0}^{v_0} f(v,v_0, \delta)\sqrt{2(V(v_0)-V(v))} dv \\
		&=& \int _{\bar v_0}^{v^*} f(v,v_0, \delta)\sqrt{2(V(v_0)-V(v))} dv+
		\int _{\bar v^*}^{v_0} f(v,v_0, \delta)\sqrt{2(V(v_0)-V(v))} dv \\
		&=&-(\bar v_0-v^*)\int _{0}^{1} f(v ^*+ t(\bar v_0-v^*),v_0, \delta)\sqrt{2(V(v_0)-V(v ^*+ t(\bar v_0-v^*)))} dt\\
		&+&(v_0-v^*)\int _{0}^{1} f(v ^*+ t(v_0-v^*),v_0, \delta)\sqrt{2(V(v_0)-V(v ^*+ t(v_0-v^*)))} dt =
	\end{eqnarray*}
	Next step is to use
	$$
	V(v)=  (v-v^*)^2 \tilde V(v), \ \tilde V(v^*)=\frac{V''(v^*)}{2} 
	$$
	obtaining:
	\begin{eqnarray*}
		&=&(\bar v_0-v^*)^2\int _{0}^{1} f(v ^*+ t(\bar v_0-v^*),v_0, \delta)\sqrt{2(\tilde V(\bar v_0)-t^2\tilde V(v ^*+ t(\bar v_0-v^*)))} dt\\
		&+& 
		(v_0-v^*)^2\int _{0}^{1} f(v ^*+ t(v_0-v^*),v_0, \delta)\sqrt{2(\tilde V(v_0)-t^2\tilde V(v ^*+ t(v_0-v^*)))} dt
	\end{eqnarray*}
	using these computations and that $\bar v_0-v^*=-(v_0-v^*)+ O(v_0-v^*)^2$,  one obtains:
	\begin{eqnarray*}
		M(v^*, \delta) & =& 0, 
		\\
		\frac{\partial M}{\partial v}(v^*, \delta)
		&=& -\frac{\pi}{2 \sqrt{ V''(v^*)}} f(v^*,v^*,\delta).
	\end{eqnarray*}
	Now, using the expression of $f$ one can see that this derivative vanishes if $\delta = \delta _H$ given in \eqref{hcurve}.
	Observe that, since $v^*$ is a zero of $M(\cdot,\delta)$ for any $\delta$ and a critical point of $M(\cdot,\delta_\mathcal{H})$, 
	the stability of the critical point $(0,v^*)$ is given by the second derivative 
	$\frac{\partial^2}{\partial v^2} M(v^*,\delta_\mathcal{H})$. 
	More precisely, this value corresponds to the Lyapunov coefficient of the Hopf bifurcation, which is subcritical when $v^*$ is a minimum and it is 
	supercritical when it is a maximum of  of $M(\cdot,\delta_\mathcal{H})$.

	Moreover, using that the function $f(v,v_0,\delta)$, and therefore the function $M(v_0,\delta)$ are lineal with respect to $\delta$ one easily obtains:
	\begin{equation} \label{deltavderivative}
	\begin{aligned}
	\frac{\partial^2}{\partial v \partial \delta} M(v^*,\delta_\mathcal{H}) = \frac{\pi \varphi'(v^*)}{4((\det{Z})_x(\0))^2\sqrt{ V''(v^*)} } \left( X^2_x \tilde Y^2-Y^2_x \tilde X^2 \right)(\0) 
	\end{aligned}
	\end{equation}
	which is, by  \autoref{visibility} and \eqref{foldsbe}, positive if both folds are invisible and negative if the folds have opposite visibility.
\end{proof}

Once we know the basic properties of the function $M$ we can  prove of \autoref{prop:melnikov}:

\begin{itemize}
	\item 
	First, as $\G_Z>0$, for any $\delta$, there exists $V_0>v^*$ independent of $\delta$, and $M(v,\delta)<0$ for $v>V_0$.
	\item
	As the bifurcation is subcritical, $\frac{\partial^2}{\partial v^2} M(v^*,\delta_\mathcal{H})>0$ and therefore the point $v=v^*$ is a minimum of 
	$ M(v,\delta_\mathcal{H})$ and $ M(v,\delta_\mathcal{H})>0$ for all $v^*<v<V_1$, and $(0,v^*)$ is unstable at the Hopf bifurcation.
	\item
	As in the invisible-invisible case $\frac{\partial^2}{\partial v \partial \delta} M(v^*,\delta_\mathcal{H})>0$ one has that 
	\begin{itemize}
		\item 
		If $\delta <\delta_\mathcal{H}$, $\frac{\partial M}{\partial v}(v^*, \delta)<0$, and consequently $ M(v^*, \delta)<0$ for $v>v^*$ close enough to $v^*$, 
		which implies that the critical point of the system is stable.
		\item
		If $\delta >\delta_\mathcal{H}$, $\frac{\partial M}{\partial v}(v^*, \delta)>0$, and consequently $ M(v^*, \delta)>0$ for $v>v^*$ close enough to $v^*$, 
		which implies that the critical point of the system is unstable.
	\end{itemize}
\end{itemize}
All this information together ensures that the function $M(v,\delta)$ satisfies:
\begin{itemize}
	\item 
	For $\delta >\delta_\mathcal{H}$, $M(v, \delta)>0$, for any $v^*<v<V_1$, therefore $Z^\alpha_\e$ has no periodic orbits near 
	the critical point $P(\alpha, \e)$. 
	Nevertheless, as $M(v,\delta)<0$ for  $v>V_0$, the function  $M(\cdot,\delta)$ 
	has a zero $v^s$, corresponding to an attracting periodic orbit $\Gamma^{\alpha,s}_\e$.  
	\item
	For $\delta<\delta_\mathcal{H}$, $M(v, \delta)<0$, for $v>v^*$   near $v^*$, but is positive near $V_1$, therefore it has a unique zero 
	$v^{u}=v^{u}(\delta)$ near $v^*$ satisfying 
	$\frac{\partial}{\partial v} M(v^{u}(\delta),\delta)>0$, therefore, by the implicit function theorem,  
	$Z^\alpha_\e$ has an repelling periodic orbit $\Delta^{\alpha,u}_\e$.
	In addition, as $M(v,\delta)$ is negative near $V_0$, it has a zero $v^s$, corresponding to an attracting periodic orbit $\Gamma^{\alpha,s}_\e$. 
	Therefore, 
	$M(v,\delta)$  has a maximum $v_M(\delta)\in (v^{u}, v^s)$.  
	\item
	In addition, as $M$ has also a minimum between $v^*$ and $v^u$, we can assure that its second derivative vanishes, at least, in one point. 
	If we assume that $\frac{\partial^3}{\partial v^3} M(v_\mathcal{S},\delta_\mathcal{S})\ne 0$, 
	we can ensure that there are no more zeros of $M$ besides $v^ *$, $v^u$, $v^s$.
	
	This guarantees the existence of a pair $( v_\mathcal{S},\delta_\mathcal{S})$ with 
	$v^* <  v_\mathcal{S}= v_M(\delta_\mathcal{S})$ and $\delta_\mathcal{S} <\delta_\mathcal{H}$ 
	such that  $M( v_\mathcal{S},\delta_\mathcal{S})=0$, $\frac{\partial}{\partial v} M(v_\mathcal{S},\delta_\mathcal{S})=0$,
	giving raise to a saddle-node bifurcation of periodic orbits.
\end{itemize}

The reasoning for the visible-invisible case is analogous, using that, for $\delta >\delta_\mathcal{C}$ the  Melnikov function is also negative for $v$ near $\bar v$. 

Even if the saddle-node bifurcation can not be computed analytically, 
it is worth to mention that we can use the Melnikov function to find it numerically solving the system of equations:
$$
M(v,\delta)=0, \quad \frac{\partial M}{\partial v}(v,\delta)=0.
$$

Observe that $M$ is linear in $\delta$ and therefore one can easily reduce this system to the problem of looking for zeros of a function of one variable.
This can be a useful computational tool to find the saddle-node bifurcations in the regularization for a concrete system.

%%%%%%
% acknowledgments
%%%%%%	
	
	\section{Acknowledgments}
	
	C. Bonet-Reves and  Tere M. Seara have been partially supported by the Spanish MINECO-FEDER Grants MTM2015-65715-P and MDM-2014-0445, 
	and the Catalan Grant 2014SGR504. Tere M-Seara  is also supported by the Russian Scientific Foundation grant 14-41-00044 and the European Marie Curie Action
	FP7-PEOPLE-2012-IRSES: BREUDS. 
	
	J. Larrosa has been supported by FAPESP grants 2011/22529-8 and 2014/13970-0 and the European Marie Curie Action FP7-PEOPLE-2012-IRSES: BREUDS. 

%%%%%
% Appendix
%%%%%	
\section{Appendix}

%%%%%
% Appendix 1 
%%%%%	

\subsection{Proof of \autoref{prop:phiaeZ}} \label{ssec:propphiaZproof}

To prove \autoref{prop:phiaeZ} one need the following lemmas.

\begin{lemma} \label{phi+expression}
	Let $\phi^{\alpha,\e}_+$ the map which goes to the section $v=1$ to the section $v=-1$ and $\phi^{\alpha,\e}_-$ 
	the map which goes to the section $v=-1$ to the section $v=+1$, then
	\begin{enumerate}[(a)]
		\item 
		There exists $x^+(\alpha,\e)$ such that for $x>x^+(\alpha,\e)$ the map $\phi^{{\alpha , \e}}_+$ is well defined and
		$\phi^{{\alpha, \e}}_+ (x) = x + g^+(x;\alpha) \e  + \mathcal{O}(\e^2)$ where $g^+(x;\alpha)$ is given in \ref{g+}.
		\item 
		There exists $x^-(\alpha,\e)$ such that for $x<x^-(\alpha,\e)$ the map $\phi^{{\alpha ,\e}}_-$ is well defined and
		$\phi^{{\alpha, \e}}_- (x) = x + g^-(x;\alpha) \e  + \mathcal{O}(\e^2)$ where $g^-(x;\alpha)=-g^+(x;\alpha)$.
	\end{enumerate}
\end{lemma}

\begin{proof} We prove item $(a)$, the reasoning to prove item $(b)$ is analogous.
	
	Since the origin is an invisible fold point of $X$, fixing $v \in (-1,1)$, it follows that there exist $\alpha_0$ and $\e_0$ and a map 
	$x(\alpha,\e v)$ such that $X^{\alpha 2}(x(\alpha,\e v),\e v)=0$ and $X^{\alpha 1} \cdot X_x^{\alpha 2}(x(\alpha,\e v),\e v)<0$ for every 
	$|\alpha|<\alpha_0$ and $\e<\e_0$. For each $\alpha$ and $\e$ fixed, set 
	\begin{equation} \label{xXmax}
	x^{+}_{X} (\alpha,\e) = \max_{v \in [-1,1]}  x(\alpha,\e v)  .
	\end{equation}
	
	Observe that for $x>x^{+}_{X} (\alpha,\e)$ we have $X^{\alpha 2}(x,\e v)<0$ for all $v \in [-1,1]$. 
	Using the same arguments, one can define $x^{+}_{Y} (\alpha,\e)$ in an analogous way.
	
	Define 
	\begin{equation} \label{x+}
	x^{+}(\alpha,\e) = \max \{ x^{+}_{X} (\alpha,\e), x^{+}_{Y} (\alpha,\e) \}.
	\end{equation}
	
	Therefore, for $x>x^{+}(\alpha,\e)$ we have that $X^{\alpha 2}(x,\e v)$ and $Y^{\alpha 2}(x,\e v)$ are smaller than zero, simultaneously. 
	This implies that the $Z^{\alpha}_{\e}$ trajectory of any initial condition $x>x^{+}(\alpha,\e)$ crosses the regularization zone, 
	since in this region 
	$$
	\dot{v}= \frac{1}{\e} F^2 (x,v;\alpha,\e) = (1+\varphi(v)) X^{\alpha 2}(x,\e v) + (1-\varphi(v)) Y^{\alpha 2}(x,\e v) <0,\quad \forall \, v \in [-1,1].
	$$
	
	Now we are able to compute the map $\phi^{\alpha \e}_{+}$. 
	Consider the equation of the orbits of system \ref{fsystem}: 
	\begin{eqnarray}
	\frac{dx}{dv} 	&=& \e \frac{F^1(x,v;\alpha,\e)}{F^2(x,v;\alpha,\e)} \label{orbitseq}
	\end{eqnarray}
	Let $x(v;\e)$ be the solution of \ref{orbitseq} satisfying $x(1,\e)=x_0 > x^{+}(\alpha,\e)$. 
	Taylor expanding this solution we obtain:
	\begin{equation}
	x(v;\e) = x_0 + \e \int_{1}^{v}{\frac{F^1(x_0,s;\alpha,0)}{F^2(x_0,s;\alpha,0)}}ds + \mathcal{O}(\e^2).
	\end{equation}
	Then the intersection between the $Z^{\alpha}_{\e}$ trajectory departing from $(x,1)$ with $x>x^{+}(\alpha,\e)$ with the section $v=-1$ is given by
	\begin{equation} \label{phiae+}
	\phi^{\alpha \e}_{+}(x) = x + g^+(x;\alpha) \e + \mathcal{O}(\e^2),
	\end{equation}
	where \begin{equation}\label{g+}
	g^{+}(x;\alpha) = \int_{1}^{-1}{\frac{F^1(x,s;\alpha,0)}{F^2(x,s;\alpha,0)}}ds.
	\end{equation}
	Observe that the function $g^+$ is regular respect to $\alpha$: 
	$g^{+}(x;\alpha)=g^{+}(x;0)+\mathcal{O}(\alpha)$.
\end{proof}

%%%%%%%%%%%%%%%%%%%%%%%%%%%%%%%

\begin{lemma} \label{lem:phiZae}
	For $\alpha,\e>0$ sufficiently small there exists a domain $\Theta^\alpha_\e$, given in \ref{theta}, where the generalized first return map 
	$\phi^{\alpha}_\e$ is well defined and satisfies
	\begin{equation} \label{phiex}
	\phi^{\alpha}_\e = \phi^{\alpha} + \mathcal{O}(\e) = \phi_Z + \mathcal{O}(\alpha,\e)
	\end{equation}
\end{lemma}

\begin{proof}
	\begin{figure}[!htb]
		\centering
		\begin{scriptsize}
			\def\svgscale{0.5}
			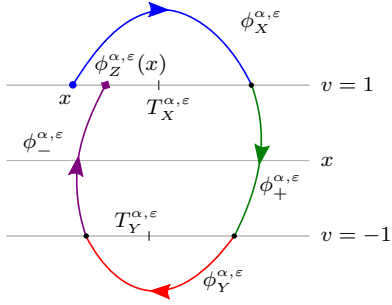
		\end{scriptsize}
		\caption{The generalized first return map $\phi^{\alpha \e}_{X}$}
		\label{fig:IIUR5}
	\end{figure}
	
	To fix ideas, lets assume that $\G_Z>0$. 
	The first step is to construct a first return map $\phi^{\alpha}_\e$ (see \autoref{fig:IIUR5}) defined in a section 
	$\Theta^\alpha_\e \subset \{ (x,1) : \ x<T^{\alpha,\e}_X \}$ as follows.
	
	Using \autoref{prop:frX} applied to the vector field $X^\alpha(x,\e v)$, we define the Poincar\'{e} map $\phi^{\alpha,\e}_X$ in the section $v=1$ for 
	$x < T^{\alpha,\e}_X$ by
	\begin{equation} \label{phiXae}
	\phi^{{\alpha, \e}}_{X}(x)=2T^{{\alpha, \e}}_{X} - x + \beta^{{\alpha, \e}}_{X} (x-T^{{\alpha, \e}}_{X})^2 + \mathcal{O}(x-T^{{\alpha, \e}}_{X})^3.
	\end{equation}
	Analogously, we have the map $\phi^{{\alpha, \e}}_{Y}$ defined at the section $v=-1$ for $x>T^{\alpha,\e}_Y$, given by
	\begin{equation} \label{phiYae}
	\phi^{{\alpha, \e}}_{Y}(x)=2T^{{\alpha, \e}}_{Y} - x + \beta^{{\alpha, \e}}_{Y} (x-T^{{\alpha, \e}}_{Y})^2 + \mathcal{O}(x-T^{{\alpha, \e}}_{Y})^3.
	\end{equation}
	Where $\beta^{{\alpha, \e}}_{X,Y}$ and $T^{{\alpha, \e}}_{X,Y}$ are given by formulas \ref{betap0} and \ref{epsXtangency}, respectively. 
	Moreover,  
	\begin{align*}
	T^{{\alpha, \e}}_{X,Y} = T^{{\alpha}}_{X,Y} + \mathcal{O}(\e)  \ \textrm{and} \  \beta^{{\alpha, \e}}_{X,Y} = \beta^{{\alpha}}_{X,Y} + \mathcal{O}(\e)
	\end{align*}
	Therefore, we conclude that $ \phi^{\alpha,\e}_{X,Y} = \phi^\alpha_{X,Y} + \mathcal{O}(\e),$ where $\phi^\alpha_{X,Y}$ 
	are the Poincar\'{e} maps defined for the vector fields $X^\alpha,Y^\alpha$.
	
	In this way, we obtain a generalized first return map $\phi^{\alpha}_\e$ defined in $\Theta_{\alpha,\e}$
	\begin{equation} \label{phiZae}
	\phi^{\alpha}_\e =  \phi^{\alpha, \e}_{-} \circ \phi^{\alpha, \e}_{Y} \circ \phi^{\alpha, \e}_{+} \circ \phi^{\alpha, \e}_{X}=\phi^{\alpha} + \mathcal{O}(\e) 
	= \phi_Z + \mathcal{O}(\alpha,\e)
	\end{equation}
	where $\phi^{\alpha}$ is given in \ref{FRunfolding} and 
	\begin{equation}\label{theta}
	\Theta_{\alpha,\e} = \{ (x,1) : \ x < \min \{ (\phi^{\alpha,\e}_X)^{-1}(x^+(\alpha,\e)), (\phi^{\alpha,\e}_Y \circ \phi^{\alpha,\e}_+ 
	\circ \phi^{\alpha,\e}_X)^{-1} (x^-(\alpha,\e)) \} \}.
	\end{equation}
\end{proof}

\begin{proof}[Proof of \autoref{prop:phiaeZ}]
	\begin{enumerate}[(a)]
		\item 
		By \autoref{corol:toptype}, for $\alpha >\alpha _\mathcal{H}$, and $\e$ small enough, the system $Z^\alpha_\e$ has an unstable focus. 
		As $\G_Z>0$ the map $\phi_Z$ is attracting, therefore using \ref{phiex}, $\phi^{\alpha}_\e$ is attracting for $a \leq x \leq b < 0$. 
		Therefore, the Poincar\'{e}-Bendixson Theorem guarantees the existence of a stable orbit $\Gamma^\alpha_\e$ for $(\alpha,\e)$ on the right of the curve $\mathcal{H}.$
	\end{enumerate}
	
	\begin{enumerate}[(b1)]
		\item 
		By \autoref{prop:IICunfolding}, if $\alpha >0$, the map $\phi^{\alpha}$ has a hyperbolic fixed point $F(\alpha)$ given in \ref{Falpha}. 
		Then the result is a consequence of the Implicit Function Theorem.
		\item 
		When $\alpha<0$ the map $\phi^{\alpha}$ has no fixed points and, by continuity, there are no fixed points for 
		$\phi^{\alpha}_\e$ when $\e>0$ is sufficiently small.
	\end{enumerate}
	\begin{enumerate}[(c1)]
		\item
		Using the properties of the Melnikov function given in \autoref{prop:propiedadesM} we have:
		\begin{itemize}
			\item 
			$M(v^*,\delta)=0$. 
			\item 
			As $\G_Z>0$, for any $\delta$, there exists $V_0>v^*$ independent of $\delta$, and $M(v,\delta)<0$ for $v>V_0$.
			\item
			As in the invisible-invisible case $\frac{\partial^2}{\partial v \partial \delta} M(v^*,\delta_\mathcal{H})>0$ one has that 
			\begin{itemize}
				\item 
				If $\delta <\delta_\mathcal{H}$, $\frac{\partial M}{\partial v}(v^*, \delta)<0$, and consequently $ M(v^*, \delta)<0$ for $v>v^*$ close enough to $v^*$, 
				which implies that the critical point of the system is stable.
				\item
				If $\delta >\delta_\mathcal{H}$, $\frac{\partial M}{\partial v}(v^*, \delta)>0$, and consequently $ M(v^*, \delta)>0$ for $v>v^*$ close enough to $v^*$, 
				which implies that the critical point of the system is unstable.
			\end{itemize}
		\end{itemize}
		\item
		For $\delta>\delta_\mathcal{H}$ Bolzano Theorem assures that the Melnikov function has a  zero $v^*<v^{s}=v^{s}(\delta)<V_0$ satisfying 
		$\frac{\partial}{\partial v} M(v^{s}(\delta))\le 0$ corresponding to 
		the attracting periodic orbit $\Gamma^{\alpha,s}_\e$.
		Moreover, $\Gamma^{\alpha,s}_\e$ is locally unique if $\frac{\partial}{\partial v} M(v^{s},\delta_\mathcal{H})<0$.
		\item 
		If we assume that $M(v, \delta)$ is strictly concave, no more zeros of $M(v, \delta)$ exist for $\delta>\delta_\mathcal{H}$ and the periodic orbit  
		$\Gamma^{\alpha,s}_\e$ is unique.
		Analogously, for $\delta<\delta_\mathcal{H}$ the function $M(v, \delta)<0$ for $v^*<v $. 
	\end{enumerate}

\end{proof}

%%%%%
% Appendix 2 
%%%%%	

\subsection{Proof of \autoref{prop:canard}}\label{proofcanard}

During this section we restrict ourselves to any compact set containing the folds and the results will be valid for $\e$ small enough depending of this fixed compact.

First, as numerical simulations indicate that there is a maximal Canard trajectory when $\alpha =\mathcal{O}(\e)$, we set $\alpha = \delta \e$.
Therefore, system \ref{vsystem} becomes
\begin{equation} \label{eq:I}
\begin{aligned}
\dot{x} &= F^1(x,v;\delta \e,\e) \\
\e \dot{v} &= F^2(x,v;\delta \e,\e)
\end{aligned}
\end{equation}

As $\alpha = \delta \varepsilon$, the critical manifolds of this system
are the same as the ones for the regularization of the vector field $Z$.
Then, as we saw in \autoref{ssec:OVcritical} there are two critical manifolds
$\Lambda_0^{s,u}$ given by:
$$
\Lambda_0^s = \{ (x,v), \ v=m_0(x), \ x<0\}, \
\Lambda_0^u = \{ (x,v), \ v=m_0(x), \ x>0\},
$$
where 
\begin{equation}\label{m0}
m_0(x)= -\varphi ^{-1} \left( \frac{X^2+Y^2}{X^2-Y^2}(x,0) \right), 
\end{equation}
which are normally hyperbolic (attracting and repelling respectively) and we restrict them to $|x|>\kappa$ for a small but fixed $\kappa>0$.
Applying Fenichel theorem, we know the existence of two normally
hyperbolic invariant manifolds
$$
\Lambda_\varepsilon^{s} = \{ (x,v), \ v=m^s (x;\varepsilon), x<-\kappa\}, \
\Lambda _\varepsilon^{u} = \{ (x,v), \ v=m^u (x;\varepsilon), x>\kappa\}.
$$

The idea is now, to ``extend'' these manifolds until they reach $x=0$ and
to see if they can coincide for some value of $\delta$ giving rise to
the so-called maximal Canards.

We first look for the asymptotic expansion of the functions
$m^{s,u} (x;\varepsilon)$ defining these manifolds:
\begin{equation} \label{solexp}
m^{s,u}(x;\e) = m_0(x) + m_1(x) \e + m_2(x) \e^2 + \mathcal{O}(\e^3).
\end{equation}
Using that $m^{s,u} (x;\varepsilon)$ satisfy the equation of the orbits:
\begin{equation} \label{eq:orbitequation}
\e \frac{dv}{dx} = \frac{F^2(x,v;\delta \e, \e)}{F^1(x,v;\delta \e ,\e)}
\end{equation}
we can obtain analytical expressions for $m_i(x)$ for $i=0,1,\cdots$, with $m_0(x)$  given in \autoref{m0},
and it is easy to check that they behave as:

$$
m_0(x) 	= \mathcal{O}(1) , \ 
m_1(x) 	= \mathcal{O}\displaystyle{\left( \frac{1}{x} \right)}, \
m_2(x) 	= \mathcal{O}\displaystyle{\left( \frac{1}{x^3} \right)} 
$$

Therefore while $|x|< \sqrt{\e}$ the series \ref{solexp} is asymptotic.
Next proposition gives  the behavior  of the stable Fenichel manifold $m^s(x;\e)$ shown in these expansions.
Analogously, one can prove the same result for the unstable Fenichel manifold $m^u(x;\e)$ reversing time.

\begin{prop} \label{block1}
	Take any $0<\lambda<\frac{1}{2}$. Fix $x_0>\kappa$. 
	Then, there exists $M>0$ big enough, $\sigma=\sigma(M)>0$ small enough, and $\e_0=\e_0(M,\sigma)>0$ such that, 
	for $0<\e<\e_0$, any solution of system \ref{vsystem} which enters the set
	$$
	B = \left\{ (x,v) : \ -x_0 \leq x \leq -\e^\lambda,\ | v - \ m_0(x) | \leq \frac{M\e}{|x|} \right\}
	$$
	leaves it through the boundary $x=-\e^\lambda.$
\end{prop}

\begin{proof}
	The proof is based in the fact that the vector field $\tilde Z^\alpha_\e$ \ref{vsystem} points inwards through the boundaries
	$$
	B^\pm = \left\{ (x,v): \  \ -x_0 \leq x \leq -\e^\lambda, \ v =  m_0(x) \mp \frac{M\e}{x} \right\}.
	$$
	This is straightforward and can be seen in full details in \cite{tere}. 
	For instance, to see that the vector field $\tilde Z^\alpha_\e$ points inwards $B$ through $B^+$, we must see that
	\begin{equation} \label{boundaries}
	\langle \left( -m_0'(x) - \frac{M\e}{x^2},1 \right), \tilde Z^\alpha_\e \rangle <0, \quad -x_0 \leq x \leq -\e^\lambda
	\end{equation}
	and for this purpose it will be enough to see that, in fact, the negative sign of the coefficient of the $\e$ order of this expression determines the sign of it.
\end{proof}

Next step is to enlarge the domain of existence of these stable and unstable Fenichel manifolds. 
We take a value  $x^* = -\e ^\lambda$ for the stable case (or $x^* =\e ^\lambda$ for the unstable case),
with $0<\lambda < 1/2$. We know that
\begin{align} \label{eqq}
|m^{s,u}(x^*;\varepsilon)-m_0(x^*)| \le M \e ^{1-\lambda}.
\end{align}
It is clear that \ref{solexp} looses its asymptotic character when $|x| = \e^{\frac{1}{2}}$.
This suggests the change $x =\e^{\frac{1}{2}} r.$ 
Now the Taylor expansion  of the Fenichel manifold in this new variable justifies the change $v = \bv + \e^{\frac{1}{2}} s,$ 
where $\bar v = m_0(0)$ (see \autoref{IVslowmanifold}).

Finally, setting $\gamma = \e^\frac{1}{2}$, we study the continuation of the Fenichel manifolds for small values of $x$, 
performing the following change to system \ref{vsystem}.
\begin{equation} \label{change}
\begin{cases}
x &= \gamma r, \\
v &= \bv + \gamma s,
\end{cases}
\end{equation}

With this change and re-scaling time $t=\gamma \tau$, system \ref{eq:I} becomes
\begin{equation} \label{eq:II}
\begin{cases}
\dot{r} 	&= M_0 + \mathcal{O}(\gamma), \\
\dot{s}	&= M_1 + M_2 \delta + M_3 r^2 + M_4 rs + \mathcal{O}(\gamma),
\end{cases}
\end{equation}
where $\mathcal{O}(\gamma)$ are terms bounded by $K\gamma$, where $K>0$ is independent of $\gamma$, for $r^*<r\le 0$, 
where $r^* =\frac{x^*}{\sqrt{\e}}= -\gamma ^{2\lambda -1}$ and the constants are given by

\begin{align} \label{mis}
\begin{aligned}
M_0 &=\left( X^1 + Y^1 + \varphi(\bv) (X^1-Y^1) \right) (\0), \\
M_1 &= \bv \left( X^2_{y} + Y^2_{y} + \varphi(\bv) \left( X^2_{y}-Y^2_{y} \right) \right) (\0), \\
M_2 &= \left( \tilde X^2 + \tilde Y^2  + \varphi(\bar v) (\tilde X^2 - \tilde Y^2)  \right)(\0), \\
M_3 &=\frac{1}{2} \left( X^2_{xx} + Y^2_{xx} + \varphi(\bv) \left( X^2_{xx}-Y^2_{xx} \right) \right)(\0), \\
M_4	&= \varphi'\left(\bv\right) \left( X^2_{x} - Y^2_{x} \right)(\0).
\end{aligned}
\end{align}

Putting $\gamma=0$ in  system \ref{eq:II} we obtain \begin{equation} \label{eq:III}
\begin{cases}
\dot{r} 	&= M_0, \\
\dot{s}	&= M_1 + M_2 \delta + M_3 r^2 + M_4 rs,
\end{cases}
\end{equation}
which gives us the so called inner equation 
\begin{equation} \label{eq:inner}
\frac{ds}{dr} = N_1 + N_2 \delta + N_3 r^2 + N_4 rs,
\end{equation}
where
$\displaystyle{N_i = \frac{M_i}{M_0}}$ for $i=1,2,3,4$.

Observe that since $\displaystyle{M_0 = \frac{-(\det{Z})_x(\0)}{\left( X^2_x - Y^2_x \right)(\0)}>0}$, then
$
\displaystyle{N_4 = \varphi'\left(\bv\right)\frac{\left( X^2_x - Y^2_x \right)^2(\0)}{-(\det{Z})_x(\0)}>0}$.

From now, one must find a solution of system \ref{eq:II} which coincides with the Fenichel manifold at the point $x^*=\gamma r^*$.
Let us recall that $\e= \gamma ^2$ and that  $r^* = -\gamma ^{2\lambda-1}$ for the stable case and $r^* =\gamma ^{2\lambda-1}$ for the unstable case.

In the new variables, the Fenichel manifolds satisfy
\begin{equation}\label{asimp1}
s ^{s,u}(r^*;\gamma)= \frac{m^{s,u}(x^*)-\bar v}{\gamma}= m'_0(0)r^* +O(\gamma^{1-2\lambda}, \gamma^{4\lambda-1}),
\end{equation}
and this suggests to look for  solutions $s^\mp_i (r)$ of the inner equation \eqref{eq:inner} satisfying:
$$
|s^\mp_0(r)-m'_0(0)r| \ \mbox{bounded as } \ r\to \mp \infty.
$$
As $N_4>0$, these particular solutions are
$$
s^-_0(r) = e^{\frac{1}{2}N_4 r^2} \int_{-\infty}^r {e^{-\frac{1}{2} N_4 t^2} \left( N_1 +N_2 \delta + N_3 t^2 \right)dt}, 
$$
$$
s^+_0(r) = - e^{\frac{1}{2}N_4 r^2} \int_{r}^\infty {e^{-\frac{1}{2} N_4 t^2} \left( N_1 +N_2 \delta + N_3 t^2 \right)dt}.
$$
Integrating by parts and using the expression for $M_3$ and $M_4$ given in \ref{mis} and the expression for $m_0(x)$ given in \ref{m0}, we get:
\begin{equation} \label{s0--}
s^-_0(r) = m_0'(0)r +\left( N_1 +N_2 \delta + \frac{N_3}{N_4}  \right)
%e^{\frac{1}{2}N_4 r^2} 
\int_{-\infty}^r {e^{\frac{1}{2} N_4(r^2- t^2)} dt},
\end{equation}
\begin{equation} \label{s0++}
s^+_0(r) = m_0'(0)r - \left( N_1 +N_2 \delta + \frac{N_3}{N_4}  \right)
%e^{\frac{1}{2}N_4 r^2} 
\int_{r}^\infty {e^{\frac{1}{2} N_4 (r^2-t^2)} dt}.
\end{equation}

Using the L'H\^{o}pital Rule
one can easily see that $s^\pm_0(r)-m_0'(0)r$ tend to zero when $r \rightarrow \pm\infty$.
More concretely:
\begin{equation} \label{error2}
s^\mp_0(r)-m_0'(0)r = \mathcal{O} \left( \frac{1}{r} \right),  \ r \to \pm \infty.
\end{equation}

Analogously to what we did in the study of the stable Fenichel manifold in the region $-x_0 < x < -\e^{\lambda}$ 
(and the unstable one for  $ \e^{\lambda}< x < x_0$), we seek solutions of \ref{eq:II} in the form
\begin{equation} \label{sexpansion}
s^{u,s}(r,\gamma) = s_0^\pm(r) + \gamma s_1^\pm(r)+ \cdots
\end{equation}

Substituting this expression in \ref{eq:II}, we obtain that the successive $s_i^\pm$ satisfy linear equations with the same homogeneous part as the one 
satisfied by  $s^\pm_0(r)$, only differing on the non homogeneous term which depends recursively of $s^\pm_{i-1}$. 
Obviously, if we want to follow the Fenichel manifolds, we must select for $s_i^\pm$ the solution with no exponential term. 
So, as we did for $s_0^\pm$, the L'H\^{o}pital Rule shows that
$$
s^\pm_1(r) = \mathcal{O}( r^2).
$$
This suggests, like in \autoref{block1}, the definition of a new block $\mathcal{B}$ for the equation \ref{eq:II}. 
Also, to see that effectively the continuation of the Fenichel manifold is well approximated by $s_0^\mp$
for $r^* \le r \le 0$ 
we have the following proposition for the behavior of the stable part. 
Analogously, reversing time, one can prove the same result for the unstable one.

\begin{prop} \label{block2}
	Take any $\frac{1}{4}< \lambda < \frac{1}{2}$. 
	Then, there exists $r_0>0$, $K>0$ and $\gamma_0 = \gamma_0 (r_0,K)$, such that for $\gamma \le \gamma_0$, any solution of system \ref{eq:II} which enters the set
	$$
	\mathcal{B} = \{ (r,s): \ r^* \le r \le 0, |s - s_0^-(r)| \le K \gamma M(r) \}
	$$
	where $r^*=-\gamma^{2 \lambda - 1}$ and the function $M(r)$ is defined by
	$$
	M(r) = \begin{cases}
	r^2,& -\infty \le r \le -r_0 <0 \\
	r_0^2,& -r_0 \le r <0
	\end{cases} .
	$$
	leaves it through the boundary $r=0$.
\end{prop}
\begin{proof}
	We proceed as in \autoref{block1} to see that the vector field \ref{eq:II} points inwards through the boundaries
	$$
	\mathcal{B}^\pm = \{ (r,s): \  \ r^* \leq r \leq 0, \ s =  s_0^-(r) \pm K \gamma M(r) \}.
	$$
	This is straightforward and details can be found in \cite{tere}. Only to remark that $r_0$ can be large, 
	but now the vector field \ref{eq:II} is regular in $\gamma$, so the behavior of the Fenichel manifolds from $r_0$ to the origin is guaranteed to be 
	$\mathcal{O}(\gamma)$.
\end{proof}
To see that the Fenichel manifold enters the block $
\mathcal{B}$ we use estimates \eqref{asimp1} and \eqref{error2}, and taking $\frac{1}{4} <\lambda<\frac{1}{3}$ we obtain the result.

Using $s_0^\mp$, we can continue the Fenichel manifolds until we reach $r=0$:
$$
v^{s,u}(x,\e)= \bar v + \gamma s^{u,s}(r;\gamma) = \bar v + \gamma s_0^{\pm}(r) + \mathcal{O}(\gamma ^2), \ r=\frac{x}{\gamma}, \ \e = \gamma ^2,
$$
and the intersection with $x=0$ is given by:
$$
v^{s,u}(0,\e)= \bar v + \gamma s_0^{\pm}(0) + \mathcal{O}(\gamma ^2).
$$
And we obtain:
\begin{eqnarray*}
	v^s(0,\e)-v^u(0,\e) &=& 
	\gamma (s ^{-}_0 (0)-s ^{+}_0(0)) + \mathcal{O}(\gamma ^2)
	= \gamma(N_1 +N_2 \delta + \frac{N_3}{N_4})\int_{-\infty}^{\infty}  e^{-\frac{1}{2} N_4 t^2} + \mathcal{O}(\gamma ^2)\\
	&=&
	\gamma(N_1 +N_2 \delta + \frac{N_3}{N_4})  \sqrt{\frac{2\pi}{N_4}} + \mathcal{O}(\gamma ^2)
\end{eqnarray*}

Calling  $\xi (\delta, \gamma)= \frac{1}{\gamma}(v^s-v^u)$ one has that $\xi (\delta_\mathcal{C}, 0)=0$, where:

\begin{align*}
\delta_\mathcal{C} = - \frac{N_3 + N_1 N_4}{N_2 N_4} =- \frac{M_0 M_3 + M_1 M_4}{M_2 M_4},
\end{align*}
where $M_i$, $i=0,1,2,3,4$ are given in \ref{mis}.
Now, using that 
$\displaystyle{\frac{\partial \xi}{\partial \delta} (\delta_\mathcal{C},0) = N_2 \sqrt{\frac{2\pi}{N_4}} \neq 0}$ 
(see \autoref{N2nozero})
one can apply the Implicit Function Theorem, obtaining  a curve
\begin{equation} \label{dce}
\delta_\mathcal{C}(\gamma) = \delta_\mathcal{C} + \mathcal{O}(\gamma),
\end{equation}
such that over this curve the trajectories $s^s(r;\gamma)$ and $s^u(r;\gamma)$ coincide.
Moreover, for $\delta = \delta_ \mathcal{C}$, one has, for $r= \mathcal{O}(1)$ (and $x=\mathcal{O}(\sqrt{\e})$):
$$
s_0^\mp (r)= m_0'(0)r, \quad v^{s,u}(x)= m_0(0)+m_0'(0)x + \mathcal{O}(\e).
$$

Coming back to the original variables $(x,v)$  and recalling that $\alpha = \delta \e$, we have a curve \begin{equation} \label{ac}
\alpha_\mathcal{C}(\e) = \delta_\mathcal{C} \e + \mathcal{O} \left( \e^\frac{3}{2} \right),
\end{equation}
where  $\Lambda^{\alpha,u}_\e$ and $\Lambda^{\alpha,s}_\e$ coincide.

In conclusion, there exists a curve $\mathcal{C}$ given by
\begin{equation} \label{ccurve}
\mathcal{C} = \left\{ (\alpha,\e) : \, \alpha= \alpha_\mathcal{C}(\e) = - \frac{M_0 M_3 + M_1 M_4}{M_2 M_4} \e + \mathcal{O}\left( \e^{\frac{3}{2}} \right) \right\},
\end{equation}
where $\Lambda^{\alpha,u}_\e = \Lambda^{\alpha,s}_\e$ giving rise to a Canard solution. 
Moreover, as $s_0(0)=0$, at this value one has that:
$$
v^{s,u}(0,\e)= \bar v +  \mathcal{O}(\e).
$$

Finally, observe that, for any $\delta$ we obtain:
\begin{equation}\label{Distcanard}
v^s-v^u=\gamma C (\delta-\delta _{\mathcal{C}})+ \mathcal{O}(\gamma ^2).
\end{equation}
where $C=\sqrt{\frac{2\pi}{N_4}}N_2 $.

\begin{rem}\label{N2nozero}
	Observe that the denominator of $\delta_\mathcal{C}$ given by $M_2 M_4 = 2 \varphi'(\bar v) (\tilde Y^2 X^2_x - \tilde X^2 Y^2_x)(\0) \neq 0$ due to the transversality condition (see \ref{foldsbe}) imposed to the unfolding $Z^\alpha$. 
\end{rem}

\begin{rem}
	In the case $(\det{Z})_x(\0)>0$, as $N_4<0$, all the solutions of the inner equation have the correct asymptotic behavior as $r\to \pm \infty$, nevertheless when
	$\delta= \delta_\mathcal{C}$ the solutions $s_0^\mp(r)= m_0'(0)r$ can be seen as the ``Canard'' solution.
\end{rem}

%%%%%
% Appendix 3 
%%%%%
\subsection{Proof of \autoref{prop:linearcanard}}\label{sec:linearcanard}

In this section we will prove that, for the case of a linear regularization $\varphi$ of the visible-invisible fold-fold singularity with 
$(\det{Z})_x<0$ one can transform the system to a general slow-fast system studied in \cite{krupa,KrupaS01,Kuehn10} and apply their
results for the existence of a maximal canard. 
We will recover all the values of $\delta_\mathcal{H}$, $\delta_\mathcal{C}$ computed in this paper, as well as the first Lyapunov coefficient at the Hopf bifurcation, 
and their relations.

In \cite{krupa,KrupaS01}, the authors proved the existence of a maximal Canard for the following general system
\begin{equation} \label{eq:KSgeneral}
\begin{cases}
\dot{x} = R^1(x,y,\e,\lambda) = -y h_1(x,y,\e,\lambda) + x^2 h_2(x,y,\e,\lambda) + \e h_3(x,y,\e,\lambda), \\
\dot{y} = \e R^2(x,y,\e,\lambda) = \e \left( x h_4(x,y,\e,\lambda) -\lambda h_5(x,y,\e,\lambda) +y h_6(x,y,\e,\lambda) \right),
\end{cases}
\end{equation}
\noindent with $h_3(x,y,\e,\lambda) = \mathcal{O}(x,y,\e,\lambda)$ and $h_j(x,y,\e,\lambda) = 1+ \mathcal{O}(x,y,\e,\lambda)$ for $j=1,2,4,5$. Moreover, let

\begin{equation} \label{KSvalues}
\begin{array}{lll}
\gamma_1 = \displaystyle{ \frac{\partial}{\partial x}} h_3(0,0,0,0), & \quad
\gamma_2 = \displaystyle{ \frac{\partial}{\partial x}} h_1(0,0,0,0), & \quad
\gamma_3 = \displaystyle{ \frac{\partial}{\partial x}} h_2(0,0,0,0), \\ & & \\
\gamma_4=  \displaystyle{ \frac{\partial}{\partial x}} h_4(0,0,0,0), & \quad
\gamma_5 = h_6(0,0,0).
\end{array}
\end{equation}

Considering the system \ref{eq:KSgeneral}, they obtain:

\begin{theo}[Krupa-Szmolyan Theorem] \label{th:KrupaSzmolyan}
	For $0< \e < \e_0$, $| \lambda| < \lambda_0$, $\e_0$, $\lambda_0$ sufficiently small and under the previous assumptions, there exists a unique critical point $P(\lambda,\e)$ of system \ref{eq:KSgeneral} in a neighborhood of the origin. The equilibrium point undergoes to a Hopf bifurcation at $\lambda_H$ with
	\begin{equation} \label{KSHopf}
	\lambda_H = - \frac{\gamma_1 + \gamma_5}{2} \e + \mathcal{O}{(\e^2)}.
	\end{equation}
	
	The slow manifolds $C_{\e,l}$ and $C_{\e,r}$ intersect/coincide at a maximal Canard at $\lambda_c$ for
	\begin{equation} \label{KSCanard}
	\lambda_c = - \left( \frac{\gamma_1 + \gamma_5}{2} + \frac{A}{8} \right) \e + \mathcal{O}{(\e^{3/2})},
	\end{equation}
	where $$A=-\gamma_2 +3\gamma_3-2\gamma_4-2\gamma_5.$$
	
	The equilibrium $P(\lambda,\e)$ is stable for $\lambda<\lambda_H$ and unstable for $\lambda>\lambda_H$. 
	The Hopf bifurcation is non degenerated for $A \neq 0$, supercritical for $A<0$ and subcritical for $A>0$.
\end{theo}

Putting $\lambda = \kappa \e$ in \eqref{eq:KSgeneral}, the next lemma is straightforward:
%and will be omitted, valid when .
\begin{lemma} \label{lem:KScritical}
	In a suitable neighborhood of the origin, one has that:
	\begin{itemize}
		\item the critical manifold $C_0$ of system \ref{eq:KSgeneral} can be written as the graph of \begin{equation} \label{eq:KScriticalmanifold}
		f(x)=x^2 + (\gamma_3-\gamma_2) x^3 + \mathcal{O}(x^4);
		\end{equation}
		\item  considering the coordinate change
		\begin{equation}
		\begin{cases}
		u 		&=x, \\
		\e w 	&= y - f(x),
		\end{cases}
		\end{equation}
		system \ref{eq:KSgeneral}, becomes
		\begin{equation} \label{KS2}
		\begin{cases}
		\dot{u}			 &= \gamma_1 u - w (1+\gamma_2 u ) + \tilde{R}^1(u,w,\e), \\	
		\e \dot{w}		 &= u - \kappa \e + (\gamma_5 + \gamma_4 - 2\gamma_1 )u^2 +\gamma_5 \e w +2u w \\
		&+(3\gamma_3 -\gamma_2) u^2 w + \tilde{R}^2(u,w,\e),
		\end{cases}
		\end{equation}
		with 
		\begin{eqnarray*}
			\tilde{R}^1(u,w,\e) 	&=& u^2K_4(u) + \e^2K_2(u,\e) + u^2 w K_5(u) + \e w K_3(u,\e) + \e w^2 K_1(u,\e w,\e), \\
			\tilde{R}^2(u,w,\e) 	&=& u^3 K_{12} (u) + u^3 w K_{13}(u) + f'(u) \left( \e^2 K_2(u,\e) + \e w K_3(u,\e) \right. \\
			&+& \left. \e w^2 K_1(u,\e w,\e) \right) + \e^2 K_7(u,\e)+ \e w K_{11}(u) + \e^2 w K_8(u,\e) \\
			&+& (\e w)^2 K_6(u,\e w,\e).
		\end{eqnarray*}
		Where the functions $K_i$ are smooth bounded functions.
	\end{itemize}
\end{lemma}

\begin{rem}
	Observe that system \ref{KS2} is a slow-fast system and its critical manifold is given by 
	$$
	\tilde{C}_0 = \{ (u,w) : u=0 \} \cup \left\{ (u,w) : w = - \frac{1}{2} + \left( - \frac{\gamma_5 + \gamma_4 -2\gamma_1}{2}+ \frac{3\gamma_3 -\gamma_2}{4} \right) u  
	+ \mathcal{O}(u^2) \right\}.
	$$
\end{rem}

Now suppose that $\varphi(v)=v$ for $v \in (-1,1)$ and let $Z^\alpha_\e$ be the $\varphi-$regularization of $Z^\alpha$ and that  $\alpha = \delta \e.$ 
The regularized system $Z^\alpha_\e$ in coordinates $(x,v)$ with $y=\e v$ has the form
\begin{equation} \label{eq:NS1}
\begin{cases}
\dot{x} 	&= F^1(x,v;\delta \e,\e),\\
\e \dot{v} 	&= F^2(x,v;\delta \e,\e).
\end{cases}
\end{equation}
where:
$$
F^i(x,v;\alpha,\e)= (X^{\alpha i}+Y^{\alpha i})(x,\e v)+  v  (X^{\alpha i}-Y^{\alpha i})(x,\e v), \, i=1,2., \ |v|\le 1
$$
Recall that for $\e=0$ system \ref{eq:NS1} has a critical point at $P(0,0)=(0,v^*)$ with 
$$
v^* = - \frac{X^1 + Y^1}{X^1 - Y^1}(\0).
$$
The next proposition shows that there exists a coordinate change, such that system  \ref{eq:NS1} is equivalent to system \ref{eq:KSgeneral} for $\e \neq 0$.
\begin{prop} \label{NSchange}
	
	There exists a change of coordinates:
	$$
	(x,v) \to (u,w)
	$$
	such that transforms system \ref{eq:NS1} into:
	\begin{equation}  \label{NS4}
	\begin{cases}
	\dot{u} 		&= \displaystyle{ \frac{A_1}{\sqrt{-A_3}} u - w \left( 1+ \frac{2A_2 \sqrt{-A_3}}{A_6}u \right) + \tilde{S}^1(u,w,\e) } \vspace{0.2cm}, \\
	\e \dot{w} 	&= \displaystyle{ u + \frac{A_6}{2A_3 \sqrt{-A_3}} \left(A_4 + A_6 \delta \right) \e +\frac{1}{\sqrt{-A_3}}(A_7 + A_8 \delta) \e w +2 u w} \\  \vspace{0.2cm} \\
	& \displaystyle{ -\frac{2 A_9}{A_6 \sqrt{-A_3}} u^2+  \frac{4 A_{10} \sqrt{-A_3}}{A_6^2}u^2 w  + \tilde{S}^2(u,w,\e)},
	\end{cases}
	\end{equation}
	where $A_i, \, i=1,\cdots,10$ are given in expressions \ref{A1exp} to \ref{A10exp} and 
	\begin{eqnarray*}
		\tilde{S}^1(u,w,\e) &=&
		u^2 T_4(u)
		+ \e T_2(u,\e)
		+ u^2 w T_5(u)
		+\e w T_3(u,\e)
		+ \e w^2 T_1(u,w,\e), \\
		\tilde{S}^2(u,w,\e) &=&
		u^3 T_9(u)
		+ \e u T_{10}(u)
		+ \e^2 T_7(u,\e)
		+ u^3 w T_{11}(u)
		+\e u w T_{12}(u) \\
		&+& \e w^2 T_8(w,\e)
		+\e w^2 T_6(u,w,\e) .
	\end{eqnarray*}
\end{prop}
Where the functions $T_i$ are smooth bounded functions.

\begin{proof} 
	We begin with a coordinate change which moves the critical point $P(0,0)$ to the origin, given by
	\begin{equation}
	\begin{cases}
	\bar{u} &= x, \\
	\bar{v} &=  (X^1 + Y^1)(\0) + v (X^1 - Y^1)(\0).
	\end{cases}
	\end{equation}
	obtaining the system 
	\begin{equation} \label{NS42}
	\begin{cases}
	\dot{\bu} &=  A_1 \bar{u} + \bar{v}(1+A_2 \bar{u} )
	+ \bu^2 T_4(\bar{u})
	+ \e T_2(\bar{u},\e)
	+ \bu^2 \bv T_5(\bar{u}) \\
	&+\e \bar{v} T_3(\bar{u},\e)
	+ \e \bar{v}^2 T_1(\bar{u},\bar{v},\e), \\
	\e \dot{\bv} 	&= A_3 \bu + (A_4 + A_5 \delta) \e + A_6 \bu \bv + (A_7 + A_8 \delta) \e \bv + A_9 \bu^2 \\
	&+ A_{10} \bu^2 \bv
	%restos
	+ \bu^3 T_9(\bu)
	+ \e \bu T_{10}(\bu)
	+ \e^2 T_7(\bu,\e)
	+ \bu^3 \bv T_{11}(\bu) \\
	&+\e \bu \bv T_{12}(\bu)
	+ \e \bv^2 T_8(\bu,\e)
	+\e \bv^2 T_6(\bu,\bv,\e),
	\end{cases}
	\end{equation}
	where the constants $A_i$ $i=1,\dots, 10$ are given by
	\begin{align}
	A_1 		&=   (X^1_x + Y^1_x)(\0)+ v^* (X^1_x - Y^1_x)(\0) ), \label{A1exp} \vspace{0.2cm} \\
	A_2 		&= \frac{(X^1_x + Y^1_x)(\0)}{(X^1-Y^1)(\0)}, \label{A2exp}  \vspace{0.2cm} \\
	A_3 		&= 2 (\det{Z})_x(\0), \label{A3exp}  \vspace{0.2cm} \\
	A_4 		&= 2(\det{Z})_y(\0) v^*, \label{A4exp}  \vspace{0.2cm} \\
	A_5 		&= 2(\tilde Y^2 X^1 - \tilde X^2 Y^1)(\0), \label{A5exp}  \vspace{0.2cm} \\
	A_6 		&= (X^2_x-Y^2_x)(\0),  \label{A6exp}  \vspace{0.2cm} \\
	A_7 		&=  \frac{2(\det{Z})_y(\0)}{(X^1-Y^1)(\0)} +(X^2_y-Y^2_y)(\0)v^*, \label{A7exp}  \vspace{0.2cm} \\
	A_8 		&= (\tilde X^2 -\tilde Y^2)(\0), \label{A8exp}  \vspace{0.2cm} \\
	A_9 		&= \frac{1}{2} (X^1-Y^1)(\0) \left( (X^2_{xx}+Y^2_{xx})(\0) + v^* (X^2_{xx}-Y^2_{xx})(\0) \right), \label{A9exp} \vspace{0.2cm} \\
	A_{10}	&= \frac{1}{2}(X^2_{xx}-Y^2_{xx})(\0). \label{A10exp}
	\end{align}
	
	Now, as $A_3<0$, we can perform the scaling 
	\begin{equation}
	\begin{cases}
	\bu 		&= \displaystyle{\frac{2\sqrt{-A_3}}{A_6}} u,\\
	\bv 		&= \displaystyle{\frac{2 A_3}{A_6} } w, \\
	\tau 	&= \left(  \sqrt{-A_3} \right) t,
	\end{cases}
	\end{equation}
	system \ref{NS42} becomes the desired one, given in \ref{NS4}.
\end{proof}

Using  \autoref{lem:KScritical} and \autoref{NSchange} we can apply \autoref{th:KrupaSzmolyan} to system \ref{NS4}, 
which correspond to the regularized system $Z^\alpha_\e$. Now we can finally prove \autoref{prop:linearcanard}

\begin{proof}{Proof of \autoref{prop:linearcanard}}
	
	By Lemma \ref{lem:KScritical} and Proposition \ref{NSchange} we have that system \ref{eq:KSgeneral} and the regularized System \eqref{NS42} can be identified with
	the following relations
	\begin{align}
	\gamma_1 						&= \frac{A_1}{\sqrt{-A_3}}, \label{g1} \\
	\gamma_2 						&= \frac{2A_2 \sqrt{-A_3}}{A_6}, \\
	\gamma_5 						&= \frac{1}{\sqrt{-A_3}}(A_7 + A_8 \delta), \\
	\gamma_4 						&=-\frac{2 A_9}{A_6 \sqrt{-A_3}} -\gamma_5 + 2 \gamma_1,  \\
	3 \gamma_3 						&= \frac{4A_{10}\sqrt{-A_3}}{A_6^2} +\gamma_2, \label{g3}\\
	\kappa 							&= -\frac{A_6}{2A_3\sqrt{-A_3}}\left( A_4 + A_5 \delta \right). \label{kappa}
	\end{align}
	
	Recall that in our case, the parameter $\lambda$ from \autoref{th:KrupaSzmolyan} is $\lambda=\kappa \e$, thus the critical point undergoes to a Hopf bifurcation at $\lambda_{H} =\kappa_{H} \e$.
	\begin{eqnarray*}
		\lambda_{H}	= \kappa_H \e	&=& -\frac{ \gamma_1 + \gamma_5}{2} \e +\mathcal{O}(\e^2),
	\end{eqnarray*}
	Then the value of $\kappa$ in which that Hopf bifurcation occurs is
	\begin{equation} \label{kappaH}
	\kappa_{H} = -\frac{ \gamma_1 + \gamma_5}{2} + \mathcal{O}(\e).
	\end{equation}
	
	Moreover, combining expressions \ref{kappa} and \ref{kappaH}, it follows that
	
	\begin{align} \label{dH}
	-\frac{1}{2\sqrt{-A_3}}\left( \frac{A_4 A_6}{A_3} +\frac{A_5 A_6}{A_3} \delta_{H} \right) &= -\frac{1}{2\sqrt{-A_3}} \left( A_1 + A_7 + A_8 \delta_{H}  \right) + \mathcal{O}(\e).
	\end{align}
	
	\noindent Isolating $\delta_{H}$ in \autoref{dH} we have
	\begin{equation} \label{dH1}
	\delta_H = \frac{1}{\displaystyle{\frac{A_5 A_6}{A_3}} - A_8} \left( A_1 + A_7 -\frac{A_4 A_6}{A_3} \right) + \mathcal{O}(\e).
	\end{equation}
	\noindent where a straightforward computation gives us
	\begin{align}
	&\begin{aligned} \label{dHnum}
	A_1 + A_7 -\frac{A_4 A_6}{A_3} &=M(Z),
	\end{aligned} \\
	&\begin{aligned} \label{dHdem}
	\frac{A_5 A_6}{A_3} - A_8 &= -  N(Z,\tilde{Z}),
	\end{aligned}
	\end{align}
	where $M(Z)$ and $N(Z,\tilde{Z})$ are the constants computed in \autoref{corol:regnumbers}. From \ref{dHnum} and \ref{dHdem}, it follows that
	\begin{equation}
	\delta_{H} = - \frac{M(Z)}{N(Z,\tilde{Z})} + \mathcal{O}(\e),
	\end{equation}
	which coincides with the Hopf bifurcation value computed in \autoref{prop:toptype} setting $\alpha = \delta \e.$

	From now on we are going to compute the Canard value $\delta_c$. From \autoref{th:KrupaSzmolyan} the Canard happens for
	$$
	\kappa_\mathcal{C} = - \left(  \frac{\gamma_1 + \gamma_5}{2} + \frac{A}{8} \right) +\mathcal{O}(\sqrt{\e}).
	$$
	Using equations \ref{g1} to \ref{g3}, we have the following expression for $A$
	\begin{align*}
	A 	= -\gamma_2 +3\gamma_3-2\gamma_4-2\gamma_5 = \frac{4}{\sqrt{-A_3}} \left( -\frac{A_{10} A_3}{A_6^2} + \frac{A_9}{A_6} - A_1 \right).
	\end{align*}
	
	Using the same argument as above, \autoref{th:KrupaSzmolyan} implies that
	\begin{align*}
	-\frac{1}{2\sqrt{-A_3}}\left( \frac{A_4 A_6}{A_3} +\frac{A_5 A_6}{A_3} \delta_{C} \right) &= -\frac{1}{2\sqrt{-A_3}} \left( A_1 + A_7 + A_8 \delta_{C} -\frac{A_{10} A_3}{A_6^2} + \frac{A_9}{A_6} - A_1  \right) \\
	&+ \mathcal{O}(\sqrt{\e}).
	\end{align*}
	Isolating $\delta_c$ and using the previous computations, we obtain
	\begin{equation} \label{dC}
	\delta_c = - \frac{M(Z)}{N(Z,\tilde{Z})} + \bar{A} + \mathcal{O}(\sqrt{\e}),
	\end{equation}
	
	with $$\bar{A}=-\frac{1}{N(Z,\tilde{Z})} \left(-\frac{A_{10} A_3}{A_6^2} + \frac{A_9}{A_6} - A_1 \right).$$
	Observe that, in this particular case, $-N(Z,\tilde{Z})>0$ then $\sgn{A} =\sgn{\bar{A}}$. Therefore, $\delta_c > \delta_{H}$, if $A>0$ and $\delta_c < \delta_{H}$, if $A<0$.
	
\end{proof}

Since $\alpha = \delta \e$, this proposition gives us the existence of a curve $\mathcal{C}$ given by
\begin{equation}
\mathcal{C} = \left\{ (\alpha,\e) : \alpha(\e) = \delta_C \e + \mathcal{O}\left( \e^\frac{3}{2} \right) \right\}.
\end{equation}

Observe that the $\delta_H$ obtained in \autoref{prop:linearcanard} coincide with the value we had obtained in general in \autoref{prop:toptype}, 
if we set $\varphi(v)=v$. 
Moreover, a straightforward computation shows us that the $\delta_C$ given by this proposition also coincide with the value obtained in \autoref{prop:canard} 
for the Canard trajectory when $\varphi(v)=v$. 

%%%%%
% References
%%%%%
	\clearpage
	\printbibliography[title=References]
	
\end{document}